\documentclass[11pt,reqno]{amsart}
\usepackage{amsthm,amssymb,amsmath}
\usepackage{textcomp}

\usepackage{enumitem}
\usepackage[hidelinks, pdfstartview=FitV, pagebackref=false]{hyperref}
\usepackage{ulem}
\usepackage{esint}

\usepackage[T1]{fontenc}
\usepackage{mlmodern}

\usepackage{fix-cm}

\newtheorem{theorem}{Theorem}[section]
\newtheorem{proposition}{Proposition}[section]
\newtheorem*{theorem*}{Theorem}
\newtheorem{lemma}{Lemma}[section]
\newtheorem{corollary}{Corollary}[section]

\theoremstyle{definition}

\newtheorem{definition}{\sc Definition}[section]
\newtheorem*{definition*}{\sc Definition}

\newtheorem{example}{\bf Example}[section]
\newtheorem{examples}{\bf Examples}[section]
\newtheorem{remark}{\bf Remark}[section]
\newtheorem*{remark*}{\bf Remark}
\newtheorem*{remarks}{\bf Remarks}

\newtheorem*{example*}{\bf Example}

\newcommand{\loc}{{\rm loc}}

\newcommand{\Real}{{\rm Re}\,}
\newcommand{\Imag}{{\rm Im}\,}

\newcommand{\sprt}{{\rm supp\,}}

\numberwithin{equation}{section}

\include{diagxy}

\begin{document}

\title[SDEs with singular drift]{SDEs with critical general distributional drifts: sharp solvability and blow-ups}

\begin{abstract}
We develop weak solution theory of SDEs having general  distributional drifts that can introduce local blow-up effects. Our assumption on the drift is closely related to the variational theory of the Kolmogorov backward equation in the standard Hilbert triple $W^{1,2} \hookrightarrow L^2 \hookrightarrow W^{-1,2}$. Under suitable smallness assumption on the drift, one has well-posedness of the Cauchy problem in this Hilbert triple.  In turn, by a result of Mazya and Verbitsky, boundedness of the drift term from $W^{1,2}$ to $W^{-1,2}$ is equivalent to a decomposition of the drift in the sum of a form-bounded drift (including, for example, Morrey or Chang--Wilson--Wolff drifts) and a divergence-free distributional drift in the ${\rm BMO}^{-1}$ space of Koch and Tataru; such decomposition is our class of drifts.

We apply our results to finite particle systems with strong attracting interactions immersed in a stationary turbulent flow. This includes particle systems of Keller-Segel type. Crucially, in dimensions $d \geq 3$, we cover almost the entire admissible range of attraction strengths, reaching nearly to the blow-up threshold.
As a further application of our results for SDEs and of the theory of Bessel processes, we obtain an improved upper bound on the constant in the many-particle Hardy inequality of Hoffmann-Ostenhof, Hoffmann-Ostenhof, Laptev, and Tidblom.
\end{abstract}

\dedicatory{To the memory of Yu.\,A.\,Sem\"{e}nov}

\author{D.\,Kinzebulatov}

\address{Universit\'{e} Laval, D\'{e}partement de math\'{e}matiques et de statistique, Qu\'{e}bec, QC, Canada}

\email{damir.kinzebulatov@mat.ulaval.ca}

\author{R.\,Vafadar}

\address{Universit\'{e} Laval, D\'{e}partement de math\'{e}matiques et de statistique, Qu\'{e}bec, QC, Canada}

\curraddr{McMaster University, Department of Mathematics and Statistics, Hamilton, ON, Canada}

\email{vafadar@mcmaster.ca}

\thanks{The research of D.K. is supported by  NSERC grant (RGPIN-2024-04236)}



\fontsize{10.3pt}{4.2mm}\selectfont

\maketitle

\vspace*{-2mm}

\setcounter{tocdepth}{1}
\tableofcontents

\vspace*{-9mm}

\section{Introduction}

\label{intro_sect}

The subject of this paper is the stochastic differential equation (SDE)
\begin{equation}
\label{sde1}
X_t=x-\int_0^t c(X_s)ds + \sqrt{2}\int_0^t \sigma (X_s)dB_s, \quad x \in \mathbb R^d,
\end{equation}
with critical general distributional drift $c$ and discontinuous diffusion coefficients $\sigma$ (see Section \ref{gen_sect} and Section \ref{diff_coeff_sect} for our precise setting). Here $\{B_t\}_{t \geq 0}$ denotes a $d$-dimensional Brownian motion. 
SDEs with singular drifts arise in various physical models, for instance, the passive-tracer model, where a time-dependent drift 
$c$ is the velocity field obtained from the Navier–Stokes equations \cite{MK}, or finite-particle approximations of the Keller–Segel model of chemotaxis \cite{CP,FJ}. In the latter case one observes local blow-ups, i.e.\,when even the weak existence for SDE \eqref{sde1} fails once one replaces $c$ by $(1+\varepsilon)c$, $\varepsilon>0$, {so a finite-time collision prevents continuation of the original SDE}. Both of these models are within the scope of the present work.

We postpone a survey of the recent literature on SDEs with singular drifts  until after we have stated and discussed our assumptions on $c$.

Our main focus in this paper is on general drifts, i.e.\,not satisfying any special structural conditions such as control of the sign of ${\rm div\,}c$. So, on the one hand, {our main working conditions \eqref{c_cond}-\eqref{delta_cond} on $c$ (see Section \ref{result_one})} include the drifts 
with sufficiently small Morrey norm
\begin{equation}
\tag{$M_{2+\varepsilon}$}
\label{M_2}
\|c\|_{M_{2+\varepsilon}}=\sup_{r>0, x \in \mathbb R^d} r\biggl(\frac{1}{|B_r|}\int_{B_r(x)}|c|^{2+\varepsilon}dx \biggr)^{\frac{1}{2+\varepsilon}},
\end{equation}
or, more generally, the drifts
in the Chang-Wilson-Wolff class: for some $\alpha>0$,
\begin{equation}
\label{cww0}
\sup_{r>0, x \in \mathbb R^d} \frac{1}{|B_r|}\int_{B_r(x)} |c|^2\, r^2 \bigl(1+(\log^+|c|^2\,r^2)^{1+\alpha} \big)\,dy \text{ is sufficiently small}
\end{equation}
(actually, the definition is more general, see \eqref{cww}).
This, for example, provides us with flexible means to construct interaction kernels in particle systems.
Crucially, we almost reach the blow-up threshold for 
$c$, i.e.\,multiplying $c$ by $1+\varepsilon$ takes us out of the weak existence regime.  On the other hand, the same conditions \eqref{c_cond}-\eqref{delta_cond}  include  drifts arising in some of the physical models mentioned above, notably of the form
\begin{equation}
\tag{$\mathbf{BMO}^{-1}$}
c=\nabla C, \quad \text{$C=-C^{\top}$ is an anti-symmetric matrix field with entries in ${\rm BMO}(\mathbb R^d)$},
\end{equation}
where $\nabla$ is the {column-divergence} operator (see notations in Section \ref{notations_sect}).
The class $\mathbf{BMO}^{-1}$ consists of divergence-free vector fields, and was identified by Koch-Tataru \cite{KT} as a large class of initial conditions for which one can prove the existence and uniqueness of mild solution to 3D Navier-Stokes equations, and which provides natural scale and translation invariant version of $L^2$ boundedness of this solution.  This class contains divergence-free drifts with entries in Besov space of distributions: $${\rm div\,}c=0 \quad \text{ and }\quad  c^i \in {\dot{B}^{-1+d/p}_{p,\infty}}$$ for some ${d<p<\infty}$, or Borel measurable divergence-free drifts with entries in the largest scaling-invariant Morrey class $M_1$:$${\rm div\,}c=0 \quad \text{ and } \quad  \langle |c|\mathbf{1}_{B_r(x)}\rangle \leq C r^{d-1}$$
for constant $C$ independent of $r$ or $x \in \mathbb R^d$; here and below, $\langle \rangle$ denotes spatial integration. Thus, although we refer in this paper to $c \in \mathbf{BMO}^{-1}$ as distributional drifts, the class $\mathbf{BMO}^{-1}$ also contains quite singular Borel measurable vector fields such as the following one of Biot-Savart type:
\begin{equation}
\label{q_}
c(x)=\biggl(\frac{x_2}{x_1^2+x_2^2}, \frac{-x_1}{x_1^2+x_2^2},0,\dots,0 \biggr), \quad x \in \mathbb R^d,
\end{equation}
that, clearly, does not belong to $[L^2_{\loc}]^d$ and thus is not in \eqref{cww0}.

{Although we will mainly focus on the singularity of drift $c$, taking $\sigma = I_d$ in most
parts of the paper, as we will see below (Section \ref{diff_coeff_sect}), our proof techniques can tackle some non-constant diffusion coefficients which can have critical discontinuities.}

\begin{example}[Brownian particles in a turbulent flow]
\label{ex1_multi} We test our results for general SDE \eqref{sde1} against some interacting particle systems.
That is, we aim at describing the dynamics of $N$ Brownian particles in $\mathbb R^d$, $d \geq 3$, 
subject to strong pair-wise attraction, while being advected by a turbulent flow whose divergence-free distributional velocity field belongs to
$\mathbf{BMO}^{-1}(\mathbb R^d)$. To this end, we
work in $\mathbb R^{dN}$ where $N$ is large, and consider SDE
\begin{equation}
\label{sde2_part}
X_t=x_0-\int_0^t c(X_s)ds + \sqrt{2}B_t, \quad x_0=(x_0^1,\dots,x_0^N) \in \mathbb R^{dN}, 
\end{equation}
where
$
X_t=(X_t^1,\dots,X_t^N)$,  $X_t^i$ is the position of the $i$-th particle at time $t$, $B_t=(B_t^1,\dots,B_t^N)$,
$\{B_t^i\}_{t \geq 0}$, $(i=1,\dots,N)$ are independent Brownian motions in $\mathbb R^d$.
We further take
$$
c=b+q
$$
where:
\begin{enumerate}
\item[--] The first drift $b:\mathbb R^{dN} \rightarrow \mathbb R^{dN}$ is given component-wise by
\begin{equation}
\label{drift_b}
b^i(x^1,\dots,x^N):=\frac{1}{N}\sum_{j=1, j \neq i}^N \sqrt{\kappa}\frac{d-2}{2}e_{i,j}(x)\frac{x^i-x^j}{|x^i-x^j|^2}, \quad x^i,x^j \in \mathbb R^d,
\end{equation}
where $e_{ij} \in L^\infty(\mathbb R^{dN})$, $\|e_{ij}\|_\infty \leq 1$.
Setting $e_{ij}=1$ introduces attraction between the particles arising, for example, in the finite particle approximation of the Keller-Segel model of chemotaxis\footnote{In this example we assume $d \geq 3$, but will  be able to include the case $d=2$ as well, albeit, at the moment, with {the restrictive condition $\kappa<C N^{-4}$ on the strength of attraction, which tends to zero as $N\rightarrow\infty$}, see Section \ref{wfb_sect}.}. 

\smallskip

\item[--]
The second drift $q$ is described by external divergence-free velocity field,
$$
q(x^1,\dots,x^N)=\big(q_0(x^1),\dots,q_0(x^N)\big), \quad q_0 \in \mathbf{BMO}^{-1}(\mathbb R^d),
$$ 
so, it is easily seen, $q \in \mathbf{BMO}^{-1}(\mathbb R^{dN})$. 
\end{enumerate}

\noindent Either of our main results, Theorem \ref{thm2} or Theorem \ref{thm_div} for $e_{ij}=1$, applies to $c=b+q$ and  provides, in particular, weak existence and approximation uniqueness for particle system \eqref{sde2_part}. 
Both theorems impose dimension-independent conditions, and so the resulting constraint on  $\kappa$ does not degenerate as the number of particles $N$ goes to infinity. See Examples \ref{ex2_multi}, \ref{ex3_multi} where we detail the particle system \eqref{sde2_part}. In fact, in Example \ref{ex3_multi} we show that when all $e_{ij}=1$ Theorem \ref{thm_div} allows us to handle all $\kappa<16$ for all $N$, {which is close to the blow-up threshold for \eqref{sde2_part} in high dimensions: if $q_0=0$ and $\kappa\geq16(d/(d-2))^2$, then a collision of all particles prevents continuation as a global weak solution.} 

Section \ref{particle_sect} compares our results (in the case $q_0=0$ and $e_{ij}=1$) with those of Cattiaux-P\'{e}d\`{e}ches \cite{CP}, Fournier-Jourdain \cite{FJ}, Fournier-Tardy \cite{FT} and Tardy \cite{T} whose methods exploit the special structure of the drift \eqref{drift_b} to obtain sharp, detailed results.
\end{example}

The search for the maximal admissible value of the strength of attraction $\kappa$ before a blow-up regime can be re-stated as the search for the minimal level of thermal excitation  that {prevents local blow-ups due to collisions between the particles}. The fact that the noise can turn local in time solutions into global ones can be viewed as another instance of regularization by noise; regarding the latter, see \cite{Fl,FGP,FGP2}. 

\subsection{Result \#1}
\label{result_one}

In Theorem \ref{thm2} and also in Theorem \ref{thm2_a},  we establish weak well-posedness of SDE \eqref{sde1} for the following class of  $\mathbb R^d \rightarrow \mathbb R^d$  drifts:
\begin{equation}
\label{c_cond}
\tag{$A_1$}
c=b+q,
\end{equation}
where $q$ is in general distribution-valued,
\begin{equation}
\label{q_cond}
\tag{$A_2$}
q \in \mathbf{BMO}^{-1} \qquad (\Rightarrow\;\;{\rm div\,}q=0),
\end{equation}
and $b$ is a Borel-measurable drift satisfying
\begin{equation}
\label{b_cond}
\tag{$A_3$}
b \in \mathbf{F}_\delta, \quad \;\text{ i.e.}\; \quad |b| \in L^2_{\loc} \text{ and } \|b\varphi\|_2^2 \leq \delta \|\nabla \varphi\|_2^2 + c_\delta\|\varphi\|_2^2 \quad \forall\,\varphi \in W^{1,2},
\end{equation}
for some finite constants $\delta$ (important) and $c_\delta$ (only its finiteness is important). (Here and below, $L^2_{\loc}=L^2_{\loc}(\mathbb R^d)$ and $\|\cdot\|_2$ is the $L^2=L^2(\mathbb R^d)$ norm, $W^{1,2}=W^{1,2}(\mathbb R^d)$ is the Sobolev space of functions that are square integrable together with their first-order derivatives.) That is, $b$ is \textit{form-bounded}. The last condition covers several important classes of singular drifts. It includes the Morrey class $M_{2+\varepsilon}$ (with form-bound $\delta$ depending on the Morrey norm $\|b\|_{M_{2+\varepsilon}}$), the larger Chang-Wilson-Wolff class recalled  in \eqref{cww} and, when we work in $\mathbb R^{dN}$, the many-particle drift \eqref{drift_b} in Example \ref{ex1_multi} where one has $\delta=(N-1)^2 N^{-2}\kappa$. We discuss these and other examples in greater detail in Section \ref{classes_sect}. 

In Remark \ref{rem_prob} we discuss probabilistic meaning of conditions $\mathbf{F}_\delta$ and $\mathbf{BMO}^{-1}$ on the drift.

The constant $\delta$ measures the magnitude or the ``strength'' of singularities. In Theorem \ref{thm2} we impose a dimension-free condition on $\delta${, valid for every $d\geq3$}:
\begin{equation}
\label{delta_cond}
\tag{$A_4$}
\delta<4.
\end{equation}
As a consequence, the attraction strength $\kappa=N^2(N-1)^{-2}\delta$ in Example \ref{ex1_multi} does not degenerate as the number of particles $N$ tends to infinity. (Theorem \ref{thm_div} relaxes this constraint on {$\delta$, and thus on $\kappa$,} by taking into account the divergence of $b$.)

Our proofs make  use of specific properties of the classes $\mathbf{F}_\delta$ and $\mathbf{BMO}^{-1}$, such as the compensated compactness estimates and results on ${\rm BMO}$-multipliers.

\subsubsection{Optimality of our conditions}
\label{optim_sect}
{Let us explain in what sense conditions \eqref{c_cond}-\eqref{delta_cond} are natural and nearly optimal. The following two observations concern the energy space of the Kolmogorov equation and the admissible strength of the drift.} 

\begin{enumerate}

\item[--] Mazya and Verbitsky \cite{MV} proved that conditions \eqref{c_cond}--\eqref{b_cond} are equivalent to the estimate
\begin{equation}
\label{fbd_cond_gen}
|\langle c \cdot \nabla \varphi,\eta\rangle| \leq \alpha \|\nabla \varphi\|_2\|\nabla \eta\|_2
\end{equation}
for all $\varphi, \eta \in C_c^\infty(\mathbb R^d)$,
for some constant $\alpha>0$ (assuming $c_\delta=0$ or up to replacing the homogeneous Sobolev spaces with their non-homogeneous counterparts, see Section \ref{diff_coeff_sect}). {In other words, \eqref{c_cond}--\eqref{b_cond}  are equivalent to having  
$$
\text{operator } -\Delta + c \cdot \nabla \text{ is bounded } W^{1,2} \rightarrow W^{-1,2}.
$$}
If $\alpha<1$, then inequality \eqref{fbd_cond_gen} ensures that the Kato-Lions-Lax-Milgram-Nelson theorem \cite[Ch.\,6,\S 2]{Ka} applies and so the Cauchy problem for the Kolmogorov backward equation is weakly well-posed in the standard Hilbert triple of Sobolev spaces $W^{1,2} \hookrightarrow L^2 \hookrightarrow W^{-1,2}$. 
%%{
%This does not assert failure of weak existence for every drift outside this class; it identifies the drifts for which the %first-order term is bounded between these energy spaces.} 
% Damir: The previous assertion is a sufficient condition, I don't think we need to say explicitly that it is not necessary.
Arguably, to the extent that one can view the latter as a minimal theory of the Kolmogorov backward equation, the present paper bridges the Eulerian and Lagrangian descriptions of diffusion:
\medskip
\begin{center}
\text{Euler \qquad \qquad $\longleftrightarrow$ \qquad \qquad Lagrange}.
\end{center}
\smallskip
Here:
\medskip

\begin{enumerate}

\item[$\bullet$] Eulerian viewpoint: one studies averaged quantities, such as temperature or concentration, governed by the heat equation or, more generally, by transport-diffusion PDEs.

\medskip

\item[$\bullet$] Lagrangian viewpoint: one follows individual molecules whose trajectories solve SDEs (Brownian motion in the simplest case).

\end{enumerate}

\medskip

Until quite recently, Lagrangian results required stronger regularity assumptions on the drift.

We refer to the equivalence \eqref{c_cond}-\eqref{b_cond} $\Leftrightarrow$ \eqref{fbd_cond_gen} as the ``generalized form-boundedness'', although Mazya and Verbitsky \cite{MV} call \eqref{fbd_cond_gen} simply the form-boundedness condition on $c$. 

Let us comment on why one would expect the Kolmogorov equation to be well-posed in the standard Hilbert triple. In fact, we can substantially relax the conditions on drift $b$ by requiring well-posedness of the Kolmogorov backward equation in the shifted triple of Bessel spaces $\mathcal W^{\frac{3}{2},2} \hookrightarrow \mathcal W^{\frac{1}{2},2} \hookrightarrow \mathcal W^{-\frac{1}{2},2}$, see Theorem \ref{thm_wfb}. {Here $\mathcal W^{s,2}:=(1-\Delta)^{-s/2}L^2$ is the Bessel potential space, i.e.\,the fractional Sobolev space of order $s$. } Doing so, however, forces us to drop the distributional part $q$, lose at least for now the almost optimal dimension-independent condition \eqref{delta_cond} and prevents us from including discontinuous diffusion coefficients.
By contrast, under \eqref{c_cond}-\eqref{delta_cond} we can allow discontinuous diffusion coefficients $\sigma$ for which the associated non-divergence operator can still be rewritten in divergence form (Theorem \ref{thm2_a}). Such diffusion coefficients keep us within the standard Hilbert triple $W^{1,2} \hookrightarrow L^2 \hookrightarrow W^{-1,2}$, as discussed in Section \ref{diff_coeff_sect}.

\medskip

\item[--] Near-optimality of the bound \eqref{delta_cond}.
Consider the SDE
\begin{equation}
\label{sde0}
X_t= x - \sqrt{\delta} \frac{d-2}{2}\int_0^t \frac{X_s}{|X_s|^2}ds + \sqrt{2}B_t,
\end{equation}
where $B_t$ is the $d$-dimensional Brownian motion, $d \geq 3$. The drift here, $b(x)=\sqrt{\delta}\frac{d-2}{2}\frac{x}{|x|^2}$, belongs to $\mathbf{F}_\delta$, see Section \ref{classes_sect} for a detailed explanation.
Theorems \ref{thm2} and \ref{thm_div} show that if \eqref{delta_cond} is satisfied, then there exists a strong Markov family of weak solutions to SDE \eqref{sde0}, and these weak solutions are unique at least among those weak solutions that can be constructed via regularization of the drift. Moreover, in the critical case $\delta=4$ there is still a sufficiently rich theory of the corresponding Kolmogorov backward PDE, see Remark \ref{delta_4_rem}. 
On the other hand, taking advantage of the anti-symmetry of the drift, one can show that $R_t=0.5|X_t|^2$ is, up to its first hitting time of zero, a squared Bessel process of dimension $d-\sqrt{\delta}(d-2)/2$ (see, e.g.\,\cite{BFGM}) and therefore:

(a) If $$
\delta \geq 4\left(\frac{d}{d-2}\right)^2,
$$
then the radial process reaches zero in finite time with probability $1$. At equality $\delta=4(\frac{d}{d-2})^2$ the process is absorbed at zero, while for larger $\delta$ it has no nonnegative continuation after hitting zero. Such absorption destroys existence of global in time weak solution to SDE \eqref{sde0} departing from $x=0$, see e.g.\,\cite{BFGM}. 

(b) If $$4<\delta<4\left(\frac{d}{d-2}\right)^2,$$ then the corresponding squared Bessel process still visits zero infinitely many times, but does not stay there, i.e.\,$\int_0^T\mathbf{1}_{\{R_t=0\}}dt=0$ a.s.\,for every $T<\infty$.

In Section \ref{particle_sect} we discuss similar counterexamples in the context of particle system introduced in Example \ref{ex1_multi}.
\end{enumerate}

\begin{remark}
The fact that \eqref{fbd_cond_gen} implies \eqref{c_cond}--\eqref{b_cond} is proved in \cite{MV} by taking $b$ and $q$ from the Hodge-type decomposition of $c$:
\begin{equation}
\label{b__}
b:=-\nabla (1-\Delta)^{-1}{\rm div\,}c + (1-\Delta)^{-1}c, 
\end{equation}
\begin{equation}
\label{q__}
q:=-{\nabla}(1-\Delta)^{-1}{\rm curl\,}c.
\end{equation}
Here, $({\rm curl\,}c)_{ij}:=\nabla_i c^j-\nabla_j c^i$, and $b$ captures the attractive or the repulsive part of the drift in the dynamics of $X_t$. We do not use decomposition \eqref{b__}, \eqref{q__} in the proofs. Let us emphasize that the form-bounded class $\mathbf{F}_\delta$, which is of interest on its own also because it provides conditional strong well-posedness for the SDE as well as well-posedness for the associated stochastic transport equation, see Theorem \ref{thm2}, is not exhausted by drifts \eqref{b__}.
\end{remark}

\begin{remark}[On probabilistic interpretation of $\mathbf{F}_\delta$ and $\mathbf{BMO}^{-1}$]
\label{rem_prob}
1.~Let $b \in \mathbf{F}_\delta$ and $q=0$. Let $b$ be also bounded and smooth: what matters below is that the constants in the estimates are independent of its boundedness or smoothness. Let $\rho$ be a positive smooth probability density on $\mathbb R^d$ with sufficiently rapid decay at infinity so that the manipulations with the equations below are justified. Let $\mathbb P_{\rho,h}^{b}$ and $\mathbb P_{\rho,h}^{0}$ be the laws on $[0,h]$ of the diffusion with drift $-b$, i.e.\,generated by $-\Delta + b \cdot \nabla$, and of the Brownian motion with generator $-\Delta$, respectively, both starting at time $t=0$ from probability density $\rho$. Recall the definition of relative entropy between two probability measures: 
$$H(P | Q)=\int \log \frac{dP}{dQ}\,dP.$$
One can show, using e.g.\,Girsanov's transform, that
$$
\frac{d}{dh}
H(\mathbb P_{\rho,h}^{b}|\mathbb P_{\rho,h}^{0})\bigg|_{h=0}
=\frac{1}{4}\langle |b|^2\rho\rangle.
$$
The left-hand side measures how rapidly the law of the process with drift becomes statistically distinguishable (i.e.\,after taking expectation  to calculate relative entropy) from that of Brownian motion as time increases (clearly, if $\mathbb P_{\rho,h}^{b}=\mathbb P_{\rho,h}^{0}$, then the relative entropy is zero).
On the other hand, Brownian motion started from $\rho$ increases its entropy $S(\rho):=-\langle\rho\log\rho\rangle $ at the initial rate
$$
\frac{d}{dh}S(e^{h\Delta}\rho)\bigg|_{h=0}
=4\langle|\nabla\sqrt\rho|^2\rangle:=I(\rho),
$$
where $I(\rho)$ is the Fisher information of $\rho$. See \cite{BGL,Le} for the proofs and more detailed discussions of entropy and Fisher information.
Now, specifying the form-boundedness condition of $b$ to the test function $\sqrt\rho$, we obtain
$$
\frac{1}{4}\langle|b|^2\rho\rangle
\leq \frac{\delta}{16}I(\rho)+\frac{c_\delta}{4}.
$$
Thus, the form-boundedness of $b$ imposes a condition: the rate at which the trajectory laws become statistically distinguishable from Brownian motion, measured by relative entropy, is controlled by the rate at which Brownian motion increases the entropy of the particle's distribution. 

2.~A drift of the form $q=\nabla Q\in\mathbf{BMO}^{-1}$ has an interpretation in terms of probability flux. For simplicity, take $b=0$ and let $\rho_t$ be the probability density of the diffusion with drift
$-q$. We assume that $Q$ is smooth and that $\rho_t$ is positive, smooth and decays at infinity sufficiently rapidly.
Since $Q=-Q^\top$,
$$
{\rm div\,}(Q\nabla\rho_t)=q\cdot\nabla\rho_t.
$$
The Fokker-Planck equation for $\rho_t$ can therefore be written as
$$
\partial_t\rho_t+{\rm div\,}J_t=0,
\quad
J_t:=-\nabla\rho_t-Q\nabla\rho_t.
$$
The Brownian part $-\nabla\rho_t$ transports probability towards
regions of lower density. The additional flux $-Q\nabla\rho_t$
is transverse to the density gradient, since
$
(Q\nabla\rho_t)\cdot\nabla\rho_t=0.
$
Thus, $Q$ describes a transverse component of the probability flux.

Let us also note 
that a constant anti-symmetric matrix produces a divergence-free transverse
flux and has no effect on the density evolution, i.e.\,its column-divergence is zero. It is the spatial
variation of $Q$ that produces the drift $q=\nabla Q$.
It is therefore natural to control the oscillations of $Q$.
The (analytic) reason that the ${\rm BMO}$ condition is sufficient
at the level of the Brownian energy space is the estimate
$$
|\langle Q\nabla f,\nabla g\rangle|
\leq C_d\|Q\|_{\rm BMO}\|\nabla f\|_2\|\nabla g\|_2,
\quad f,g\in W^{1,2},
$$
see Proposition \ref{cc_lem}. 
\end{remark}

\subsubsection{Local blow-ups}
In Section \ref{div_drifts_sect} we will also discuss drifts having critical (form-bounded) divergence, which allows to relax the assumptions on the drift to some super-critical conditions, i.e.\,passing to the small scales actually increases the norm of the drift. The sub-critical/critical/super-critical classification of the spaces of vector fields, widely used in the literature, is recalled in subsection \ref{super_rem}. 
It should be added, however, that this classification is not so relevant to the main body of the present paper (except for Section \ref{div_drifts_sect}) because:

(a) Our main focus is on general drifts, for which one only has the dichotomy sub-critical/critical.

(b) This classification does not distinguish between critical spaces that do, and those that do not, reach blow-ups.

(Regarding (b), for example, both $L^d$ space and weak $L^d$ space (we recall its definition in Section \ref{notations_sect}) are critical, but only the weak $L^d$ space contains drift $C\frac{x}{|x|^2}$ whose attracting singularity at the origin is strong enough to kill the weak well-posedness of the SDE if $C$ is too large. In other words, in the critical case, a lot depends on the definition  of the norm of the drift.) 

If a critical class of drifts is broad enough to contain blow-up examples, a well-posedness theorem must include a smallness condition on the drift norm, such as e.g.\,\eqref{delta_cond}.

Compared with blow-ups for the Navier–Stokes equations, the mechanism of blow-ups in particle systems of the kind treated in Example \ref{ex1_multi} is much better understood \cite{CP,CPZ,FJ,FT,JL}. 
That said, as was noted in \cite{FJ}, at the time of writing of their article 
there was still a substantial gap between (i) the drift singularities that the general theory of SDEs with singular drifts could handle and (ii) the even stronger singularities they themselves had to treat. One of the purposes of the present work, together with \cite{KiS_sharp, Ki_multi, KiS_feller}, is to close this gap:
\begin{equation}
\label{conn}
\text{SDEs with general singular drift \qquad \qquad $\rightarrow$ \qquad \qquad particle systems},
\end{equation}
i.e.\,to bring the general theory of SDEs ``up to the task'' so that it can handle blow-ups.
See, in particular, Section \ref{wfb_sect} where we demonstrate how the weak well-posedness of the finite-particle Keller-Segel  SDE in $\mathbb R^{2N}$ can be reached from an earlier result in \cite{KiS_brownian} on SDEs with general drifts, albeit at the moment under additional rather restrictive conditions on the strength of attraction. Interestingly, the path that leads to this passes through  non-local operators (Theorem \ref{thm_wfb} and Corollary \ref{cor1}). 

The same connection \eqref{conn} was already pursued by Krylov and Röckner \cite{KR}, but with a different interaction kernel: its attractive part can be extremely singular, but it is always dominated on average (not pointwise) by the repulsive part, so no blow-up occurs.

\subsection{Result \#2}

In subsection \ref{super_rem} we will also discuss super-critical drifts (necessarily under additional assumptions on their divergence). Specifically, if the positive part of the divergence ${\rm div\,}b$ is a form-bounded potential, then this enables us to relax the form-boundedness condition on $b$ to a super-critical form-boundedness condition:
$$
|b|^{\frac{1+\nu}{2}} \in \mathbf{F}_\delta, \quad \nu \in ]0,1[, \quad \delta<\infty.
$$
In this case one still has weak existence for every initial point \cite{KiS_sharp}, moreover, as was mentioned there, one can combine such a drift with an ordinary form-bounded component, again reaching the blow-up threshold, and may also include discontinuous diffusion coefficients of the type treated in Section \ref{diff_coeff_sect}.

Generally speaking, in super-critical settings many standard regularity properties of the diffusion process are lost. Some of them can be saved, but to recover most of them one must impose additional critical conditions on the drift. 
In our framework, to establish e.g.\,the Markov property, Theorem \ref{thm1} supplements the super-critical assumption {with the restriction $\delta_+<2$ on the form-bound of $({\rm div\,}b)_+^{1/2}$ and }with the multiplicative form-boundedness condition $b \in \mathbf{MF}_\delta$, i.e.
\begin{equation*}
\langle |b|\varphi,\varphi\rangle \leq \delta\|\nabla \varphi\|_2\|\varphi\|_2+c_\delta\|\varphi\|_2^2, \quad \varphi \in W^{1,2}, \quad \delta<\infty.
\end{equation*}
This condition is critical and is substantially more general than $b \in \mathbf{F}_\delta$ because somehow it implicitly presumes the existence of ${\rm div\,}b$.  
In contrast to the form-bounded class $\mathbf{F}_\delta$, the multiplicative form-bounded class $\mathbf{MF}_\delta$ can be completely characterized in elementary terms: $b \in \mathbf{MF}_\delta$ if and only if  $$\langle |b|\mathbf{1}_{B_r(x)}\rangle \leq C r^{d-1}$$
with constant $C$ independent of $r$ or $x \in \mathbb R^d$. The latter is the largest scaling-invariant Morrey\footnote{For the theory of Morrey spaces and related estimates, see Adams-Xiao \cite{AX} and Krylov \cite{Kr_b2}.} class  $M_1$. In fact, we have
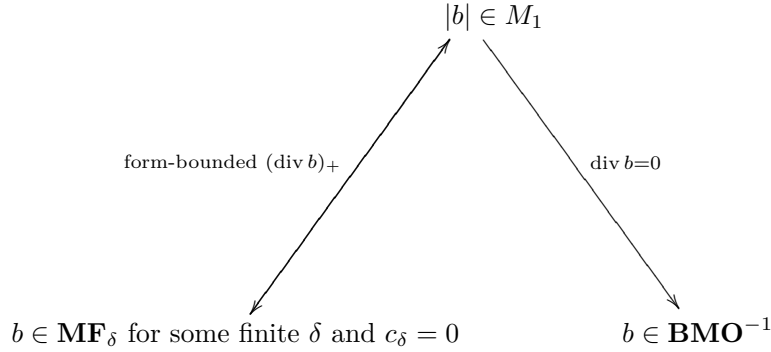
\begin{equation*}
\bfig
\node a1(0,0)[\qquad |b| \in M_1]
\node a2(-800,-1100)[b \in \mathbf{MF}_\delta \text{ for some finite $\delta$ and }c_\delta=0]
\node a3(800,-1100)[b \in \mathbf{BMO}^{-1}]
\arrow[a1`a2; \text{form-bounded }({\rm div\,}b)_+]
\arrow[a2`a1;]
\arrow[a1`a3;{\rm div\,}b=0]
\efig
\end{equation*}
where the arrow $\rightarrow$ in this diagram means inclusion. (These inclusions were proved in \cite[Theorem 1.4.7]{M} and \cite[Theorem V]{MV2}. See discussion before Theorem \ref{thm1} and Krylov's proof in Appendix \ref{m_sect}.) 

As we explain below, we use De Giorgi's method to obtain PDE results that, in turn, yield SDE solution theory. The point of departure for this method is a local energy estimate -- Caccioppoli's inequality, cf.\,Lemma \ref{cacc_prop}. For $b \in \mathbf{F}_\delta$ we prove  Caccioppoli's inequality in a straightforward manner, using integration by parts and Cauchy-Schwarz. However, for divergence-free drifts in Morrey class $M_1$ the proof is more involved. That is, choosing the inclusion of $M_1$ into $\mathbf{MF}_\delta$ leads to the proof via an extra iteration procedure, see Remark \ref{cacc_rem} for details;
choosing the inclusion into $\mathbf{BMO}^{-1}$ amounts to absorbing the stream matrix $Q$ (for $b=\nabla Q$) into the diffusion coefficients and obtaining a Caccioppoli-type inequality from there \cite{H,SSSZ}. We discuss this in Section \ref{div_drifts_sect}.

\subsection{Result \#3}
As an application of our results on SDEs (Theorem \ref{thm_div}) and of the theory of Bessel processes, we obtain an upper bound on the best possible constant $C_{d,N}$ in the many-particle Hardy inequality
\begin{equation*}
C_{d,N}\sum_{1 \leq i<j \leq N}\int_{\mathbb R^{dN}}\frac{|\varphi(x)|^2}{|x^i-x^j|^2}dx \leq  \int_{\mathbb R^{dN}}|\nabla \varphi(x)|^2 dx \quad \forall\,\varphi \in C_c^\infty(\mathbb R^{dN})
\end{equation*}
(Theorem \ref{cor2}). 
It shows that the lower bound on $C_{d,N}$ obtained earlier by Hoffmann-Ostenhof--Hoffmann-Ostenhof-Laptev-Tidblom \cite{HHLT} is nearly optimal  in high dimensions.
To our knowledge, this is the first time a probabilistic argument is used in the analysis of the best possible constant in a Hardy inequality.

\subsection{Literature}
Let us now comment on the existing literature on SDEs with singular drifts, focusing mostly on general drifts. In our setting, even if we restrict our attention to Borel measurable drifts, then the Zvonkin transform is not applicable, and our drifts are not of Girsanov type, see Remark \ref{girsanov_rem}.

1) A pioneering paper of Kovalenko-Sem\"{e}nov \cite{KS} constructed Feller semigroup for $-\Delta + b \cdot \nabla$ with $b \in \mathbf{F}_\delta$ provided that $\delta<c_d$, where $c_d \ll 1$ for $d \gg 3$. The present paper continues \cite{KS,KiS_sharp, Ki_multi, KiS_feller} and also \cite{KiS_Osaka,KiS_brownian}. In \cite{KiS_brownian}, after \cite{Ki_super}, Sem\"{e}nov and one of us constructed weak solutions to SDEs with drifts in a larger class than $\mathbf{F}_\delta$, namely, in the class of \textit{weakly form-bounded} vector fields \cite{S}.
This larger class contains, for example, the Morrey class $M_{1+\varepsilon}$, and therefore vector field \eqref{q_}; however, that result requires more restrictive condition $\delta<\frac{c}{d^2}$ (see Section \ref{wfb_sect}).
In all those earlier works one has $q=0$ and $\sigma=I$, with the exception of \cite{KiS_Osaka} that dealt with the class of  discontinuous diffusion coefficients treated in Section \ref{diff_coeff_sect}.

2) If we restrict attention to Borel-measurable drifts, the most closely related results are the recent works of Krylov  \cite{Kr1,Kr2,Kr3, Kr3_5, Kr4} and of R\"{o}ckner and Zhao \cite{RZ_weak}. In \cite{Kr1}, Krylov treats a class of diffusion coefficients much larger than ours (his diffusion coefficients are in ${\rm VMO}$ class, or have small ${\rm BMO}$ norm), but a smaller class of drifts (namely, the Morrey class $M_{(\frac{d}{2} \vee 2) + \varepsilon}$). 
The irregular diffusion coefficients require estimates on second-order derivatives of solutions to the Kolmogorov backward equation, whereas we neither assume nor have such estimates under our assumptions on the drift. 
Krylov's drifts can also produce blow-up phenomena, but he assumes that the size of the singularity, measured by the Morrey norm, remains below a small, dimension-dependent constant. By contrast, our goal is to reach the maximal admissible strength of singularity (in Example \ref{ex1_multi}, the interaction strength $\kappa$; or the supremum \eqref{cww0}). Other recent results of Krylov are described below.

3) Some techniques such as De Giorgi's method  also connect our paper to the works of Zhang-Zhao \cite{ZZ} and Hao-Zhang \cite{HZ}, which focus on super-critical divergence-free drifts (in \cite{HZ}, distributional). Some of their results deal with non-zero divergence, but they do not reach blow-ups. It should be added that in the study of Navier-Stokes equations (the SDEs connected to the N-S equations are the main interest of \cite{ZZ,HZ}) the blow ups of the type discussed above, i.e.\,attracting, do not appear.

4) There is rich literature on SDEs with general distributional drifts, also motivated by physical applications. In all the results we are aware of, the diffusion coefficients $\sigma$ have to be at least H\"{o}lder continuous.  Many of these works treat time-inhomogeneous drifts, however, in this brief discussion we will specify these results to time-homogeneous drifts. 
Flandoli-Issoglio-Russo \cite{FIR} and
Zhang-Zhao \cite{ZZ2} study drifts that belong to Bessel potential spaces with negative index and employ the Zvonkin transform to obtain weak well-posedness for SDE \eqref{sde1}.
Chaudru de Raynal-Menozzi \cite{CM} proved, among other results, weak well-posedness with drift $b$ in Besov space $B_{p,q}^{\beta}$ with $-\frac{1}{2}<\beta \leq 0$ and $p>d$. They addressed, in particular, the problem of defining the product of two distributions $b \cdot \nabla v$, where $v$ is a solution to the Kolmogorov backward equation. The latter, in turn, dictates their restriction $\beta>-\frac{1}{2}$. 
Earlier, Delarue-Diel \cite{DD} and Cannizzaro-Chouk \cite{CC}  proved, in dimensions $d=1$ and $d \geq 2$, respectively, well-posedness of the martingale problem for general distributional drifts in H\"{o}lder-Besov space $B_{\infty,\infty}^{\beta}$ for all $-\frac{2}{3}<\beta \leq 0$. In the more singular regime $-\frac{2}{3} <\beta \leq -\frac{1}{2}$ they assume that the drift can be enhanced to a rough distribution, in order to apply the theory of paracontrolled distributions. This allowed them to consider random drifts of the form
$
b=\nabla h$, where $h$ is a solution of a KPZ-type equation,
and thereby construct the polymer measure with white noise potential. 
As mentioned earlier, the class of divergence-free drifts  $\mathbf{BMO}^{-1}$ contains drifts whose components lie in Besov space ${\dot{B}^{-1+d/p}_{p,\infty}}$, ${d<p<\infty}$, i.e.\,the exponent $\beta=-1+d/p$ can go up to $-1$. This, however, cannot serve as a comparison with the previously cited results since by definition $\mathbf{BMO}^{-1}$ drifts are divergence-free, but this justifies to some extent why we do not address in this paper the problem of multiplying distributions in the Kolmogorov equation, i.e.\,since our $\beta \not > -\frac{1}{2}$. 

5) Let us also  mention recent papers by Chaudru de Raynal-Jabir-Menozzi \cite{CJM,CJM2} where the authors handle McKean-Vlasov SDEs with distributional Besov drifts. The regularization by convolution allows them to venture substantially farther in the assumptions on the drift, compared to the papers cited in 4). It is quite noteworthy, since they can start with a delta-function in the initial distribution. 

The papers on general distributional drifts mentioned in 4) and 5) do not reach blow-ups, in contrast to our Theorems \ref{thm2}, \ref{thm2_a}, \ref{thm_div}. That said, the cited works cover some other highly irregular drifts that fall outside the scope of our hypotheses.

Finally, we mention Bresch-Jabin-Wang \cite{BJW} who consider gradient-form singular attracting drifts that allow blow-ups. Their focus is different, namely, it is quantitative propagation of chaos at the PDE level, but they are also interested in reaching the blow-up thresholds. Their class of drifts is not contained in our class, and vice versa, although there is a  very substantial overlap between these classes.

6) Some singular drifts, as well as some degenerate diffusion coefficients, can be treated using the Ambrosio-Figalli-Trevisan superposition principle. The latter allows to conclude weak solvability of the SDE from some integrability and weak continuity properties of  solution to the corresponding Fokker-Planck equation (see \eqref{sup_pr}, \eqref{int_cond}).
See Trevisan \cite{Tr} and Bogachev-R\"{o}ckner-Shaposhnikov \cite{BRS}, see also Grube \cite{Gru} and the survey of the literature therein. The theory of Fokker-Planck equations with locally unbounded drifts and degenerate diffusion coefficients is discussed in detail in \cite{BKRS}.

7) Regarding super-critical drifts, we refer again to  Zhang-Zhao \cite{ZZ}, as well as to recent  papers by Hao-Zhang \cite{HZ} and Gr\"{a}fner-Perkowski \cite{GP} where the authors consider super-critical distributional drifts. In \cite{GP}, the authors treat divergence-free Besov drifts $b \not \in B^{-1}_{2d(d+2)^{-1},2}$, provided that the initial density is absolutely continuous with respect to the Lebesgue measure; they also consider quite irregular attracting/repulsing component of the drift, although it does not reach the blow-ups.

Comprehensive surveys of the literature on SDEs with singular and distributional drifts can be found in  \cite{CM} and \cite{HZ}.

\subsection{Main instruments:  De Giorgi's method, Trotter's theorem and compensated compactness estimates}
Since the drift $c$ satisfying \eqref{fbd_cond_gen} is in general distribution-valued ($c=b+q$) we  have to give a proper meaning to term
$c(X_s)$ in SDE \eqref{sde2}. We will do it in two ways:

(a) Theorem \ref{thm2}(\textit{iii}): At the level of the martingale problem and with the It\^{o} expansion, we have to define $(c\cdot \nabla v)(X_s)$ for suitable test functions. The usual test functions in $C_c^\infty$ cannot be used, but it is still possible to find a sufficiently rich space of test functions (in particular, dense in $C_\infty$; see notations in Section \ref{notations_sect}) that will give us a continuous martingale. In fact, this space of test functions will be the domain of the Feller generator $\Lambda \supset -\Delta + c \cdot \nabla$.

(b) Theorem \ref{thm2}(\textit{iv}): By constructing the ``limiting drift'' process (called ``formal dynamics'' in \cite{CM}) for disperse initial data, assuming $\delta<1$. This is a result of Bass-Chen type \cite{BC}, see also Zhang-Zhao \cite{ZZ2}, Chaudru de Raynal-Menozzi \cite{CM} and Hao-Zhang \cite{HZ} who have similar results. The proof will use convergence of the martingale solutions of the approximating SDEs \eqref{sde2_approx} that will follow as a by-product of (a).

To construct the sought Feller semigroup, we will have to employ some deep results from the operator theory, the theory of PDEs and harmonic analysis:

\begin{enumerate}
\item[--]
 Trotter's approximation theorem, 
 
\smallskip
 
 \item[--]
 De Giorgi's method in $L^p$ for $p>\frac{2}{2-\sqrt{\delta}}$, i.e.\,in the context of Example \ref{ex1_multi}, $p$ will be an explicit function of the strength of attraction between the particles $\kappa$. 
 
 \smallskip
 
 \item[--] Compensated compactness estimates and results on ${\rm BMO}$-multipliers.
 
\end{enumerate}

A crucial feature of Trotter's approximation theorem (Theorem \ref{trotter_thm}) is that it requires no a priori knowledge of the limiting object, i.e.\,in our case, the limiting Feller semigroup. 
By contrast, in some other of our results on singular SDEs (such as Theorem \ref{thm_wfb}) we first construct an explicit candidate for the Feller resolvent. This simplifies approximation arguments, but it also automatically provides strong gradient bounds that, in turn, introduce dimension-dependent restrictions on the form bound $\delta$. 
As the dimension grows, the admissible range of 
$\delta$ shrinks to the empty set, making problematic possible applications to many-particle systems that exist in spaces of very large dimension. (Example \ref{ex2_multi} explains how a bound on $\delta$ translates into a bound on the attraction strength $\kappa$.) Using De Giorgi's method in the proof of Theorem \ref{thm2} decouples gradient bounds from the bounds that provide the existence of the weak solution (i.e.\,tightness argument, construction of the Feller semigroup, etc). Consequently, we can impose dimension-free assumptions on $\delta$ (and on the entire drift), which is a crucial point for applications to many-particle systems.

Another way to view the relationship between the dimension-free assumptions on the drift and the Sobolev embedding is as follows. The Sobolev embedding becomes weaker as the dimension increases, and so if it is applied only once, then one arrives at dimension-dependent conditions on the drift. In De Giorgi's method, however, the Sobolev embedding is applied infinitely many times, which allows to overcome this dependence on the dimension. {Here dimension-free refers to the restriction $\delta<4$, which is valid for every $d\geq3$, however, the constants in the PDE estimates may still depend on $d$.}

The present paper strengthens the results \cite{Ki_multi,KiS_feller} that dealt with Borel-measurable drifts. Allowing a non-zero distributional $q \neq 0$ drift, as in \eqref{c_cond}-\eqref{delta_cond}, required substantial modifications of the arguments in  those papers.

De Giorgi's method for divergence-form operators with form-bounded drift and/or drift in $\mathbf{BMO}^{-1}$ has also been applied, in different contexts, by Hara \cite{H} and Seregin-Silvestre-\v{S}verak-Zlato\v{s} \cite{SSSZ}. We also refer to recent  paper by Liang-Wang-Zhao \cite{LWZ} where the authors establish strong estimates on the modulus of continuity of solutions of elliptic equations with distributional coefficients.

\subsection{Structure of the paper} 

\begin{enumerate}

\item[Section \ref{particle_sect}] 
We begin with a preliminary discussion of blow-up thresholds for the particle system in \ref{ex1_multi}. These thresholds both illustrate the sharpness of our SDE results and serve as input in Section \ref{div_drifts_sect},  where they enter the proof of a sharper upper bound for the constant in the many-particle Hardy inequality (Theorem \ref{cor2}).

\item[Section \ref{gen_sect}] Theorem \ref{thm2} is our main result for general drifts when the diffusion matrix is constant.

\item[Section \ref{div_drifts_sect}] Theorems \ref{thm_div} and \ref{thm1} 
treat drifts whose divergence satisfies a critical (scaling-invariant) form-boundedness condition. Under this assumption we can relax the constraint on $|b|$, allowing some super-critical form-boundedness conditions. As an application, Theorem \ref{cor2} gives an improved upper bound on the constant in the many particle Hardy inequality.

\item[Section \ref{diff_coeff_sect}] Theorem \ref{thm2_a} extends the analysis to diffusion matrices that may have critical discontinuities.

\item[Section \ref{wfb_sect}] Corollary \ref{cor1} shows how well-posedness of the finite-particle Keller–Segel system can be derived from our earlier result on SDEs with general singular drifts (Theorem \ref{thm_wfb}), {although under a condition on the strength of attraction degenerating to zero as the number of particles goes to infinity}. The proof of that theorem uses an operator-theoretic approach, based on fractional resolvent representations, and covers the larger class of \textit{weakly form-bounded drifts} that contains both the Morrey class $M_{1+\varepsilon}$ and the Kato class studied by Bass-Chen \cite{BC}.

{
\item[Sections \ref{proof_thm2_sect}--\ref{proof_thm2a_sect}]
Section \ref{proof_thm2_sect} is dedicated to the proof of Theorem \ref{thm2}, including
the assertions for $q=0$. In Sections \ref{proof_thmdiv_sect} and
\ref{proof_thm1_sect} we prove Theorems \ref{thm_div} and
\ref{thm1}(\textit{iii}), respectively. The crucial PDE estimates
and convergence arguments are provided in Sections
\ref{aux_sect}--\ref{grad_lem_proof}.
We prove Theorem \ref{thm2_a} in Section \ref{proof_thm2a_sect}.}

\end{enumerate}

\subsection*{Acknowledgements} We are grateful to Galia Dafni, Jean-Fran{\c{c}}ois Jabir and Kodjo Rapha\"{e}l Madou for  very useful discussions. {We also thank the anonymous referees for thorough reading of the manuscript and their helpful comments.}

\bigskip

\section{Notations and auxiliary results} 

\label{notations_sect}

Throughout this work we use the following notations.

\medskip

\textbf{1.}~$\mathcal B(X,Y)$ is the space of bounded linear operators $X \rightarrow Y$ between Banach spaces $X$, $Y$, endowed with the operator norm $\|\cdot\|_{X \rightarrow Y}$. Put $\mathcal B(X):=\mathcal B(X,X)$. 

The space of $d$-dimensional vectors with entries in $X$ is denoted by $[X]^d$. We reserve the upper index to denote the components $q^i \in X$ of $q \in [X]^d$.

The notation ``$\overset{w}{\rightarrow}$ in $X$'' stands for the weak convergence in $X$. Similarly for $[X]^d$.

We write $$T=s\mbox{-} Y \mbox{-}\lim_n T_n$$ for $T$, $T_n \in \mathcal B(X,Y)$ if $$\lim_n\|Tf- T_nf\|_Y=0 \quad \text{ for every $f \in X$}.
$$ 

By $T \upharpoonright X$ we denote the restriction of operator $T$ to a subspace $X \subset D(T)$. 

By 
$
\big[T \upharpoonright X\big]_{Y \rightarrow Y}^{\rm clos}
$
we denote the closure of the restriction $T \upharpoonright X$ (when it exists).

\medskip

\textbf{2.}
The spaces $L^p=L^p(\mathbb R^d,dx)$, $W^{1,p}=W^{1,p}(\mathbb R^d,dx)$ correspond to the Lebesgue and to the Sobolev space, respectively. Let $L^p_\rho=L^p(\mathbb R^d,\rho(x) dx)$ denote the weighted Lebesgue space with weight $\rho$ (defined by formula \eqref{rho_def} below).

Set $\|\cdot\|_p:=\|\cdot\|_{L^p}$
and denote operator norm
$
\|\cdot\|_{p \rightarrow q}:=\|\cdot\|_{L^p \rightarrow L^q}
$.

Given $1 < p < \infty$, we set $p':=\frac{p}{p-1}$.

Put
$$
\langle f,g\rangle = \langle f \bar{g}\rangle :=\int_{\mathbb R^d}f \bar{g} dx.$$ 
For vector fields $b$, $\mathsf{f}:\mathbb R^d \rightarrow \mathbb R^d$, we put
$$
\langle b,\mathsf{f}\rangle:=\langle b \cdot \mathsf{f}\rangle \qquad \text{($\cdot$ is the scalar product in $\mathbb R^d$)}.
$$

$C_c$ (resp.\, $C_c^\infty$) denotes the space of continuous (infinitely differentiable) functions on $\mathbb R^d$ having compact support. 

$C_b$ is the space of bounded continuous functions on $\mathbb R^d$ endowed with the $\sup$-norm, and $C_b^k$ is the subspace of bounded continuous functions with bounded continuous derivatives up to order $k$.

$C_\infty$ is a closed subspace of $C_b$ consisting of functions vanishing at infinity.

We denote by $\mathcal S$ the Schwartz space, and by $\mathcal S'$ the space of tempered distributions on $\mathbb R^d$.

Let $\nabla_i=\partial_{x_i}$, $1 \leq i \leq d$. {Thus $\nabla_i\nabla_j=\partial_{x_i}\partial_{x_j}$.}

Given a matrix field $a:\mathbb R^d \rightarrow \mathbb R^{d \times d}$, define the {column divergence} operator $a \mapsto \nabla a$ via
$$
(\nabla a)_j:=\sum_{i=1}^d \nabla_i a_{ij}, \quad 1 \leq j \leq d. 
$$

{For a vector $v$ and a matrix $Q$, the vector $v\cdot Q$ has components $(v\cdot Q)_j=\sum_i v_iQ_{ij}$; thus $v\cdot Q=Q^{\scriptscriptstyle\top}v$.}

We denote by $H_{\xi}$ ($\xi>0$) the set of bounded symmetric uniformly elliptic Borel measurable matrix fields $a:\mathbb R^d \rightarrow \mathbb R^{d \times d}$:
$$
a=a^{\scriptscriptstyle \top}, \quad \xi I \leq a(x) \leq \xi^{-1}I\text{ for a.e.\,} x \in \mathbb R^d,
$$
for $I$ the $d \times d$ identity matrix.

Put $$\Gamma_c(t,x)=\Gamma(ct,x):=(4c\pi t)^{-\frac{d}{2}}e^{-\frac{|x|^2}{4ct}},$$
the Gaussian density.

Set
$$
\gamma(x):=\left\{
\begin{array}{ll}
c\exp\left(\frac{1}{|x|^2-1}\right)& \text{ if } |x|<1, \\
0, & \text{ if } |x| \geqslant 1,
\end{array}
\right.
$$
where $c$ is adjusted to $\int_{\mathbb R^d} \gamma(x)dx=1$, and put $\gamma_\varepsilon(x):=\frac{1}{\varepsilon^{d}}\gamma\left(\frac{x}{\varepsilon}\right)$, $\varepsilon>0$, $x\in \mathbb R^d$.
Define the De Giorgi mollifier of a function $h \in L^1_{\loc}$ (or a vector field with entries in $L^1_{\loc}$) by $$E_\varepsilon h:=e^{\varepsilon \Delta} h.$$

$B_r(x)$ denotes the open ball of radius $r$ centered at $x \in \mathbb R^d$. If $x=0$, we simply write $B_r$.

Given a function $f \in L^1_{\loc}$, we denote by $(f)_{B_r(x)}$ its average over the ball $B_r(x)$:
$$
(f)_{B_r(x)}:=\frac{1}{|B_r|}\int_{B_r(x)} f dx.
$$
If $x=0$, then we write $(f)_r \equiv (f)_{B_r}$.

We denote the positive and negative parts of function $f$ by 
$$(f)_+:=f \vee 0, \quad (f)_-:=-(f \wedge 0).$$

Define weight
\begin{equation}
\label{rho_def}
\rho(y) \equiv \rho_{\epsilon_0}(y)=(1+\sigma|y|^2)^{-\frac{d+\epsilon_0}{2}}, \quad \epsilon_0>0, \quad \sigma>0.
\end{equation}
This weight has property
\begin{equation}
\label{rho_est}
|\nabla \rho| \leq \frac{d+\epsilon_0}{2}\sqrt{\sigma} \rho.
\end{equation}
In the same way, $|\nabla \rho^{-1}|\;(= \frac{|\nabla \rho|}{\rho^2})\;\leq \frac{d+\epsilon_0}{2}\sqrt{\sigma} \rho^{-1}$.
(The last two inequalities will allow us to replace all occurrences of $\nabla_i\rho$ or $\nabla_i\rho^{-1}$ resulting from the integration by parts in the analysis of PDEs in weighted spaces by the weight $\rho$ itself or $\rho^{-1}$, respectively, times a constant that is proportional to $\sqrt{\sigma}$. This constant can be made arbitrarily small by fixing $\sigma$ sufficiently small. We will use this to get rid of the terms containing $\nabla_i\rho$ or $\nabla_i\rho^{-1}$.)
Put $$\rho_x(y):=\rho(x-y).$$

\smallskip

\textbf{3.}\,Let $\mathbf{C}$ and $\mathbf{D}$ denote the canonical spaces of continuous and c\`{a}dl\`{a}g trajectories from $[0,\infty[$ to $\mathbb R^d$ equipped with the uniform topology and the Skorohod topology, respectively, endowed with the natural filtration $\mathcal B_t=\sigma\{\omega_s \mid 0 \leq s \leq t\}$,
where $\omega_t$ is the coordinate process.

Let $\mathcal P(\mathbf C)$ and $\mathcal P(\mathbf D)$ denote the spaces of probability measures on $\mathbf C$ and $\mathbf D$, respectively.

Recall that a  probability measure $\mathbb P_{s,x}$ ($s \geq 0$, $x \in \mathbb R^d$) on $\mathbf{C}$
is said to be
 a classical martingale solution to SDE
\begin{equation}
\label{sde4}
X_t=x-\int_s^t b(r,X_r)dr + \sqrt{2}(B_t-B_s), \quad x \in \mathbb R^d, \quad t \geq s,
\end{equation} 
with a time-inhomogeneous drift $b \in L^1_{\loc}(\mathbb R^{1+d})^d$ 
if $\mathbb P_{s,x}[\omega_s=x]=1$,  
$$
\mathbb E_{\mathbb P_{s,x}}\int_s^t |b(r,\omega_r)| dr<\infty
$$
and for every $v \in C_c^2$
$$
t \mapsto v(\omega_t)-v(x) + \int_s^t (-\Delta + b \cdot \nabla)v (\omega_r) dr, \quad t \geq s,
$$
is a continuous martingale with respect to $\mathbb P_{s,x}$.

If $b$ does not depend on time, then we take $s=0$, and in the above martingale problem write $\mathbb P_x$.

\medskip

\textbf{4.}~\textit{BMO functions.\,}(a) 
 A function $g \in L^1_{\loc}$ is in class ${\rm BMO}={\rm BMO}(\mathbb R^d)$ if 
$$
\|g\|_{\rm {\rm BMO}}:= \sup_{x \in \mathbb R^d, R>0}\frac{1}{|B_R|}\int_{B_R(x)}|g-(g)_{B_R(x)}|dy<\infty.
$$
One can also define the ${\rm BMO}$ semi-norm as
\begin{equation}
\label{carl_bmo}
\|g\|_{\rm {\rm BMO}}=\biggl(\sup_{x \in \mathbb R^d, R>0}\frac{1}{|B_R|}\int_{B_R(x)}\int_0^{R^2} |\nabla e^{t\Delta}g|^2 dt dy \biggr)^\frac{1}{2}
\end{equation}
(this is Carleson's characterization of ${\rm BMO}$ functions). 

\medskip

(b) We will need the compensated compactness estimate of Coifman-Lions-Meyer-Semmes:

\begin{proposition}[\cite{CLMS}] 
\label{cc_lem}
There exists a constant $C_d$ such that for every anti-symmetric matrix field $Q$ with entries in ${\rm BMO}$ one has
$$
|\langle Q \cdot \nabla f,\nabla g\rangle| \leq C_d\|Q\|_{\rm BMO}\|\nabla f\|_2\|\nabla g\|_2, \quad \forall\,f,g \in W^{1,2}.
$$
\end{proposition}

(c) Qian-Xi \cite{QX} employed in their work the following modification of the previous proposition. We will need it as well.

\begin{proposition}
\label{cc_lem_qx}
There exists a constant $C_d$ such that, for each function $h \in {\rm BMO}$, $i=1,\dots,d$,
$$
|\langle h,g \nabla_i g\rangle| \leq C_d\|h\|_{\rm BMO}\|\nabla g\|_{2}\|g\|_{2}, \quad \forall\,g \in W^{1,2}.
$$

\end{proposition}

\medskip

(d) The following result on the multipliers in space ${\rm BMO}$ on $\mathbb R^d$ follows from the analysis of Nakai-Yabuta in \cite{NY} (see, in particular, Sect.\,5 in their paper).

\begin{lemma}
\label{NY_lem}
Set $\xi(y):=\frac{y_i}{1+\sigma |y|^2}$ ($\sigma>0$). Then we have
$$\|\xi h \|_{\rm BMO} \leq C_{d,\sigma}\|h\|'_{\rm BMO} \quad \text{$\forall\,h \in {\rm BMO}$},$$
where
$
\|h\|'_{\rm BMO}:={|\langle \mathbf{1}_{B_1(0)}h\rangle|} + \|h\|_{\rm BMO}
$
 is the ${\rm BMO}$-norm.
\end{lemma}

We will apply Lemma \ref{NY_lem} in the proof of Theorem \ref{thm2}, to $\xi(y)=\frac{\nabla_i \rho}{\rho}={-\frac{(d+\epsilon_0)\sigma y_i}{1+\sigma|y|^2}}$, where $\rho$ is the weight introduced above.

\medskip

\section{Classes of singular drifts}
\label{classes_sect}

\subsection{Definitions and examples} The following classes of singular or distributional drifts are covered by our Theorems \ref{thm2}, \ref{thm_div}, \ref{thm2_a}.

\subsubsection{Form-bounded vector fields}

\begin{definition}
\label{def_fbd}
A vector field $b \in [L^2_{\loc}]^d$ is said to be form-bounded with \textit{form-bound} $\delta>0$ (abbreviated as $b \in \mathbf{F}_\delta$) if
\begin{equation}
\label{fbd_cond}
\|b\varphi\|_2^2 \leq \delta \|\nabla \varphi\|_2^2 + c_\delta\|\varphi\|_2^2 \quad \forall\,\varphi \in W^{1,2}
\end{equation}
for some constant $c_\delta<\infty$.
\end{definition}

The constant $c_\delta$ does not affect the well-posedness of SDE \eqref{sde2} (one needs $c_\delta>0$ to include $L^\infty$ drifts). By contrast, $\delta$ is crucial: if $\delta$ exceeds the critical threshold $\delta=4$, then in general SDE \eqref{sde2} ceases to have global in time weak solution (a blow-up occurs). On the other hand, as Theorem \ref{thm2}  shows, whenever $\delta<4$, weak well-posedness holds. Moreover, as $\delta \downarrow 0$, this theory becomes more detailed, e.g.\,one has conditional weak uniqueness, strong well-posedness, etc. The form-bound  $\delta$  appears explicitly in the examples below.

\begin{examples}
\label{ex_fbd}

1.~If $b \in [L^d]^d+[L^\infty]^d$ ($=$ sums of vector fields in $[L^d]^d$ and in $[L^\infty]^d$), then $b \in \mathbf{F}_\delta$ with form-bound $\delta$ that can be chosen arbitrarily small at the expense of increasing the constant $c_\delta$.  Indeed, for every $\epsilon>0$ one can represent $b=b_1+b_2$ with $\|b_1\|_d<\epsilon$ and $\|b_2\|_\infty<\infty$.  By the Sobolev embedding theorem,
\begin{align*}
\|b\varphi\|_2^2 
 & \leq 2\|b_1\|_d^2 \|\varphi\|_{\frac{2d}{d-2}}^2 + 2\|b_2\|_\infty^2 \|\varphi\|_2^2 \\
& \leq C_S 2 \|b_1\|_d^2 \|\nabla \varphi\|_2^2 + 2\|b_2\|_\infty^2 \|\varphi\|_2^2,
\end{align*}
so $b \in \mathbf{F}_{\delta}$ with $\delta=C_S 2\epsilon^2$ and $c_\delta=2\|b_2\|_\infty^2$.

2.~The class of form-bounded vector fields  contains the weak $L^d$ class:
\begin{align*}
\|b\|_{d,\infty} := & \sup_{s>0}s|\{y \in \mathbb R^d \mid |b(y)|>s\}|^{1/d}<\infty \\[2mm]
& \Rightarrow \quad b \in \mathbf{F}_\delta \quad \text{ with } \delta={\biggl(\|b\|_{d,\infty}|B_1(0)|^{-\frac{1}{d}}\frac{2}{d-2}\biggr)^2}
\end{align*}
with $c_\delta=0$.
This was proved in \cite[Lemma 2.7]{KPS}.

3.~The weak $L^d$ class includes itself the Hardy drift:
\begin{equation}
\label{hardy}
b(x)=\pm \sqrt{\delta}\frac{d-2}{2}|x|^{-2}x \in \mathbf{F}_\delta \text{ with } c_\delta=0 \quad (\text{but  }b \not \in \mathbf{F}_{\delta'} \text{ with any } \delta'<\delta, c_{\delta'}<\infty).
\end{equation}
In fact, inclusion \eqref{hardy} is a re-statement of the usual Hardy inequality $$\||x|^{-1}\varphi\|_2^2 \leq \frac{4}{(d-2)^2}\|\nabla \varphi\|_2^2.$$ The plus sign in \eqref{hardy} corresponds in SDE \eqref{sde2} to the attraction towards the origin, the minus corresponds to the repulsion. 

\medskip

4.~The previous example can be refined using the weighted Hardy inequality of Hoffmann-Ostenhof--Laptev \cite{HL}. Fix $$0 \leq \Phi \in L^s(S^{d-1})\quad \text{ for some } s \geq \frac{(d-2)^2}{2(d-1)}+1,$$ where $S^{d-1}$ is the unit sphere in $\mathbb R^d$. If
$$
|b(x)|^2 \leq \delta \frac{(d-2)^2}{4} c\frac{\Phi(x/|x|)}{|x|^2}, \qquad \text{where $c:=\frac{|S^{d-1}|^{\frac{1}{s}}}{\|\Phi\|_{L^s(S^{d-1})}}$},
$$
then $b \in \mathbf{F}_\delta$ with $c_\delta=0$. Using this example, one can e.g.\,cut holes in the drift \eqref{hardy} while still controlling the value of $\delta$.

\medskip

5.~One can also refine example \eqref{hardy} using a Hardy-type inequality of Felli-Marchini-Terracini \cite[Lemma 3.5]{FMT}. Assume that
$$
|b(x)|^2 \leq \delta \frac{(d-2)^2}{4} \sum_{i=1}^\infty \frac{\mathbf{1}_{B_r(a_i)}(x)}{|x-a_i|^2}, \quad x \in \mathbb R^d,
$$ 
where the loci of singularities $\{a_i\}_{i=1}^\infty$ are sufficiently ``spread out'':
$$
\sum_{i=1}^\infty |a_i|^{-d+2}<\infty, \quad \sum_{k=1}^\infty |a_{i+k}-a_i|^{-d+2} \text{ is bounded uniformly in $i$,}
$$
and $|a_i-a_m| \geq 1$ for all $i \neq m$. Then there exists $r$ sufficiently small (so, the singularities are strictly local) so that $b \in \mathbf{F}_\delta$ with $c_\delta=0$.

\medskip

6.~The following simple lemma applies, in particular, to the multi-particle Hardy drift $b:\mathbb R^{dN} \rightarrow \mathbb R^{dN}$ defined by \eqref{drift_b} in Example \ref{ex1_multi}. 

\begin{lemma}[{\cite[Lemma 1]{Ki_multi}}]
\label{particle_lem}
If $K \in \mathbf{F}_\kappa(\mathbb R^d)$, then the drift $b=(b_1,\dots,b_N):\mathbb R^{dN} \rightarrow \mathbb R^{dN}$ with components defined by
\begin{equation*}
b_i(x^1,\dots,x^N):=\frac{1}{N}\sum_{j=1, j \neq i}^N K(x^i-x^j), \quad x^i,x^j \in \mathbb R^d
\end{equation*}
is in $\mathbf{F}_\delta$ with $$\delta=\frac{(N-1)^2}{N^2}\kappa,$$
i.e.\,there is almost equality between $\delta$ and $\kappa$. (In the context of Example \ref{ex1_multi} $\kappa$ is the strength of attraction between the particles.) 
\end{lemma}

{(The lemma is valid in every dimension $d\geq1$. Its application to the Hardy kernel in Example \ref{ex1_multi} requires $d\geq3$.)} 

\medskip

7.~Critical Morrey classes. Every vector field $b \in M_{2+\varepsilon}$ for some $\varepsilon>0$ small, i.e.\,
\begin{equation*}
\|b\|_{M_{2+\varepsilon}}:=\sup_{r>0, x \in \mathbb R^d} r\biggl(\frac{1}{|B_r|}\int_{B_r(x)}|b(y)|^{2+\varepsilon}dy \biggr)^{\frac{1}{2+\varepsilon}}<\infty,
\end{equation*}
is in $\mathbf{F}_\delta$ with $\delta=C(d,\varepsilon){\|b\|_{M_{2+\varepsilon}}^2}$ and $c_\delta=0$. The constant $C=C(d,\varepsilon)$ depends on the constants in some fundamental inequalities of harmonic analysis \cite{F}. 
There exist far-reaching and deep extensions of this inclusion due to Adams \cite{A} (Appendix \ref{adams_sect}) and Chiarenza-Frasca \cite{CF}.

The Morrey class $M_{2+\varepsilon}$ is substantially larger than the weak $L^d$ class, e.g.\,it includes vector fields having strong hypersurface singularities. It is easily seen that the class $M_{2+\varepsilon}$ gets larger as $\varepsilon$ gets smaller. Note, however, that by passing through the Morrey class one to a large extent loses the control over the form-bound $\delta$.

On the other hand, the form-bounded class $\mathbf{F}_\delta$ (with $c_\delta=0$) is contained in the Morrey class $M_2$; this is not difficult to see by selecting cutoff functions as test functions $\varphi$ in the definition of $\mathbf{F}_\delta$. Thus, to summarize,
$$
\cup_{\varepsilon>0}M_{2+\varepsilon} \quad \subsetneq \quad \cup_{\delta>0}\mathbf{F}_\delta\;(\text{with $c_\delta=0$}) \quad \subsetneq \quad M_2.
$$

\medskip

8.~A sub-class of $\mathbf{F}_\delta$ larger than $\cup_{\varepsilon>0}M_{2+\varepsilon}$ was found by Chang-Wilson-Wolff \cite{CWW}, that is, $|b| \in L^2_{\loc}(\mathbb R^d)$ and
\begin{equation}
\label{cww}
\sup_{r>0, x \in \mathbb R^d} \frac{1}{|B_r|}\int_{B_r(x)} |b(y)|^2\, r^2 \xi\big(|b(y)|^2\,r^2 \big) dy<\infty,
\end{equation}
where
$\xi:\mathbb R_+ \rightarrow [1,\infty[$ is a fixed increasing function such that
$$
\int_1^\infty \frac{ds}{s\xi(s)}<\infty.
$$
For instance, one can take $\xi(s)=1+(\log^+s)^{1+\epsilon}$ or $\xi(s)=1+\log^+s (\log\log^+s)^{1+\epsilon}$ for some $\epsilon>0$ (but not $\xi(s)=1+\log^+s$).

The Lebesgue, weak Lebesgue, Morrey, and Chang–Wilson–Wolff classes are all elementary sub‐classes  of $\mathbf{F}_\delta$. We regard estimating an integral over a ball as an ``elementary'' calculation. This is admittedly subjective, but in practice one can carry it out quite easily for a concrete vector field. By contrast, computing an operator norm for a given vector field (or equivalently checking a quadratic‐form inequality) often requires more advanced tools, such as various forms of Hardy's inequality.

9. There are also deep  necessary and sufficient conditions for ensuring that $b \in \mathbf{F}_\delta$, such as the criterion of Kerman-Sawyer \cite{KSa} or the criterion of Mazya \cite[Theorem 2.4.1]{M}. The latter, {for some $\delta<\infty$ with $c_\delta=0$,} is
\begin{equation}
\label{m_ineq}
\langle \mathbf{1}_E |b|^2\rangle \leq K {\rm cap}(E) \quad \forall\,\text{compact } E \subset \mathbb R^d, 
\end{equation}
where, recall,
$$
{\rm cap}(E)=\inf\{{\|\nabla v\|_{2}^2} \mid v \in C_c^\infty, v \geq 1 \text{ on } E\}.
$$
{For general $c_\delta$, one uses the inhomogeneous capacity, replacing $\|\nabla v\|_2^2$ by $\|\nabla v\|_2^2+\|v\|_2^2$ in its definition.}
So, criterion \eqref{m_ineq} requires computing the capacity of an arbitrary compact set (one cannot restrict to dyadic cubes only).
Kerman-Sawyer's inequality that needs to be verified is more complex than \eqref{m_ineq}, but the calculations are confined to dyadic cubes, so it is in some sense more practical. That said, precisely because these are necessary and sufficient conditions, their verifications can be non-trivial (for example, verifying that the Chang-Wilson-Wolff class \eqref{cww} satisfies these conditions, see \cite{CWW}). 
\end{examples}

Efforts to characterize inequalities of type \eqref{fbd_cond}, known as trace inequalities, remain an active research area. The interest was originally motivated by the problems related to estimating the spectrum of Schr\"{o}dinger operators, cf.\,\cite{F,MV}.

Finally, we note that we can combine the previous examples: 
$$
b_1 \in \mathbf{F}_{\delta_1}, \quad b_2 \in \mathbf{F}_{\delta_2} \quad \Rightarrow \quad b_1+b_2 \in \mathbf{F}_\delta, \quad \delta=(\sqrt{\delta_1}+\sqrt{\delta_2})^2.
$$
This extends to series of form-bounded drifts.

\medskip

{In the introduction we also mentioned the class of multiplicatively form-bounded drifts $\mathbf{MF}_\delta$. It will be discussed in Section \ref{div_drifts_sect}.}

\subsubsection{$\text{Bounded mean oscillation}^{-1}$}

\begin{definition}
\label{def_bmo-1}
 A divergence-free vector field $q=(q_i)_{i=1}^d \in [\mathcal S']^d$ is said to be in class $\mathbf{BMO}^{-1}$ if there exist functions $Q^{ij}=-Q^{ji} \in {\rm BMO}(\mathbb R^d)$ such that
$$
q^i=\sum_{j=1}^d \nabla_j {Q^{ji}}, \quad 1 \leq i \leq d,
$$ 
i.e.\,$q=\nabla Q$, where $\nabla$ is the {column-divergence} operator. 
\end{definition}

Examples of  ${\rm BMO}$ functions include
$$g \in L^\infty \quad \text{ or } \quad g(x)=\log|p(x)| \text{ for a polynomial $p$.}$$ In the last example the ${\rm BMO}$ semi-norm does not depend on the coefficients of $p$.

\begin{examples}
1.~A divergence-free vector field $q$ belongs to  $\mathbf{BMO}^{-1}$ if and only if the caloric extensions of its components $q^i$, $i=1,\dots,d$, satisfy
\begin{equation*}
\sup_{x \in \mathbb R^d, R>0}\frac{1}{|B_R|}\int_{[0,R^2] \times B_R(x)} | e^{t\Delta}q^i|^2 dt dy<\infty,
\end{equation*}
where $e^{t\Delta}$ is the heat semigroup,
see \cite{KT}. 

Equivalently, $q^i$ can be characterized as elements of homogeneous Triebel-Lizorkin space ${\dot{F}^{-1}_{\infty,2}}$.

2.~The divergence-free vector fields with entries in Morrey space $M_1$, i.e.\,such that
\begin{equation*}
\langle |q|\mathbf{1}_{B_r(x)}\rangle \leq C r^{d-1}
\end{equation*}
with constant $C$ independent of $r$ or $x \in \mathbb R^d$,
 are in $\mathbf{ BMO}^{-1}$ \cite[Sect.\,3.4.5]{M}.

3.~The divergence-free vector fields with entries in Besov space ${\dot{B}^{-1+d/p}_{p,\infty}}$, ${d<p<\infty}$, are in $\mathbf{ BMO}^{-1}$. This follows by recalling that ${\dot{B}^{-1+d/p}_{p,\infty}}$ consists of tempered distributions $h$ such that
$$
\|e^{t\Delta}h\|_p \leq Ct^{-\frac{1-\frac{d}{p}}{2}}, \quad {t>0},
$$
see \cite{KT} for details.
\end{examples}

In \cite{KT}, Koch and Tataru established, among other results, the existence and the uniqueness of global in time mild solution to Cauchy problem for the 3D Navier-Stokes equations in the critical space $Z$ of functions $v:\mathbb R_+ \times \mathbb R^d \rightarrow \mathbb R^3$ satisfying 
\begin{equation}
\label{X_def}
\|v\|_Z:=\sup_{t \geq 0}t^{\frac{1}{2}}\|v(t)\|_\infty + \biggl(\sup_{x \in \mathbb R^d, R>0}\frac{1}{|B_R|}\int_{[0,R^2] \times B_R(x)} |v(t,y)|^2 dt dy\biggr)^{\frac{1}{2}}<\infty,
\end{equation}
provided that the initial data $v(0) \in \mathbf{BMO}^{-1}(\mathbb R^3)$ have sufficiently small norm. (An even larger $\mathbf{BMO}^{-1}$ type space is considered in the recent paper \cite{CE}.)

\subsection{``Physical'' approximations} We now introduce the classes of  bounded smooth approximations of vector fields in $\mathbf{F}_\delta$ and ${\mathbf{BMO}}^{-1}$ that preserve the structure constants of the latter.

\label{approx_sect}

\begin{definition}
\label{b_n}
Given $b \in \mathbf{F}_\delta$, we denote by $[b]$ the set of sequences of vector fields $\{b_n\} \subset [C_b \cap C^\infty]^d$ such that
\begin{equation}
\label{conv_b_n}
\left\{
\begin{array}{l}
b_n \in \mathbf{F}_\delta \text{ with the same $c_\delta$ as }b, \\
b_n \rightarrow b  \text{ in } [L^2_{\loc}]^d
\end{array}
\right.
\end{equation}
(note that we can always increase $c_\delta$, if needed; what matters is that $c_\delta$ does not depend on $n$).
\end{definition}

For instance, the sequence $\{b_n\}$ defined by $b_n:=E_{\varepsilon_n}b$ (De Giorgi mollifier $E_\varepsilon$ is defined in Section \ref{notations_sect}) for any  $\varepsilon_n\downarrow 0$ is in $[b]$. See \cite{KiS_MAAN} or \cite[Sect.\,6]{Ki_multi} for the proof. 

In the quantum-mechanical context, the form-boundedness condition on potential $|b|^2$ expresses smallness of the potential energy with respect to the kinetic energy in the system described by the Hamiltonian $-\Delta - |b|^2$. Our conditions on regularizations $\{b_n\}$ are thus that they, essentially, do not increase the potential energy.

\medskip

Given a potential $0 \leq V \in L^1_{\loc}$, we write, with some abuse of notation,  $V^{\frac{1}{2}} \in \mathbf{F}_{\delta_+}$  if
$$
\langle V, \varphi^2 \rangle \leq \delta_+\langle |\nabla \varphi|^2\rangle + c_{\delta_+}\langle \varphi^2 \rangle \quad \forall\,\varphi \in C_c^\infty.
$$

\begin{definition}
\label{def4}
Let
 $$
b \in \mathbf{F}_\delta, \quad
 ({\rm div\,}b)^{\frac{1}{2}}_+ \in \mathbf{F}_{\delta_+}, \quad ({\rm div\,}b)_- \in L^1+L^\infty,$$
we denote by $[b]'$ the set of sequences of vector fields $\{b_n\} \subset [C^1_b \cap C^\infty]^d$ that satisfy the previous inclusions with the same constants $\delta$, $c_\delta$, $\delta_+$, $c_{\delta_+}$ (so, independent of $n$) and such that
 $$
 b_n \rightarrow b  \text{ in } [L^2_{\loc}]^d, \quad {\rm div\,}b_n \rightarrow {\rm div\,}b \text{ in } L^1_{\loc}.
 $$
\end{definition}

Once again, we can take e.g.\,$b_n:=E_{\varepsilon_n}b$, see \cite[Sect.\,6]{Ki_multi} for the proof.

\begin{definition} 
\label{q_m}
Given $q=\nabla Q \in \mathbf{BMO}^{-1}$, where $Q$ is the corresponding anti-symmetric matrix field with entries in ${\rm BMO}$ (so, ${\rm div\,}q=0$), we denote by $[q]$ the set of sequences 
$$\{q_m=\nabla Q_m \text{ for anti-symmetric }  Q_m \in  [C^\infty \cap W^{1,\infty}]^{d \times d}\}$$ such that
$$
\left\{
\begin{array}{l}
\|Q_m\|_{{\rm BMO}} \leq C\|Q\|_{{\rm BMO}} \text{ for a constant $C$ independent of $m$}, \\[1mm]
Q_m \rightarrow Q\text{ in } [L^s_{\loc}]^{d \times d} \text{ for any } 1 \leq s<\infty. \\ [1mm]
\end{array}
\right.
$$
\end{definition}
For example, one possible choice is
$$
Q_m:=E_{\varepsilon_m} (Q \wedge U_{\varepsilon_m} \vee V_{\varepsilon_m}),
$$
where the maximum and the minimum are taken componentwise, $$U_\varepsilon:=(-c\log|x|+\varepsilon^{-1}) \wedge \varepsilon^{-1} \vee 0, \quad V_\varepsilon:=(c\log|x|-\varepsilon^{-1}) \wedge 0 \vee (-\varepsilon^{-1}), \quad \varepsilon_m \downarrow 0$$ 
with $c$ chosen so that $\|c\log|x|\|_{{\rm BMO}} \leq \|Q\|_{{\rm BMO}}$. 
The last two functions are compactly supported and are in ${\rm BMO}$. Since ${\rm BMO}$ is a lattice, the components of the matrix fields $Q_m$ (still anti-symmetric) are in ${\rm BMO}$. This regularization of $Q$ was employed earlier in \cite{QX}.

\bigskip

\section{Preliminary discussion: blow-up thresholds for Brownian particles}
\label{particle_sect}

We continue discussing the particle system in Example \ref{ex1_multi}, i.e.
\begin{equation}
\label{syst_m5}
X^i_t=x^i_0-\sqrt{\kappa}\frac{d-2}{2}\frac{1}{N}\sum_{j=1, j \neq i}^N\int_0^t \frac{X_s^i-X_s^j}{|X_s^i-X_s^j|^2}ds + \int_0^t q_0(X_s^i)ds +  \sqrt{2}B^i_t, \quad i=1,\dots,N,
\end{equation}
where $q_0 \in \mathbf{BMO}^{-1}(\mathbb R^d)$, $d \geq 3$. 

\medskip

1. First, we present positive well-posedness results for \eqref{syst_m5} that follow from Theorems \ref{thm2} and \ref{thm_div}.

\medskip

2. Next, we exhibit counterexamples to well-posedness of particle system \eqref{syst_m5} when the attraction strength 
$\kappa$ is too large. On one hand, these counterexamples test the sharpness of Theorems \ref{thm2} and \ref{thm_div}; on the other, they play a crucial role in the proof of Theorem \ref{cor2}, which provides an improved upper bound on the constant in the many-particle Hardy inequality.

\medskip

3. Finally, we discuss the two- and the one-dimensional cases.

\subsection{Positive results}Theorem \ref{thm_div} covers a large portion of the admissible range for the attraction parameter $\kappa$ in \eqref{syst_m5}:
\begin{equation}
\label{k_int}
\kappa \in \bigg[0, 16 \biggl(1 \vee \frac{N}{1+\sqrt{1+\frac{3(d-2)^2}{{2}(d-1)^2}(N-1)(N-2)}} \biggr)^2\bigg[.
\end{equation}
Here, the right endpoint arises from Hoffmann-Ostenhof, Hoffmann-Ostenhof, Laptev and Tidblom's lower bound on the constant in the many-particle Hardy inequality \cite{HHLT}. 
Since the optimal constant in that inequality is not yet known, one expects the true admissible interval for 
$\kappa$ to be strictly larger. {Note that for $N\geq3$ the interval \eqref{k_int} reduces to $\kappa \in [0,16[$ if $\frac{3(d-2)^2}{2(d-1)^2}\geq\frac{N}{N-1}$.} 

By contrast, Theorem \ref{thm2} handles only $\kappa \in [0,4\frac{N^2}{(N-1)^2}[$,  but it offers greater flexibility in modifying the interaction kernel in \eqref{syst_m5}. For example, one can multiply the attracting interaction kernel by any function of $L^\infty$ norm at most one, without altering the assumption on $\kappa$. In particular, one can ``cut holes'' in the interaction kernel so that the particles do not interact along certain directions (cf.\,Examples \ref{ex_fbd}.4).

One can also describe how the particles behave as they approach collision via  estimates on the heat kernel of the corresponding Kolmogorov operator that, necessarily, involve a desingularizing weight, see \cite{BK}. 

\subsection{Counterexamples}These were the positive results for \eqref{syst_m5} that we now balance with counterexamples, i.e.\,analogues of (a), (b) in Section \ref{optim_sect}. For simplicity, assume there is no divergence-free distributional component of the drift (i.e.\,$q_0=0$). 
The right endpoint of the interval \eqref{k_int} lies just below the first blow-up threshold for \eqref{syst_m5}, or the ``non-sticky collisions threshold''. More precisely, following \cite{Fo}, set $$R_t:=\frac{1}{4N}\sum_{i,j=1}^N|X_t^i - X_t^j|^2.$$ It is not difficult to see that $R_t$ is a local squared Bessel process, i.e.
\begin{equation}
\label{Bessel}
R_t=R_0 + 2 \int_0^t {\sqrt{R_s}dW_s} + \mu t, 
\end{equation}
where $W_t$ is a one-dimensional Brownian motion, and 
$$
\mu=(N-1)\left(d-\sqrt{\kappa}\frac{d-2}{4}\right)
$$
is its ``dimension''; the value of $\mu$ controls how often $R_t$ hits zero.
 In what follows, by a collision of particles in \eqref{syst_m5} we mean a collision of all $N$ particles at the same time, i.e.\,when $R_t=0$. So, by standard theory of Bessel processes (see \,\cite[Ch.\,XI, \S 1]{RevuzYor}):

(a') If $\kappa \geq 16(\frac{d}{d-2})^2$, {then the radial process reaches zero in finite time a.s. At equality $\mu=0$ and the squared Bessel process is absorbed at zero; for larger $\kappa$ one has $\mu<0$, so there is no nonnegative continuation after hitting zero. Absorbing the particles at a collision prevents existence of a weak solution to \eqref{syst_m5}. In particular, the particle system has no global weak solution departing from a point where all particles coincide, cf.\,(a) in Section \ref{optim_sect}.}

(b') If $16\big(\frac{d}{d-2}\big)^2\bigl(1-\frac{2}{d(N-1)} \bigr)^{{2}}<\kappa< 16(\frac{d}{d-2})^2$, {then $0<\mu<2$, so the corresponding squared Bessel process visits zero infinitely many times, but $\int_0^T\mathbf{1}_{\{R_t=0\}}dt=0$ a.s.\ for every $T<\infty$. These are the non-sticky collisions.}

Comparing (a'), (b') with the admissible range \eqref{k_int} in Theorem \ref{thm_div}, we see that the latter provides a result that is close to optimal.

\subsection{Two- and one-dimensional cases}For the two-dimensional counterpart of \eqref{syst_m5}, namely, the finite-particle approximation of the Keller-Segel model,
$$
X^i_t=x^i-\sqrt{\varkappa}{\frac{1}{N}\sum_{j=1, j\neq i}^N}\int_0^t \frac{X_s^i-X_s^j}{|X_s^i-X_s^j|^2}ds + \sqrt{2}B^i_t, \quad x=(x^1,\dots,x^N) \in \mathbb R^{2N}, 
$$
the attracting interaction kernel is no longer locally square integrable and therefore is not form-bounded, so one cannot apply our Theorems \ref{thm2}, \ref{thm_div} (see, however, Section \ref{wfb_sect} where we to some extent address $d=2$).
Cattiaux-P\'{e}d\`{e}ches \cite{CP} and Fournier-Jourdain \cite{FJ}, Fournier-Tardy \cite{FT} exploited the special structure of the drift and constructed process $(X_t^1,\dots,X_t^N)$. They work either via a suitable Dirichlet form or by exhibiting a weak solution to the SDE, respectively. In particular, they cover the full critical range $\varkappa \in [0,16[$, where $16$ is the sticky collisions threshold in dimension two. In this context, note:

1) The arguments of \cite{FJ,FT} work in dimensions $d \geq 3$ as well, and allow to handle $16 \leq \kappa<16(\frac{d}{d-2})^2$ in \eqref{syst_m5}, which includes non-sticky collisions, see (b') above.

2) Ohashi-Russo-Teixeira \cite{ORT} {consider Bessel processes in the low-dimensional regime $0<\nu<1$ (i.e.\,non-sticky collisions). They characterize the process as the unique nonnegative solution of the strong martingale problem associated with the formal SDE}
$$
dX_t=\frac{{\nu-1}}{2} \frac{dt}{X_t}+dB_t, \quad X_0=x_0>0, \quad 0<\nu<1,
$$
{whose drift $\frac{\nu-1}{2x}$ is not in $L^1_{\loc}$ and is interpreted as $\frac{\nu-1}{2}\partial_x\log|x|$ in the sense of distributions.} 

It is not yet clear whether  Theorems \ref{thm2}, \ref{thm_div} can be extended in some form to the non-sticky collisions part of the interval of admissible values of $\kappa$. However, taking into account the positive results in 1) and in 2), such extension is conceivable.

We refer to Cattiaux \cite{C} and Fournier \cite{Fo} for recent surveys (among new results) on the Keller-Segel model and its finite particle approximations.

\bigskip

\section{General drifts} 

\label{gen_sect}

In this section we consider SDE
\begin{equation}
\label{sde2}
X_t=x-\int_0^t \bigl(b(X_s) + q(X_s)\bigr)ds + \sqrt{2}B_t, \quad x \in \mathbb R^d, \quad t \geq 0,
\end{equation}
with 
$$
b \in \mathbf{F}_\delta, \quad q \in \mathbf{BMO}^{-1} \;\;(\Rightarrow {\rm div\,}q=0), 
$$
with the form-bound $\delta$ of $b$ going all the way up to (but staying strictly less than) the critical threshold $\delta=4$. 
As mentioned in the introduction, by a result of Mazya and Verbitsky, this assumption on $b$ and $q$ is equivalent to having generalized form-boundedness \eqref{fbd_cond_gen} of $c=b+q$.

We fix bounded smooth approximations $\{b_n\} \in [b]$, $\{q_m\} \in [q]$ {(as in Section \ref{approx_sect})}. Consider the approximating SDEs
\begin{equation}
\label{sde2_approx}
X^{n,m}_t=x-\int_0^t \big(b_n(X^{n,m}_s) + q_m(X^{n,m}_s)\big)ds + \sqrt{2}B_t,
\end{equation}
on a complete probability space $\mathfrak F=(\Omega,\mathcal F,\{\mathcal F_t\}_{t \geq 0},\mathbf{P})$, with $B_t$ being a $\mathcal F_t$-Brownian motion. By classical theory, for every $x \in \mathbb R^d$ and every $n,m=1,2,\dots,$ there exists a unique strong solution $\{X_t^{n,m}\}_{t \geq 0}$ to \eqref{sde2_approx}.
The corresponding Kolmogorov operators
$$
\Lambda(b_n,q_m):=-\Delta + (b_n + q_m) \cdot \nabla, \quad D\big(\Lambda(b_n,q_m)\big)=(1-\Delta)^{-1}C_\infty
$$
generate strongly continuous Feller semigroups on $C_\infty$ such that
$$
e^{-t\Lambda(b_n,q_m)}f(x)=\mathbf E[f(X^{n,m}_t)].$$
Set
 $\mathbb P_x^{n,m}:=\mathbf{P} (X_t^{n,m})^{-1}$.

\begin{theorem}
\label{thm2} Let $d \geq 3$.
Assume that $b$ and $q$ are, respectively, Borel measurable and distribution-valued vector fields $\mathbb R^d\rightarrow \mathbb R^d$  that satisfy 
\begin{equation}
\label{cond}
\left\{
\begin{array}{l}
b \in \mathbf{F}_\delta \text{ with } \delta<4, \\
q \in \mathbf{BMO}^{-1}.
\end{array}
\right.
\end{equation}
Let $\{b_n\} \in [b]$, $\{q_m\} \in [q]$ as in Definitions \ref{b_n} and \ref{q_m}. The following are true:

\begin{enumerate}[label=(\roman*)]

\item {\rm (Feller semigroup)} The limit 
$$
s\mbox{-}C_\infty \mbox{-}\lim_n \lim_m e^{-t\Lambda(b_n,q_m)} \text{ (loc.\,uniformly in $t \geq 0$)}
$$
exists and determines a strongly continuous Feller semigroup, say, $e^{-t\Lambda} = e^{-t\Lambda(b,q)}$. (The order in which we take the limits is essential.) The generator $\Lambda$ of $e^{-t\Lambda}$ is an operator realization of the formal differential expression $-\Delta + (b+q) \cdot \nabla$ in $C_\infty$.

\medskip

\item {\rm (A relaxed approximation uniqueness)} The limit in {(\textit{i})} does not depend on the choice of $\{b_n\} \in [b]$ and $\{q_m\} \in [q]$. 
Furthermore, assume $\delta<1$ and  $\{b_n\} \in [b] \cap [L^2]^d$, but let the strong convergence in the definition of $[b]$ be replaced with the weak convergence
$$
b_n \overset{w}{\rightarrow} b \quad \text{in $[L^2]^d$}.
$$
Then the approximating Feller semigroups still converge to the same limit from (\textit{i}):
$$
e^{-t\Lambda(b,q)}=s\mbox{-}C_\infty \mbox{-}\lim_n \lim_m e^{-t\Lambda(b_n,q_m)} \quad \text{(loc.\,uniformly in $t \geq 0$)}.
$$
Thus, when $\delta<1$, we can extend the class of admissible approximations of $b$.

\medskip

\item {\rm (Generalized martingale solution)} There exists a strong Markov family of probability measures $\{\mathbb P_x\}_{x \in \mathbb R^d}$ on the {canonical space of trajectories $\mathbf C$}  such that 
$$
e^{-t\Lambda(b,q)}f(x)=\mathbb E_{\mathbb P_x}[f(\omega_t)], \quad f \in C_\infty, \quad x \in \mathbb R^d,\;t \geq 0,$$
$$
\mathbb P_x=w\mbox{-}\mathcal P(\mathbf C)\mbox{-}\lim_n\lim_m \mathbb P_x^{n,m},
$$
and for every test function $v$ in the domain $D\big(\Lambda(b,q)\big)$,  a dense subspace of $C_\infty$, the process
$$
t \mapsto v(\omega_t)-v(x) + \int_0^t \Lambda(b,q) v(\omega_s) ds
$$
is a continuous martingale under $\mathbb P_x$. 

\smallskip

Selecting test functions $v$ from the domain of the Feller generator allows to address the problem of defining the term $\int_0^t q(X_s)\cdot \nabla v(X_s)ds$ in the martingale problem.

\medskip
\medskip

\item {\rm (Weak solution for disperse initial data)} Let $\delta<1$ and let us also assume that $b$, $q$, $b_n$, $q_m$ have supports in a ball of fixed radius. Given an initial (smooth) probability density $\nu_0$ satisfying $\langle \nu_0^{2r}\rangle<\infty$ for some $1<r<\frac{1}{\sqrt{\delta}}$,
there exist a probability space $\mathfrak F'=(\Omega',\mathcal F',\{\mathcal F'_t\}_{t \geq 0},\mathbf{P}')$ and a continuous process $X_t$ on this space such that the limit
$$
A_t:=L^2(\Omega')\mbox{-}\lim_n\lim_m \int_0^t \big(b_n(X_s)+q_m(X_s)\big)ds
$$
exists, and we have a.s.
$$
X_t=X_0-A_t + \sqrt{2}B'_t, \quad t>0,
$$ 
for a $\mathcal F_t'$-Brownian motion {$B'_t$}, for $\mathbf{P}' X_0^{-1}$ having density $\nu_0$. (The compact support assumption can be removed with a few additional efforts at expense of requiring $\langle \nu_0^{2r} \rho^{-\alpha}\rangle<\infty$, where the weight $\rho$ is defined by \eqref{rho_def}, for appropriate $\alpha>0$, see Remark \ref{nu_rem}.)

\medskip

If $q=0$, then more can be said:

\medskip

\item {\rm (Classical weak solution)} For every $x \in \mathbb R^d$, 
$$
\mathbb E_{\mathbb P_x}\int_0^{{t}}|b(\omega_s)|ds<\infty, \quad t \geq 0,
$$
and, for every test function $v \in C_c^2$, the process
$$
t \mapsto v(\omega_t)-v(x) + \int_0^t (-\Delta + b \cdot \nabla) v(\omega_s) ds
$$
is a continuous martingale with respect to $\mathbb P_x$. Moreover, 
$$
B_t(\omega):=\frac{1}{\sqrt{2}}\bigg(\omega_t-x+\int_0^t b(\omega_s)ds\bigg), \quad t \geq 0,
$$
is a Brownian motion, so we have a weak solution to SDE \eqref{sde2}.

\medskip

\item {\rm (Strong solvability)} {Let $\delta<\frac{C}{d^2}$ for a constant $C$ that, in principle, can be written explicitly.} The strong solutions $X_t^n$ of the approximating SDEs
$$
X^n_t=x-\int_0^t b_n(X^n_s)ds + \sqrt{2}B_t,
$$
considered on a fixed complete probability space, converge a.s.\,after passing to a subsequence independent of $x$, to a strong solution $X_t$ to the SDE
\begin{equation}
\label{sde_b}
X_t=x-\int_0^t b(X_s)ds + \sqrt{2}B_t
\end{equation}
(so, $X=\Phi(B)$ for a Borel measurable function $\Phi$; the solution depends only on the Brownian motion).

\medskip

\item {\rm (Another kind of approximation uniqueness)}
Let {$\delta<1$}. Let $\{\mathbb Q_x\}_{x \in \mathbb R^d}$ be a family of solutions to the martingale problem in (\textit{v}) that are constructed via approximation, i.e.\,are such that
$$
\mathbb Q_x=w{\mbox-}\mathcal P(\mathbf{C})\mbox{-}\lim_n \mathbb P_x(\tilde{b}_n) \quad \text{for every $x \in \mathbb R^d$},
$$
where $\tilde{b}_n \in \mathbf{F}_\delta \cap [C_b \cap C^\infty]^d$ with $c_\delta$ independent of $n$. Then $$\{\mathbb Q_x\}_{x \in \mathbb R^d}=\{\mathbb P_x\}_{x \in \mathbb R^d},$$ where $\{\mathbb P_x\}_{x \in \mathbb R^d}$ were constructed in (\textit{iii}). Here we do not require any convergence of $\tilde{b}_n$ to $b$.

\medskip

\item {\rm (Stochastic transport equation)} Let $b \in \mathbf{F}_\delta$ with $\delta<(1+4rd)^{-2}$ for a given $r=1,2,\dots$. Then, for every $f \in W^{1,4r}$, there exists a unique weak solution to Cauchy problem for the stochastic transport equation
\begin{equation}
\label{ste}
dv+b \cdot \nabla vdt + \sqrt{2}\nabla v \circ dB_t=0, \quad v|_{t=0}=f, 
\end{equation}
with $\circ$ denoting the Stratonovich {operator}. It 
satisfies
\begin{equation*}
\sup_{0 \leq \alpha \leq 1}\bigl\|\mathbb E|\nabla v|^{2r}\bigr\|_{L^{\frac{2}{1-\alpha}}([0,t],L^{\frac{2d}{d-2+2\alpha}})} \leq 
C_1e^{C_2t}\|\nabla f\|^{2r}_{4r}.
\end{equation*}
In particular,  if $2r>d$, then by the Sobolev embedding theorem for a.e. $\omega \in \Omega$ the function $x \mapsto v(t,x,\omega)$ is H\"{o}lder continuous, possibly after 
modification on a set of measure zero in $\mathbb R^d$ (in general, depending on $\omega$).
\end{enumerate}

\end{theorem}

 The novelty of Theorem \ref{thm2} is in assertions (\textit{i})-(\textit{iv}) {and in the extension of assertion (\textit{vii}) to $\delta<1$}. Assertions (\textit{v})-(\textit{viii}) were proved in \cite{KiS_sharp}, \cite{KiM_strong}, \cite{KiS_Osaka}, \cite{KiSS_transport}, respectively, with the additional restriction $\delta<\frac{4}{(d-2)^2}$ in assertion (\textit{vii}).  The strong solvability (\textit{vi}) is proved via a strengthening of the approach of R\"{o}ckner-Zhao.

 \medskip

 Assertion (\textit{vii}) is a form of approximation uniqueness that does not require convergence of the approximating drifts to $b$, provided that the limiting laws solve the martingale problem for $b$. Its proof in Section \ref{proof_thm2_sect} uses the a priori H\"older continuity of Proposition \ref{thm_holder} in place of the gradient estimates, which is what allows to improve the condition on the form-bound to $\delta<1$. 
 
 \medskip

In the case when the distributional component $q=0$ and the form-bound $\delta$ of $b$ is sufficiently small, one can furthermore prove uniqueness of the constructed weak or strong solutions among solutions satisfying reasonable occupation time estimates, i.e.\,Krylov-type bounds. For details, see \cite{KiM_JDE}, \cite{Ki_Morrey} regarding conditional weak uniqueness and \cite{KiM_strong} regarding conditional strong uniqueness.

\medskip

In the case $q$ is non-trivial, we do not establish conditional uniqueness or strong solvability. Nevertheless, already the existence of the Feller semigroup and its approximation uniqueness cover the following rather common scenario for
singular SDE: one mollifies a singular drift and studies the corresponding SDE as an approximation of the true dynamics, for instance, when investigating long‐term behavior such as sub‐ or super‐diffusivity, see e.g.\,\cite{ABK}. Although the mollified drift is bounded and smooth, and so the corresponding SDE is well-posed, one still must ensure that:

-- the resulting stochastic dynamics does not depend on the choice of the regularization of the drift (in particular, on the choice of the mollifier);

-- in the limit one still gets a diffusion process.

The former is addressed by assertion (\textit{ii}). The latter is addressed by assertions (\textit{i}), (\textit{iii}), (\textit{iv}).

\begin{example} 
\label{ex2_multi}
Let us return to the problem of describing the dynamics $
X_t=(X_t^1,\dots,X_t^N)$ of $N$ interacting particles immersed in a velocity field in $\mathbf{BMO}^{-1}$, $d \geq 3$. That is, we are in the setting of Example \ref{ex1_multi}, where, recall,
\begin{equation}
\label{syst_m}
X_t=x_0-\int_0^t \bigl(b(X_s) + q(X_s)\bigr)ds + \sqrt{2}B_t, \quad x_0=(x_0^1,\dots,x_0^N) \in \mathbb R^{dN}, \quad B_t=(B_t^1,\dots,B_t^N),
\end{equation}
and $b(x)=(b^1(x),\dots,b^N(x))$, $q(x)=(q_0(x^1),\dots,q_0(x^N))$ ($ x=(x^1,\dots,x^N) \in \mathbb R^{dN}$),
\begin{equation*}
b^i(x^1,\dots,x^N):=\frac{1}{N}\sum_{j=1, j \neq i}^N \sqrt{\kappa}\frac{d-2}{2}\frac{x^i-x^j}{|x^i-x^j|^2}, \quad q_0 \in \mathbf{BMO}^{-1}(\mathbb R^d).
\end{equation*}
Then, by Lemma \ref{particle_lem}, $b \in \mathbf{F}_\delta(\mathbb R^{dN})$ with $\delta=\frac{(N-1)^2}{N^2}\kappa$. Also, as  mentioned in the introduction, $q \in \mathbf{BMO}^{-1}(\mathbb R^{dN})$. Therefore, if the strength of attraction between the particles $\kappa$ satisfies 
\begin{equation}
\label{kappa_hyp_int}
\tag{$\kappa_{\rm hyp}$}
\kappa<4\frac{N^2}{(N-1)^2}
\end{equation}
(so that $\delta<4$), then Theorem \ref{thm2} applies and ensures the existence and the approximation uniqueness for this particle system. Importantly, the assumption on $\kappa$ basically does not depend on the number of particles $N$ (assumed to be large). 

In the case $d=2$, which is of interest e.g.\,in the Keller-Segel model, the drift $b:\mathbb R^{2N} \rightarrow \mathbb R^{2N}$ defined above is not form-bounded. However, it is weakly form-bounded, which still allows us to say something about the corresponding SDE \eqref{syst_m}, see Section \ref{wfb_sect}.
\end{example}

\begin{remark}[On strong solutions] There is a well known link between the stochastic transport equation \eqref{ste} and the SDE
\begin{equation}
\label{sde_st}
X_t=x{+}\int_0^t b(X_r)dr+\sqrt{2}B_t.
\end{equation}
Namely, when $b$ is bounded and smooth, the solution $v$ to the STE \eqref{ste} can be represented as
\begin{equation}
\label{flow}
v(t)=f(\Psi_{t}^{-1}), \quad t \geqslant 0,
\end{equation}
where $\Psi_{t}:\mathbb R^d \times \Omega \rightarrow \mathbb R^d$ is the stochastic flow for the SDE \eqref{sde_st} {on the canonical Wiener space. Let $\theta_s$ be the Wiener shift, so that $B_t(\theta_s\omega)=B_{t+s}(\omega)-B_s(\omega)$. Then} there exists $\Omega_0 \subset \Omega$, $\mathbb P(\Omega_0)=1$, such that, for all $\omega \in \Omega_0$,
{$\Psi_{t}(\cdot,\theta_s\omega)\circ\Psi_{s}(\cdot,\omega) = \Psi_{t+s}(\cdot,\omega)$}, $\Psi_{0}(x,\omega)=x$, and

1) for every $x \in \mathbb R^d$, the process $t \mapsto \Psi_{t}(x,\omega)$ is a strong solution to \eqref{sde_st},

2) $\Psi_{t}(x,\omega)$ is continuous in $(t,x)$, $\Psi_{t}(\cdot,\omega):\mathbb R^d \rightarrow \mathbb R^d$ are homeomorphisms  and $\Psi_{t}(\cdot,\omega)$, $\Psi_{t}^{-1}(\cdot,\omega) \in C^\infty(\mathbb R^d,\mathbb R^d)$.

Beck-Flandoli-Gubinelli-Maurelli \cite{BFGM} reversed this connection when $b$ is singular (for time-homogeneous drifts their condition reads as $|b| \in L^d+L^\infty$). They used the stochastic transport equation \eqref{ste} to construct, for a.e.\,initial point $x \in \mathbb R^d$, a strong solution to SDE \eqref{sde_st}. Having Theorem \ref{thm2}(\textit{viii}), one can extend the argument of  \cite{BFGM}  to $b \in \mathbf{F}_\delta$, see \cite[Remark 1]{KiSS_transport}.
{Applying the same argument to $-b$, which has the same form-bound as $b$, gives a strong solution to \eqref{sde_b} for a.e.\,initial point.}
However, since this approach excludes a measure zero set of initial points, it does not imply Theorem \ref{thm2}(\textit{vi}).

Regarding recent progress on strong solutions of SDEs with singular drifts, we also refer to Krylov \cite{Kr4} who develops a different approach based on It\^{o}-Duhamel series that can be viewed, to some extent, as the Duhamel series for the stochastic transport equation.

\end{remark}

\begin{remark}[Krylov-type bound in the distributional setting]
\label{krylov_distr_rem}
Consider the assumptions of Theorem \ref{thm2}.
Then the following a priori Krylov-type bound holds. 
Let $U \in \mathbf{F}_{\delta_1}$, $\delta_1<\infty$, be a form-bounded function, i.e.\,$U \in L^2_{\loc}$ and 
\begin{equation}
\label{fbd_pot}
\langle U^2,\varphi^2 \rangle \leq \delta_1 \langle |\nabla \varphi|^2\rangle + c_{\delta_1}\langle \varphi^2 \rangle, \quad \varphi \in W^{1,2}.
\end{equation}
 Let $W={\rm div\,}w$ for some vector field $w$ whose components lie in ${\rm BMO}$. Fix some smooth approximations $\{U_n\} \in [U]$ as in Section \ref{approx_sect}. Let $w_m \in [C^\infty \cap W^{1,\infty}]^d$ satisfy $\sup_m\|w_m\|_{{\rm BMO}}<\infty$ and $w_m \rightarrow w$ in $[L^s_{\loc}]^d$ for every $1 \leq s<\infty$, and put $W_m={\rm div\,}w_m$. Fix $1<\theta<\frac{d}{d-2}${, put $\theta':=\theta/(\theta-1)$, and fix} $p \geq 2$ such that $p>\frac{2}{2-\sqrt{\delta}}$. Then, for all $f \in C_c^\infty$,
\begin{equation}
\label{krylov_distr}
\sup_{x \in \mathbb R^d}\biggl|\,\mathbb E_{\mathbb P_x}\int_{0}^t (U_n+W_m)(\omega_s)f(\omega_s) ds\, \biggr| \leq K  \|A_m\|_{p\theta} \vee \|A_m\|_{p\theta'},  \quad t \in [0,T],
\end{equation}
where $$A_m=|w_m| |\nabla f|+(1+|w_m|)|f|,$$
and constant $K$ does not depend on $n$, $m$ or $f$. Informally, this bound shows that a solution of SDE \eqref{sde2} cannot spend too much time near the singularities of {$U$} and $W$. {The proof is obtained using Proposition \ref{emb_thm}, which is an elliptic counterpart of \eqref{krylov_distr}, upon representing the expectation via the semigroup and writing
\begin{align*}
\int_0^t e^{-s\Lambda(b,q)}(U_n+W_m)f ds & =\int_0^t (\mu+\Lambda(b,q)) e^{-s\Lambda(b,q)}R_\mu(U_n+W_m)f ds \\
& = {R_\mu(U_n+W_m)f-e^{-t\Lambda(b,q)}R_\mu(U_n+W_m)f} \\
& + \mu\int_0^t e^{-s\Lambda(b,q)}R_\mu(U_n+W_m)f ds, \qquad R_\mu:=(\mu+\Lambda(b,q))^{-1},
\end{align*}
so it remains to apply $L^\infty$ contractivity of semigroup $e^{-t\Lambda(b,q)}$ to arrive at
$$
\left\|\int_0^t e^{-s\Lambda(b,q)}(U_n+W_m)f ds \right\|_\infty \leq (2+\mu t)\|(\mu+\Lambda(b,q))^{-1}(U_n+W_m)f\|_\infty. 
$$
To arrive at the sought Krylov-type bound under the assumptions of Theorem \ref{thm2}, we take as  $U$ and $W$ the components of vector fields $b$ and $q$.}

By running parabolic De Giorgi's iterations, one refines  \eqref{krylov_distr} to
\begin{equation}
\label{t9}
\biggl|\,\mathbb E_{\mathbb P^{n,m}_x}\int_{0}^\varepsilon ({U_n}+W_m)(\omega_s)f(\omega_s) ds\, \biggr| \leq {H_f(\varepsilon)},
\end{equation}
{where, for each fixed $f\in C_c^\infty$, $H_f(\varepsilon) \downarrow 0$ ($\varepsilon \downarrow 0$) is independent of $n$, $m$ and $x \in \mathbb R^d$.} If we could place the absolute value  under the integral, then, after taking $U=b^i$ and $W=q^i$, a standard argument would allow to conclude tightness of $\{\mathbb P_x^{n,m}\}$. However, since $W$ is a distribution, we cannot do this. Still, an argument of Hao-Zhang \cite{HZ} shows that one can conclude tightness of $\{\mathbb P_x^{n,m}\}$ from, basically, \eqref{t9}, by applying It\^{o}'s formula to $\sqrt{\sigma+|x-x_0|^2}$ with $\sigma>0$ small, which allows to control the smallness of the increments of solutions of the approximating SDEs after taking $\sigma \downarrow 0$.
That said, \cite{HZ} need a tightness argument since they are dealing with divergence-free super-critical drifts, while we are dealing with general critical drifts and obtain stronger convergence results for $\{\mathbb P_x^{n,m}\}$ provided by the theory of Feller semigroups. 
\end{remark}

\begin{remark}[Dispersion estimate, local maximum principle and gradient bounds] 
\label{disp_rem}

We now make a few remarks regarding the theory of the Kolmogorov operator $-\Delta + (b+q)\cdot \nabla$ behind SDE \eqref{sde2}.

1.~Under the assumptions of Theorem \ref{thm2}(\textit{i}), we can descend
from $C_\infty$ to $L^p$ and show that for every $p>\frac{2}{2-\sqrt{\delta}}$ the operators
$$
e^{-t\Lambda_p(b,q)}:=\biggl[e^{-t\Lambda(b,q)} \upharpoonright C_\infty \cap L^p \biggr]^{\rm clos}_{L^p \rightarrow L^p}
$$
are bounded on $L^p$ and constitute a strongly continuous semigroup.
Moreover, for all $\frac{2}{2-\sqrt{\delta}} < p \leq r < \infty$, 
\begin{equation*}
\|e^{-t\Lambda_p(b,q)}f\|_{r} \leq C_{\delta,d}e^{\omega_{p} t}t^{-\frac{d}{2}(\frac{1}{p}-\frac{1}{r})}\|f\|_p, \quad f \in L^p, \quad \omega_p=\frac{c_\delta}{2(p-1)}.
\end{equation*}
The latter and the Dunford-Pettis theorem yield that $e^{-t\Lambda(b,q)}$, $t>0$, are integral operators. 

If additionally $\delta<1$, then 
$v(t):=e^{-t\Lambda_2(b,q)}f$, $f \in L^2$, is the unique weak solution to Cauchy problem 
$$
(\partial_t-\Delta + (b+q) \cdot \nabla)v=0, \quad v|_{t=0}=f,
$$
in the standard Hilbert triple  $W^{1,2} \hookrightarrow L^2 \hookrightarrow W^{-1,2}$.

The proof of this dispersion estimate uses Nash's argument, see e.g.\,\cite[proof of Theorem 4.2]{KiS_theory}. The semigroup is constructed in the proof of Proposition \ref{thm_conv} in Section \ref{thm_conv_sect}.
% The construction of the semigroup in $L^p$ follows closely \cite[proof of Theorem 4.2]{KiS_theory}. In fact, we basically construct this semigroup in the proof of Proposition \ref{thm_conv}. 
The proof of the uniqueness of the weak solution follows the classical Lions's argument and the compensated compactness estimate (Proposition \ref{cc_lem}), see \cite{QX} for details.

2.~Also under the assumptions of Theorem \ref{thm2}(\textit{i}), for every $f \in L^{p\theta} \cap L^{p\theta'}$, $u:=(\mu+\Lambda(b,q))^{-1}f$ satisfies for each $x \in \mathbb R^d$ the following local maximum principle
\begin{equation*}
\sup_{B_{\frac{1}{2}}(x)}|u|  \leq K \biggl( \langle |f|^{p\theta}\rho_x\rangle^{\frac{1}{p\theta}}  +\langle |f|^{p\theta'}\rho_x\rangle^{\frac{1}{p\theta'}} \biggr), \quad \mu>\mu_0>0,
\end{equation*}
for fixed $1<\theta<\frac{d}{d-2}$ and $p \geq 2$ such that $p>\frac{2}{2-\sqrt{\delta}}$. Here $\rho_x(y)=(1+\sigma |y-x|)^{-\frac{d}{2}+}$ for some $\sigma>0$. The constants $K$ and $\mu_0$ do not depend on $f$ or $x$. This is the content of Proposition \ref{sep_thm}.

\smallskip

3.~Let now $q=0$. 

Assume that $\delta<\frac{4}{(d-2)^2} \wedge 1$. 
 Then the unique weak solution $u$ to the elliptic equation $$(\mu-\Delta + b \cdot \nabla)u=f$$ satisfies, for every $r \in [2,\frac{2}{\sqrt{\delta}}[$,
\begin{equation}
\label{i}
\tag{\cite{KS}}
\|\nabla u\|_r\leq K_1(\mu-\mu_0)^{-\frac{1}{2}}\|f\|_r,\quad {\|\nabla |\nabla u|^{\frac{r}{2}} \|_2^{\frac{2}{r}}}  \leq K_2(\mu-\mu_0)^{-\frac{1}{2}+\frac{1}{r}}\|f\|_r, 
\end{equation}
\begin{equation}
\label{ii}
\tag{\cite{Ki_revisited}}
\|(\mu-\Delta)^{\frac{1}{2}+\frac{1}{s}}u\|_r \leq K\|(\mu-\Delta)^{-\frac{1}{2}+\frac{1}{\ell}}f\|_r, \quad \text{ for all } 2 \leq \ell<r<s
\end{equation}
for all $\mu$ greater than some generic constant $\mu_0$.

The corresponding parabolic gradient bounds \cite{KiS_note} impose more restrictive conditions on $\delta$.
Namely, assume that form-bound $\delta$ satisfies, for some $r=d+\varepsilon$ (with this choice of $r$ the Sobolev embedding theorem will give H\"{o}lder continuity of solution)
\[
\sqrt{\delta}<\left\{
\begin{array}{ll}
\big(\sqrt{r-1}-\frac{r-2}{2}\big)\frac{2}{r} & \text{ in dimensions } d=3,4,\\
(1-\mu)\frac{r-1}{r-2}\frac{1}{r} & \text{ in dimensions } d\geq 5,
\end{array}
\right.
\]
where $0<\mu<1$, $16\mu>(1-\mu)^4\frac{(r-1)^2}{(r-2)^4}$. (For instance, these assumptions on $\delta$ are satisfied if $\delta<\frac{1}{d^2}$, $d \geq 3$.) Then the unique weak solution $v$ to Cauchy problem $$(\partial_t-\Delta + b\cdot \nabla)v=0, \quad v|_{t=0}=f,$$ satisfies
\begin{align}
\label{grad_parab}
\sup_{0 \leq s \leq t } \|\nabla v(s)\|_r^r &  + C_1\int_0^t\||\nabla v|^\frac{r-2}{2}  \partial_s v\|_2^2 ds + C_2\int_0^t \langle |\nabla |\nabla v|^{\frac{r}{2}}|^2 \rangle ds  \leq e^{C_3t} \|\nabla f\|_r^r
\end{align}
for constants $C_i>0$ ($i=1,2,3$) that depend only on $d$, $\delta$ and $c_\delta$.

\end{remark}

\leavevmode{
\begin{remark}
\label{girsanov_rem}
The construction of the drift-diffusion via  Girsanov transform is not applicable to form-bounded drifts. Indeed, already for the Hardy drift 
$
b(x)=\sqrt{\delta}\frac{d-2}{2}\frac{x}{|x|^2}, 
$
which is in $\mathbf{F}_\delta$,
one has, regardless of how small $\delta>0$ is, for $W_s$ being a Brownian motion starting at the origin:
$$
\int_0^T|b(W_s)|^2\,ds
=\delta\frac{(d-2)^2}{4}\int_0^T\frac{ds}{|W_s|^2}=\infty \quad\text{a.s.},\quad T>0.
$$
Hence the It\^{o} integral
$
\int_0^T b(W_s) dW_s
$
appearing in the Girsanov exponential is not defined for any $T>0$.
\end{remark}
}

\begin{remark}[On the threshold $\delta=4$]
\label{delta_4_rem}
1.~In the construction of the Feller semigroup in Theorem \ref{thm2}, we run De Giorgi's method in $L^p$ for 
\begin{equation}
\label{p_cond}
p>\frac{2}{2-\sqrt{\delta}}
\end{equation}
(and so we need $\delta<4$; the counterexamples discussed in the introduction show that  $\delta<4$ is sharp at least in high dimensions). 
The condition \eqref{p_cond} comes from the following elementary calculation for the Kolmogorov backward equation. Let $b\in\mathbf F_\delta$ and $q \in \mathbf{BMO}^{-1}$ be additionally bounded and smooth so that the manipulations with the equations are justified, but the constants in the estimates will not depend on the smoothness of $b$ and $q$. Let us also assume for simplicity that $c_\delta=0$ (if not, then we need to add a constant term in the Kolmogorov equation to absorb $c_\delta>0$). Consider Cauchy problem
 $$(\partial_t - \Delta +(b+q) \cdot \nabla)v=0, \quad v|_{t=0}=v_0 \in C_c^\infty.$$ Without loss of generality, $v_0 \geq 0$, and so $v \geq 0$. Multiply equation by $v^{p-1}$ and integrate by parts. Since ${\rm div\,}q=0$, one finds
$$
\frac{1}{p}\langle \partial_t v^p\rangle + \frac{4(p-1)}{p^2}\langle |\nabla v^{\frac{p}{2}}|^2 \rangle + \frac{2}{p}\langle b \cdot \nabla v^{\frac{p}{2}},v^{\frac{p}{2}}\rangle=0,
$$
or, equivalently,
$$
\partial_t\langle  v^p\rangle + \frac{4(p-1)}{p}\langle |\nabla v^{\frac{p}{2}}|^2 \rangle =- 2\langle b \cdot \nabla v^{\frac{p}{2}},v^{\frac{p}{2}}\rangle.
$$
Applying the Cauchy-Schwarz inequality in the last term gives
$$
 \langle \partial_t v^p\rangle + \frac{4(p-1)}{p}\langle |\nabla v^{\frac{p}{2}}|^2 \rangle \leq 2\biggl(\alpha \langle |b|^2,v^p \rangle + \frac{1}{4\alpha} \langle |\nabla v^{\frac{p}{2}}|^2 \rangle\biggr).
$$
Now, applying $b \in \mathbf{F}_\delta$ and selecting $\alpha=\frac{1}{2\sqrt{\delta}}$, we obtain 
\begin{equation}
\label{en_ineq}
\langle \partial_t v^p\rangle + \biggl[\frac{4(p-1)}{p}-2\sqrt{\delta} \biggr]\langle |\nabla v^{\frac{p}{2}}|^2 \rangle \leq 0.
\end{equation}
To keep the dispersion term positive, one needs $\frac{4(p-1)}{p}-2\sqrt{\delta}> 0$, i.e.\,$p > \frac{2}{2-\sqrt{\delta}}$.
This calculation reappears (in slightly different form, e.g.\,for sub-solutions of the Kolmogorov equation) in the proof of Theorem \ref{thm2}(\textit{i}). 

The observation that one should work in $L^p$, $p>\frac{2}{2-\sqrt{\delta}}$, was made already in \cite{KS} in the context of the $L^p$ semigroup theory of $-\Delta + b \cdot \nabla$, $b \in \mathbf{F}_\delta$. See also \cite{CPZ, JL} where the authors use energy methods in $L^p$ to study regularity of solutions of the elliptic-parabolic Keller-Segel model of chemotaxis (i.e.\,of the corresponding McKean-Vlasov PDE), although the optimal choice of $p$ is not really discussed in these papers. 

2.~It is interesting to abstract away the previous calculation to see  if there are other  test functions that allow for a similar energy analysis of Cauchy problem $(\partial_t - \Delta +b \cdot \nabla)v=0$, $v|_{t=0}=v_0$, where $0 \leq v_0 \in C_c^\infty$. We seek test functions $\varphi(v)=\varphi(v(t))$ such that $\varphi \geq 0$, $\varphi' \geq 0$, and so that there exists a ``conjugate'' function $\psi$ satisfying $\psi(0)=0$ and
\begin{equation}
\label{phi_psi}
\psi'=\sqrt{\varphi'}, \quad \varphi=a\psi \psi'
\end{equation}
for some constant $a>0$ to be chosen (hence $\psi \geq 0$). Multiplying the parabolic equation by $\varphi(v)$ and integrating by parts, we obtain
\begin{align*}
 \langle \partial_t v,\varphi(v)\rangle& =a\langle \partial_t v,\psi(v)\psi'(v)\rangle = \frac{a}{2}\langle \partial_t (\psi(v))^2\rangle, \\
\langle -\Delta v,\varphi(v)\rangle & = \langle \nabla v,\varphi'(v)\nabla v\rangle = \langle 
|\nabla \psi(v)|^2\rangle, \\[2mm]
\end{align*}
and
\begin{align*}
\langle b \cdot \nabla v,\varphi(v)\rangle & = \langle b \cdot \nabla\psi(v),a\psi(v)\rangle \\
& \leq a\biggl(\alpha \langle |b|^2,(\psi(v))^2\rangle + \frac{1}{4\alpha}\langle |\nabla \psi(v)|^2 \rangle \biggr) \\
&  \text{(take $\alpha=\frac{1}{2\sqrt{\delta}}$ and apply $b \in \mathbf{F}_\delta$ (with $c_\delta=0$))} \\
& \leq a \sqrt{\delta}\langle |\nabla \psi(v)|^2 \rangle.
\end{align*}
Thus, we obtain an energy inequality of the form
$$
\frac{a}{2}\partial_t \langle (\psi(v))^2\rangle + \bigl(1-a\sqrt{\delta}\bigr)\langle |\nabla \psi(v)|^2 \rangle \leq 0.
$$
In order for the dispersion term to remain non-negative, we need to take in \eqref{phi_psi} $a=\frac{1}{\sqrt{\delta}}$ (or smaller, but the equality is least restrictive on $\varphi$). So far, no constraint on $\delta$ has appeared, but it will appear once we  solve the resulting (from \eqref{phi_psi}) system 
$$
\left\{
\begin{array}{l}
\psi'=\sqrt{\varphi'}, \\
\varphi=\frac{\psi \psi'}{\sqrt{\delta}}
\end{array}
\right. \qquad \Rightarrow \qquad \varphi'=\frac{1}{\sqrt{\delta}}(\psi\psi''+(\psi')^2),
$$
so that $\psi$ must satisfy
$\psi\psi''=(\sqrt{\delta}-1)(\psi')^2$.
Assuming $\psi>0$ for $v>0$, we can reduce order to obtain
\begin{equation}
\label{ode1}
\psi'=C\psi^{\sqrt{\delta}-1}.
\end{equation}
We look for non-trivial solutions $\psi$ that are defined globally.
Hence the right-hand side of \eqref{ode1} must grow at most linearly, which forces
$
\sqrt{\delta}-1 \leq 1$, i.e.\, $\delta \leq 4.
$

 If $-1<\sqrt{\delta}-1<1$, i.e.\,$0<\delta<4$, then we find $\psi(v)=C_1v^{\frac{1}{2-\sqrt{\delta}}}$, among other possible solutions of this type (note that there is no uniqueness here). Then, using equation $
\psi'=\sqrt{\varphi'}$, we find $$\varphi(v)=c v^{\frac{2}{2-\sqrt{\delta}}-1}= c v^{p-1}, \quad p:=\frac{2}{2-\sqrt{\delta}},$$ covering the classical power-type test function discussed in the beginning of this remark.

 If $\delta=4$, then ODE \eqref{ode1} becomes linear. Taking, for instance, $C=\frac{1}{2}$ in \eqref{ode1} yields the non-trivial solution
$
\psi(v)=2e^{\frac{v}{2}}$,
$\varphi(v)=e^{v}.
$
This is essentially the test function that allows us to treat the critical threshold $\delta=4$, see below.
Although $\psi(0) \neq 0$, which fails the original requirement $\psi(0)=0$ needed above to apply the form-boundedness of $b$, one can overcome this by working with test functions $\varphi(v)=e^{v}-1$ or $\varphi(v)=e^{v}-e^{-v}$ multiplied by a cutoff function, or working on torus instead of $\mathbb R^d$ \cite{Ki_Orlicz, KiS_feller}.

3.~Let $$b \in \mathbf{F}_\delta \text{ with threshold $\delta=4$}.$$ Assume additionally that $b$ and the approximating vector fields $\{b_n\} \in [b]$ have supports in a fixed ball (or, more generally, decay sufficiently rapidly at infinity uniformly in $n$); since here we are interested in admissible local singularities of $b$, this is not a particularly restrictive assumption. Let $q \in \mathbf{BMO}^{-1}$. By following closely the proof of \cite[Theorem 1]{Ki_Orlicz} (see also \cite[Theorem 2]{KiS_feller}) i.e.\,using test function $e^{v}-e^{-v}$ in the analysis of Cauchy problem for the Kolmogorov PDE $(\partial_t-\Delta + b \cdot \nabla)v=0$, one can show that the limit
$$
s\mbox{-}L_{{\rm \cosh - 1}} \mbox{-}\lim_n \lim_m e^{-t\Lambda(b_n,q_m)} \text{ (loc.\,uniformly in $t \geq 0$)}
$$
exists and determines a strongly continuous Markov semigroup $e^{-t\Lambda(b,q)}$ on the Orlicz space 
$$L_{\cosh -1}:=\text{the closure of the Schwartz space $\mathcal S$ with respect to norm}$$
$$
\|f\|_{\cosh - 1}=\inf\left\{c>0\mid\langle\cosh\frac{f}{c}-1\rangle\leq 1\right\}.
$$ 
On the torus, in the case $q=0$ (i.e.\,no distributional component of the drift), \cite{Ki_Orlicz} established the following energy inequality for $v(t)=e^{-t\Lambda(b_n)}v_0$:
$$
 \frac{1}{2}\sup_{s \in [0,t]}\langle e^{v^{p}(s)}  \rangle +  4\frac{(p-1)}{p}\int_0^t \langle (\nabla v^{\frac{p}{2}})^2e^{v^{p}}\rangle ds \leq \langle e^{v_0^{p}}  \rangle, \quad p=2,4,\dots,
$$
provided $\frac{c_\delta}{\sqrt{\delta}}t<\frac{1}{2}$. The small time restriction can be removed using the semigroup property. One can compare this to the usual $L^p$ energy inequality \eqref{en_ineq} when $\delta<4$. At first sight, letting $\delta \uparrow 4$ seems to eliminate the dispersion term; however, it turns out that one retains an energy inequality once appropriate exponential factors are included.
One also obtains uniqueness of weak solution to Cauchy problem for Kolmogorov PDE at least for sufficiently regular initial functions \cite{Ki_Orlicz}. 

Orlicz space $L_{\cosh -1}$ can be viewed as the limit of $L^p$ spaces as $p>\frac{2}{2-\sqrt{\delta}}$ tends to $\infty$ as $\delta \uparrow 4$. This can be made precise: the theory of the Kolmogorov equation in the critical regime $\delta=4$  can be viewed as the limit $p\rightarrow \infty$  of the asymptotic $L^p$ theory around $1$, i.e.\,with $L^p$ test function
$$
\varphi(v)=\biggl(1+\frac{v}{p}\biggr)^p-1,
$$ 
see \cite{Ki_intrinsic}.
\end{remark}

\bigskip

\section{Critical divergence and the constant in many-particle Hardy inequality} 
\label{div_drifts_sect}

Throughout this section, $d\geq3$.

\textbf{1.~}We can substantially relax  condition \eqref{kappa_hyp_int} on the strength of attraction between the particles in Example \ref{ex2_multi} by employing the many-particle Hardy inequality of \cite{HHLT} and the following variant of Theorem 
\ref{thm2}.

\medskip

Let $b \in [L^1_{\loc}]^d$ be a vector field with divergence ${\rm div\,}b \in L^1_{\loc}$. Let $({\rm div\,}b)_+$ denote the positive part of ${\rm div\,}b$.

\begin{definition}
We say that  $({\rm div\,}b)_+$ is form-bounded, and write 
$({\rm div\,}b)^{\scriptscriptstyle 1/2}_+ \in \mathbf{F}_{\delta_+}$, if 
$$
\langle ({\rm div\,}b)_+, \varphi^2 \rangle \leq \delta_+\langle |\nabla \varphi|^2\rangle + c_{\delta_+}\langle \varphi^2 \rangle \quad \forall\,\varphi \in C_c^\infty
$$
for some constants $\delta_+$ and $c_{\delta_+}$.
\end{definition}

\begin{theorem}
\label{thm_div}
Let
\begin{equation*}
\left\{
\begin{array}{l}
b \in \mathbf{F}_\delta \text{ with } \delta<\infty, \quad
 ({\rm div\,}b)^{\frac{1}{2}}_+ \in \mathbf{F}_{\delta_+} \text{ with } \delta_+<4, \quad ({\rm div\,}b)_- \in L^1+L^\infty, \\
q \in \mathbf{BMO}^{-1}.
\end{array}
\right.
\end{equation*}
Let $\{b_n\} \in [b]'$ (see Definition \ref{def4}), $\{q_m\} \in [q]$. {Then assertions (\textit{i})-(\textit{v}) of Theorem \ref{thm2} remain valid. }
\end{theorem}

\begin{example}
\label{ex3_multi}
Let us establish weak well-posedness of particle system \eqref{syst_m} using Theorem 
\ref{thm_div} rather than Theorem \ref{thm2}. The difference between the two theorems is in the assumptions on the drift $b(x)=(b^1(x),\dots,b^N(x))$, 
\begin{equation*}
b^i(x):=\frac{1}{N}\sum_{j=1, j \neq i}^N \sqrt{\kappa}\frac{d-2}{2}\frac{x^i-x^j}{|x^i-x^j|^2}, \quad x=(x^1,\dots,x^N).
\end{equation*}
We already know that this drift is form-bounded, but what matters in Theorem \ref{thm_div} is the form-bound of potential
\begin{equation*}
({\rm div\,}b(x))_+={\rm div\,}b(x) =\sqrt{\kappa}\frac{(d-2)^2}{N}\sum_{1 \leq i<j\leq N}\frac{1}{|x^i-x^j|^2}.
\end{equation*}
To verify the form-boundedness of $({\rm div\,}b)_+$, we invoke the many-particle Hardy inequality: for $d \geq 3$, all  $N \geq 2$, 
\begin{equation}
\label{multi_hardy}
C_{d,N}\sum_{1 \leq i<j \leq N}\int_{\mathbb R^{dN}}\frac{|\varphi(x)|^2}{|x^i-x^j|^2}dx \leq  \int_{\mathbb R^{dN}}|\nabla \varphi(x)|^2 dx
\end{equation}
for all $\varphi \in W^{1,2}(\mathbb R^{dN})$, where, from now on, $C_{d,N}$ denotes the best possible constant in \eqref{multi_hardy}. Hence
$$
({\rm div\,}b)_+^{\frac{1}{2}} \in \mathbf{F}_{\delta_+}, \quad \delta_+=\sqrt{\kappa}\frac{(d-2)^2}{N}C^{-1}_{d,N}.
$$
To the best of our knowledge, the problem of finding the exact value of $C_{d,N}$ is still open. It is not difficult to obtain a crude lower bound on $C_{d,N}$ by summing up the ordinary Hardy inequalities for the inverse square potential $x^i \mapsto |x^i-x^j|^{-2}$, each in its own copy of $\mathbb R^d$. However, as is pointed out by Hoffmann-Ostenhof, Hoffmann-Ostenhof, Laptev and Tidblom in \cite{HHLT}, this lower bound on $C_{d,N}$ is quite suboptimal. They provided a finer argument that gives a much better lower bound 
\begin{equation}
\label{HHLT_est}
C_{d,N} \geq (d-2)^2 \max\bigg\{\frac{1}{N},\frac{1}{1+\sqrt{1+\frac{3(d-2)^2}{2(d-1)^2}(N-1)(N-2)}} \bigg\}.
\end{equation}
Therefore, it suffices for us to require
\begin{equation}
\label{kappa_hyp}
\tag{$\kappa_{\rm hyp 2}$}
\kappa<16
\end{equation}
which guarantees  $\delta_+<4$ and allows us to apply Theorem \ref{thm_div}.

\end{example}

\textbf{2.~}We argue that the relationship between the many-particle Hardy inequality and the particle system 
\begin{equation}
\label{syst_m2}
X^i_t=x^i_0-\sqrt{\kappa}\frac{d-2}{2}\frac{1}{N}\sum_{j=1, j \neq i}^N\int_0^t \frac{X_s^i-X_s^j}{|X_s^i-X_s^j|^2}ds + \sqrt{2}B^i_t
\end{equation}
 goes both ways. Namely, we can use the counterexample in (a') of Section \ref{particle_sect} to the weak well-posedness of \eqref{syst_m2}, i.e.\,when the strength of attraction $\kappa$ is too large, to obtain an \textit{upper bound} on the best possible constant $C_{d,N}$ in \eqref{multi_hardy}.

\begin{theorem}[An upper bound on the constant in the many particle Hardy inequality \eqref{multi_hardy}]
\label{up_cor}
\label{cor2}
$$
C_{d,N} \leq \frac{d(d-2)}{N}.
$$

\end{theorem}
\begin{proof}
By Theorem \ref{thm_div} and the calculation in the previous example, \eqref{syst_m2} has a weak solution for every initial configuration of the particles provided that
$$\sqrt{\kappa}\frac{(d-2)^2}{N}C^{-1}_{d,N} <4.$$
On the other hand, by the counterexample in (a') of Section \ref{particle_sect}, if $\kappa>16 (\frac{d}{d-2})^2$, then \eqref{syst_m2} does not have a weak solution, so we must have $$4\frac{d}{d-2}\frac{(d-2)^2}{N}C^{-1}_{d,N} \geq 4,$$ otherwise a weak solution would exist. This gives the sought upper bound on $C_{d,N}$.
\end{proof}

The lower bound \eqref{HHLT_est} of \cite{HHLT} is thus close to optimal, at least in high dimensions.

In \cite{HHLT}, the authors also provided, among other results, the following upper bound:
\begin{equation}
\label{up_bd}
C_{d,N} \leq \frac{d(d-2)}{N-1},
\end{equation}
which is only slightly worse than ours for $N$ large\footnote{There is a small error in the calculations in \cite{HHLT}; their upper bound, as it is stated, grows factorially in $d$, but their proof actually gives \eqref{up_bd}.}.
% Their argument uses particular test functions in \eqref{multi_hardy}. Theorem \ref{up_cor} improves the dependence on the dimension $d$ in \eqref{up_bd}, i.e.\,we now have polynomial growth versus factorial growth in $d$. Of course, the simplicity of the proof of Theorem \ref{cor2} is only seeming: we apply Theorem \ref{thm_div} whose proof uses De Giorgi's method. 

% Theorem \ref{cor2} shows that the lower bound \eqref{HHLT_est} of \cite{HHLT} is close to optimal, at least in high dimensions.

\medskip

\textbf{3.~}We can relax the assumptions on $b$ in Theorem \ref{thm_div} as follows.

\begin{definition}[Multiplicative form-boundedness]
\label{mf_def}
A vector field $b \in [L^1_{\loc}(\mathbb{R}^d)]^d$ is said to be multiplicatively form-bounded if
\begin{align*}
\langle |b|\varphi, \varphi\rangle \leq \delta \|\nabla \varphi\|_2\|\varphi\|_2+  c_\delta\|\varphi\|^2_2 \qquad \forall\,\varphi \in W^{1,2}
\end{align*}
for some constants $\delta$ and $c_\delta$ (we will see below that only their finiteness is important).
This will be abbreviated as $b \in \mathbf{MF}_\delta$.
\end{definition}

Once again, the constant $c_\delta>0$ plays a secondary role when it comes to handling local singularities of $b$ (e.g.\,$c_\delta>0$ allows to include $L^\infty$ drifts).

\medskip

Mazya \cite[Sect.\,1.4.7]{M} proved that 
$$
\langle |b|\varphi, \varphi\rangle \leq \delta \|\nabla \varphi\|_2\|\varphi\|_2\;\;\forall \varphi \in W^{1,2} \text{ for some $\delta<\infty$}  \quad \Leftrightarrow \quad \sup_{x \in \mathbb R^d }\langle |b|\mathbf{1}_{B_r(x)}\rangle \leq C r^{d-1} \quad \forall\,r>0
$$
for some $C<\infty$, i.e.\,there is a complete characterization of $\mathbf{MF}_\delta$ in terms of Morrey spaces:
\begin{equation}
\label{mf_char_}
\cup_{\delta>0}\mathbf{MF}_\delta \;(\text{with $c_\delta=0$})\quad =\quad M_1
\end{equation}
(see Appendix \ref{m_sect} for the proof).
Let us emphasize that for the class of form-bounded vector fields one only has inclusions
\begin{equation}
\label{f_char_}
M_{2+\varepsilon} \quad \subset \quad  \cup_{\delta>0}\mathbf{F}_\delta\;(\text{with $c_\delta=0$}) \quad  \subset \quad  M_2, 
\end{equation}
where $ \varepsilon>0$ is fixed arbitrarily small, i.e.\,there is no complete characterization of $\mathbf{F}_\delta$ in terms of Morrey spaces. See discussion in Section \ref{classes_sect}.

Comparing \eqref{mf_char_} and \eqref{f_char_}, one sees that one gains quite a lot in admissible singularities of $b$ by passing from form-bounded drifts  to multiplicatively form-bounded drifts. Of course, this comes at the expense of imposing conditions on  ${\rm div\,}b$.

\begin{theorem}
\label{thm1} 
The following are true:

 \begin{enumerate}[label=(\roman*)]

\item {\rm (Classical martingale solutions)} If $$|b|^{\frac{1+\nu}{2}} \in \mathbf{F}_\delta \quad \nu \in ]0,1], \quad \delta<\infty,$$ 
and 
\begin{equation}
\label{div_cond5}
({\rm div\,}b)_+^{1/2} \in \mathbf{F}_{\delta_+}, \quad \delta_+<4, \qquad ({\rm div}\,b)_- \in L^1 + L^\infty,
\end{equation}
then, for every $x \in \mathbb R^d$, SDE 
\begin{equation*}
X_t=x-\int_0^t b(X_s)ds + \sqrt{2}B_t
\end{equation*}
has a  martingale solution $\mathbb P_{x}$.

\medskip

\item  {\rm (Approximation uniqueness and Markov property)} If, in addition to the assumptions of (\textit{i}),
{$\delta_+<2$ and }$b \in \mathbf{MF}_{\delta}$ for some $\delta<\infty$,
then there exists
$0<\gamma<1$  such that, regardless of the choice of $\{b_n\} \in [b]$ (defined in the same way as in Section \ref{approx_sect}, i.e.\,to preserve the structure constants of $b$), provided that $\{b_n\}$ additionally satisfies
$$
b_n \rightarrow b \quad \text{ in $[L^{1+\gamma}]^d$,}
$$
we have convergence 
$$
\mathbb P_{x}^{n} \rightarrow \mathbb P_{x} \quad \text{ weakly in $\mathcal P(\mathbf{C})$},
$$
of the martingale solutions $\{\mathbb P_{x}^{n}\}$ to the approximating SDEs 
\begin{equation}
\label{approx_sdes5}
X^{n}_t=x-\int_0^t b_{n}(X^{n}_r)dr + \sqrt{2}B_t.
\end{equation}
Furthermore, $\{\mathbb P_x\}_{x \in \mathbb R^d}$ is a Markov family.
\medskip

\item {\rm (Feller semigroup)} Under the assumptions of (\textit{ii}),
$$
T_tf(x):=\mathbb E_{\mathbb{P}_x}[f(\omega_t)], \quad f \in C_\infty
$$
is a strongly continuous Feller semigroup on $C_\infty$, say, $T_t=:e^{-t\Lambda}$, where the generator $\Lambda$ is thus an appropriate operator realization of the formal operator $-\Delta + b \cdot \nabla$ in $C_\infty$.
\end{enumerate}
\end{theorem}

 The first two statements were proved in \cite{KiS_sharp} and \cite{Ki_multi}, respectively.
The novelty is in assertion (\textit{iii}). Its proof uses the Trotter approximation theorem in the same way as the proof of Theorem \ref{thm2}.

\begin{remark}[$L^2$ vs $L^{1+\gamma}$ for some $\gamma<1$]
\label{cacc_rem}

If, in the setting of Theorem \ref{thm1}(\textit{i}), we additionally require $b \in [L^2_{\loc}(\mathbb R^{d})]^d$, then it is also possible to prove a.e.\,approximation uniqueness. The last condition is actually satisfied if time-inhomogeneous $b$ is a Leray-Hopf solution of the 3D Navier-Stokes equations, i.e.\,then one has $b \in L^\infty([0,1],[L^2_{\loc}(\mathbb R^{d})]^d)$. This was explored by a number of authors, see subsection \ref{super_rem}.
There are, however, other classes of solutions to 3D N-S equations that are not uniformly in $t$ square integrable, such as the critical class \eqref{X_def} of Koch and Tataru. So, we are interested in finding different additional conditions on $b$ that do not require square integrability, but still allow us to prove, among other results, the approximation uniqueness. {These are the conditions $b \in \mathbf{MF}_\delta$ and $\delta_+<2$ in Theorem \ref{thm1}(\textit{ii}).}

The proof of the approximation uniqueness in Theorem \ref{thm1}(\textit{ii}) uses an $L^{\frac{1+\gamma}{\gamma}}(\mathbb R^d)$ gradient bound on solutions of the corresponding elliptic Kolmogorov equation (cf.\,Lemma \ref{grad_lem}). It is proved by means of the Gehring-Giaquinta-Modica's lemma (Lemma \ref{gehring_prop}), so $\gamma$ can be estimated explicitly, see Remark \ref{gehring_expl_est}.

To use Gehring-Giaquinta-Modica's lemma, we need  Caccioppoli's inequality. The proof of Caccioppoli's inequality for multiplicatively form-bounded drifts  employs an extra iteration procedure (``Caccioppoli's iterations'') which was introduced in our previous paper \cite{KiV} to study regularity of solutions of Dirichlet problem for the drift-diffusion equation. Namely, for  $v=(u-k)_+$ and cutoff function $\eta$ one has
\begin{align*}
\langle b \cdot \nabla u, \eta v \rangle & = \frac{1}{2}\langle b \cdot \nabla v^2, \eta \rangle \\
& =-\frac{1}{2}\langle b \cdot \nabla \eta, v^2\rangle-\frac{1}{2}\langle {\rm div\,}b,\eta v^2\rangle \\
& \geq-\frac{1}{2}\langle |b|, \psi^2 \rangle-\frac{1}{2}\langle ({\rm div\,}b)_+,\eta v^2\rangle, \quad \text{ where }\psi:=v\sqrt{|\nabla \eta|}.
\end{align*}
The last term is controlled by the form-bound of $({\rm div\,}b)_+^{1/2}$, applied to $v\sqrt\eta$. We focus on the cutoff term $\langle |b|,\psi^2\rangle$, which requires the additional iteration.

By $b \in \mathbf{MF}_\delta$ (for simplicity, take $c_\delta=0$),
\begin{align*}
\langle |b|, \psi^2 \rangle
& \leq \delta \|\nabla (v\sqrt{|\nabla \eta|})\|_2\|v\sqrt{|\nabla \eta|}\|_2\\
& \leq \delta \bigg(\|(\nabla v)\sqrt{|\nabla \eta|}\|_2 + \|v \nabla \sqrt{|\nabla \eta|}\|_2 \bigg) \, \|v\sqrt{|\nabla \eta|}\|_2,
\end{align*}
so
$$
\langle |b|, \psi^2 \rangle \leq \frac{C_1}{r_2-r_1}\|(\nabla v)\mathbf{1}_{B_{r_2}}\|_2\| v\mathbf{1}_{B_{r_2}}\|_2 + \frac{C_1}{(r_2-r_1)^2} \|v\mathbf{1}_{B_{r_2}}\|^2_2,
$$
provided that $\eta$ is equal to $1$ on $B_{r_1}$, is zero outside of $B_{r_2}$, and its derivatives satisfy appropriate estimates.
The first term in the RHS contains both $\nabla v$ and the indicator function of the ball of larger radius, so we cannot simply apply Cauchy-Schwarz' inequality to  obtain the Caccioppoli inequality. But it is possible to arrive at the Caccioppoli inequality by iterating over a sequence of intermediate balls with radii between $r_1$ and $r_2$.
\end{remark}

\begin{remark}[Heat kernel bounds] 
\label{heat_rem}
Although in the results mentioned so far it is the singular positive part of ${\rm div\,}b$ that presents an obstacle to the well-posedness of the SDE, something can be said about the case of positive divergence. Namely, assume that $a \in H_{\xi}$, i.e.\,we have a bounded symmetric uniformly elliptic matrix field. Let $$b \in \mathbf{F}_\delta, \quad \delta<4\xi^2,$$ and
$$
{\rm div\,}b \geq 0.
$$ 
Then the heat kernel $p(t,x,y)$ of $-\nabla \cdot a \cdot \nabla + b \cdot \nabla$, defined as the integral kernel of the corresponding $C_0$ semigroup in $L^p$, $p>\frac{2}{2-\xi^{-1}\sqrt{\delta}}$, constructed via a suitable regularization of $a$ and $b$, satisfies, possibly after a modification on a measure zero set, the Gaussian lower bound
\begin{equation}
\label{lgb}
  c_1 \Gamma_{c_2}(t,x-y) e^{-c_3 t} \leq p(t,x,y),
\end{equation}
where $\Gamma_c(t,x):=(4\pi c t)^{-\frac{d}{2}}e^{-\frac{|x|^2}{4ct}}$ and $c_1, c_2>0$, $c_3 \geq 0$, see \cite{KiS_MAAN}.

Under the above assumptions on $b$ there is no Gaussian upper bound on $p(t,x,y)$. In fact, the counterexample is given by the Brownian particles considered in Example \ref{ex1_multi}, see the end of Section \ref{particle_sect}. Consequently, the proof of the Gaussian lower bound \eqref{lgb} does not use the Gaussian upper bound. As is mentioned in \cite{KiS_MAAN}, to the best of authors' knowledge, this is the first result of this type. 
\end{remark}

\begin{remark}[More on the case ${\rm div\,}b=0$] 
1.~Let $b=b(x)$. In the case ${\rm div\,}b=0$, there is an alternative approach to the proof of the approximation uniqueness for $b \in \mathbf{MF}_\delta$. Namely, 
Mazya-Verbitsky \cite[Theorem 5]{MV} proved equivalence 
\begin{equation}
\label{MV_equiv}
|\langle b \varphi,\varphi\rangle| \leq \delta \|\nabla \varphi\|_2\|\varphi\|_2 \quad \forall\,\varphi \in C_c^\infty \quad \Leftrightarrow \quad b=\nabla Q \text{ for some } \quad Q \in [{\rm BMO}]^{d \times d},
\end{equation}
where the LHS is, clearly, more general than $b \in \mathbf{MF}_\delta$.
So, one can use this representation for divergence-free $b$, put the anti-symmetric matrix $Q$ in the diffusion coefficients,  
and then prove uniqueness of the weak solution to Cauchy problem for the Kolmogorov parabolic equation by working in the standard Hilbert triple $W^{1,2} \hookrightarrow L^2 \hookrightarrow W^{-1,2}$, see \cite{QX}.

2. Assuming that $b \in \mathbf{MF}_\delta$, ${\rm div\,}b=0$, Sem\"{e}nov \cite{S} proved two-sided 
Gaussian bounds
\begin{equation}
\label{S}
C_1\Gamma_{c_2}(t-s,x-y) \leq p(t,s,x,y) \leq C_3 \Gamma_{c_4}(t-s,x-y),
\end{equation}
where $p(t,s,x,y)$ is the heat kernel of the Kolmogorov operator $-\nabla \cdot a \cdot \nabla + b \cdot \nabla$ with measurable symmetric uniformly elliptic $a$. He used Moser's method to prove the upper bound. His proof of the lower bound is based on a substantial modification of Nash's method. Next, Qian-Xi \cite{QX} established two-sided Gaussian bounds for all $q=\nabla Q \in \mathbf{BMO}^{-1}$. 

Having two-sided Gaussian bounds greatly simplifies the analysis of the corresponding diffusion process, e.g.\, one obtains right away the Feller propagator, the continuity of trajectories follows easily from the Kolmogorov continuity criterion. Let us show  how the tightness of the martingale solutions $\mathbb P_x^n$ of the approximating SDEs \eqref{approx_sdes5} follows from the upper Gaussian bound (our underlying goal here is to demonstrate how natural condition $b \in \mathbf{MF}_\delta$ is).
The tightness argument requires, as its point of departure, the estimate
\begin{equation}
\label{t0t1}
\mathbb E_{\mathbb P^n_x}\int_{t_0}^{t_0+\varepsilon} |b_n(\omega_r)|dr \leq H(\varepsilon), \quad 0 \leq t_0\leq 1, 0<\varepsilon<1,
\end{equation}
for a continuous function $H$ such that $H(\varepsilon) \downarrow 0$ as $\varepsilon \downarrow 0$. By It\^{o}'s formula,
$
\mathbb E_{\mathbb P^n_x}\int_{t_0}^{t_1} |b_n(\omega_r)|dr=\mathbb E_{\mathbb P^n_x}[u_n(t_0,\omega_{t_0})],
$
where $u_n$ solves
$
\partial_t u_n + \Delta u_n - b_n \cdot \nabla u_n + |b_n|=0$, $u_n(t_1)=0.
$
After reversing the direction of time, we can deal instead with the initial-value problem
\begin{align}
(\partial_t - \Delta + b_n \cdot \nabla)v_n=|b_n|, \quad v_n|_{t=0}=0.
\end{align} 
(with some abuse of notation, we continue denoting the drift by $b_n$). By the Duhamel formula,
\begin{align*}
v_n(t,x)=\int_0^t \langle p(t,s,x,\cdot)|b_n(\cdot)|\rangle ds,
\end{align*}
so, applying the upper Gaussian bound on the heat kernel $p=p_n$, we obtain (put $\Gamma_{t,x}:=\Gamma_{c_4}(t-s,x-\cdot)$):
\begin{align}
|v_n(t,x)| & \leq C_3\int_0^t \langle |b_n|\sqrt{\Gamma_{t,x}},\sqrt{\Gamma_{t,x}} \rangle ds \notag \\
& (\text{we use $b_n \in \mathbf{MF}_\delta$}) \label{M_line} \\
& \leq C_3\int_0^t \biggl(\delta\|\nabla \sqrt{\Gamma_{t,x}}\|_2 \|\sqrt{\Gamma_{t,x}}\|_2+c_\delta \|\sqrt{\Gamma_{t,x}}\|^2_2\biggr) ds \notag \\
& = C_3\int_0^t \biggl(\delta\|\nabla\sqrt{\Gamma_{t,x}}\|_2+c_\delta \biggr) ds \notag \\ 
& = C_3\int_0^t \biggl(\delta \sqrt{\frac{d}{8c_4}\frac{1}{t-s}} + c_\delta \biggr)ds, \notag
\end{align}
which yields \eqref{t0t1} for {$H(t)=C(\sqrt{t}+c_\delta t)$}. 
Let us note that the upper Gaussian bound that we assumed above is valid e.g.\,if $b \in \mathbf{MF}_\delta$ and ${\rm div\,}b \leq 0$ or, more generally, $({\rm div\,}b)_+$ is in the Kato class, see \cite{KiS_MAAN}.
\end{remark}

\begin{remark}The following is a variant of the particle system in Examples \ref{ex1_multi}, \ref{ex2_multi} and \ref{ex3_multi} that additionally pushes the center of mass of the interacting particles towards the origin:
\begin{align*}
X^i_t =x^i_0-\frac{1}{N}\sum_{j=1, j \neq i}^N \sqrt{\kappa}\frac{d-2}{2} \int_0^t \frac{X_s^i-X_s^j}{|X_s^i-X_s^j|^2}{ds} - \sqrt{\kappa}\int_0^t {\frac{d-2}{4N}}\frac{\frac{1}{N}\sum_{k=1}^N X_s^k}{|\frac{1}{N}\sum_{k=1}^N X_s^k|^2}ds + \sqrt{2}B^i_t,  
\end{align*}
where $i=1,\dots,N$. Theorem \ref{thm_div} applies whenever $\kappa<16$.
Indeed, the corresponding drift $b=(b^1,\dots,b^N)$ is given by 
$$
b^i(x^1,\dots,x^N):=  \frac{\sqrt{\kappa}}{N}\sum_{j=1, j \neq i}^N\frac{d-2}{2} \frac{x^i-x^j}{|x^i-x^j|^2} + \sqrt{\kappa}{\frac{d-2}{4N}}\frac{\frac{1}{N}\sum_{k=1}^N x^k}{|\frac{1}{N}\sum_{k=1}^N x^k|^2},
$$
so it suffices for us to apply to ${\rm div\,}b$, in the same way as it was done in Example \ref{ex3_multi},  the following many-particle Hardy-type inequality, also proved in \cite{HHLT}:
\begin{equation}
\label{many_hardy_2}
\frac{(d-2)^2}{N}\sum_{1 \leq i<j \leq N}\int_{\mathbb R^{dN}}\frac{\varphi^2(x)}{|x^i-x^j|^2}dx + \frac{(d-2)^2N}{4} \int_{\mathbb R^{dN}}\frac{\varphi^2 (x)}{\big|\sum_{k=1}^N x^k\big|^2}dx \leq  \int_{\mathbb R^{dN}}|\nabla \varphi(x)|^2 dx.
\end{equation}

The proof of \eqref{multi_hardy} in \cite{HHLT} departs from \eqref{many_hardy_2}. In turn, the proof of the latter is rather close to the proof of the usual Hardy inequality (i.e.\,via representing the potential as the divergence of appropriately chosen vector field and integrating by parts, see \cite[proof of Theorem 2.1]{HHLT} for details). Despite that, to our knowledge, the question of the optimality of the constants in inequality \eqref{many_hardy_2} has not been addressed yet. 

\end{remark}

\subsection{Critical divergence, super-critical drift}
\label{super_rem}
Some results for the operator $-\Delta +b \cdot \nabla$ depend only on ${\rm div\,}b$. For instance, if  $({\rm div\,}b)_+^{1/2} \in \mathbf{F}_{\delta_+}$, $\delta_+<4$, then the solution $v$ to the Kolmogorov backward equation 
$$
(\partial_t-\Delta + b\cdot \nabla)v=0, \quad v|_{t=0}=v_0,
$$
satisfies,
for all $\frac{2}{2-\sqrt{\delta_+}}<q \leq p \leq \infty$, the dispersion estimate
$$
\|v(t)\|_p \leq Ce^{\omega t}{t^{-\frac{d}{2}(\frac{1}{q}-\frac{1}{p})}}\|v_0\|_q, \quad t>0.
$$
The proof is due to J.\,Nash, see \cite{KiS_theory} for details. While one needs smoothness and boundedness of $b$ and $v_0$ to carry out integration by parts, the constants $C$ and $\omega$ depend only on {$d$, $\delta_+$, $c_{\delta_+}$, $q$ and $p$}. {They do not} depend on any {other }integral characteristics of $b$.

However, to obtain more detailed information about the diffusion process with drift $b$, one must impose some conditions on $b$. For instance, in Theorem \ref{thm_div} we also required $b \in \mathbf{F}_\delta$. 

Between these two types of assumptions, there are intermediate conditions, such as in assertion (\textit{i}) of Theorem \ref{thm1}. In this assertion, {where $|b|^{\frac{1+\nu}{2}}\in\mathbf{F}_\delta$,} selecting $\nu$ close to zero, one can treat $b$ that can be essentially twice more singular than the vector fields in $\mathbf{F}_\delta$. This is a super-critical condition on the drift in the sense of scaling. Let us recall the sub-critical/critical/super-critical classification of the spaces of vector fields. Given a vector field $b$, put $b_\lambda(x):=\lambda b(\lambda x)$. Let $Y$ be a translation-invariant Banach space of distribution-valued vector fields $b$ such that
$$
\|b_\lambda\|_Y=\lambda^{a} \|b\|_Y.
$$
Now,
\begin{align*}
& \text{$a>0$, then $Y$ is sub-critical, i.e.\,passing to the small scales decreases the norm,}\\
& \text{$a=0$, then $Y$ is critical,}\\
& \text{$a<0$, then $Y$ is super-critical}.
\end{align*}
In the last two cases ``zooming in'' does not change the norm or makes the norm larger.
For example, $L^p$ is sub-critical, critical or super-critical according to whether $p>d$, $p=d$ or $p<d$.
This classification is widely used in the study of Navier-Stokes equations. There one applies it to spaces of solutions or initial data.

The super-critical condition on $b$ appearing in assertion (\textit{i}) was introduced in the work of Q.S.\,Zhang \cite{Z_supercritical}.  He considered the time-inhomogeneous counterpart of $|b|^{\frac{1+\nu}{2}} \in \mathbf{F}_\delta$, namely,
$b \in [L^{1+\nu}_{\loc}(\mathbb R^{1+d})]^d$ and for a.e.\,$t \in \mathbb R$
\begin{align}
\label{super_fbd}
\langle |b(t,\cdot)|^{1+\nu}\varphi,\varphi\rangle   \leq \delta \|\nabla \varphi\|_2^2+g_\delta(t)\|\varphi\|_2^2 \qquad \forall\,\varphi \in W^{1,2}
\end{align}
where $0 \leq g_{\delta} \in L^1_{\loc}(\mathbb R)$ describes how irregular $b$ can be in time. He established, among other results, local boundedness of any weak solution to the parabolic equation
$$
(\partial_t-\Delta + b \cdot \nabla)v=0 \text{ on } \mathbb R^{1+d},
$$
{provided that}
\begin{equation}
\label{b_L2}
{b \in [L^2_{\loc}(\mathbb R^{1+d})]^d, \quad {\rm div}\,b\leq 0.}
\end{equation}
The last condition is satisfied if $b$ is taken to be a Leray-Hopf solution of 3D Navier-Stokes equations, which motivated \cite{Z_supercritical}. 

The proof of Theorem \ref{thm1}(\textit{i}) uses a tightness estimate for solutions of the approximating SDEs with bounded smooth drifts (cf.\,\eqref{t0t1}). The proof of that estimate, in turn, uses the idea from \cite{Z_supercritical} for handling cutoff functions in presence of $b$ satisfying \eqref{super_fbd}. 

The first result on SDEs with {supercritical} divergence-free drifts belongs to X.\,Zhang and G.\,Zhao \cite{ZZ}. They considered
\begin{equation}
\label{sde9}
X_t=x-\int_s^t b(r,X_r)dr + \sqrt{2}(B_t-B_s), \quad t \geq s,
\end{equation}
with divergence-free drift $b$ additionally satisfying the square integrability condition \eqref{b_L2}, and included in the super-critical Ladyzhenskaya-Prodi-Serrin condition
\begin{equation}
\label{sLPS}
\tag{SLPS}
|b| \in L^q([0,T],L^p(\mathbb R^d)), \quad p,q \geq 2, \quad \frac{d}{p}+\frac{2}{q}<2.
\end{equation}
They proved that for every initial data $(s,x) \in \mathbb R^{1+d}$ the SDE \eqref{sde9} has a weak solution satisfying a Krylov type estimate. Moreover, using hypothesis \eqref{b_L2}, they proved that outside a measure zero set of $(s,x)$ one has approximation uniqueness and a.s.\,Markov property for these weak solutions. 
Consequently, the weak well-posedness result of \cite{ZZ} justifies the passive tracer model in the Leray-Hopf setting  under the a priori assumption \eqref{sLPS}. They also allow the positive part $({\rm div}\,b)_+$ of the divergence of $b$ to be singular, provided it satisfies condition \eqref{sLPS} (possibly with different exponents $p$, $q$). 
In recent paper \cite{HZ}, Z.\,Hao and X.\,Zhang extended the results in \cite{ZZ} to divergence-free super-critical distributional drifts.

Let us add that condition \eqref{b_L2} is quite powerful. For instance, if one assumes only ${\rm div\,}b \leq 0$ and \eqref{b_L2}, then it is already sufficient to prove weak uniqueness results for the backward Kolmogorov equation in $L^1$ \cite{GS}.

It is easy to show, using H\"{o}lder's inequality, that \eqref{sLPS} is a subclass of \eqref{super_fbd}. It is a proper subclass. Indeed, \eqref{super_fbd} contains some vector fields having strong hypersurface singularities that are not covered by \eqref{sLPS}. 

The main focus of Theorem \ref{thm1}(\textit{i}) was reaching the blow-up threshold for $\delta_+$ for $({\rm div}\,b)_+$. As a by-product, it also closes the gap (at the level of weak existence for SDE \eqref{sde9}) between the hypotheses on the drift in \cite{ZZ} and in \cite{Z_supercritical}.

There remains nontrivial (as it seems to us) work left to establish weak existence for  SDEs whose drifts lie in an even larger class of super-critical divergence-free drifts, namely, those considered by Q.S.\,Zhang in \cite{Z_supercritical2}:
\begin{align}
\label{super_fbd2}
\big\langle |b(t)|\log(1+|b(t)|)^2\varphi,\varphi\big\rangle   \leq \delta {\|\nabla\varphi\|_2^2}+g_\delta(t)\|\varphi\|_2^2 \qquad \forall\,\varphi \in W^{1,2} \quad \text{for a.e.\,$t \in \mathbb R$}.
\end{align}

\medskip

We conclude by returning to heat kernel bounds. Q.S.\,Zhang \cite{Z_supercritical} and Qian-Xi \cite{QX_supercritical} also obtained non-Gaussian upper bounds on the heat kernel of $-\Delta + b \cdot \nabla$ with supercritical $b$. It would be interesting to ``test'' the optimality of these bounds and try to deduce the tightness estimate of \cite{ZZ} from them, in the same way as it was done above for $b \in \mathbf{MF}_\delta$ and the Gaussian upper bound. Of course, the calculations become more complicated due to the lack of a scaling invariance in these bounds.

\bigskip

\section{Diffusion coefficients with form-bounded $\nabla a$}

\label{diff_coeff_sect}

By \cite[Theorem 6.1]{MV}, for a distributional vector field $c$ on $\mathbb R^d$, one has
\begin{equation}
\label{op_emb}
-\Delta + c \cdot \nabla \in \mathcal B(W^{1,2},W^{-1,2}),
\end{equation}
if and only if $c$ admits a decomposition
\begin{equation}
\label{c_decomp}
c=b+q \quad \text{ for some $b \in \mathbf{F}_\delta$ and $q \in \mathbf{BMO}^{-1}_\sharp$.}
\end{equation}
Here
$\mathbf{BMO}^{-1}_\sharp$ consists of divergence-free vector fields $q=\nabla Q$ whose $d \times d$ anti-symmetric primitive $Q=(Q^{ij})_{i,j=1}^d$ satisfies
$$
\|Q^{ij}\|_{ {\rm BMO}_\sharp}=\sup_{x \in \mathbb R^d, 0<R \leq 1}\frac{1}{|B_R|}\int_{B_R(x)}|Q^{ij}-(Q^{ij})_{B_R(x)}|dy<\infty.
$$
It follows that Theorem \ref{thm2} covers exactly the same class of drifts that guarantee the embedding \eqref{op_emb}, up to replacing $\mathbf{BMO}^{-1}_\sharp$ with $\mathbf{BMO}^{-1}$, i.e.\,imposing rather mild assumptions on the growth of $Q$ at infinity. In fact, \cite[Theorem 1]{MV} yields a similar necessary and sufficient condition for the embedding \eqref{op_emb} for the homogeneous Sobolev spaces, in which case $ \mathbf{BMO}^{-1}_\sharp$ in \eqref{c_decomp} gets replaced with $\mathbf{BMO}^{-1}$, and one has $c_\delta=0$ in the form-boundedness condition for $b$.

As noted above, in light of the result of Mazya and Verbitsky \cite[Theorem 6.1]{MV}, the drift conditions in Theorem \ref{thm2} are essentially optimal if one expects the Kolmogorov operator $-\Delta + c \cdot \nabla$ to be $W^{1,2} \rightarrow W^{-1,2}$ bounded. In Theorem \ref{thm_wfb}, however, we will be dealing with {a larger class than $\mathbf{F}_\delta$, namely, a class of Borel measurable drifts} such that the Kolmogorov operator is $\mathcal W^{\frac{3}{2},2} \rightarrow \mathcal W^{-\frac{1}{2},2}$ (Bessel spaces) bounded. This apparent ambiguity and the question why one might expect Kolmogorov operator to be $W^{1,2} \rightarrow W^{-1,2}$ bounded is clarified by noting that both Theorem \ref{thm2} and \cite[Theorem 6.1]{MV} are valid in a greater generality. Let $a$ be a bounded uniformly elliptic symmetric matrix field on $\mathbb R^d$, i.e.\,$a \in H_{\xi}$ for some $\xi>0$. {Thus $\xi I\leq a(x)\leq\xi^{-1}I$ a.e.\,on $\mathbb R^d$.} The cited result of Mazya and Verbitsky states that one has
\begin{equation}
\label{op_emb2}
-\nabla \cdot a \cdot \nabla + c \cdot \nabla \in \mathcal B(W^{1,2},W^{-1,2}) \quad \Leftrightarrow \quad c=b+q \text{ as in \eqref{c_decomp}},
\end{equation}
where $
-\nabla \cdot a \cdot \nabla$ does not let us deviate from the embedding $W^{1,2} \rightarrow W^{-1,2}$. 
In Theorem \ref{thm2_a}, we treat non-divergence form operators $$-a \cdot \nabla^2 + c \cdot \nabla,$$ where $\nabla a \in [L^2_{\loc}]^d$ is in $\mathbf{F}_\delta$. Since $$-a \cdot \nabla^2 + c \cdot \nabla=-\nabla \cdot a \cdot \nabla + (\nabla a + c)\cdot \nabla,$$ 
\cite[Theorem 6.1]{MV} again applies. For such diffusion coefficients, which are considered in the next theorem, our condition on the drift is, arguably, close to being optimal.

Let $a \in H_\xi$.
Put $\sigma=\sqrt{a}$. In Theorem \ref{thm2_a} we consider SDE
\begin{equation}
\label{sde_sigma}
X_t=x-\int_0^t \big(b(X_s) + q(X_s)\big)ds + \sqrt{2}\int_0^t \sigma(X_s) dB_s, 
\end{equation}
with $\nabla a + b \in \mathbf{F}_\delta$ and $q \in \mathbf{BMO}^{-1}$.

\begin{example}
\label{ex100}
{Let} $$a(x)=I+c\frac{x \otimes x}{|x|^2}, \quad x \in \mathbb R^d, \quad c>-1.$$ Then
$
\nabla a(x)=c(d-1)\frac{x}{|x|^2},
$ so, in view of \eqref{hardy}, $$\nabla a \in \mathbf{F}_{\delta_1}, \quad \delta_1=\frac{4c^2(d-1)^2}{(d-2)^2}.$$ This matrix field produces diffusion coefficients $\sigma$ with discontinuity at the origin that is strong enough to make the weak solution to SDE \eqref{sde_sigma} (even with $b=q=0$)  arrive at the origin with positive probability, see e.g.\,\cite[Ch.\,V, Sect.\,3]{B}.
This example can be extended to a matrix field $a$ in $\mathbb R^{dN}$  similar to the many-particle drift \eqref{drift_b}, i.e.\,corresponding to $N$ particles in $\mathbb R^d$ interacting via diffusion coefficients.
\end{example}

\begin{definition} Given a matrix field $a \in H_{\xi}$ and a vector field $b$ such that $\nabla a + b \in \mathbf{F}_\delta$,  we denote by $[a,b]$ the set of sequences $\{a_n\} \in H_{\xi} \cap [C_b \cap C^\infty]^{d \times d}$, $\{b_n\} \in [C_b \cap C^\infty]^{d}$ such that 
$$
\nabla a_n + b_n \in \mathbf{F}_{\delta} \text{ with $c_\delta$ independent of $n$},
$$
and
$$
\nabla a_n + b_n \rightarrow \nabla a + b \text{ in } [L^2_{\loc}]^d, \quad
a_n \rightarrow a \quad \text{ a.e.\,on }\mathbb R^d
$$
as $n \rightarrow \infty$.
\end{definition}
For example, we can take $a_n=E_{\varepsilon_n}a$, see \cite[Sect.\,4.4]{KiS_theory} for the proof.

Let $\{a_n\}$, $\{b_n\}\in[a,b]$, $\{q_m\}\in[q]$ and $\sigma_n=\sqrt{a_n}$. Let $-\Lambda(a_n,b_n,q_m)$ denote the generators in $C_\infty$ of the Feller semigroups corresponding to the smooth bounded coefficients $a_n$, $b_n+q_m$. Thus, $C_c^\infty\subset D(\Lambda(a_n,b_n,q_m))$ and
$$
\Lambda(a_n,b_n,q_m)f=-a_n\cdot\nabla^2f+(b_n+q_m)\cdot\nabla f,\qquad f\in C_c^\infty.
$$
By classical theory, these semigroups are strongly continuous on $C_\infty$. One has

$$
e^{-t\Lambda(a_n,b_n,q_m)}f(x)=\mathbf E[f(X^{n,m}_t)],$$
where $X^{n,m}_t$ is the unique strong solution to SDE
\begin{equation}
X^{n,m}_t=x-\int_0^t \big(b_n(X^{n,m}_s) + q_m(X^{n,m}_s)\big)ds + \sqrt{2}\int_0^t \sigma_n(X^{n,m}_s) dB_s, 
\end{equation}
considered on a fixed complete probability space $\mathfrak F=(\Omega,\mathcal F,\{\mathcal F_t\}_{t \geq 0},\mathbf{P})$, $B_t$ is a $d$-dimensional Brownian motion on this space.
Set
 $\mathbb P_x^{n,m}:=\mathbf{P} (X_t^{n,m})^{-1}$.

\begin{theorem}
\label{thm2_a} Let $d \geq 3$.
Let $a \in H_{\xi}$ and let $b$ and $q$ be, respectively, Borel measurable and distribution-valued vector fields $\mathbb R^d\rightarrow \mathbb R^d$  that satisfy 
\begin{equation*}
\left\{
\begin{array}{l}
\nabla a + b \in \mathbf{F}_\delta \text{ with } \delta<\xi^2, \\
q \in \mathbf{BMO}^{-1}.
\end{array}
\right.
\end{equation*}
Let $\{a_n\}, \{b_n\} \in [a,b]$, $\{q_m\} \in [q]$. The following are true:

\begin{enumerate}[label=(\roman*)]
\item {\rm (Feller semigroup)} The limit 
$$
s\mbox{-}C_\infty \mbox{-}\lim_n \lim_m e^{-t\Lambda(a_n,b_n,q_m)} \text{ (loc.\,uniformly in $t \geq 0$)}
$$
exists and determines a strongly continuous Feller semigroup on $C_\infty$, say, $e^{-t\Lambda}=e^{-t\Lambda(a,b,q)}$, {where $\Lambda$ is an operator realization in $C_\infty$ of the formal differential expression $-a \cdot \nabla^2 + (b+q) \cdot \nabla$.}

\medskip

\item 
{\rm (A relaxed approximation uniqueness)} The limit in (\textit{i}) does not depend on the choice of $\{a_n\}$, $\{b_n\}$ and $\{q_m\}$. In fact, in the definition of $[a,b]$ we can replace the strong convergence of $\nabla a_n+b_n$ to $\nabla a+b$ in $[L^2_{\loc}]^d$ by the weak convergence
$$
\nabla a_n+b_n\overset{w}{\rightarrow}\nabla a+b\quad\text{in }[L^2_{\loc}]^d
$$
to have convergence of the Feller semigroups
$e^{-t\Lambda(a_n,b_n,q)}\overset{s}{\rightarrow}e^{-t\Lambda(a,b,q)}$ in $C_\infty$ (loc.\,uniformly in $t\geq0$). All the other conditions on $a_n$, $b_n$ being unchanged, this weak convergence holds, in particular, if $b_n\overset{w}{\rightarrow}b$ in $[L^2]^d$.

\medskip

\item {\rm (Generalized martingale solution)} There exists a strong Markov family of probability measures $\{\mathbb P_x\}_{x \in \mathbb R^d}$ on the canonical space $\mathbf C$ of continuous trajectories $\omega$ such that 
$$
e^{-t\Lambda(a,b,q)}f(x)=\mathbb E_{\mathbb P_x}[f(\omega_t)], \quad f \in C_\infty, \quad x \in \mathbb R^d,\;t \geq 0,$$
$$
\mathbb P_x=w\mbox{-}\mathcal P(\mathbf C)\mbox{-}\lim_n\lim_m \mathbb P_x^{n,m},
$$
and for every $v$ in the domain $D\big(\Lambda(a,b,q)\big) \subset C_\infty$ of operator {$\Lambda(a,b,q)$}, a dense subspace of $C_\infty$, the process
$$
t \mapsto v(\omega_t)-v(x) + \int_0^t \Lambda(a,b,q) v(\omega_s) ds
$$
is a continuous martingale with respect to $\mathbb P_x$. 

\medskip

\item {\rm (Weak solution for disperse initial data)} Assume additionally that $b$, $q$ and $b_n$, $q_m$ have supports in a fixed ball of finite radius. Given an initial (smooth) probability density $\nu_0$ satisfying $\langle \nu_0^{2r}\rangle<\infty$ for some $1<r<\delta^{-\frac{1}{2}}$,
there exist a probability space $\mathfrak F'=(\Omega',\mathcal F',\{\mathcal F'_t\}_{t \geq 0},\mathbf{P}')$ and a continuous process $X_t$ on this space such that the limit
$$
A_t:=L^2(\Omega')\mbox{-}\lim_n\lim_m \int_0^t \big(b_n(X_s)+q_m(X_s)\big)ds
$$
exists, we have a.s.
$$
X_t=X_0-A_t + \sqrt{2}\int_0^t \sigma(X_s)dB_s, \quad t>0,
$$ 
for a $\mathcal F_t'$-Brownian motion $B_t$, and $\mathbf{P}' X_0^{-1}$ has density $\nu_0$.

\end{enumerate}
\end{theorem}

Theorem \ref{thm2_a}, when applied to $a=I$ (so $\xi=1$), imposes a more restrictive condition on the form-bound $\delta$ than Theorem \ref{thm2}, i.e.\,$\delta<1$ instead of $\delta<4$ (the reason for this is that in the proof of Theorem \ref{thm2_a} we pass to the limit in $L^2$, and thus we need its $L^2$ theory).
In fact, under the assumptions of Theorem \ref{thm2_a} but with $\delta<4\xi^2$, after passing to a subsequence $\{n_k\}$ we still get Feller semigroup
$$
e^{-t\Lambda(a,b,q)}=s\mbox{-}C_\infty \mbox{-}\lim_k \lim_m e^{-t\Lambda(a_{n_k},b_{n_k},q_m)} \text{ (loc.\,uniformly in $t \geq 0$)}.
$$
That is, assertions (\textit{i}) and (\textit{iii}) of Theorem \ref{thm2_a}, with convergence along this subsequence, remain valid for $\delta<4\xi^2$.

Further, assuming that $q=0$ and $\delta<\xi^2$, some other assertions  of 
Theorem \ref{thm2} remain valid in the setting of Theorem \ref{thm2_a}. See \cite{KiS_sharp} and \cite{KiS_Osaka} regarding the weak solutions.

Earlier, Veretennikov \cite{V} and Zhang \cite{Z_a}, Zhang-Zhao \cite{ZZ2} established strong well-posedness for diffusion coefficients having derivatives in $L^p$ with $p$ strictly larger than $2d$ or $d$, respectively. These assumptions, however, make diffusion coefficients H\"{o}lder continuous, so they exclude e.g.\,Example \ref{ex100}.

Similar conditions on diffusion coefficients on the scale of Morrey spaces are considered by Krylov, see \cite{Kr1, Kr2}.

\begin{remark} We have a counterpart of Remark \ref{disp_rem}. Namely, for every $p>\frac{2}{2-\xi^{-1}\sqrt{\delta}}$, the operators
$$
e^{-t\Lambda_p(a,b,q)}:=\biggl[e^{-t\Lambda(a,b,q)} \upharpoonright C_\infty \cap L^p \biggr]^{\rm clos}_{L^p \rightarrow L^p}
$$
are bounded on $L^p$, constitute a strongly continuous semigroup, and for all $\frac{2}{2-\xi^{-1}\sqrt{\delta}} < p \leq r < \infty$, 
\begin{equation*}
\|e^{-t\Lambda_p(a,b,q)}f\|_{r} \leq C_{\delta,d}e^{\omega_{p} t}t^{-\frac{d}{2}(\frac{1}{p}-\frac{1}{r})}\|f\|_p, \quad f \in L^p,
\end{equation*}
and so, by the Dunford-Pettis theorem, $e^{-t\Lambda(a,b,q)}$, $t>0$, are integral operators. 

Furthermore,
$v(t):=e^{-t\Lambda_2(a,b,q)}f$, $f \in L^2$, is the unique weak solution to Cauchy problem 
$$
(\partial_t-a\cdot\nabla^2 + (b+q) \cdot \nabla)v=0, \quad v|_{t=0}=f,
$$
in the standard Hilbert triple  $W^{1,2} \hookrightarrow L^2 \hookrightarrow W^{-1,2}$.

{The construction of the semigroup and the proof of the dispersion estimates follow the arguments in Remark \ref{disp_rem} and in the proof of Proposition \ref{thm_conv}, with $-\Delta$ replaced by $-\nabla\cdot a\cdot\nabla$ and $b$ replaced by $\nabla a+b$. The energy inequality now contains the positive coefficient $\frac{4(p-1)}{p^2}\xi-\frac{2}{p}\sqrt{\delta}$. Since $\delta<\xi^2$, this coefficient is positive also for $p=2$, and the uniqueness of the weak solution follows from the same Lions's argument and the compensated compactness estimate of Proposition \ref{cc_lem}.}

The local maximum principle: for every $f \in L^{p\theta} \cap L^{p\theta'} \cap C_\infty$, solution $u=(\mu+\Lambda(a,b,q))^{-1}f$ to the elliptic Kolmogorov equation $(\mu-a\cdot \nabla^2 + (b+q)\cdot \nabla)u=f$ satisfies for each $x \in \mathbb R^d$:
\begin{equation*}
\sup_{B_{\frac{1}{2}}(x)}|u|  \leq K \biggl(\langle |f|^{p\theta}\rho_x\rangle^{\frac{1}{p\theta}}  + \langle |f|^{p\theta'}\rho_x\rangle^{\frac{1}{p\theta'}} \biggr), 
\end{equation*}
for fixed $1<\theta<\frac{d}{d-2}$ and $p>\frac{2}{2-\xi^{-1}\sqrt{\delta}}$, for all $\mu$ strictly greater than certain $\mu_0$. The constants $K$ and $\mu_0$ do not depend on $f$ or $x$.

\medskip

Also, a priori Krylov-type bound analogous to \eqref{krylov_distr} holds.

{The local maximum principle follows from the analogue of Proposition \ref{sep_thm} for $\Lambda(a_n,b_n,q_m)$, discussed in the proof of Theorem \ref{thm2_a}(\textit{i}), by passing to the limit in the approximations. The Krylov-type bound is obtained from the corresponding analogue of Proposition \ref{emb_thm} in the same way as in Remark \ref{krylov_distr_rem}.}

\medskip

The elliptic gradient bounds in the case $q=0$ and $\delta<\frac{c}{d^2}$ are proved in \cite{KiS_Osaka}.

\end{remark}

\bigskip

\section{Weakly form-bounded drifts and Keller-Segel finite particles}

\label{wfb_sect}

In the previous sections we tested out results for SDEs with form-bounded drifts against the interacting particle system in Example \ref{ex1_multi}. This, however, was limited to dimensions $d \geq 3$. In dimension $d=2$, which is of interest e.g.\,in the study of the Keller-Segel model of chemotaxis, the particle system in Example \ref{ex1_multi} is more difficult to handle since its drift \eqref{drift_b} is not in $L^2_{\loc}(\mathbb R^2)$ and, thus, is not form-bounded. We can address this issue, at least to some extent, by pursuing a different approach to proving weak well-posedness of SDEs. It works for substantially larger class of weakly form-bounded drifts.

\begin{definition}
\label{wfb}
A vector field $b \in [L^1_{\loc}]^d$ is said to be weakly form-bounded if there exists constant $\delta>0$ such that
$$
\langle |b|\varphi,\varphi\rangle \leq \delta \|(\lambda-\Delta)^{\frac{1}{4}}\varphi\|_2^2, \quad \forall\,\varphi \in W^{1,2},
$$
for some $\lambda=\lambda_\delta \geq 0$. This will be abbreviated as $b \in \mathbf{F}_\delta^{\scriptscriptstyle \frac{1}{2}}$.
\end{definition}

\begin{example}
1.~ Morrey class $M_{1+\varepsilon}$ is a large subclass of $\mathbf{F}_\delta^{\scriptscriptstyle \frac{1}{2}}$ defined in elementary terms:
\begin{equation*}
\|b\|_{M_{1+\varepsilon}}:=\sup_{r>0, x \in \mathbb R^d} r\biggl(\frac{1}{|B_r|}\int_{B_r(x)}|b|^{1+\varepsilon}dx \biggr)^{\frac{1}{1+\varepsilon}}<\infty.
\end{equation*}
The inclusion follows by D.\,R.\,Adams' theorem \cite[Theorem 7.3]{A}. The value of $\delta$ will be proportional to the Morrey norm, with a coefficient that depends on the constants in some fundamental inequalities of Harmonic Analysis.

2.~The class of form-bounded drifts $\mathbf{F}_{\delta^2}$ considered in the previous section, i.e.
$$
\langle |b|^2\varphi,\varphi\rangle \leq \delta^2 \|(-\Delta)^{\frac{1}{2}}\varphi\|_2^2 + c_{\delta^2}\|\varphi\|_2^2 \quad \biggl(=\delta^2 \|(\lambda-\Delta)^{\frac{1}{2}}\varphi\|_2^2, \;\;\lambda=\frac{c_{\delta^2}}{\delta^2}\biggr),
$$
 is a proper subclass of $\mathbf{F}_{\delta}^{\scriptscriptstyle 1/2}$. This is seen easily by applying Heinz' inequality. 
Alternatively, one can invoke the inclusion $\mathbf{F}_\delta~(\text{with $c_\delta=0$}) \subset M_2$, see Examples \ref{ex_fbd} in Section \ref{classes_sect}, and, next, apply $M_2 \subset M_{1+\varepsilon}$ if $\varepsilon<1$, and then use the previous example. That said, if one follows this path, one to a large extent loses the control over the value of the form-bound $\delta$, which is in our focus in this paper.

3.~It is instructive to compare how $\mathbf{F}_\delta$ and $\mathbf{F}_{\delta}^{\scriptscriptstyle 1/2}$ handle the weak $L^d$ class.
Namely, for $|b| \in L^{d,\infty}$,
we verify, using \cite[Prop.~2.5, 2.6, Cor.~2.9]{KPS},
\begin{align*}
d \geq 2, \quad b \in \mathbf{F}_{\delta}^{\scriptscriptstyle \frac{1}{2}} \text{ with } \sqrt{\delta}&=\||b|^\frac{1}{2} (- \Delta)^{-\frac{1}{4}} \|_{2 \rightarrow 2}  \leqslant 
\|(|b|^\ast)^\frac{1}{2} (- \Delta)^{-\frac{1}{4}} \|_{2 \rightarrow 2} \\
& \leqslant \biggl(\|b\|_{d,\infty} \Omega_d^{-\frac{1}{d}}\biggr)^{\frac{1}{2}} \||x|^{-\frac{1}{2}} (- \Delta)^{-\frac{1}{4}} \|_{2 \rightarrow 2} = \biggl(\|b\|_{d,\infty} \Omega_d^{-\frac{1}{d}}\biggr)^{\frac{1}{2}} 2^{-\frac{1}{2}} \frac{\Gamma\bigl(\frac{d-1}{4} \bigr)}{\Gamma\bigl(\frac{d+1}{4} \bigr)},
\end{align*}
where $\Omega_d={\frac{\pi^{\frac{d}{2}}}{\Gamma(\frac{d}{2}+1)}}$, and $|b|^\ast$ is the symmetric decreasing rearrangement of $|b|$.
Similarly,
\begin{align*}
d \geq 3, \quad b \in \mathbf{F}_{\delta_1} \text{ with } \sqrt{\delta_1}&=\||b| (- \Delta)^{-\frac{1}{2}} \|_{2 \rightarrow 2} \\ & \leqslant 
\|b\|_{d,\infty} \Omega_d^{-\frac{1}{d}} \||x|^{-1} (- \Delta)^{-\frac{1}{2}} \|_{2 \rightarrow 2} \\ & \leqslant \|b\|_{d,\infty} \Omega_d^{-\frac{1}{d}}2^{-1}\frac{\Gamma\bigl(\frac{d-2}{4} \bigr)}{\Gamma\bigl(\frac{d+2}{4} \bigr)}=\|b\|_{d,\infty} \Omega_d^{-\frac{1}{d}} \frac{2}{d-2}.
\end{align*}
In particular,  using \cite[Cor.~2.9]{KPS}, we have 
\begin{equation}
\label{hardy1}
d \geq 2, \quad x|x|^{-2} \in \mathbf{F}^{\scriptscriptstyle \frac{1}{2}}_{\delta},
\quad \sqrt{\delta}=2^{-\frac{1}{2}} \frac{\Gamma\bigl(\frac{d-1}{4} \bigr)}{\Gamma\bigl(\frac{d+1}{4} \bigr)},
\end{equation}
\begin{equation}
\label{hardy2}
d \geq 3, \quad x|x|^{-2} \in \mathbf{F}_{\delta_1}, \quad \sqrt{\delta_1}=\frac{2}{d-2}.
\end{equation}
In fact, \eqref{hardy2}  coincides with the classical Hardy inequality.

4.~An important proper subclass of $\mathbf{F}_{\delta}^{\scriptscriptstyle 1/2}$ that is not contained in the Morrey class $M_{1+\varepsilon}$, regardless of how small $\varepsilon>0$ is, is the Kato class. The Kato class consists of vector fields $b \in [L^1_{\loc}]^d$ such that
$$
\|(\lambda-\Delta)^{-\frac{1}{2}}|b|\|_{\infty}  \leq \sqrt{\delta}
$$
for some $\delta>0$ and $\lambda=\lambda_\delta \geq 0$ (The inclusion Kato class $\subset \mathbf{F}_{\delta}^{\scriptscriptstyle 1/2}$ follows e.g.\,by duality and interpolation.) SDEs with Kato class drifts were treated  by Bass-Chen \cite{BC}, who studied Brownian motion on fractals such as the Sierpinski gasket. To be more precise, they considered measure-valued $b$ with the total variation $|b|$ satisfying the Kato class condition. Moreover, when $\delta$ is sufficiently small, one obtains two-sided Gaussian bounds on the heat kernel of $-\Delta + b \cdot \nabla$ \cite{Z_Kato}.
Note that the Kato class does not contain $[L^d]^d$, but, for every fixed $\varepsilon>0$, it contains some vector fields that are not in $[L^{1+\varepsilon}_{\loc}]^d$.
\end{example}

The proof of the next theorem is based on the resolvent representation \eqref{frac_repr} where we, crucially, work with the fractional powers $|b|^{\frac{1}{r}}$ for $r>d-1$ (so, here we take advantage of the fact that $b$ is not distributional or measure-valued). 

\begin{theorem}[{\cite{Ki_super, KiS_brownian}}]
\label{thm_wfb} Let $d \geq 2$. Assume that $b \in \mathbf{F}^{\scriptscriptstyle 1/2}_\delta$ with weak form-bound $\delta$ satisfying
$$
\delta<m_d^{-1}\left\{
\begin{array}{ll}
\frac{4(d-2)}{(d-1)^2} & \text{ if $d \geq 4$,} \\
1 & \text{ if $d=2,3$},
\end{array}
\right.
$$
where
$
m_d := 
\pi^{\frac{1}{2}} (2e)^{-\frac{1}{2}} d^\frac{d}{2} (d-1)^{\frac{1-d}{2}}
$.
The following are true:

\begin{enumerate}[label=(\roman*)]

\item {\rm (Weak solution to SDE)} There exists a strong Markov family of probability measures $\{\mathbb P_x\}_{x \in \mathbb R^d}$ on the canonical space of continuous trajectories $\mathbf C$ that deliver, for every $x \in \mathbb R^d$, a weak solution
to SDE 
\begin{equation}
\label{sde1__}
X_t=x-\int_0^t b(X_r)dr + \sqrt{2}B_t.
\end{equation}

\item {\rm (Feller semigroup)}
\begin{equation*}
(e^{-t\Lambda(b)} f)(x):=\mathbb E_{\mathbb P_x}[f(X_t)], \quad x \in \mathbb R^d,
\end{equation*}
is a strongly continuous Feller semigroup on $C_\infty$. 

\medskip

\item {\rm (Uniqueness of weak solution to Kolmogorov backward PDE \cite{KiS_JDE})}
$v(t):=e^{-t\Lambda(b)}f$, $f \in C_\infty \cap {\mathcal W^{\frac{1}{2},2}}$, is the unique weak solution to Cauchy problem 
$$
(\partial_t-\Delta + b \cdot \nabla)v=0, \quad v|_{t=0}=f,
$$
in the ``shifted'' triple of Bessel potential spaces $\mathcal W^{\frac{3}{2},2} \hookrightarrow \mathcal W^{\frac{1}{2},2} \hookrightarrow \mathcal W^{-\frac{1}{2},2}$. This result yields approximation uniqueness for $\{\mathbb P_x\}_{x \in \mathbb R^d}$.

\medskip

\item {\rm (Another kind of approximation uniqueness)}
If $\{\mathbb Q_x\}_{x \in \mathbb R^d}$ is another weak solution to \eqref{sde1__} such that
$$
\mathbb Q_x=w{\mbox-}\mathcal P(\mathbf{C})\mbox{-}\lim_n \mathbb P_x(\tilde{b}_n) \quad \text{for every $x \in \mathbb R^d$},
$$
for some $\{\tilde{b}_n\} \subset \mathbf{F}^{\scriptscriptstyle 1/2}_{\delta_1} \cap [C_b \cap C^\infty]^d$ with {$m_d\delta_1<\frac{4(d-2)}{(d-1)^2}$} if $d \geq 4$ or {$m_d\delta_1<1$} if $d=2,3$, and {$\lambda_{\delta_1}$} independent of $n$, then $\{\mathbb Q_x\}_{x \in \mathbb R^d}=\{\mathbb P_x\}_{x \in \mathbb R^d}.$  

\medskip

\item {\rm (Elliptic gradient bounds)} $u:=(\mu+\Lambda(b))^{-1}f$, $f \in C_\infty \cap L^r$, \newline $r \in ]{(d-1)\vee\frac{2}{1+\sqrt{1-m_d\delta}}},\frac{2}{1-\sqrt{1-m_d\delta}}[$, satisfies 
$$
\|(\mu-\Delta)^{\frac{1}{2}+\frac{1}{2s}}u\|_r \leq K\|(\mu-\Delta)^{-\frac{1}{2}+\frac{1}{2\ell}}f\|_r, \quad \text{ for all } 1 \leq \ell<r<s,
$$
for all $\mu$ greater than a generic $\mu_0$ (cf.\,\eqref{ii}). In particular, since $r>d-1$, we can select $s$ sufficiently close to $r$ so that by the Sobolev embedding theorem $u$ is H\"{o}lder continuous.
\end{enumerate}

\end{theorem}

\begin{remarks}
1.~Replacing condition $b \in \mathbf{F}_{\delta}$ by more general condition $b \in \mathbf{F}_{\delta}^{\scriptscriptstyle 1/2}$ comes at a cost. Although one can still include some diffusion coefficients, these may no longer be discontinuous (cf.\,\cite[Sect.\,14]{Ki_survey}). Also, distributional drifts $q \in \mathbf{BMO}^{-1}$ are out of reach. It is, however, possible to consider drifts $b+q$, where $b \in \mathbf{F}_{\delta}^{\scriptscriptstyle 1/2}$ and $q$ is measure-valued with total variation in the Kato class \cite{Ki_measure}.

2.~In Theorem \ref{thm_wfb}, we construct the candidate for resolvent of the Feller generator \textit{a priori}. It is the following formal Neumann series for $$\mu+\Lambda \supset \mu-\Delta + b \cdot \nabla,$$ possibly after a modification on a measure zero set \cite{Ki_super}:
\begin{equation}
\label{frac_repr}
(\mu+\Lambda)^{-1}f:=(\mu - \Delta)^{-1}f - (\mu - \Delta)^{-\frac{1}{2}-\frac{1}{2s}} Q_{r} (1 + T_r)^{-1} G_{r} (\mu - \Delta)^{-\frac{1}{2}+\frac{1}{2\ell}}f,
\end{equation}
where $f \in L^r \cap C_\infty$, and
\begin{align*}
 Q_r:=(\mu -\Delta )^{-\frac{1}{2}+\frac{1}{2s}}|b|^{\frac{1}{r'}}, \quad G_r&:=b^{\frac{1}{r}} \cdot \nabla (\mu -\Delta )^{-\frac{1}{2}-\frac{1}{2\ell}} \quad \text{ are bounded on }L^r,
\end{align*}
\begin{align*}
T_r&:=b^{\frac{1}{r}}\cdot \nabla(\mu - \Delta)^{-1}|b|^{\frac{1}{r'}} \quad \text{ is bounded on } L^r,
\end{align*}
where 
$$b^{\frac{1}{r}}:=|b|^{-1+\frac{1}{r}}b, \quad \ell, s \text{ satisfy }1 \leq \ell<r<s,$$
and, of course, one gets stronger regularity result by choosing $\ell$, $s$ close to $r$.
The proof of the boundedness of $Q_r$, $G_r$ and $T_r$ is based on the Stroock-Varopoulos inequalities for symmetric Markov generators, i.e.\,this is an elliptic argument (to bound $T_r$, we first apply pointwise estimate $|\nabla_x(\mu - \Delta)^{-1}(x,y)| \leq m_d(\kappa_d\mu-\Delta)^{-\frac{1}{2}}(x,y)$ for appropriate $m_d$, $\kappa_d>0$). {The smallness condition on $\delta$ and the restriction on $r$ in Theorem \ref{thm_wfb} ensure} $\|T_r\|_{r \rightarrow r}<1$, so that  $(1 + T_r)^{-1} $ converges as geometric series in $L^r$. Earlier, similar estimates were employed in \cite{BS,LS} to refine the $L^2$ theory of Schr\"{o}dinger operators with the usual form-bounded potentials to an $L^r$ theory, which allowed the authors, for instance, to obtain additional information about the Sobolev regularity of the eigenfunctions of Schr\"{o}dinger operators.

3.~A priori, it is not clear why \eqref{frac_repr} should determine the resolvent of a strongly continuous semigroup in $L^r$. The latter is, in fact, true: the proof uses Hille's theory of pseudoresolvents \cite{Ki_super}. When $r=2$, one can give a different proof using Lions' variational approach, but it requires working in a quintuple of Hilbert spaces (instead of the usual triple) \cite{KiS_theory}.

4.~Having an explicit candidate for the limiting object, i.e.\,the Feller resolvent, greatly simplifies the approximation arguments. In Theorem \ref{thm2}, no such representation is available, so one must rely on Trotter's approximation theorem, whose key feature is that it does not require any a priori representation of the limiting operator.

5.~Although these broad assumptions on $b$ destroy the usual $L^r$ estimates for second-order derivatives of solution $u$ to $(\mu-\Delta + b \cdot \nabla)u=f$, one can still use \eqref{frac_repr}  to obtain \textit{some} $L^r$ bounds on $\nabla^2 u$. However, either one needs restriction $r<d$ or these estimates are valid only in weighted space $L^r(\mathbb R^d, (1+|b(x)|)^{-r+1}dx)$. (The latter follows by applying $(1+|b|)^{-\frac{1}{r'}}(\mu-\Delta)$ to \eqref{frac_repr}; note that the information about the second derivatives of $u$ disappears at the points where $|b|$ is infinite, but in a controlled way, see \cite{Ki_survey} for more detailed discussion.)

6.~A straightforward dual variant of the resolvent representation \eqref{frac_repr} produces, in particular, strongly continuous semigroup $e^{-t\Lambda^\ast}$ for the Fokker-Planck operator 
\begin{equation}
\label{l_ast}
\Lambda^\ast \supset -\Delta - \nabla \cdot b, \quad b \in \mathbf{F}_{\delta}^{\scriptscriptstyle 1/2},
\end{equation}
in $L^r$, where $r>1$ can be chosen as close to $1$ as needed at expense of assuming that the weak form-bound $\delta$ is sufficiently small.

Now, let us recall the Ambrosio-Figalli-Trevisan superposition principle: one can construct a weak solution to SDE
\begin{equation}
\label{sde_aft}
X_t=X_0-\int_0^t b(X_r)dr + \sqrt{2}\int_0^t \sigma(X_s) dB_s
\end{equation}
for bounded $b$ and $\sigma$,
provided that
\begin{equation}
\label{sup_pr}
t \mapsto \int_{\mathbb R^d} \varphi d\mu_t \text{ is continuous {on $[0,T]$} for every $\varphi \in C_b(\mathbb R^d)$},
\end{equation}
where $\mu_t$ (a probability measure for each ${t\in[0,T]}$) is a weak solution of the corresponding Fokker-Planck equation with $\mu_0={\rm Law\,}(X_0)$; then one has ${\rm Law\,}(X_t)=\mu_t$ for all ${t\in[0,T]}$ \cite{Tr}. {For unbounded coefficients, assuming that $b$ and $a=\sigma\sigma^{\top}$ are locally integrable with respect to $\mu_t(dx)dt$, the superposition principle is valid under the additional integrability condition of \cite{BRS}:}
\begin{equation}
\label{int_cond}
\int_0^T \int_{\mathbb R^d} \frac{|\langle b(x),x\rangle| + |a(x)|}{1+|x|^2}\mu_t(dx)dt<\infty.
\end{equation}
 This principle, and the appropriate regularity results on the Fokker-Planck equations, can be used to treat SDEs with some quite singular drifts that are not covered by the results in the present paper, and some fairly degenerate diffusion coefficients, see \cite{Gru} and references therein. 

An application of this principle requires constructing a probability solution of the Fokker-Planck equation and verifying \eqref{sup_pr} and \eqref{int_cond}. The principle does not require the initial law to have a density; in particular, $\mu_0$ may be a delta-function.

In Theorem \ref{thm_wfb}, we construct directly a Feller family of weak solutions starting from every point and establish approximation uniqueness. These properties do not follow from the superposition principle.

\end{remarks}

\medskip

\begin{corollary}[Finite particle approximation of the elliptic-parabolic Keller-Segel model]
\label{cor1}
In $\mathbb R^{2N}$, {$N\geq2$, }consider SDE
\begin{equation}
\label{ks_eq}
X_t=x_0-\int_0^t b(X_s)ds + \sqrt{2}B_t, \quad x_0=(x_0^1,\dots,x_0^N) \in \mathbb R^{2N}, 
\end{equation}
where $B_t=(B_t^1,\dots,B_t^N)$ is a Brownian motion in $\mathbb R^{2N}$, and
\begin{equation}
\label{b_hardy}
b_i(x^1,\dots,x^N):=\frac{\sqrt{\kappa}}{N}\sum_{j=1, j \neq i}^N \frac{x^i-x^j}{|x^i-x^j|^2}.
\end{equation}
Then, provided that $\kappa<\frac{C}{N^{{4}}}$ {for a sufficiently small constant $C>0$ independent of $N$}, the assertions of Theorem \ref{thm_wfb} are valid for this particle system.
\end{corollary}

\begin{proof} Thus defined drift $b:\mathbb R^{2N} \rightarrow \mathbb R^{2N}$ is in $\mathbf{F}_\delta^{\scriptscriptstyle 1/2}$. In fact, it is in the Morrey class $M_{1+\varepsilon}$, a subclass of  $\mathbf{F}_\delta^{\scriptscriptstyle 1/2}$. To see this, it suffices to prove this inclusion for a single term
$$
\mathbb R^{2N} \ni (x^1,\dots,x^N) \mapsto \frac{x^1-x^2}{|x^1-x^2|^2}
$$
which, after a change of variable, reduces to proving that the scalar function $(x^1,\dots,x^N) \mapsto |x^1|^{-1}$ is in the Morrey class $M_{1+\varepsilon}$. Put $C_r(x)=D_r(x^1) \times \dots \times D_r(x^N)$ (the direct product of $N$ discs centered at $x^i$). We have
\begin{align*}
\||x^1|^{-1}\|_{M_{1+\varepsilon}} & \leq c \sup_{r>0}r\biggl(\frac{1}{r^{2N}}\langle \mathbf{1}_{C_r(0)}|x^1|^{-(1+\varepsilon)}\rangle \biggr)^{\frac{1}{1+\varepsilon}} \\
& = c\sup_{r>0}r\biggl(\frac{1}{r^{2}} \int_{D_r(0)}|x^1|^{-(1+\varepsilon)}dx^1 \biggr)^{\frac{1}{1+\varepsilon}} \\
& = c\sup_{r>0} r\biggl(\frac{1}{r^{2}} \int_{0}^r t^{-(1+\varepsilon)+1}dt \biggr)^{\frac{1}{1+\varepsilon}} = c\sup_{r>0} r \biggl(\frac{1}{r^{2}} \frac{r^{-\varepsilon+1}}{-\varepsilon+1} \biggr)^{\frac{1}{1+\varepsilon}} <\infty.
\end{align*}

To estimate the weak form-bound, put
$$
C_0:=\frac{1}{2}\biggl(\frac{\Gamma(\frac{1}{4})}{\Gamma(\frac{3}{4})}\biggr)^2.
$$
By \eqref{hardy1} in dimension $2$, applied in the variable $x^i$ with the other variables fixed, we have, for every $\varphi\in C_c^\infty(\mathbb R^{2N})$,
\begin{align*}
\langle |b|\varphi,\varphi\rangle
&\leq \frac{\sqrt{\kappa}}{N}\sum_{i=1}^N\sum_{j\neq i}\big\langle |x^i-x^j|^{-1}\varphi,\varphi\big\rangle\\
&\leq C_0\sqrt{\kappa}\frac{N-1}{N}\sum_{i=1}^N\|(-\Delta_{x^i})^{\frac{1}{4}}\varphi\|_2^2\\
&\leq C_0\sqrt{\kappa}\frac{N-1}{\sqrt{N}}\|(-\Delta)^{\frac{1}{4}}\varphi\|_2^2.
\end{align*}
The last inequality follows from $\sum_i|\xi^i|\leq\sqrt{N}|\xi|$.
Thus one can take $\delta=C_0\sqrt{\kappa}(N-1)/\sqrt{N}$ and $\lambda_\delta=0$.
Since $m_{2N}\leq\sqrt{\pi N}$, condition $\kappa<4/(\pi C_0^2N^4)$ implies
$$
m_{2N}\delta<\frac{2(N-1)}{N^2}<\frac{4(2N-2)}{(2N-1)^2},
$$
as required in Theorem \ref{thm_wfb}.

\end{proof}

The main, quite unacceptable drawback of Corollary \ref{cor1} 
is that the condition on $\kappa$ degenerates as $N \rightarrow \infty$.
Fournier-Jourdain \cite{FJ}, Fournier-Tardy \cite{FT} and Tardy \cite{T} exploit the special form of the interaction kernel in \eqref{b_hardy} and establish for \eqref{ks_eq}, among other results, weak existence and the existence of mean field limit as $N \rightarrow \infty$ for all $\kappa \in [0,16[$, where $16$ is the sticky collisions threshold for \eqref{ks_eq}, \eqref{b_hardy}. One can also apply the Dirichlet forms approach, see Cattiaux-P\'{e}d\`{e}ches \cite{CP}. Already the weak existence results of \cite{FJ} are thus much stronger  than  Corollary \ref{cor1}. Our point here, however, is different. Corollary \ref{cor1} shows that one, in fact, can reach the Keller-Segel finite particle system \eqref{ks_eq}, \eqref{b_hardy} by applying results on general singular SDEs.

One advantage of Theorem \ref{thm_wfb}, compared to \cite{CP,FJ,FT,T}, is that we can easily modify the drift in Corollary \ref{cor1}. For example, multiplying each interaction kernel by a function with $L^\infty$ norm less than or equal to one does not affect the conclusion.

\bigskip

\section{Proof of Theorem \ref{thm2}}
\label{proof_thm2_sect}

Let us write, to shorten notations,
$$
\Lambda_{n,m}:=\Lambda(b_n,q_m) \equiv -\Delta + (b_n+q_m)\cdot \nabla, \quad D(\Lambda_{n,m})=(1-\Delta)^{-1}C_\infty.
$$

\medskip

\noindent \emph{Proof of ({\textit{i}})} This assertion  will follow from the Trotter approximation theorem (see e.g.\,{\cite[IX.2.5]{Ka}}). Applied to contraction semigroups $\{e^{-t\Lambda_{n,m}}\}_{n,m \geq 1}$ in $C_\infty$, this theorem is stated as follows:

\begin{theorem}[Trotter's approximation theorem]
\label{trotter_thm}
Assume that there exists constant $\mu_0>0$ independent of $n$, $m$ such that

\smallskip

{\rm $1^\circ$)} $\sup_{n,m \geq 1}\|(\mu+\Lambda_{n,m})^{-1}f\|_\infty \leq \mu^{-1}\|f\|_\infty$, $\mu \geq \mu_0$.

\smallskip

{\rm $2^\circ$)} there exists $\text{\small $s\text{-}C_\infty\text{-}$}\lim_{n}\lim_m (\mu+\Lambda_{n,m})^{-1}$ for some $\mu \geq \mu_0$.

\smallskip

{\rm $3^\circ$)} $\mu (\mu+\Lambda_{n,m})^{-1} \rightarrow 1$ in $C_\infty$ as $\mu \uparrow \infty$ uniformly in $n$, $m$.

\medskip

\noindent Then there exists a contraction strongly continuous semigroup $e^{-t\Lambda}$ on $C_\infty$ such that
$$
e^{-t\Lambda}=s\mbox{-}C_\infty\mbox{-}\lim_n\lim_m e^{-t\Lambda_{n,m}}
$$
locally uniformly in $t \geq 0$.
\end{theorem}

\subsection{Main PDE results}

\begin{proposition}[Embedding property]
\label{emb_thm}
Let $w_{n,m}$ be the classical solution to elliptic equation
\begin{equation}
\label{eq11}
\big(\mu -\Delta + (b_n+q_m) \cdot \nabla\big)w_{n,m}=(b_n^i+q_m^i)f, \quad f \in \mathcal S,
\end{equation}
where we have fixed $1 \leq i \leq d$ and have denoted
$
q^i_m=\sum_{j=1}^d \nabla_j Q_m^{{ji}}.
$
Put
$$
A_m:=|Q^i_m| |\nabla f|+(1+|Q^i_m|)|f| \quad \text{ where } Q_m^i \text{ denotes the $i$-th {column} of $Q_m$}.
$$
Then, for every $p>\frac{2}{2-\sqrt{\delta}}$, $p \geq 2$ and $1<\theta<\frac{d}{d-2}$, there exist constants $\mu_0>0$, $0<\beta<1$ and $K_j$, $j=1,2$, independent of $n$, $m$, such that
\begin{align}
\|w_{n,m}\|_\infty & \leq K_1 (\mu-\mu_0)^{-\frac{\beta}{p}}\|A_m\|_{p\theta'} 
\notag \\
& + K_2 (\mu-\mu_0)^{-\frac{1}{p\theta}}\|A_m\|_{p\theta}, \label{w_est_nm}
\end{align}
for all $\mu >\mu_0$, {where $\theta':=\frac{\theta}{\theta-1}$ is the conjugate exponent of $\theta$}.
\end{proposition}

We prove Proposition \ref{emb_thm} in Section \ref{emb_thm_sect}.

\medskip

Let $u_{n,m}$ be the classical solutions to elliptic equation
\begin{equation}
\label{eq10}
\big(\mu -\Delta + (b_n+q_m) \cdot \nabla\big)u_{n,m}=f, \quad \mu>0,\quad {f\in\mathcal S}.
\end{equation}

\begin{proposition}[A priori H\"{o}lder continuity]
\label{thm_holder}For every $\mu>0$,
$\{u_{n,m}\}$ are locally H\"{o}lder continuous uniformly in $n$, $m$, {with constants depending on $f$ only through $\|f\|_\infty$}.
\end{proposition}

That is, $u_{n,m}$ are H\"{o}lder continuous, in every unit ball, with constants that do not depend on $n$, $m$ or the center of the ball. These constants are, however, allowed to depend on $\|f\|_\infty$.
We prove Proposition \ref{thm_holder} in Section \ref{thm_holder_sect}.

\begin{proposition}[Convergence]
\label{thm_conv}
There exists $\mu_0>0$ such that for every $\mu  \geq \mu_0$, for all $p>\frac{2}{2-\sqrt{\delta}}$, $p \geq 2$, and any $x \in \mathbb R^d$, there exists the limit $$u:=L_{\rho_x}^p\mbox{-}\lim_n\lim_m u_{n,m},$$
where 
$\rho_x(y):=\rho(y-x)$,
provided that constant $\sigma$ in the definition of weight $\rho$ (this is \eqref{rho_def}) is chosen sufficiently small (independently of $x$).
\end{proposition}

We prove Proposition \ref{thm_conv} in Section \ref{thm_conv_sect}.

\begin{remark}
{The order of the limits in Proposition \ref{thm_conv} comes from its proof: first $m\rightarrow\infty$ with $n$ fixed, using the boundedness of $b_n$ in the $L^2$ argument of Step 4 in Section \ref{thm_conv_sect}; then $n\rightarrow\infty$ in Steps 5 and 6. The bounds used at the first step may depend on $n$.}
The proof of Proposition \ref{thm_conv} can be extended to show the existence of the limit $L^p\mbox{-}\lim_{n,m} u_{n,m}$,  but at expense of imposing additional assumptions on drifts $b$ or $q$, such as form-bound $\delta$ of $b$ being strictly less than $1$, or the stream matrix $Q$ of $q$ having entries in ${\rm VMO}$. See Remark \ref{stronger_conv_rem} in the end of the proof of Proposition \ref{thm_conv}.
\end{remark}

\begin{proposition}[Separation property/local maximum principle] 
\label{sep_thm}
Fix some $1<\theta<\frac{d}{d-2}$ and $p>\frac{2}{2-\sqrt{\delta}}$, $p \geq 2$. {Here again $\theta'=\frac{\theta}{\theta-1}$.}
There exist constants $K$, $\mu_0>0$ and  $\sigma$ (in the definition of weight $\rho$) independent of $n$, $m$ such that for all $\mu \geq \mu_0$, for every $x \in \mathbb R^d$, 
\begin{equation}
\label{sep_prop}
\sup_{B_{\frac{1}{2}}(x)} |u_{n,m}| \leq K \biggl( \langle |f|^{p\theta}\rho_x\rangle^{\frac{1}{p\theta}} + \big\langle |f|^{p\theta'}\mathbf{1}_{B_1(x)}\big\rangle^{\frac{1}{p\theta'}} \biggr).
\end{equation}
\end{proposition}

We prove Proposition \ref{sep_thm} in Section \ref{sep_thm_sect}.

\medskip

We are in a position to verify conditions of Trotter's theorem for $u_{n,m}=(\mu+\Lambda_{n,m})^{-1}f$:

\subsection{Proof of {\rm $1^\circ$)}} This condition is a direct consequence of the fact that, by the classical theory, $e^{-t\Lambda_{n,m}}$ are $L^\infty$ contractions.

\subsection{Proof of {\rm $2^\circ$)}} By {\rm $1^\circ$)}, it suffices to verify the existence of the limit for all $f$ belonging to a countable dense subset of $C_c^\infty$. 
Proposition \ref{thm_holder} and the Arzel\`{a}-Ascoli theorem yield: for every $r>0$, $\{u_{n,m}\}$ is relatively compact in $C(\bar{B}_r)$. Proposition \ref{thm_conv} allows to further conclude that  $\{u_{n,m}\}$ converges in $C(\bar{B}_r)$, for every $r>0$:
\begin{equation}
\label{loc_unif}
u \upharpoonright \bar{B}_r=s\mbox{-}C(\bar{B}_r)\mbox{-}\lim_n \lim_m u_{n,m} \upharpoonright \bar{B}_r.
\end{equation}

\begin{remark}
We need Proposition \ref{thm_conv} to ensure that any two partial limits of $u_{n,m}$ in $C(\bar{B}_r)$  coincide. The choice of the topology, and thus the weight, is secondary.
\end{remark}

We need to improve \eqref{loc_unif} to global uniform convergence:
\begin{equation}
\label{glob_unif}
u=s\mbox{-}C_\infty \mbox{-}\lim_n \lim_m u_{n,m}.
\end{equation}
To this end, we combine Proposition \ref{sep_thm} and convergence \eqref{loc_unif}. Namely, since $f \in C_c^\infty$ and weight $\rho_x$ vanishes at infinity, it follows from \eqref{sep_prop} that solution $u_{n,m}$ is small uniformly in $n$, $m$ when considered far away from the support of $f$. (Hence the name ``separation property'' for \eqref{sep_prop}.) 

We have verified condition $2^\circ)$ of Trotter's theorem.

\begin{remark}
Using Proposition \ref{sep_thm}, it is easy to obtain the preservation of probability, i.e.\,that $$e^{-t\Lambda(b_n,q_m)}(1-\mathbf{1}_{B_R}) \rightarrow 0 \quad \text{ as $R \uparrow \infty$ uniformly in $n$, $m$.}$$ From here it follows easily that $e^{-t\Lambda(b,q)}1=1$. {It should be added that, for each $x$, the Riesz representation theorem gives a finite positive Borel measure $p_t(x,dy)$ such that $$e^{-t\Lambda(b,q)}f(x)=\int_{\mathbb R^d} f(y)p_t(x,dy).$$ This representation allows to extend $e^{-t\Lambda(b,q)}$ to bounded Borel measurable functions (evidently, without the strong continuity property) and so, in particular, the expression $e^{-t\Lambda(b,q)}1$ is well-defined.}
\end{remark}

\subsection{Proof of {\rm $3^\circ$)}} In view of {\rm $1^\circ$)}, it suffices to verify {\rm $3^\circ$)} on a dense subset of $C_\infty$, e.g.\,$C_c^\infty$. Fix $g \in C_c^\infty$. By the resolvent identity,
\begin{align*}
\mu (\mu+\Lambda_{n,m})^{-1} g - \mu(\mu-\Delta)^{-1} g & = {-}\mu (\mu+\Lambda_{n,m})^{-1} (b_n+q_m) \cdot \nabla (\mu-\Delta)^{-1} g \\
& = {-}(\mu+\Lambda_{n,m})^{-1} (b_n+q_m) \cdot \mu (\mu-\Delta)^{-1} \nabla g.
\end{align*}
{Here $(\mu-\Delta)^{-1}\nabla g$ denotes the vector with components $(\mu-\Delta)^{-1}\nabla_i g$, $1\leq i\leq d$.}
Since $ \mu(\mu-\Delta)^{-1} g \rightarrow g$ uniformly on $\mathbb R^d$ as $\mu \rightarrow \infty$, it suffices to show the convergence 
\begin{align}
\label{conv}
\|(\mu+\Lambda_{n,m})^{-1} (b_n+q_m) \cdot \mu (\mu-\Delta)^{-1} \nabla g\|_\infty \rightarrow 0 \text{ as } \mu \rightarrow \infty \quad \text{ uniformly in $n$, $m$}.
\end{align}
To that end, we apply Proposition \ref{emb_thm} to $w_{n,m}:=(\mu+\Lambda_{n,m})^{-1} (b^i_n+q^i_m) f$
with $f$ taken to be
\begin{equation}
\label{f_def}
f:=\mu (\mu-\Delta)^{-1} \nabla_i g. 
\end{equation}
Our goal is thus to prove \begin{equation}
\label{w_nm_conv}
\|w_{n,m}\|_\infty \rightarrow 0 \quad \text{ as $\mu \rightarrow \infty$ uniformly in $n$, $m$.}
\end{equation}
If we can obtain the bound
\begin{equation}
\label{bd_task}
\sup_{\mu \geq 1, m}\|A_m\|_{p\theta}, \sup_{\mu \geq 1, m}\|A_m\|_{p\theta'}<\infty \quad \text{for $f$ given by \eqref{f_def}},
\end{equation}
then the convergence \eqref{w_nm_conv} will follow thanks to the factors $(\mu-\mu_0)^{-\frac{\beta}{p}}$, $(\mu-\mu_0)^{-\frac{1}{p\theta}}$ in  \eqref{w_est_nm}. 
The only slightly non-trivial aspect of proving \eqref{bd_task} is that $|Q_m^i|$ can grow at infinity. But since the entries of $Q_m^i$ are ${\rm BMO}$ functions, this growth cannot be arbitrary, see Lemma \ref{est_44_lem}. 

\smallskip

\textit{Proof of \eqref{bd_task}}. So, from now on, let $f$ be given by \eqref{f_def} where, recall, $g$ has compact support (we will need this when we apply Lemma \ref{est8_lem}). In what follows, for brevity, $s=p\theta$ or $s=p\theta'$.
We have
\begin{align}
\||Q^i_m||f|\|^s_{s} & =\langle |Q^i_m|^s |\mu (\mu-\Delta)^{-1} \nabla_i g|^s \rangle \notag \\
& (\text{write $1=(1+|x|)^{-(d+\epsilon_0)s}(1+|x|)^{(d+\epsilon_0)s}$)} \notag \\
& \leq \big\langle |Q^i_m|^{s \nu} (1+|x|)^{-(d+\epsilon_0)s\nu} \big\rangle^{\frac{1}{\nu}} \big\langle  (1+|x|)^{(d+\epsilon_0)s\nu'} |\mu (\mu-\Delta)^{-1} \nabla_i g|^{s\nu'} \big\rangle^{\frac{1}{\nu'}}, \quad \nu>1. \label{Q_8_est}
\end{align}
We will need the following elementary lemma.

\begin{lemma}
\label{est8_lem}
For every $r>0$, $g \in C_c^\infty$ and $\mu\geq1$,
\begin{equation}
\label{est8}
|(1+|x|^{r}) \mu (\mu-\Delta)^{-1} \nabla_i g (x)| \leq c_R \mu (c\mu-\Delta)^{-1}|\nabla_i g|(x),
\end{equation}
for some positive constants $c$, $c_R$ that depend only on $d$, $r$ and {the radius $R$ of a fixed ball that contains the support of $|\nabla_i g|$.}
\end{lemma}

\begin{proof}For the reader's convenience, we provide the proof.
We estimate
\begin{align}
\label{weight_mu}
(1+|x|^{r}) \mu (\mu-\Delta)^{-1} \nabla_i g (x) & =\mu\int_0^\infty \int_{\mathbb R^d} e^{-\mu s} (4\pi s)^{-\frac{d}{2}}   (1+|x|^{r}) e^{-\frac{|x-y|^2}{4s}}  \nabla_i g (y) dy ds
\end{align}
as follows.
For all $\mu \geq 1$, provided that $1 \leq s<\infty$,
\begin{align}
 e^{-\mu s} (4\pi s)^{-\frac{d}{2}} \big|(1+|x|^{r}) e^{-\frac{|x-y|^2}{4s}}  \nabla_i g (y) \big| & {\leq e^{-\mu s} (4\pi s)^{-\frac{d}{2}} c_r(1+|y|^{r}+|x-y|^r) e^{-\frac{|x-y|^2}{4s}}  |\nabla_i g (y) |} \notag\\
 &\leq C_R e^{-c_1 \mu s}(4 \pi s)^{-\frac{d}{2}} e^{-\frac{|x-y|^2}{c_2s}}  |\nabla_i g (y)| \label{g_est}
\end{align}
for some $c_2>4$, $0<c_1<1$. To see this, it suffices to show that
\begin{equation}
\label{abs_est}
 e^{-\mu s} (4\pi s)^{-\frac{d}{2}} (1+|x|^{r}) e^{-\frac{|x|^2}{4s}}  \leq C e^{-c_1 \mu s}(4 \pi s)^{-\frac{d}{2}} e^{-\frac{|x|^2}{c_2 s}}, \quad x \in \mathbb R^d, s \geq 1,
\end{equation}
since $y$ varies only in $B_R$, and the sought estimate is non-trivial when $|x| \gg R$.
Inequality \eqref{abs_est} reduces to 
$$
e^{- \gamma_1 \mu s} (1+|x|^{r})e^{-\frac{|x|^2}{\gamma_2 s}} \leq C 
$$
for $0<\gamma_1<1$ and $\gamma_2>4$ ($\gamma_1=1-c_1$, $\frac{1}{\gamma_2}=\frac{1}{4}-\frac{1}{c_2}$).
So, denoting $x'=\frac{x}{\sqrt{s}}$, we obtain for $\mu \geq 1$, for all $x' \in \mathbb R^d$, $s \geq 1$,
\begin{align*}
e^{- \gamma_1 \mu s} (1+|x|^{r})e^{-\frac{|x|^2}{\gamma_2 s}} & \leq e^{-\gamma_1 \mu s} s^{\frac{r}{2}}(1+|x'|^{r})e^{-\frac{|x'|^2}{\gamma_2}} \\
& \leq C e^{-\frac{\gamma_1}{2} \mu s} e^{-\frac{|x'|^2}{2\gamma_2}} \leq C,
\end{align*}
which gives us the previous estimate and hence \eqref{abs_est}.
We thus have \eqref{g_est} for all $1 \leq s<\infty$. Note that \eqref{g_est} is trivial for $0<s<1$. So, armed with \eqref{g_est} for all $0<s<\infty$, we estimate \eqref{weight_mu}:
\begin{align*}
\mu\int_0^\infty \int_{\mathbb R^d} e^{-\mu s} (4\pi s)^{-\frac{d}{2}} (1+|x|^{r}) e^{-\frac{|x-y|^2}{4s}}  |\nabla_i g (y)| dy ds \leq c_R \mu (c\mu-\Delta)^{-1}|\nabla_i g|(x) ,\quad c=4c_1 c_2^{-1},
\end{align*}
i.e.\,we have obtained \eqref{est8}. 
\end{proof}

We will also need

\begin{lemma}[{see e.g.\,\cite[Prop.\,7.1.5]{Graf_M}}]
\label{est_44_lem}
For every $f \in {\rm BMO}(\mathbb R^d)$, for all $1 \leq r<\infty$,
\begin{equation}
\label{est_44}
\langle|f-(f)_{B_1}|^r \rho^r \rangle \leq C_{d,r,\varepsilon}\|f\|_{{\rm BMO}}^r,
\end{equation}
where $\rho=\rho_{\varepsilon}$ is defined by \eqref{rho_def}.
Hence
\begin{equation}
\label{est_44_final}
\langle |f|^r \rho^r \rangle \leq C'_{d,r,\varepsilon} \bigl(\|f\|_{{\rm BMO}}^r + \|f\mathbf{1}_{B_1}\|^r_{L^1}\bigr)<\infty.
\end{equation}
\end{lemma}

\medskip

By Lemma \ref{est8_lem} (with $r:=d+\epsilon_0$), we have pointwise estimate
$
(1+|x|)^{(d+\epsilon_0)} |\mu (\mu-\Delta)^{-1} \nabla_i g(x)| \leq c_R \mu (c\mu-\Delta)^{-1}|\nabla_i g|(x)
$
with constants $c_R$, $c$ independent of $\mu \geq 1$, so we can estimate the second multiple in \eqref{Q_8_est}:
$$
\big\langle  (1+|x|)^{(d+\epsilon_0)s\nu'} |\mu (\mu-\Delta)^{-1} \nabla_i g|^{s\nu'} \big\rangle \leq
C \|\mu (c\mu-\Delta)^{-1} |\nabla_i g|\|_{s\nu'}^{s\nu'}.
$$
In turn, by the contractivity of the heat semigroup in Lebesgue spaces, for all $\mu \geq 1$, $$\|\mu (c\mu-\Delta)^{-1} |\nabla_i g|\|_{s\nu'} \leq c^{-1}\|\nabla_i g\|_{s\nu'}.$$ Thus, the second multiple in \eqref{Q_8_est} can be estimated as follows: for all $\mu \geq 1$,
$$
\big\langle  (1+|x|)^{(d+\epsilon_0)s\nu'} |\mu (\mu-\Delta)^{-1} \nabla_i g|^{s\nu'} \big\rangle \leq C\|\nabla_i g\|_{s\nu'}^{s\nu'},
$$
where $C$ depends on the support of $g$, but does not depend on $\mu$.

On the other hand, regarding the first multiple in \eqref{Q_8_est}, we have by Lemma \ref{est_44_lem} (that is, \eqref{est_44_final} with $r=s\nu$) after taking into account that the ${\rm BMO}$ semi-norms of $Q^i_m$ are uniformly in $m$ bounded,
$$
\sup_m \big\langle |Q^i_m|^{s \nu} (1+|x|)^{-(d+\epsilon_0)s\nu} \big\rangle < \infty.
$$
We can thus conclude from \eqref{Q_8_est}:
$$
\sup_{\mu \geq 1, m}\||Q^i_m||f|\|_s<\infty.
$$

In the same way, since $\nabla f= \mu (\mu-\Delta)^{-1} \nabla_i \nabla g$ and the entries of $\nabla g$ have compact supports,
$$
\sup_{\mu \geq 1, m}\||Q^i_m||\nabla f|\|_s<\infty.
$$
The last two bounds yield \eqref{bd_task}. In view of the previous discussion, condition {\rm $3^\circ$)} of Trotter's theorem is thus verified.

\medskip

Now, Trotter's theorem applies and gives us assertion (\textit{i}) of Theorem \ref{thm2}.

\medskip

\noindent \emph{Proof of (\textit{ii})} The first statement, i.e.\,the approximation uniqueness, is immediate from the existence of the limit in (\textit{i}) and the fact that $[b]$, $[q]$ are closed with respect to passing to a sub-sequence. Let us prove the second statement, i.e.\,a relaxed approximation uniqueness. It suffices for us to replace Proposition \ref{thm_conv} with the following:
for all sufficiently large $\mu$  and every $f \in \mathcal S$, $u_n=(\mu+\Lambda(b_n,q))^{-1}f$ converge to the same limit: 
\begin{equation}
\label{u_n_conv5}
u_n \rightarrow u=(\mu+\Lambda(b,q))^{-1}f \quad \text{ in } L^2_{\rho}.
\end{equation}
The rest repeats the proof of assertion (\textit{i}).

\textit{Proof of \eqref{u_n_conv5}}. Since $b \in \mathbf{F}_\delta$, $\delta<1$ and ${\rm div\,}q=0$ provide, via Cauchy-Schwarz inequality and the compensated compactness estimate (Proposition \ref{cc_lem}), coercivity and boundedness of the quadratic form of $-\Delta + (b+q) \cdot \nabla$ in $L^2$, it is readily seen that function $u=(\mu+\Lambda(b,q))^{-1}f$, $f \in C_\infty \cap L^2$, constructed via approximation in (\textit{i}), is a weak solution to elliptic equation $(\mu-\Delta + (b+q) \cdot \nabla)u=f$, i.e.
$
\mu \langle u,\varphi\rangle + \langle \nabla u,\nabla \varphi\rangle + \langle \nabla u,b\varphi\rangle + \langle \nabla u,q\varphi\rangle=\langle f,\varphi\rangle $ for all $\varphi \in W^{1,2}$.

{In particular, $u\in W^{1,2}$. Indeed, Proposition \ref{elem_lem}(\textit{i}) gives uniform $W^{1,2}$ bounds for the approximations used in (\textit{i}); weak compactness and their locally uniform convergence identify the weak $W^{1,2}$ limit with $u$. Thus $\|\nabla u\|_2<\infty$ is known before considering the new approximations $u_n$.}

Put $h_n:=u_n-u$. Our goal is to show convergence $h_n \rightarrow 0$ in $L^2_\rho$. Notice that, in contrast to the proof of assertion (\textit{i}), we already have the limiting object $u$, which greatly simplifies the analysis.

Step 1.~{Passing to the limit in Proposition \ref{sep_thm} (``separation property'') along the approximations defining $u$ and $u_n$, and applying the triangle inequality,} yields: for a fixed $1<\theta<\frac{d}{d-2}$ there exist constants $K$  and $\mu_0>0$ independent of $n$ such that for all $\mu \geq \mu_0$, for every $x \in \mathbb R^d$, 
\begin{equation*}
\sup_{n \geq 1}\sup_{B_{\frac{1}{2}}(x)} |h_n| \leq K \biggl( \langle |f|^{2\theta}\rho_x\rangle^{\frac{1}{2\theta}} + \big\langle |f|^{2\theta'}\mathbf{1}_{B_1(x)}\big\rangle^{\frac{1}{2\theta'}} \biggr),
\end{equation*}
provided constant $\sigma$ in the definition of $\rho$ is fixed sufficiently small. {This is the choice in Proposition \ref{sep_thm}: by \eqref{rho_est}, the terms containing derivatives of $\rho$ carry a factor proportional to $\sqrt\sigma$ and are absorbed in the weighted energy estimate.} Since $f \in \mathcal S$, it follows that for every $\varepsilon>0$ there exists sufficiently large $R>0$ such that
\begin{equation}
\label{h_n_conv}
\sup_n \sup_{x \in \mathbb R^d \setminus B_R}|h_n(x)|<\varepsilon.
\end{equation}

Step 2.~The difference $h_n=u_n-u$ satisfies
$$
\mu \langle h_n,\varphi\rangle + \langle \nabla h_n,\nabla \varphi\rangle + \langle b_n \cdot \nabla h_n,\varphi\rangle + \langle q \cdot \nabla h_n,\varphi\rangle=\langle (b-b_n)\cdot \nabla u,\varphi\rangle \quad \text{ for all } \varphi \in W^{1,2}.
$$
Hence, taking $\varphi=h_n \rho$ and repeating the proof of the energy inequality of Proposition \ref{elem_lem}(\textit{ii}) (take there $s=2$, which is possible since now $b$, $b_n$ have form-bound $\delta<1$), with $\sigma$ in the definition of $\rho$ fixed sufficiently small, we obtain
$$
(\mu-\mu_0)\langle |h_n|^2\rho\rangle + C_1\langle |\nabla h_n|^2\rho\rangle \leq \langle (b-b_n)\cdot \nabla u,h_n\rho\rangle,
$$
where $C_1>0$ is independent of $n$.
Thus, it suffices to show that $\langle (b-b_n)\cdot \nabla u,h_n \rho \rangle \rightarrow 0$ as $n \rightarrow \infty$.

Step 3.~We represent
\begin{align}
\langle (b-b_n)\cdot \nabla u,h_n \rho\rangle & = \langle \mathbf{1}_{\mathbb R^d \setminus B_R}(b-b_n) \cdot \nabla u,h_n \rho\rangle + \langle \mathbf{1}_{B_R}(b-b_n) \cdot \nabla u,h_n \rho\rangle \label{b_b_n_repr} \\
& =:I_1+I_2 \notag
\end{align}

By Step 1 and $\rho \leq \sqrt{\rho}$, the term $I_1$ can be estimated as follows: for every $\varepsilon>0$, for all $R=R(\varepsilon)>0$ sufficiently large
$$
|I_1|= |\langle \mathbf{1}_{\mathbb R^d \setminus B_R}(b-b_n) \cdot \nabla u,h_n \rho\rangle| \leq K_b \|\nabla u\|_2 \varepsilon,
$$
$$
K^2_b:=2 \sup_{n}\langle |b_n|^2\rho\rangle \vee \langle |b|^2\rho\rangle<\infty \text{ by \eqref{rho_global} (this is the place where we need weight $\rho$)},
$$
where $\|\nabla u\|_2<\infty$ by the energy inequality (Proposition \ref{elem_lem}(\textit{i})).

Let $R$ be as above (for some small $\varepsilon>0$). Let us deal with the second term $\langle \mathbf{1}_{B_R}(b-b_n) \cdot \nabla u,h_n \rho \rangle$ in \eqref{b_b_n_repr}. Here we work over a compact set, so the weight $\rho$ plays no role.
By the energy inequality of Proposition \ref{elem_lem}(\textit{i}), $\sup_n\|\nabla u_n\|_2<\infty$, and thus $\sup_n\|\nabla h_n\|_2<\infty$. Therefore, by the Rellich-Kondrashov theorem, there is a subsequence of $\{h_n\}$ (without loss of generality, $\{h_n\}$ itself) such that $$h_n \upharpoonright B_R \rightarrow g \text{ in $L^2(B_R)$ for some $g \in (L^2 \cap L^\infty)(B_R)$.}$$ (Here we have used  a priori estimate $\|h_n\|_\infty \leq 2\mu^{-1}\|f\|_\infty$.) So,
\begin{align*}
|I_2|=|\langle \mathbf{1}_{B_R}(b-b_n) \cdot \nabla u,h_n \rho\rangle| \leq |\langle \mathbf{1}_{B_R}(b-b_n) \cdot \nabla u, g\rho\rangle| + |\langle \mathbf{1}_{B_R}(b-b_n) \cdot \nabla u,(h_n-g)\rho\rangle|,
\end{align*}
where $$
\langle \mathbf{1}_{B_R}(b-b_n) \cdot \nabla u,g\rho\rangle \rightarrow 0 \quad \text{since $\mathbf{1}_{B_R}(\nabla u) g\rho \in L^2$ and $b_n \rightarrow b$ weakly in $L^2$},
$$
and
\begin{align*}
|\langle \mathbf{1}_{B_R}(b-b_n) \cdot \nabla u,(h_n-g)\rho\rangle| & \leq \|\mathbf{1}_{B_R}(b-b_n)\sqrt{\rho}\|_2 \|{\mathbf{1}_{B_R}}(h_n-g)\nabla u\|_2 \quad (\text{we have used $\rho \leq 1$})\\
& \leq K_b\|\mathbf{1}_{B_R} \nabla u\|_{2+\varepsilon_1}\|\mathbf{1}_{B_R}(h_n-g)\|_{\frac{2(2+\varepsilon_1)}{\varepsilon_1}}, \qquad \varepsilon_1>0.
\end{align*}
Next, we apply

\begin{lemma}
\label{grad_lem}
$\|\nabla u\|_{2+\varepsilon_1}<\infty$
provided that $\varepsilon_1>0$ is sufficiently small.
\end{lemma}

We prove Lemma \ref{grad_lem}, which is of interest on its own, in Section \ref{grad_lem_proof}. Its proof uses Gehring-Giaquinta-Modica's lemma.

\medskip

In turn, $$\|{\mathbf{1}_{B_R}}(h_n-g)\|_{\frac{2(2+\varepsilon_1)}{\varepsilon_1}} \rightarrow 0 \text{ as } n \rightarrow \infty$$ follows by interpolating between ${\mathbf{1}_{B_R}}(h_n-g) \rightarrow {0}$ in $L^2$ and $\sup_n\|{\mathbf{1}_{B_R}}(h_n-g)\|_\infty \leq 4\mu^{-1}\|f\|_\infty<\infty$.

Combining the above estimate on $I_1$ and the convergence $I_2 \rightarrow 0$ as $n \rightarrow \infty$, we obtain convergence \eqref{u_n_conv5}, which ends the proof of assertion (\textit{ii}).

\medskip

\noindent \emph{Proof of (\textit{iii})} Since $e^{-t\Lambda(b,q)}$ is a strongly continuous Feller semigroup on $C_\infty$, there exist probability measures 
$\{\mathbb P_x\}_{x \in \mathbb R^d}$ on the canonical space $\mathbf D$ of c\`{a}dl\`{a}g trajectories $\omega$ such that 
$$
e^{-t\Lambda(b,q)}f(x)=\mathbb E_{\mathbb P_x}[f(\omega_t)], \quad f \in C_\infty$$
and, for every $v \in D\big(\Lambda(b,q)\big)$, the process
$$
t \mapsto v(\omega_t)-v(x) + \int_0^t \Lambda(b,q) v(\omega_s) ds
$$
is a martingale with respect to $\mathbb P_x$, see e.g.\,\cite[Ch.\,VII, \S 1]{RevuzYor}. Since both $e^{-t\Lambda(b_n,q_m)}$,  $e^{-t\Lambda(b,q)}$ are strongly continuous Feller semigroups, the convergence of their finite-dimensional distributions, provided by assertion (\textit{i}), yields
$$
\mathbb P_x=w\mbox{-}\mathcal P(\mathbf D)\mbox{-}\lim_n\lim_m \mathbb P_x^{n,m},
$$
see e.g.\,\cite[Ch.4, Theorem 2.5]{EKu}.
Since $\mathbf C$ is closed in $\mathbf D$ and  $\mathbb P_x^{n,m}(\mathbf C)=1$, it follows from the weak convergence that $\mathbb P_x(\mathbf C)=1$. Thus, probability measures $\{\mathbb P_x\}_{x \in \mathbb R^d}$ are concentrated on  $\mathbf C$, and we have 
\begin{equation}
\label{P_n}
\mathbb P_x=w\mbox{-}\mathcal P(\mathbf C)\mbox{-}\lim_n\lim_m \mathbb P_x^{n,m},
\end{equation}
as claimed.

\medskip

\noindent\emph{Proof of (\textit{iv})} We argue as in \cite{HZ}.
Set
$
\mathbb P_{\nu_0}:=\int_{\mathbb R^d}\mathbb P_x \nu_0(x)dx.
$
Define in the same way  $\mathbb P_{\nu_0}^{n,m}$. Then, in view of  \eqref{P_n}, $$\mathbb P_{\nu_0}=w\mbox{-}\mathcal P(\mathbf C)\mbox{-}\lim_n\lim_m \mathbb P_{\nu_0}^{n,m}.$$ By the Skorohod representation theorem,  there exists a probability space {$\mathfrak F'=(\Omega',\mathcal F',\mathbf{P}')$}  and continuous processes $X^{n,m}_t$, $X_t$ defined on this space such that $\mathbb P^{n,m}_{\nu_0}$, $\mathbb P_{\nu_0}$ are the laws of $X^{n,m}_t$, $X_t$
and, up to a subsequence, $\mathbf{P}'$-a.s.\,on $\Omega'$
\begin{equation}
\label{X_conv}
X_t^{n,m}(\omega') \rightarrow X_t(\omega'), \quad (t \geq 0, \;\omega' \in \Omega')
\end{equation}
(see e.g.\,\cite[Ch.\,1, Sect.\,6]{Bil}). In particular, 
$
\mathbf{P}'(X^{n,m}_0)^{-1}, \mathbf{P}'  X_0^{-1}$ have density $\nu_0$.

Fix $1 \leq i \leq d$. Our goal is to show that
\begin{equation}
\label{goal_conv}
\lim_{n_1,n_2}\lim_{m_1,m_2} \mathbf{E}'\biggl|\int_0^t \big(b_{n_1}^i+q_{m_1}^i - (b^i_{n_2}+q^i_{m_2}) \big)(X_r)dr \biggr|^2= 0.
\end{equation}
It suffices to obtain

\begin{lemma}
\label{A_conv_lem}
\begin{equation}
\label{conv_res}
\lim_{n_1,n_2}\lim_{m_1,m_2} \mathbf{E}'\biggl|\int_0^t \big(b_{n_1}^i+q_{m_1}^i - (b_{n_2}^i+q_{m_2}^i)\big)(X_r^{n,m}) dr \biggr|^2= 0 \quad \text{ uniformly in $n$, $m$}.
\end{equation}
\end{lemma}

Indeed, having \eqref{conv_res}, we can appeal to \eqref{X_conv} and the Dominated convergence theorem to show that, for any fixed $n_1$, $m_1$, $n_2$, $m_2$,
$$
\lim_{n}\lim_{m} \mathbf{E}'\biggl|\int_0^t \big(b_{n_1}^i+q_{m_1}^i - (b_{n_2}^i+q_{m_2}^i)\big)(X_r^{n,m}) dr - \int_0^t \big(b_{n_1}^i+q_{m_1}^i - (b_{n_2}^i+q_{m_2}^i)\big)(X_r) dr \biggr|^2= 0,
$$
which then yields \eqref{goal_conv}.

\begin{proof}[Proof of Lemma \ref{A_conv_lem}]
Put for brevity  $F:=b_{n_1}^i+q_{m_1}^i - (b_{n_2}^i+q_{m_2}^i)$, so \eqref{conv_res} becomes
\begin{equation}
\label{conv_res2}
\lim_{n_1,n_2}\lim_{m_1,m_2}\mathbf{E}'\biggl|\int_0^t F(X_r^{n,m}) dr\biggr|^2 =0.
\end{equation}
Let us rewrite the expression under the limit signs as follows. Let $u$ be the classical solution to the terminal-value problem
\begin{equation}
\label{A_u_eq}
\big(\partial_s + \Delta - (b_n+q_m) \cdot \nabla \big)  u(s) =-F, \quad s<t, \quad u(t,\cdot)=0.
\end{equation}
 Then
\begin{align*}
\mathbf{E}'\biggl|\int_0^t F(X_r^{n,m}) dr\biggr|^2 & =2\mathbf{E}'\int_0^t F(X_s^{n,m}) \int_s^t F(X_r^{n,m})dr ds \\
& = 2\mathbf{E}'\int_0^t F(X_s^{n,m}) \mathbf{E}'\biggl[ \int_s^t F(X_r^{n,m})dr \mid {\sigma\{X_r^{n,m}:0\leq r\leq s\}}  \biggr]ds \\
& = 2\mathbf{E}'\int_0^t F(X_s^{n,m}) u(s,X_s^{n,m})ds = 2\int_0^t \langle F u,\nu\rangle ds,
\end{align*}
where $\nu$ is the probability density of $X^{n,m}$, i.e.\,solution to Cauchy problem for the Fokker-Planck equation
\begin{equation}
\label{A_nu_eq}
\partial_t \nu(t) - \Delta \nu(t) - {\rm div\,}\big((b_n+q_m) \nu(t)\big)=0, \quad t>0, \quad \nu(0,\cdot)=\nu_0.
\end{equation}
Now it is seen that convergence \eqref{conv_res2} will follow once we prove 
\begin{align}
\label{IJ_conv}
I:=\int_0^T \langle (b^i_{n_1}-b^i_{n_2})u,\nu\rangle ds \rightarrow 0, \quad J:=\int_0^T \langle (q^i_{m_1}-q^i_{m_2})u,\nu\rangle ds \rightarrow 0
\end{align}
as $n_1,n_2,m_1,m_2 \rightarrow \infty$ uniformly in $n$, $m$.

Indeed, take $t=T$ in \eqref{A_u_eq}. Choose $r>1$ sufficiently close to $1$ so that $r<\min\{2,\delta^{-1/2}\}$ and $\nu_0\in L^{2r}$. This is possible by the assumption on $\nu_0$ and interpolation with $\nu_0\in L^1$. Below we will use the following bound:
\begin{equation}
\label{prelim_bd}
\sup_{n,m,n_1,n_2,m_1,m_2}\int_0^T
\bigl\langle |\nabla u|^2+u^4+|u|^{2r'}+\nu^{2r}+|\nabla\nu|^2\bigr\rangle ds<\infty.
\end{equation}
For the terms containing $\nu$, this bound follows right away from Corollary \ref{elem_cor_fp} with $s=2r$, $s=2$, respectively. Since $\delta$ can be close to $1$, the exponent $2r$ must be close to $2$.

For the terms containing $u$, we argue as follows. {By the compact support assumption of Theorem \ref{thm2}(\textit{iv}),} $b_n$, $q_m$ have supports in the same ball $B_R$ independent of $n$, $m$. {Thus} choose $f\in C_c^\infty$, $0\leq f\leq1$, identically $1$ on $B_R$, and put $K=\sprt f$. With $Q_m^i$ denoting the $i$-th {column} of $Q_m$, as in Proposition \ref{emb_thm}, we have
\begin{align*}
F={}&(b_{n_1}^i-b_{n_2}^i)f
+{\rm div\,}\bigl[f(Q_{m_1}^i-Q_{m_2}^i)\bigr]\\
&-(Q_{m_1}^i-Q_{m_2}^i)\cdot\nabla f.
\end{align*}
The form-bound of $b_{n_1}-b_{n_2}$ can be chosen independently of $n_1$, $n_2$. Also, by the definition of $[q]$, the norms $\|Q_m^i\|_{L^s(K)}$ are uniformly bounded in $m$ for every finite $s$.
We invoke the energy inequality of Proposition \ref{elem_lem_par}(\textit{i}) with $s=2,4,2r'$, and estimate $\langle F,u|u|^{s-2}\rangle$ in the same way as in Step 1 of the proof of Proposition \ref{emb_thm}. More precisely, the above decomposition, integration by parts, the form-bound and Young's inequality give, for every $\varepsilon>0$,
$$
|\langle F,u|u|^{s-2}\rangle|
\leq\varepsilon\|\nabla|u|^{s/2}\|_2^2
+C_{\varepsilon,s}\|u\|_s^s+C_{\varepsilon,s},
$$
with constants independent of all the approximation indices. Since $\delta<1$, the coefficient $\frac{4(s-1)}{s^2}-\frac{2}{s}\sqrt\delta$ in the energy inequality is positive for all $s\geq2$. Choosing $\varepsilon$ sufficiently small and applying Gronwall's inequality after reversing time, we obtain the bounds on $u$ in \eqref{prelim_bd}.

Armed with \eqref{prelim_bd}, we estimate
$$
|I|\leq
\biggl(T\|b^i_{n_1}-b^i_{n_2}\|_2^2\biggr)^{\frac12}
\biggl(\int_0^T\|u\|_{2r'}^{2r'}ds\biggr)^{\frac1{2r'}}
\biggl(\int_0^T\|\nu\|_{2r}^{2r}ds\biggr)^{\frac1{2r}}.
$$
The first factor tends to $0$ as $n_1,n_2\rightarrow\infty$, since the drifts have supports in a fixed ball, while the other two factors are uniformly bounded by \eqref{prelim_bd}.

In turn, using $q_{m_1}^i-q_{m_2}^i=f(q_{m_1}^i-q_{m_2}^i)$ and integrating by parts, we obtain
\begin{align*}
J={}&-\int_0^T\langle Q_{m_1}^i-Q_{m_2}^i,f\nu\nabla u\rangle ds
-\int_0^T\langle Q_{m_1}^i-Q_{m_2}^i,fu\nabla\nu\rangle ds\\
&-\int_0^T\langle Q_{m_1}^i-Q_{m_2}^i,u\nu\nabla f\rangle ds
=:J_1+J_2+J_3.
\end{align*}
By H\"older's inequality in space and time,
\begin{align*}
|J_1|^2
&\leq\int_0^T\langle|\nabla u|^2\rangle ds
\biggl(T\langle|Q_{m_1}^i-Q_{m_2}^i|^{2r'}\mathbf1_K\rangle\biggr)^{\frac1{r'}}
\biggl(\int_0^T\langle\nu^{2r}\rangle ds\biggr)^{\frac1r},\\
|J_2|^2
&\leq\biggl(T\langle|Q_{m_1}^i-Q_{m_2}^i|^4\mathbf1_K\rangle\biggr)^{\frac12}
\biggl(\int_0^T\langle u^4\rangle ds\biggr)^{\frac12}
\int_0^T\langle|\nabla\nu|^2\rangle ds,
\end{align*}
and
\begin{align*}
|J_3|\leq \|\nabla f\|_\infty
\biggl(T\langle|Q_{m_1}^i-Q_{m_2}^i|^2\mathbf1_K\rangle\biggr)^{\frac12}\biggl(\int_0^T\langle|u|^{2r'}\rangle ds\biggr)^{\frac1{2r'}}
\biggl(\int_0^T\langle\nu^{2r}\rangle ds\biggr)^{\frac1{2r}}.
\end{align*}
The factors containing $u$ and $\nu$ are uniformly bounded by \eqref{prelim_bd}. Since $K$ is compact and $Q_m\rightarrow Q$ in $[L^s_{\loc}]^{d\times d}$ for every finite $s$, all the factors containing $Q_{m_1}^i-Q_{m_2}^i$ tend to $0$ as $m_1,m_2\rightarrow\infty$. This proves $J_1,J_2,J_3\rightarrow0$ uniformly in the remaining indices and ends the proof of \eqref{IJ_conv} and thus yields Lemma \ref{A_conv_lem}.
\end{proof}

To complete the proof of (\textit{iv}), we use the preceding estimates to identify the martingale part. Choose $n_k,m_k\rightarrow\infty$ so that $\mathbb P_{\nu_0}^{n_k,m_k}\rightarrow\mathbb P_{\nu_0}$ weakly on $\mathbf C$, and put $X^k=X^{n_k,m_k}$,
$$
A_t^{n,m}(Y):=\int_0^t(b_n+q_m)(Y_s)ds,\qquad
M_t^k:=X_t^k-X_0^k+A_t^{n_k,m_k}(X^k).
$$
On the original probability space, $M^k=\sqrt2 B$, so
$$
\mathbf E|M_t^k-M_s^k|^4\leq C|t-s|^2,\qquad
\sup_k\mathbf E\sup_{s\leq T}|M_s^k|^4<\infty.
$$
The laws of $(X^k,M^k)$ are therefore tight on $\mathbf C\times\mathbf C$. Passing to a subsequence and applying the Skorohod representation theorem jointly, we may assume that $(X^k,M^k)\rightarrow(X,M)$ a.s., locally uniformly in time, with $X$ having law $\mathbb P_{\nu_0}$.
For fixed $n,m$, boundedness and continuity of $b_n+q_m$ give $A_t^{n,m}(X^k)\rightarrow A_t^{n,m}(X)$ in $L^2(\Omega')$. We write
\begin{align*}
A_t^{n_k,m_k}(X^k)-A_t
={}&[A_t^{n_k,m_k}(X^k)-A_t^{n,m}(X^k)]
+[A_t^{n,m}(X^k)-A_t^{n,m}(X)]\\
&+[A_t^{n,m}(X)-A_t].
\end{align*}
By \eqref{conv_res} and \eqref{goal_conv}, first letting $k\rightarrow\infty$ and then $n,m\rightarrow\infty$, the right-hand side tends to zero in $L^2(\Omega')$. Hence $M_t=X_t-X_0+A_t$, and $A$ has the continuous version $M-X+X_0$.

For every $k$, the processes $M^{k,i}$ and $M^{k,i}M^{k,j}-2\delta_{ij}t$ are martingales in the filtration generated by $(X^k,M^k)$. Testing their martingale identities against bounded continuous functions of finitely many earlier values of $(X^k,M^k)$, joint convergence and the uniform fourth moment bound allow us to pass to the limit. A monotone class argument then shows that $M$ is a continuous martingale, with $\langle M^i,M^j\rangle_t=2\delta_{ij}t$, in the usual augmentation $\{\mathcal F'_t\}$ of the filtration generated by $(X,M)$. By L\'evy's characterization of Brownian motion, $B'_t=M_t/\sqrt2$ is a $\mathcal F'_t$-Brownian motion. This yields (\textit{iv}).

\begin{remark}
\label{nu_rem}
It is not difficult to remove the compact support assumption on $b$ and $q$ by inserting identity $\rho\rho^{-1}=1$ in the definitions of $I$ and $J$, i.e.\,
\begin{align*}
I=\int_0^T \langle (b^i_{n_1}-b^i_{n_2})u \rho,\nu \rho^{-1}\rangle ds, \quad J=\int_0^T \langle (q^i_{m_1}-q^i_{m_2}) u\rho,\nu \rho^{-1}\rangle ds
\end{align*}
and then arguing as above, but using the global weighted $L^2$ convergence result of Lemma \ref{lem_b} and replacing the energy inequalities of Section \ref{aux_sect} (i.e.\,Proposition \ref{elem_lem_par} and Corollary \ref{elem_cor_fp}) by their weighted counterparts with the weights $\rho$ and $\rho^{-1}$ (the last weight is not discussed in Section \ref{aux_sect}, but the arguments do not change since $|\nabla \rho^{-1}|$ is majorated by $\rho^{-1}$, see Section \ref{notations_sect}). This comes at expense of imposing condition $\langle \nu^{2r}\rho^{-\alpha}\rangle<\infty$ for $1<r<\frac{1}{\sqrt{\delta}}$ and some $\alpha>0$. 
\end{remark}

\bigskip

\noindent\emph{Proof of (\textit{vii})} Put
$$
R^n_{\mu}f(x):=\mathbb E_{\mathbb P_x(\tilde{b}_n)}\int_0^\infty e^{-\mu s} f(\omega_s)ds\quad \biggl(=(\mu+\Lambda(\tilde{b}_n))^{-1}f(x)\biggr), \quad f \in C_c^\infty,$$
$$R^Q_{\mu}f(x):=\mathbb E_{\mathbb Q_x}\int_0^\infty e^{-\mu s} f(\omega_s)ds, \quad \mu>0.
$$
We will show that $R^n_\mu f\rightarrow(\mu+\Lambda(b))^{-1}f$ in $C_\infty$ for all $f\in C_\infty$ and all sufficiently large $\mu$, where $\Lambda(b)$ is the Feller realization of $-\Delta+b\cdot\nabla$ constructed in (\textit{i}). The Trotter approximation theorem will then identify the limiting path laws.

Step 1: $$
R^n_\mu f(x) \rightarrow R_\mu^Q f(x) \text{ as }n \rightarrow \infty \quad \forall\,x \in \mathbb R^d,$$ 
as follows right away from $\mathbb Q_x=w{\mbox-}\mathcal P(\mathbf{C})\mbox{-}\lim_n \mathbb P_x(\tilde{b}_n)$.

\medskip

Step 2: $$\|R_\mu^Qf\|_2 \leqslant (\mu-\mu_0)^{-1}\|f\|_2 \quad \text{ for all $\mu>\mu_0$,}$$
 for some $\mu_0$ independent of $n$. Indeed, by the elliptic energy inequality (see e.g.\,Proposition \ref{elem_lem}), $\|R_\mu^nf\|_2 \leqslant (\mu-\mu_0)^{-1}\|f\|_2$ for all $n$.
Now Step 2 follows from Step 1 by a weak compactness argument in $L^2$.

\smallskip

By Step 2, operators $R_\mu^Q$  admit unique extensions by continuity to $L^2$, which we denote by $R_{\mu,2}^Q$.

On the other hand, operators $(\mu+\Lambda( b))^{-1}|_{C_c^\infty \cap L^2}$ are bounded on $L^2$ and, in fact, constitute the resolvent of the generator $\Lambda_2(b)$ of a strongly continuous semigroup in $L^2$, i.e.
$$
(\mu+\Lambda_2(b))^{-1}:=\bigg[(\mu+\Lambda( b))^{-1}|_{C_c^\infty \cap L^2} \biggr]_{L^2 \rightarrow L^2}^{\rm clos}.
$$
One can also construct $\Lambda_2(b)$ directly (using e.g.\,quadratic forms), see \cite{KiS_theory}.

\medskip

Step 3: For sufficiently large $\mu$,
$$
\|b\cdot\nabla (\mu-\Delta)^{-1}\|_{2\rightarrow2}\leq\sqrt\delta.
$$
Indeed, putting $u=(\mu-\Delta)^{-1}f$ and applying the form-bound to $|\nabla u|$, we obtain
$$
\|b\cdot\nabla u\|_2^2\leq\delta\|\Delta u\|_2^2+c_\delta\|\nabla u\|_2^2\leq\delta\|f\|_2^2
$$
provided $2\mu\delta\geq c_\delta$, since $\|f\|_2^2=\mu^2\|u\|_2^2+2\mu\|\nabla u\|_2^2+\|\Delta u\|_2^2$.

\smallskip

Step 4: 
\begin{equation}
\label{mu_id}
(\mu+\Lambda_{2}(b))^{-1}f = (\mu-\Delta)^{-1}\big(1+b\cdot\nabla (\mu-\Delta)^{-1}\big)^{-1} f \text{ in } L^2.
\end{equation}
Indeed, since by our assumptions we have $\delta<1$, in view of Step 3 the right-hand side of the previous formula is well defined. Now, we have to appeal for a moment to a ``good'' approximation $\{b_n\} \in [b]$, i.e.\,an approximation that actually does converge to $b$ (recall that we do not require from $\{\tilde{b}_n\}$ any kind of convergence to $b$). 

The sought identity \eqref{mu_id} holds for every $b_n$. This follows by rearranging the usual Neumann series representation for $(\mu+\Lambda_{2}(b_n))^{-1}$ while taking into account the estimate of Step 3. So,
$$
(\mu+\Lambda_{2}(b_n))^{-1}f = (\mu-\Delta)^{-1}\big(1+b_n\cdot\nabla (\mu-\Delta)^{-1}\big)^{-1} f. 
$$
It remains to pass to the limit in $n$. In the left-hand side one has $
(\mu+\Lambda_{2}(b_n))^{-1}f \rightarrow 
(\mu+\Lambda_{2}(b))^{-1}f$ in $L^2$ (see \cite{KiS_sharp}, but it is not difficult to prove this directly, see e.g.\,the proof of Proposition \ref{thm_conv}; here we need a simpler version of this in $L^2$). In the right-hand side the denominator of the geometric series $b_n\cdot\nabla (\mu-\Delta)^{-1}g \rightarrow b\cdot\nabla (\mu-\Delta)^{-1}g$ in $L^2$ for every $g \in L^2$. This is immediate on $g \in C_c^\infty$, so it remains to apply the estimate of Step 3. So, we can pass to the limit in the right-hand side, arriving at the identity \eqref{mu_id}.

\smallskip

Step 5: $$(\mu+\Lambda(b))^{-1} f = R_\mu^Q f \quad \text{ a.e.~on  $\mathbb R^d$.}$$
Indeed, since, by our assumptions, $\mathbb Q_x$ is a weak solution of the SDE \eqref{sde1}, we have by It\^{o}'s formula 
$$
(\mu-\Delta)^{-1} h= R_\mu^Q[\big(1+b\cdot\nabla (\mu-\Delta)^{-1}\big)h], \quad h \in C_c^\infty.
$$
Since $1+b\cdot\nabla (\mu-\Delta)^{-1} \in \mathcal B(L^2)$ by Step 3, we have, in view of Step 2,
$$
(\mu-\Delta)^{-1} g= R_{\mu,2}^Q[\big(1+b\cdot\nabla (\mu-\Delta)^{-1}\big)g], \quad g \in L^2.
$$
Take $g=\big(1+b\cdot\nabla (\mu-\Delta)^{-1}\big)^{-1}f$, $f \in C_c^\infty$, which is possible by Step 3 and $\delta<1$. Then, by Step 4,
$
(\mu+\Lambda_2(b))^{-1}f=R_{\mu,2}^Q f
$. By the consistency property $(\mu+\Lambda( b))^{-1}|_{C_c^\infty \cap L^2}=(\mu+\Lambda_2(b))^{-1}|_{C_c^\infty \cap L^2}$ the result follows.

\smallskip

Step 6: We apply Proposition \ref{thm_holder} with $q=0$ and $b_n=\tilde b_n$. Its proof uses only the common form-bound $\delta<4$ and the constant $c_\delta$, and does not require convergence of $\tilde b_n$ to $b$. Moreover,
$$
\mu\|R_\mu^n f\|_\infty\leq\|f\|_\infty,\qquad
\|f-\mu R_\mu^n f\|_\infty\leq2\|f\|_\infty.
$$
Thus the functions $R_\mu^n f$ are locally H\"older continuous uniformly in $n$. They are uniformly bounded by $\mu^{-1}\|f\|_\infty$. Thus Step 1 and the Arzel\`a-Ascoli theorem imply local uniform convergence to a continuous version of $R_\mu^Qf$. By Step 5, this version equals $(\mu+\Lambda(b))^{-1}f$ everywhere.

To obtain convergence in $C_\infty$, apply the local maximum principle in Remark \ref{disp_rem}, part 2, with $q=0$ and $b=\tilde b_n$:
$$
|R_\mu^n f(x)|\leq K\bigl(\langle|f|^{p\theta}\rho_x\rangle^{\frac1{p\theta}}+\langle|f|^{p\theta'}\rho_x\rangle^{\frac1{p\theta'}}\bigr),
$$
where $p\geq2$, $p>2/(2-\sqrt\delta)$, $1<\theta<d/(d-2)$, and $K$ is independent of $n,x$. For $f\in C_c^\infty$, the right-hand side tends to zero as $|x|\rightarrow\infty$, uniformly in $n$. Together with the local uniform convergence, this proves
$$
R_\mu^n f\rightarrow(\mu+\Lambda(b))^{-1}f\quad\text{in }C_\infty.
$$
By contractivity and density, this holds for every $f\in C_\infty$. The Trotter approximation theorem therefore gives
$$
e^{-t\Lambda(\tilde b_n)}f\rightarrow e^{-t\Lambda(b)}f\quad\text{in }C_\infty,
$$
locally uniformly in $t\geq0$. Iterating the Markov property of the approximating processes shows convergence of their finite-dimensional distributions to those of $\mathbb P_x$. On the other hand, their assumed weak convergence on $\mathbf C$ implies convergence of the same finite-dimensional distributions to those of $\mathbb Q_x$. Since finite-dimensional distributions determine a probability measure on $\mathbf C$, we have $\mathbb Q_x=\mathbb P_x$ for every $x\in\mathbb R^d$, as claimed.

\bigskip

\section{Proof of Theorem \ref{thm_div}} 
\label{proof_thmdiv_sect}

The proof of Theorem \ref{thm_div} follows closely the proof of Theorem \ref{thm2}, with only one calculation done differently. We stay at the level of a priori estimates, so $b$, $q$ are additionally bounded and smooth. Assuming for simplicity that $c_\delta=c_{\delta_+}=0$ in the conditions on $b$, $({\rm div\,}b)_+$, we consider Cauchy problem
 $(\partial_t - \Delta + (b+q) \cdot \nabla)v=0$, $v|_{t=0}=v_0 \in C_c^\infty$ (without loss of generality, $v_0 \geq 0$), multiply the parabolic equation by $v^{p-1}$ and integrate by parts: 
$$
\frac{1}{p}\langle \partial_t v^p\rangle + \frac{4(p-1)}{p^2}\langle |\nabla v^{\frac{p}{2}}|^2 \rangle + \frac{2}{p}\langle b \cdot \nabla v^{\frac{p}{2}},v^{\frac{p}{2}}\rangle=0,
$$
where we have used ${\rm div\,}q=0$. Thus,
$$
\partial_t\langle  v^p\rangle + \frac{4(p-1)}{p}\langle |\nabla v^{\frac{p}{2}}|^2 \rangle =- 2\langle b \cdot \nabla v^{\frac{p}{2}},v^{\frac{p}{2}}\rangle.
$$
In turn,
$
- \langle b \cdot \nabla v^{\frac{p}{2}},v^{\frac{p}{2}}\rangle = \langle b v^{\frac{p}{2}}, \nabla v^{\frac{p}{2}}\rangle + \langle {\rm div\,}b, v^{p}\rangle, 
$
hence
$
 \langle \partial_t v^p\rangle + \frac{4(p-1)}{p}\langle |\nabla v^{\frac{p}{2}}|^2 \rangle =  \langle {\rm div\,}b, v^{p}\rangle,
$
and so
$$
 \langle \partial_t v^p\rangle + \frac{4(p-1)}{p}\langle |\nabla v^{\frac{p}{2}}|^2 \rangle \leq  \langle ({\rm div\,}b)_+, v^{p}\rangle,
$$
Now, applying $({\rm div\,}b)^{\frac{1}{2}}_+ \in \mathbf{F}_{\delta_+}$, we obtain energy inequality
\begin{equation*}
\langle \partial_t v^p\rangle + \biggl(\frac{4(p-1)}{p}-\delta_+ \biggr)\langle |\nabla v^{\frac{p}{2}}|^2 \rangle \leq 0
\end{equation*}
with $\frac{4(p-1)}{p}-\delta_+>0$ provided that $p>\frac{4}{4-\delta_+}$.

\bigskip

\section{Proof of Theorem \ref{thm1}(\textit{\lowercase{iii}})}
\label{proof_thm1_sect}

\label{thm_mf_proof}

 We repeat the proof of Theorem \ref{thm2}(\textit{i}). Fix approximations $\{b_n\}$ as in Theorem \ref{thm1}(\textit{ii}), and put
$$
\Lambda_n:=-\Delta+b_n\cdot\nabla, \quad
D(\Lambda_n)=(1-\Delta)^{-1}C_\infty, \quad
T_t^n:=e^{-t\Lambda_n}.
$$
By the classical theory, $T_t^n$ are strongly continuous Feller semigroups associated with the approximating SDEs \eqref{approx_sdes5}. Theorem \ref{thm1}(\textit{ii}) yields, for every $f\in C_\infty$, $x\in\mathbb R^d$ and $t\geq0$,
$$
T_t^n f(x)=\mathbb E_{\mathbb P_x^n}[f(\omega_t)]
\rightarrow \mathbb E_{\mathbb P_x}[f(\omega_t)]=T_tf(x).
$$
We verify conditions $1^\circ)$--$3^\circ)$ of Trotter's theorem (Theorem \ref{trotter_thm}), following the proof of Theorem \ref{thm2}(\textit{i}) in Section \ref{proof_thm2_sect}.

\medskip

\noindent\textit{Proof of $1^\circ)$}. Since $T_t^n$ are $L^\infty$ contractions,
$$
\sup_n\|(\mu+\Lambda_n)^{-1}f\|_\infty
\leq\mu^{-1}\|f\|_\infty, \quad \mu>0.
$$

\medskip

\noindent\textit{Proof of $2^\circ)$}. Let $f\in C_c^\infty$ and set $u_n:=(\mu+\Lambda_n)^{-1}f$. By the convergence above and the dominated convergence theorem,
$$
u_n(x)=\int_0^\infty e^{-\mu t}T_t^nf(x)dt
\rightarrow \int_0^\infty e^{-\mu t}T_tf(x)dt=:u(x).
$$
By \cite[Theorem 5]{Ki_multi}, $\{u_n\}$ are locally H\"{o}lder continuous uniformly in $n$. Thus, the Arzel\`{a}-Ascoli theorem and the pointwise convergence yield convergence in $C(\bar B_r)$ for every $r>0$.

To obtain convergence in $C_\infty$, we use the separation property. Namely, \cite[Propositions 5 and 6]{Ki_multi}, applied with $|h|=1$ to the source term $|f|$, yield, for sufficiently large $\mu$,
$$
\sup_{B_{\frac12}(x)}|u_n|
\leq C\biggl(
\langle |f|^{p\theta}\rho_x\rangle^{\frac{1}{p\theta}}
+\langle |f|^{p\theta'}\rho_x\rangle^{\frac{1}{p\theta'}}
\biggr),
$$
where $C$ is independent of $n$ and $x$. Here we fix $1<\theta<\frac{d}{d-2}$ and $p>\frac{4}{4-\delta_+}$, $p\geq2$, sufficiently large that $p'\leq1+\nu$, and take weight $\rho$ in \eqref{rho_def} with $\epsilon_0=2$ and sufficiently small $\sigma>0$. Since $f$ has compact support, the right-hand side tends to zero as $|x|\to\infty$. Combining this with the local uniform convergence, we obtain
$$
u=s\mbox{-}C_\infty\mbox{-}\lim_n u_n.
$$
Condition $1^\circ)$ and a density argument extend this convergence to every $f\in C_\infty$.

\medskip

\noindent\textit{Proof of $3^\circ)$}. Fix $g\in C_c^\infty$. We have
$$
\mu(\mu+\Lambda_n)^{-1}g-g
=(\mu+\Lambda_n)^{-1}\Delta g
-(\mu+\Lambda_n)^{-1}(b_n\cdot\nabla g).
$$
The first term on the right-hand side has norm at most $\mu^{-1}\|\Delta g\|_\infty$. To estimate the second term, we apply \cite[Theorem 6]{Ki_multi} with $h=b_n$ and $f=\nabla_i g$, $1\leq i\leq d$, using positivity of $(\mu+\Lambda_n)^{-1}$. Since the approximations preserve the structure constants in the assumptions on $b$, this gives
$$
\sup_n\|(\mu+\Lambda_n)^{-1}(b_n\cdot\nabla g)\|_\infty
\leq C_g\bigl((\mu-\mu_0)^{-\frac{1}{p\theta}}
+\mu^{-\frac{\beta}{p}}\bigr), \quad \mu>\mu_0\vee1,
$$
for some $\mu_0>0$, $0<\beta<1$ and $C_g$ independent of $n$ and $\mu$. Hence
$$
\sup_n\|\mu(\mu+\Lambda_n)^{-1}g-g\|_\infty\rightarrow0
\quad\text{as }\mu\uparrow\infty.
$$
By $1^\circ)$ and the density of $C_c^\infty$ in $C_\infty$, the same holds for every $g\in C_\infty$.

\medskip

Trotter's theorem now yields a strongly continuous contraction semigroup on $C_\infty$ as the strong limit of $T_t^n$, locally uniformly in $t\geq0$. The pointwise convergence established at the beginning of the proof identifies this limit with $T_t$. Positivity is inherited from $T_t^n$, and preservation of probability follows from the representation by $\mathbb P_x$. Thus $T_t=e^{-t\Lambda}$ is the sought strongly continuous Feller semigroup.

\bigskip

\section{Auxiliary results used in the proofs of Propositions \ref{emb_thm}-\ref{sep_thm}}

\label{aux_sect}

\subsection{Energy inequalities}
In the next two propositions we assume that 
\begin{equation}
\label{hyp_2}
\left\{
\begin{array}{l}
b \in \mathbf{F}_\delta \text{ with } \delta<4, \\
q \in \mathbf{BMO}^{-1}
\end{array}
\right.
\end{equation}
(that is, we are under the assumptions of Theorem \ref{thm2}), and fix some $\{b_n\} \in [b]$, $\{q_m\} \in [q]$.

\begin{proposition} 
\label{elem_lem}
Assume that hypothesis \eqref{hyp_2} holds. Let $u=u_{n,m}$ denote the classical solution of the elliptic equation
\begin{equation}
\label{eq_el}
\big(\mu -\Delta + (b_n+q_m) \cdot \nabla\big)u=f, \quad \mu \geq 0, \quad f \in C_c^\infty.
\end{equation}
Fix some
$s>\frac{2}{2-\sqrt{\delta}}$, $s \geq 2.$
Then the following are true:

\begin{enumerate}
\item[($i$)] There exist positive constants $\mu_0$, $C_1$ independent of $n$, $m$ such that
\begin{equation}
\label{gen_ineq0}
(\mu-\mu_0)\langle |u|^s\rangle + C_1\langle |\nabla |u|^{\frac{s}{2}}|^2\rangle \leq \langle f,u|u|^{s-2}\rangle
\end{equation}
for all $\mu \geq \mu_0$.

\smallskip

\item[($ii$)] {\rm [Weighted variant]} Provided that constant $\sigma>0$ in weight $\rho(y)=(1+\sigma|y|^2)^{-\frac{d+\epsilon_0}{2}}$
is fixed sufficiently small, there exist positive constants $\mu_0$, $C_1$ independent of $n$, $m$  such that, for all $x \in \mathbb R^d$,
\begin{equation}
\label{gen_ineq}
(\mu-\mu_0)\langle |u|^s\rho_x\rangle + C_1\langle |\nabla |u|^{\frac{s}{2}}|^2 \rho_x\rangle \leq \langle f,u|u|^{s-2}\rho_x\rangle, \quad \rho_x(y)=\rho(x-y),
\end{equation}
for all $\mu \geq \mu_0$.  

\end{enumerate}
\end{proposition}

\begin{remark}
We have added condition $s \geq 2$ to save ourselves some efforts since to cover the values of the form-bound $\delta$ close to $4$ we need to select $s$ large anyway. But, generally speaking, $s \in ]1,2[$ does not pose a substantial difficulty, see e.g.\,\cite[proof of Theorem 4.2]{KiS_theory}.
\end{remark}

In the proof of assertion (\textit{ii}) we will need to control a term resulting from the interaction between the weight $\rho$ and the stream matrix $Q$ of drift $q$. This will be done using Proposition \ref{cc_lem_qx} (a compensated compactness type estimate) and Lemma \ref{NY_lem} on ${\rm BMO}(\mathbb R^d)$ multipliers.

\begin{proof}[Proof of Proposition \ref{elem_lem}] We will only prove (\textit{ii}). We write for brevity $b=b_n$, $q=q_m$. In the proof we will need estimates
\begin{equation}
\label{rho_est0}
|\nabla \rho_x(y)| \leq \frac{d+\epsilon_0}{2} \frac{2\sigma |x-y|}{1+\sigma|x-y|^2} \rho_x(y) \leq \frac{d+\epsilon_0}{2} \sqrt{\sigma} \rho_x(y).
\end{equation}
We multiply equation \eqref{eq_el} by $u|u|^{s-2}\rho_x$ and integrate by parts:
{
\begin{align*}
\mu \langle |u|^s\rho_x \rangle + \frac{4(s-1)}{s^2}\langle |\nabla |u|^{\frac{s}{2}}|^2 \rho_x\rangle & + \frac{2}{s}\langle \nabla |u|^{\frac{s}{2}},|u|^{\frac{s}{2}}\nabla \rho_x\rangle \\
& + \frac{2}{s}\langle b \cdot \nabla |u|^{\frac{s}{2}},|u|^{\frac{s}{2}}\rho_x\rangle + \frac{2}{s}\langle q \cdot \nabla |u|^{\frac{s}{2}},|u|^{\frac{s}{2}}\rho_x \rangle=\langle f,u|u|^{s-2}\rho_x\rangle.
\end{align*}
}
{
To estimate the term containing $q$, put
$$
\widehat Q_x:=Q-(Q)_{B_{\sigma^{-1/2}}(x)}.
$$
Subtracting this constant anti-symmetric matrix does not change the pairing, since, for every constant anti-symmetric matrix $A$,
$$
\langle A\nabla |u|^{\frac{s}{2}},|u|^{\frac{s}{2}}\nabla\rho_x\rangle
=-\frac12\langle |u|^s,{\rm div\,}(A^{\top}\nabla\rho_x)\rangle=0.
$$
Taking into account anti-symmetry of $\widehat Q_x$, we obtain
\begin{align*}
\frac{2}{s}|\langle q \cdot \nabla |u|^{\frac{s}{2}},|u|^{\frac{s}{2}}\rho_x \rangle|
&=\frac{2}{s}|\langle \widehat Q_x \cdot \nabla |u|^{\frac{s}{2}},|u|^{\frac{s}{2}}\nabla \rho_x\rangle|\\
&=\frac{2}{s}\bigg|\bigg\langle \widehat Q_x \cdot \nabla (|u|^{\frac{s}{2}}\sqrt{\rho_x}),|u|^{\frac{s}{2}}\sqrt{\rho_x}\frac{\nabla \rho_x}{\rho_x}\bigg\rangle\bigg|\\
&\leq\frac{2C}{s}\bigg\|\frac{\nabla \rho_x}{\rho_x}\cdot\widehat Q_x\bigg\|_{\rm BMO}
\|\nabla (|u|^{\frac{s}{2}}\sqrt{\rho_x})\|_2\||u|^{\frac{s}{2}}\sqrt{\rho_x}\|_2,
\end{align*}
where we have applied Proposition \ref{cc_lem_qx}. Now set $z=\sqrt{\sigma}(y-x)$. In these variables,
$$
\frac{\nabla\rho_x(y)}{\rho_x(y)}=-(d+\epsilon_0)\sqrt{\sigma}\frac{z}{1+|z|^2},
$$
and $\widehat Q_x(x+\sigma^{-1/2}z)$ has mean zero on $B_1(0)$. Applying Lemma \ref{NY_lem} with $\sigma=1$ and using the invariance of the ${\rm BMO}$ semi-norm under translations and dilations, we obtain
$$
K:=\sup_{x\in\mathbb R^d}\bigg\|\frac{\nabla\rho_x}{\rho_x}\cdot\widehat Q_x\bigg\|_{\rm BMO}
\leq C\sqrt{\sigma}\|Q\|_{\rm BMO}
\leq C\sqrt{\sigma}\sup_m\|Q_m\|_{\rm BMO}<\infty.
$$
Thus $K$ is bounded independently of $m$, and
\begin{equation}
\frac{2}{s}|\langle q \cdot \nabla |u|^{\frac{s}{2}},|u|^{\frac{s}{2}}\rho_x\rangle|
\leq\frac{2CK}{s}\|\nabla (|u|^{\frac{s}{2}}\sqrt{\rho_x})\|_2
\langle |u|^s\rho_x\rangle^{\frac12}.
\label{as}
\end{equation}
Next, by Cauchy-Schwarz and \eqref{rho_est0},
$$
\|\nabla (|u|^{\frac{s}{2}}\sqrt{\rho_x})\|_2^2
\leq (1+\varepsilon_1)\langle |\nabla |u|^{\frac{s}{2}}|^2\rho_x\rangle
+C\sigma(1+\varepsilon_1^{-1})\langle |u|^s\rho_x\rangle.
$$
Applying the quadratic inequality in \eqref{as}, we can make the constant in front of $\langle |\nabla |u|^{\frac{s}{2}}|^2\rho_x\rangle$ as small as we want, at the expense of increasing the constant in front of $\langle |u|^s\rho_x\rangle$.
}

Next,
\begin{align}
\frac{2}{s}|\langle b \cdot \nabla |u|^{\frac{s}{2}},|u|^{{\frac{s}{2}}}\rho_x\rangle| & \leq \frac{2}{s}\biggl(\beta \langle |b|^2,|u|^s \rho_x \rangle + \frac{1}{4\beta}\langle |\nabla |u|^{\frac{s}{2}}|^2\rho_x\rangle \bigg) \label{b_est_quad}\\
& (\text{use $b \in \mathbf{F}_\delta$}) \notag \\
& \leq \frac{2}{s}\biggl(\beta(\delta\|\nabla (|u|^{\frac{s}{2}}\sqrt{\rho_x})\|_2^2 + c_\delta\langle |u|^s\rho_x\rangle) + \frac{1}{4\beta}\langle |\nabla |u|^{\frac{s}{2}}|^2\rho_x\rangle \bigg). \notag
\end{align}
Take $\beta=\frac{1}{2\sqrt{\delta}}$ and then apply inequality in the end of the previous paragraph, but with $\varepsilon_1$ chosen small.

Applying the last inequality in \eqref{rho_est0}, we can replace in the previous inequalities all occurrences of $|\nabla \rho_x|$ by $C\sqrt{\sigma}\rho_x$.
That way, we arrive at the inequality \eqref{gen_ineq} with constant
$$
C_1=\frac{4(s-1)}{s^2}-\frac{2}{s}\sqrt{\delta} - \text{(constant terms proportional to $\sqrt{\sigma}$ and $\varepsilon_1$)},
$$ 
where $\mu_0$ is given in terms of $c_\delta$, $\frac{CK}{s}$ and $\varepsilon_1^{-1}$.
Since $\frac{4(s-1)}{s^2}-\frac{2}{s}\sqrt{\delta}>0$ ($\Leftrightarrow s>\frac{2}{2-\sqrt{\delta}}$) by our choice of $s$, we can fix $\sigma$ and $\varepsilon_1$ sufficiently small so that $C_1>0$. This ends the proof of Proposition \ref{elem_lem}.
\end{proof}

\begin{corollary}
\label{elem_cor}
Under the assumptions and notations of Proposition \ref{elem_lem}, we also have

\begin{enumerate}
\item[($i$)] There exist positive constants $\mu_0$, $C_1$, $C_2$ independent of $n$, $m$ such that
\begin{equation}
\label{gen_ineq0_}
(\mu-\mu_0)\langle |u|^s\rangle + C_1\langle |\nabla |u|^{\frac{s}{2}}|^2\rangle \leq C_2\langle |f|^s\rangle
\end{equation}
for all $\mu \geq \mu_0$.

\smallskip

\item[($ii$)] {\rm [Weighted variant]} Provided that constant $\sigma$ in the definition of weight $\rho$
is fixed sufficiently small, there exist positive constants $\mu_0$, $C_1$, $C_2$  independent of $n$, $m$  such that, for all $x \in \mathbb R^d$,
\begin{equation}
\label{gen_ineq_}
(\mu-\mu_0)\langle |u|^s\rho_x\rangle + C_1\langle |\nabla |u|^{\frac{s}{2}}|^2 \rho_x\rangle \leq C_2\langle |f|^s \rho_x\rangle
\end{equation}
for all $\mu \geq \mu_0$.  
\end{enumerate}

\end{corollary}
\begin{proof}
In the above proof of Proposition \ref{elem_lem} we can take one step further and apply Young's inequality in order to estimate $$\langle f,u|u|^{s-2}\rho_x\rangle \leq \frac{\varepsilon^{s'}}{s'}\langle |u|^s \rho_x\rangle + \frac{1}{\varepsilon^s s}\langle |f|^s \rho_x\rangle.$$
\end{proof}

The previous energy inequalities have their parabolic counterparts:

\begin{proposition} 
\label{elem_lem_par}
Assume that  hypothesis \eqref{hyp_2} is satisfied.
 Let $v=v_{n,m}$ be the classical solution of Cauchy problem 
\begin{equation}
\label{eq_pl}
\big(\mu + \partial_t -\Delta + (b_n+q_m) \cdot \nabla\big)v=0, \quad v(0)=f \in C_c^\infty, \quad \mu \geq 0.
\end{equation}
Fix some
$s>\frac{2}{2-\sqrt{\delta}}$, $s \geq 2.$
Then there exists $\mu_0 \geq 0$ independent of $n$, $m$ such that for all $\mu \geq \mu_0$ the following are true:
\begin{enumerate}
\item[($i$)]
\begin{equation}
\label{gen_ineq0_parab}
(\mu-\mu_0)\int_0^t \langle |v|^s\rangle  + \frac{1}{s}\sup_{r \in [0,t]}\langle |v(r)|^{s}\rangle + C_1\int_0^t \langle |\nabla |v|^{\frac{s}{2}}|^2\rangle \leq \frac{2}{s}\langle |f|^s\rangle 
\end{equation}
for constant $C_1>0$ independent of $n$, $m$.

\smallskip

\item[($ii$)] {\rm [Weighted variant]} Provided that $\sigma$ in the definition of weight $\rho$
is chosen sufficiently small, we have
\begin{equation}
\label{gen_ineq_parab}
(\mu-\mu_0)\int_0^t \langle |v|^s\rho_x\rangle  + \frac{1}{s}\sup_{r \in [0,t]}\langle |v(r)|^{s}\rho_x\rangle + C_1\int_0^t \langle |\nabla |v|^{\frac{s}{2}}|^2 \rho_x\rangle \leq \frac{2}{s}\langle |f|^s\rho_x\rangle
\end{equation}
for constant $C_1>0$ independent of $n$, $m$ and $x \in \mathbb R^d$.
\end{enumerate}

\end{proposition}
\begin{proof}[Proof of Proposition \ref{elem_lem_par}]
Let us prove (\textit{ii}). We multiply the parabolic equation in \eqref{eq_pl} by $v|v|^{s-2}\rho_x$ and integrate over $[0,r] \times \mathbb R^d$. All the terms in the resulting integral identity, except the one containing $\partial_t v$, are dealt with as in the proof of the previous proposition. In turn, the term containing $\partial_t v$ is evaluated as follows:
$$
\int_0^r \langle \partial_t v, v|v|^{s-2} \rho_x\rangle = \frac{1}{s}\int_0^r \langle \partial_t |v|^s \rho_x \rangle = \frac{1}{s}\bigl(\langle |v(r)|^s\rho_x\rangle - \langle |f|^s\rho_x\rangle\bigr).
$$
This gives us
\begin{equation}
\label{id_par}
(\mu-\mu_0)\int_0^r \langle |v|^s\rho_x\rangle + \frac{1}{s}\langle |v(r)|^{s}\rho_x\rangle + C_1\int_0^r \langle |\nabla |v|^{\frac{s}{2}}|^2 \rho_x\rangle \leq \frac{1}{s}\langle |f|^s\rho_x\rangle.
\end{equation}
for appropriate $\mu_0 \geq 0$ and $C_1>0$, both independent of $n$, $m$.
Let $\mu \geq \mu_0$. We have, in particular, 
$$
\frac{1}{s}\langle |v(r)|^{s}\rho_x\rangle \leq \frac{1}{s}\langle |f|^s\rho_x\rangle, \quad (\mu-\mu_0)\int_0^r \langle |v|^s\rho_x\rangle  + C_1\int_0^r \langle |\nabla |v|^{\frac{s}{2}}|^2 \rho_x\rangle \leq \frac{1}{s} \langle |f|^s\rho_x \rangle.
$$
We can pass to the supremum in $r$ in both inequalities since their right-hand side does not depend on $r$. Adding up the resulting inequalities, we arrive at \eqref{gen_ineq_parab}.
\end{proof}

In the proof of assertion (\textit{iv}) of Theorem \ref{thm2} we use energy inequality for the Fokker-Planck equation:

\begin{corollary}
\label{elem_cor_fp}
Assume that hypothesis \eqref{hyp_2} is satisfied with $\delta<1$. Let $\nu=\nu_{n,m}$ denote the classical solution to Cauchy problem
\begin{equation}
\label{eq_el_}
\mu \nu + \partial_t \nu -\Delta \nu - {\rm div}\big[(b_n+q_m)\nu\big]=0, \quad \nu(0)=\nu_0 \in C_c^\infty, \quad \mu \geq 0.
\end{equation}
Fix some
$2 \leq s<\frac{2}{\sqrt{\delta}}$.
Then there exist positive constants $\mu_0$, $C_1$ independent of $n$, $m$ such that
\begin{equation*}
(\mu-\mu_0)\int_0^t \langle |\nu|^s\rangle  + \frac{1}{s}\sup_{r \in [0,t]}\langle |\nu(r)|^{s}\rangle + C_1\int_0^t \langle |\nabla |\nu|^{\frac{s}{2}}|^2\rangle \leq \frac{2}{s}\langle |\nu_0|^s\rangle 
\end{equation*}
for all $\mu \geq \mu_0$.
\end{corollary}

One also has a straightforward weighted counterpart of this inequality as in assertion (\textit{ii}) of Proposition \ref{elem_lem_par}.

We will be applying Corollary \ref{elem_cor_fp} in the case when $\nu_0 \geq 0$, $\langle \nu_0 \rangle=1$.

\begin{proof}
All the terms in the corresponding integral identity, except the next one, are dealt with in the same way as in the proof of Proposition \ref{elem_lem}. Let $b=b_n$. Then
\begin{align*}
-\langle {\rm div}(b \nu),\nu|\nu|^{s-2}\rangle=(s-1)\langle b\nu,|\nu|^{s-2}\nabla \nu\rangle=\frac{2}{s}(s-1)\langle b \cdot \nabla |\nu|^{\frac{s}{2}},|\nu|^{\frac{s}{2}} \rangle.
\end{align*}
Hence we arrive at the counterpart of \eqref{id_par} with the coefficient of the dispersion term
$$
C=\frac{4(s-1)}{s^2}-\frac{2}{s}(s-1)\sqrt{\delta} - \text{(constant terms proportional to $\varepsilon_1$)}
$$
that must be positive.
We can fix $\varepsilon_1$ as small as needed. What matters is the value of $\delta$ that ensures that $\frac{4(s-1)}{s^2}-\frac{2}{s}(s-1)\sqrt{\delta}>0$. The latter is equivalent to $s<\frac{2}{\sqrt{\delta}}$, which is satisfied by our assumptions.
\end{proof}

\subsection{Global weighted $L^2$ summability of a form-bounded drift}

A form-bounded vector field $b \in \mathbf{F}_\delta$ is a priori only locally summable: $|b| \in L^2_{\loc}$. In fact, condition $b \in \mathbf{F}_\delta$ implies global square summability of $|b|$, but with respect to weight $\rho_x$.
Indeed, 
selecting  $\sqrt{\rho_x}$ as the test function in the definition of $b \in \mathbf{F}_\delta$ and using \eqref{rho_est}, we obtain
\begin{align}
\langle |b|^2\rho_x\rangle & \leq \frac{\delta}{4} \langle \frac{|\nabla \rho_x|^2}{\rho_x}\rangle + c_\delta \langle \rho_x \rangle \notag \\
& \leq \biggl( \frac{\delta}{4}\frac{(d+\epsilon_0)^2}{4}\sigma + c_\delta \biggr)\langle \rho \rangle<\infty,
\label{rho_global}
\end{align}
where we have used $\langle \rho_x \rangle=\langle \rho\rangle$.
Moreover, we have the following global convergence result:

\begin{lemma} 
\label{lem_b}
Let $b \in \mathbf{F}_\delta$, $\{b_n\} \in [b]$. Then, for every $x \in \mathbb R^d$,
$\langle |b_{n_1}-b_{n_2}|^2\rho_x\rangle \rightarrow 0$ as $n_1$, $n_2 \rightarrow \infty$.
\end{lemma}

\begin{proof}
First, let us show that 
\begin{equation}
\label{unif_c_b}
\lim_{R \rightarrow \infty} \langle |b_{n_1}-b_{n_2}|^2\mathbf{1}_{\mathbb R^d \setminus B_{R+1}}\rho_x\rangle = 0 \quad \text{ uniformly in $n_1$, $n_2$.} 
\end{equation}
Indeed, replacing $\mathbf{1}_{\mathbb R^d \setminus B_{R+1}}$ by greater function $\eta^2_R$, where
$$
\eta_R(y):=\xi_R(|y|), \quad \xi_R(r):=\left\{
\begin{array}{ll}
0, & 0 \leq r < R, \\
r-R, & R \leq r \leq R+1, \\
1, &  r > R+1,
\end{array}
\right.
$$
and noting that $|\nabla \eta_R(y)| \leq \mathbf{1}_{R<|y|<R+1}$,
we estimate
\begin{align*}
\langle |b_{n_1}-b_{n_2}|^2\mathbf{1}_{\mathbb R^d \setminus B_{R+1}}\rho_x\rangle & \leq \langle |b_{n_1}-b_{n_2}|^2\eta^2_R\rho_x\rangle \\
& \leq 4 \delta \|\nabla (\eta_R \sqrt{\rho_x})\|_2^2 + 4c_\delta \|\eta_R \sqrt{\rho_x}\|_2^2,
\end{align*}
where $\|\eta_R \sqrt{\rho_x}\|_2 \rightarrow 0$ as $R \rightarrow \infty$, and so, 
in view of the second estimate in \eqref{rho_est0}, $\| \eta_R \nabla (\sqrt{\rho_x})\|_2 \rightarrow 0$ as $R \rightarrow \infty$. Also, taking without loss of generality $x=0$, $$\|(\nabla \eta_R)\sqrt{\rho}\|_2^2 = \langle \mathbf{1}_{R\leq |\cdot| \leq R+1}\rho \rangle \leq C R^{-d-\epsilon_0}R^d=C R^{-\epsilon_0} \rightarrow 0$$
as $R \rightarrow \infty$. This yields \eqref{unif_c_b}.

In turn, inside the ball $B_R$ we have $\langle |b_{n_1}-b_{n_2}|^2\mathbf{1}_{B_R}\rho_x\rangle \rightarrow 0$ as $n_1, n_2 \rightarrow \infty$ since $b_{n_1}-b_{n_2} \rightarrow 0$ in $L^2_\loc$. This ends the proof.
\end{proof}

\bigskip

\bigskip

\section{Proof of Proposition \ref{emb_thm} (Embedding property)} 
\label{emb_thm_sect}

We will need the following lemma.

\begin{lemma}[{\cite[Lemma 7.1]{G}}]
\label{dg_lemma}
If $\{z_i\}_{i=0}^\infty \subset \mathbb R_+$ is a sequence of positive real numbers  such that
$$
z_{i+1} \leq N C_0^i z^{1+\alpha}_i
$$
for some $C_0>1$, $\alpha>0$, and
$$
z_0 \leq N^{-\frac{1}{\alpha}}C_0^{-\frac{1}{\alpha^2}}.
$$
Then
$
\lim_i z_i=0.
$
\end{lemma}

Throughout the proof, we write for brevity $$b=b_n, \quad q=q_m, \quad Q^i=Q_m^i, \quad w=w_{n,m}.$$ 
It suffices to estimate the positive part of $w$:
\begin{align}
\sup_{\mathbb R^d} w_+ & \leq K_1 (\mu-\mu_0)^{-\frac{\beta}{p}}\big\langle 
{|Q^i|^{p\theta'}|\nabla f|^{p\theta'}+(1+|Q^i|)^{p\theta'}|f|^{p\theta'}}
\big\rangle^{\frac{1}{p\theta'}} \notag \\
& + K_2 (\mu-\mu_0)^{-\frac{1}{p\theta}}\langle {|Q^i|^{p\theta}|\nabla f|^{p\theta}+(1+|Q^i|)^{p\theta}|f|^{p\theta}}\rangle^{\frac{1}{p\theta}}. \label{dg_ineq}
\end{align}

\textit{Step 1.} First, we prove the following energy inequality: there exist generic constants $\mu_0 \geq 0$ and $C_0$, $C$ such that, for every $k \geq 0$ and all $\mu > \mu_0$, the positive part $v:=(w-k)_+$ of $w-k$ satisfies
\begin{equation}
\label{cacc_no}
(\mu-\mu_0) \|v^{\frac{p}{2}}\|_2^2 + C_0\|\nabla v^{\frac{p}{2}}\|_2^2 \leq C\bigl[\langle |Q^i|^p|\nabla f|^p\mathbf{1}_{v>0}\rangle + \langle (1+|Q^i|^p) |f|^{p}\mathbf{1}_{v>0}\rangle \bigr].
\end{equation}
Proof of \eqref{cacc_no}. We obtain from equation \eqref{eq11}, using that both $\mu$ and $k$ are non-negative,
$
(\mu-\Delta + (b+q) \cdot \nabla)(w-k) \leq (b^i+q^i)f.
$
We now basically apply the energy inequality of Proposition \ref{elem_lem}(\textit{i}) with $s=p$. So, we multiply by $v^{p-1}$, integrate, apply $b \in \mathbf{F}_\delta$ and use anti-symmetry of the stream matrix $Q$ of drift $q$:
\begin{align}
(\mu-\mu_0) \langle v^p\rangle + C_1 \langle |\nabla v^{\frac{p}{2}}|^2\rangle  \leq  \langle (b^i+q^i) f,v^{p-1}\rangle. \label{est_J}
\end{align}
Let us estimate the terms in the right-hand side. First,
\begin{align*}
|\langle b^i f,v^{p-1}\rangle| & \leq \varepsilon \langle |b^i|^2,v^p\rangle + \frac{1}{4\varepsilon} \langle |f|^2 v^{p-2}\rangle \\
& \leq \varepsilon \biggl(\delta \langle |\nabla v^{\frac{p}{2}}|^2\rangle + c_\delta \langle v^p\rangle\biggr) \\
& +\frac{1}{4\varepsilon}\biggl(\frac{p-2}{p}\langle v^p\rangle + \frac{2}{p}\langle |f|^p\mathbf{1}_{v>0}\rangle \biggr).
\end{align*}
Second,
\begin{align*}
|\langle q^i f,v^{p-1}\rangle| & =|\langle {\rm div\,}Q^i,fv^{p-1}\rangle| \\
& \leq |\langle Q^i,(\nabla f)v^{p-1}\rangle | + \frac{2(p-1)}{p}|\langle Q^i,f v^{\frac{p}{2}-1}\nabla v^{\frac{p}{2}}\rangle|=:K_1+\frac{2(p-1)}{p} K_2,
\end{align*}
where
\begin{align*}
K_1 & \leq \frac{1}{2}\langle |Q^i|^2  |\nabla f|^2 v^{p-2}\rangle + \frac{1}{2}\langle v^{p}\rangle\\
& \leq \frac{1}{2}\biggl( \frac{2}{p}\langle |Q^i|^p|\nabla f|^p\mathbf{1}_{v>0}\rangle + \frac{p-2}{p}\langle v^p \rangle \biggr)  + \frac{1}{2}\langle v^{p}\rangle 
\end{align*}
and
\begin{align*}
K_2 & \leq \varepsilon \langle |\nabla v^{\frac{p}{2}}|^2\rangle + \frac{1}{4\varepsilon}\langle |Q^i|^2 f^2 v^{p-2}\rangle    \\
& \leq \varepsilon \langle |\nabla v^{\frac{p}{2}}|^2\rangle + \frac{1}{4\varepsilon}\biggl( \frac{2}{p}\langle |Q^i|^p |f|^{p}\mathbf{1}_{v>0}\rangle +\frac{p-2}{p}\langle v^p\rangle\biggr).
\end{align*}
Thus,
$$
|\langle q^i f,v^{p-1}\rangle| \leq \frac{2(p-1)}{p} \varepsilon \langle |\nabla v^{\frac{p}{2}}|^2\rangle + c_1\bigl(\langle |Q^i|^p|\nabla f|^p\mathbf{1}_{v>0}\rangle + \langle |Q^i|^p |f|^{p}\mathbf{1}_{v>0}\rangle \bigr) + c_2 \langle v^p\rangle,
$$
where $c_1$, $c_2$ depend on $\varepsilon$ and $p$. 

\medskip

Substituting the resulting estimates in \eqref{est_J} and selecting $\varepsilon$ sufficiently small, we obtain estimate \eqref{cacc_no}.

\medskip

\textit{Step 2.} In what follows, we will be selecting $k>0$. {Since $\{v>0\}=\{w>k\}$ and $w \in C_\infty$, we have $|\{v>0\}|<\infty$.} {Only finiteness of this measure is used here. Below we estimate it in terms of $L^{p\theta}$ norms, with constants independent of $n$, $m$.} We obtain from \eqref{cacc_no}, using the Sobolev embedding theorem {applied to $v^{p/2}$ and the identity $\|v^{p/2}\|_{2d/(d-2)}^2=\|v\|_{pd/(d-2)}^p$},
\begin{equation}
\label{est_40}
(\mu-\mu_0) \|v\|_{p}^p + C_S\|v\|^{p}_{\frac{pd}{d-2}}\leq C\big\langle \big(|Q^i|^p|\nabla f|^p+(1+|Q^i|^p)|f|^p\big)\mathbf{1}_{v>0}\big\rangle.
\end{equation}
We estimate the left-hand side of \eqref{est_40} using {interpolation followed by Young's inequality}:
$$
(\mu-\mu_0)^{\beta} \|v\|_{p\theta_0}^p \leq \beta(\mu-\mu_0) \|v\|_{p}^p + (1-\beta)\|v\|_{L^{\frac{pd}{d-2}}}^p, \quad 0<\beta<1, \quad \frac{1}{p\theta_0}=\beta\frac{1}{p}+(1-\beta)\frac{d-2}{pd},
$$
where $1<\theta_0<\frac{d}{d-2}$.
So,
$$
(\mu-\mu_0)^{\beta} \|v\|_{p\theta_0}^p \leq C_2\bigl[\big\langle \big(|Q^i|^p|\nabla f|^p+(1+|Q^i|^p)|f|^p\big)\mathbf{1}_{v>0}\big\rangle \bigr]
$$
Let us fix $\beta$ small enough so that we have $\theta_0>\theta$. (Recall that $1<\theta<\frac{d}{d-2}$ was fixed in the statement of the proposition.) 
Applying H\"{o}lder's inequality, we obtain
$$
(\mu-\mu_0)^{\beta} \|v\|_{p\theta_0}^p \leq C_3 H|\{v>0\}|^{\frac{1}{\theta}},
$$
where 
$$
H:=\big\langle \big(|Q^i|^{p\theta'}|\nabla f|^{p\theta'}+(1+|Q^i|^{p\theta'})|f|^{p\theta'}\big)\mathbf{1}_{v>0}\big\rangle^{\frac{1}{\theta'}}.
$$
On the other hand, again by H\"{o}lder's inequality,
$$
 \|v\|_{p\theta}^{p\theta} \leq  \|v\|_{p\theta_0}^{p\theta} |\{v>0\}|^{1-\frac{\theta}{\theta_0}}.
$$
Therefore, we obtain
$$
 \|v\|_{p\theta}^{p\theta} \leq \tilde{C}(\mu-\mu_0)^{-\beta\theta}H^{\theta}|\{v>0\}|^{2-\frac{\theta}{\theta_0}}.
$$

\textit{Step 3.}~Now, put $v_m:=(w-k_m)_+$, $k_m:=\xi(1-2^{-m}) \uparrow \xi,$ where constant $\xi>0$ will be chosen later.

\begin{remark}
We have $k_0=0$, but we will not encounter the volume of $\{w>0\}$ in the proof (clearly, $|\{w>0\}|$ can be infinite). 
\end{remark}

So,
\begin{align*}
\frac{1}{\xi^{p\theta}}  \|v_{m+1}\|_{p\theta}^{p\theta} \leq \tilde{C}\frac{1}{\xi^{p\theta}} (\mu-\mu_0)^{-\beta\theta} H^{\theta}|\{w>k_{m+1}\}|^{2-\frac{\theta}{\theta_0}}.
\end{align*}
From now on, we require constant $\xi$ to satisfy $\xi^p \geq (\mu-\mu_0)^{-\beta}H$, 
so
\begin{align*}
\frac{1}{\xi^{p\theta}} {\|v_{m+1}\|_{p\theta}^{p\theta}} \leq \tilde{C}|\{w>k_{m+1}\}|^{2-\frac{\theta}{\theta_0}}. \notag
\end{align*}
Now, 
\begin{align*}
|\{w>k_{m+1}\}| & = \bigg|\bigg\{\biggl(\frac{(w-k_m)_+}{k_{m+1}-k_m}\biggr)^{p\theta}>1\bigg\}\bigg| \\
&\leq (k_{m+1}-k_m)^{-p\theta} \langle v^{p\theta}_m  \rangle  = \xi^{-p\theta}2^{p\theta(m+1)} \|v_m\|_{p\theta}^{p\theta},
\end{align*}
so, applying the previous two inequalities, we obtain
$$
\frac{1}{\xi^{p\theta}} \|v_{m+1}\|_{p\theta}^{p\theta} \leq C  2^{p\theta m(2-\frac{\theta}{\theta_0})} \biggl(\frac{1}{\xi^{p\theta}}{\|v_m\|_{p\theta}^{p\theta}}\biggr)^{2-\frac{\theta}{\theta_0}}.
$$

\medskip

\textit{Step 4.}~Denote
$
z_m:=\frac{1}{\xi^{p\theta}}\|v_m\|_{p\theta}^{p\theta}.
$
Then
$$
z_{m+1} \leq C\gamma^m z_m^{1+\alpha}, \quad m=0,1,2,\dots, \quad \alpha:=1-\frac{\theta}{\theta_0},\;\; \gamma:=2^{p\theta (2-\frac{\theta}{\theta_0})}
$$
and $z_0 = \frac{1}{\xi^{p\theta}}\langle w_+^{p\theta}\rangle \leq C^{-\frac{1}{\alpha}}\gamma^{-\frac{1}{\alpha^2}}$
provided that we fix $\xi$ by $$\xi^{p\theta}:=C^{\frac{1}{\alpha}}\gamma^{\frac{1}{\alpha^2}}\langle w_+^{p\theta}\rangle + (\mu-\mu_0)^{-\beta\theta}H^\theta$$
(so that it also satisfies the previous requirement $\xi^p \geq (\mu-\mu_0)^{-\beta}H$).
Hence, by Lemma \ref{dg_lemma}, $z_m \rightarrow 0$ as $m \rightarrow \infty$. Therefore,
$
w_+ \leq \xi.
$
Thus, we obtain inequality 
\begin{equation}
\label{sup_ineq_int}
\sup_{\mathbb R^d} w_+ \leq K \biggl( \langle w_+^{p\theta}\rangle^{\frac{1}{p\theta}} + (\mu-\mu_0)^{-\frac{\beta}{p}}\big\langle \big(|Q^i|^{p\theta'}|\nabla f|^{p\theta'}+(1+|Q^i|^{p\theta'})|f|^{p\theta'}\big)\mathbf{1}_{v>0}\big\rangle^{\frac{1}{p\theta'}} \biggr).  
\end{equation}

\medskip

\textit{Step 5.}~It remains to estimate $\langle w_+^{p\theta}\rangle^{\frac{1}{p\theta}}$. We already did this in \eqref{cacc_no} (use $p\theta>p$):
$$
(\mu-\mu_0)^{\frac{1}{p\theta}}\|w_+\|_{p\theta}  \leq C^{\frac{1}{p\theta}}\bigl[\langle |Q^i|^{p\theta}|\nabla f|^{p\theta}\rangle + \langle (1+|Q^i|^{p\theta}) |f|^{p\theta}\rangle \bigr]^{\frac{1}{p\theta}}.
$$
This inequality, applied in \eqref{sup_ineq_int}, yields \eqref{dg_ineq} and thus ends the proof of Proposition \ref{emb_thm}. \hfill \qed

\bigskip

\section{Proof of Proposition \ref{thm_holder} (H\"{o}lder continuity)}

\label{thm_holder_proof}
\label{thm_holder_sect}

The following is a well-known consequence of the John-Nirenberg inequality:

\begin{proposition}
There exists constant $C=C(d)$ such that,  for every $g \in {\rm BMO}$, 
\begin{equation}
\label{JN_ineq0}
\sup_{x \in \mathbb R^d,R>0}\frac{1}{|B_R(x)|} \langle e^{\frac{C|g - (g)_{B_R(x)}|}{\|g\|_{\rm BMO}}}\mathbf{1}_{B_R(x)} \rangle \leq C.
\end{equation}
In particular, for every $1 \leq s <\infty$, we have $g \in L^s_{\loc}$ and 
\begin{equation}
\label{JN_ineq}
\sup_{x \in \mathbb R^d,R>0}\frac{1}{|B_R(x)|} \langle |g - (g)_{B_R(x)}|^s \mathbf{1}_{B_R(x)} \rangle \leq C(d,s) \|g\|_{{\rm BMO}}^s.
\end{equation}
\end{proposition}

Put for brevity $b=b_n$, $q=q_m$, $Q=Q_m$ and $u=u_{n,m}$. So,
$$
\big(\mu -\Delta + (b+q) \cdot \nabla\big)u=f.
$$
Let us fix some $1 < \theta < \frac{d}{d-2}$. Let {$p\geq2$,} $p>\frac{2}{2-\sqrt{\delta}}$. By our assumption $\delta<4$, so such $p$ exist.

\begin{lemma}[Caccioppoli's inequality]
\label{cacc_prop}
Let $v:=(u-k)_+$, $k \in \mathbb R$. For all $0<r<R \leq 1$, we have
\begin{equation}
\label{cacc_ineq}
\|\nabla v^{\frac{p}{2}}\|_{L^2(B_r)}^2 \leq 
 \frac{K_1}{(R-r)^2} |B_{R}|^{\frac{1}{\theta'}} (1+\|Q\|_{{\rm BMO}}^2) \|v^{\frac{p}{2}}\|_{L^{2\theta}(B_R)}^2 + K_2 \| |f - \mu u|^{\frac{p}{2}} \mathbf{1}_{v>0} \|_{L^2(B_R)}^2
\end{equation}
for constants $K_1$, $K_2$ independent of $k$, $r$, $R$ and $n$, $m$. 
\end{lemma}

\begin{proof}
Let $\{\eta=\eta_{r,R}\}$ be a family of $[0,1]$-valued smooth cut-off functions satisfying
\begin{equation}
\label{eta_ineq}
\eta =1 \text{ in } B_{r}, \quad \eta = 0 \text{ in } \mathbb R^d \setminus B_{R},
\end{equation}
\begin{equation}
\label{eta_ineq2}
|\nabla \eta| \leq \frac{c}{R-r}\mathbf{1}_{B_R}, \quad \frac{|\nabla \eta|^2}{\eta} \leq \frac{c}{(R-r)^{2}}\mathbf{1}_{B_R}
\end{equation}
with constant $c$ independent of $r$, $R$. We rewrite the equation for $u$ in the form
$
\big(-\Delta + (b+q) \cdot \nabla\big)u=f-\mu u,
$
multiply it by $v^{p-1}\eta$ and integrate, obtaining
\begin{align*}
\frac{4(p-1)}{p^2}\langle \nabla v^{\frac{p}{2}},(\nabla v^{\frac{p}{2}})\eta\rangle & +  \frac{2}{p}\langle \nabla v^{\frac{p}{2}},v^{\frac{p}{2}}\nabla \eta \rangle \\
&  \leq - \frac{2}{p}\langle b \cdot \nabla v^{\frac{p}{2}},v^{\frac{p}{2}}\eta\rangle - \frac{2}{p}\langle q \cdot \nabla v^{\frac{p}{2}},v^{\frac{p}{2}}\eta\rangle + \langle f-\mu u,v^{p-1}\eta\rangle.
\end{align*}
Hence, by Cauchy-Schwarz,
\begin{align}
\label{pre-caccioppoli}
\left(\frac{4(p-1)}{p}-\frac{4}{p}\epsilon\right)\langle |\nabla v^{\frac{p}{2}}|^2\eta\rangle &  \leq \frac{p}{4\epsilon}\big\langle v^p\frac{|\nabla \eta|^2}{\eta} \big\rangle  - 2\langle b \cdot \nabla v^{\frac{p}{2}},v^{\frac{p}{2}}\eta\rangle - 2\langle q \cdot \nabla v^{\frac{p}{2}},v^{\frac{p}{2}}\eta\rangle + p \langle f-\mu u,v^{p-1}\eta\rangle \notag \\[2mm]
&=:I_1+I_2+I_3+I_4. 
\end{align}
Let us estimate terms $I_1$-$I_4$. 
We start with the term $I_3$ containing the distribution-valued vector field $q=\nabla Q$. The other terms $I_1$, $I_2$ and $I_4$ will be estimated in such a way as to fit the estimate on $I_3$. The following argument was used in \cite{H}.
Set $\tilde{Q}^{ij} := Q^{ij} - (Q^{ij})_R$, where, recall, $(Q^{ij})_R=(Q^{ij})_{B_R}$ is the average of $Q^{ij}$ over ball $B_R$. We have 
\begin{align*}
I_3 = -2\langle q \cdot \nabla v^{\frac{p}{2}},v^{\frac{p}{2}}\eta\rangle &= 
-\sum_{i=1}^d \langle q^i \nabla_i v^p, \eta\rangle \\
& (\text{use identity $q^i=\sum_{j=1}^d \nabla_j {Q^{ji}}=\sum_{j=1}^d \nabla_j {\tilde{Q}^{ji}}$ and integrate by parts}) \\
&= \sum_{i,j=1}^d \langle {\tilde{Q}^{ji}} \nabla_{j}\nabla_{i} v^p , \eta\rangle + \sum_{i,j=1}^d \langle {\tilde{Q}^{ji}} \nabla_{i} v^p, \nabla_{j}\eta\rangle \\
& = \sum_{i,j=1}^d \langle {\tilde{Q}^{ji}} \nabla_{j}\nabla_{i} v^p , \eta\rangle + 2\sum_{i,j=1}^d \langle {\tilde{Q}^{ji}} \nabla_{i} v^\frac{p}{2}, v^{\frac{p}{2}}\nabla_{j}\eta\rangle.
\end{align*}
Due to the anti-symmetry of $Q$, the first sum on the right hand side is zero, so
$
I_3 =  2\langle \tilde{Q} \cdot \nabla v^\frac{p}{2}, v^\frac{p}{2} \nabla\eta\rangle.
$
{Put $|\tilde{Q}|:=\max_{1\leq i\leq d}\sum_{j=1}^d|\tilde{Q}^{ij}|$.} Hence
\begin{align}
\label{int_est}
|I_3| \leq \varepsilon_1 \langle |\nabla v^\frac{p}{2}|^2 \eta\rangle + \frac{1}{\varepsilon_1}\langle |\tilde{Q}|^2 v^p \frac{|\nabla \eta|^2}{\eta}\rangle.
\end{align}
The second term in the RHS of \eqref{int_est} is bounded as follows, using \eqref{eta_ineq2}:
\begin{align}
\label{F_estimate2}
\langle |\tilde{Q}|^2 v^p \frac{|\nabla \eta|^2}{\eta} \rangle &\leq \frac{c}{(R-r)^2} \left\langle |\tilde{Q}|^{2\theta'} \mathbf{1}_{B_{R}}\right\rangle ^{\frac{1}{\theta'}} \langle v^{p\theta} \mathbf{1}_{B_{R}}\rangle^\frac{1}{\theta} \notag \\ 
&\text{(use \eqref{JN_ineq})} \notag \\ 
&\leq  \frac{C}{(R-r)^2} |B_{R}|^{\frac{1}{\theta'}}\|Q\|_{{\rm BMO}}^2 \langle v^{p\theta} \mathbf{1}_{B_{R}}\rangle^\frac{1}{\theta}. 
\end{align}
Thus, to summarize, we have
$$
|I_3| \leq \varepsilon_1 \langle |\nabla v^\frac{p}{2}|^2 \eta\rangle + \frac{1}{\varepsilon_1} \frac{C}{(R-r)^2} |B_{R}|^{\frac{1}{\theta'}}\|Q\|_{{\rm BMO}}^2 \langle v^{p\theta} \mathbf{1}_{B_{R}}\rangle^\frac{1}{\theta}.  
$$

\begin{remark}
If the entries of $Q$ are in $L^\infty$, then we can take $\theta = 1$. Indeed, in this case we can obtain \eqref{F_estimate2} directly:
$\langle |\tilde{Q}|^2 v^p \frac{|\nabla \eta|^2}{\eta} \rangle \leq \frac{4c}{(R-r)^2} \|Q\|^2_\infty \langle v^{p} \mathbf{1}_{B_{R}}\rangle$.
\end{remark}

Now, we estimate the remaining terms $I_1$, $I_2$ and $I_4$.
By \eqref{eta_ineq2},
$$
I_1  \leq \frac{cp}{4\epsilon (R-r)^2} |B_R|^\frac{1}{\theta'}\langle v^{p\theta} \mathbf{1}_{B_{R}}\rangle^\frac{1}{\theta}.
$$
Next, 
\begin{align*}
\frac{1}{2}|I_2| & \leq \langle |b| |\nabla v^{\frac{p}{2}}|,v^{\frac{p}{2}}\eta\rangle \leq \alpha {\langle |\nabla v^{\frac{p}{2}}|^2\eta\rangle} + \frac{1}{4\alpha}\langle |b|^2,v^p\eta\rangle, \quad \alpha=\frac{\sqrt{\delta}}{2} \\
& (\text{use  $b \in \mathbf{F}_\delta$}) \\
& \leq \frac{\sqrt{\delta}}{2} \langle |\nabla v^{\frac{p}{2}}|^2\eta\rangle + \frac{1}{2\sqrt{\delta}} \biggl(\delta \langle |\nabla (v^{\frac{p}{2}}\sqrt{\eta})|^2\rangle +c_\delta \langle v^p \eta \rangle  \biggr) \\
& \leq \frac{\sqrt{\delta}}{2} \langle |\nabla v^{\frac{p}{2}}|^2\eta\rangle + \frac{\sqrt{\delta}}{2}\left((1+\varepsilon_0) \langle |\nabla v^{\frac{p}{2}}|^2\eta\rangle+  \frac{1}{4}\big(1+\frac{1}{\varepsilon_0}\big)\langle v^p  \frac{|\nabla \eta|^2}{\eta}\rangle \right) +  \frac{c_\delta}{2\sqrt{\delta}}\langle v^p \eta \rangle \\
& \leq  \frac{\sqrt{\delta}}{2} (2+\varepsilon_0) {\langle |\nabla v^{\frac{p}{2}}|^2\eta\rangle} + \left( \frac{\sqrt{\delta}}{8}(1+\frac{1}{\varepsilon_0}) \frac{c}{(R-r)^2} + \frac{c_\delta}{2\sqrt{\delta}} \right) |B_R|^\frac{1}{\theta'}\langle v^{p\theta} \mathbf{1}_{B_{R}}\rangle^\frac{1}{\theta}.
\end{align*}
Finally, by Young's inequality, for every $\varepsilon_2>0$,
\begin{align*}
\frac{1}{p}|I_4|= |\langle f-\mu u,v^{p-1}\eta\rangle| \leq  \frac{1}{p\varepsilon_2^p} \langle|f-\mu u|^p \mathbf{1}_{\{v>0 \} } \mathbf{1}_{B_R} \rangle + \frac{\varepsilon_2^{p'}}{p'}|B_R|^\frac{1}{\theta'}\langle v^{p\theta} \mathbf{1}_{B_{R}}\rangle^\frac{1}{\theta}.
\end{align*}
Now, applying these estimates on $I_1$-$I_4$ in \eqref{pre-caccioppoli}, we obtain:
\begin{align*}
\left(\frac{4(p-1)}{p}- \frac{4}{p} \varepsilon-(2+\varepsilon_0)\sqrt{\delta}   - \varepsilon_1\right)\langle |\nabla v^{\frac{p}{2}}|^2 \eta\rangle &\leq  \frac{C_1}{(R-r)^2} |B_{R}|^{\frac{1}{\theta'}}(1+\|Q\|_{{\rm BMO}}^2 )\langle v^{p\theta} \mathbf{1}_{B_{R}}\rangle^\frac{1}{\theta} \\
&+ C_2 \||f-\mu u|^\frac{p}{2} \mathbf{1}_{v>0} \mathbf{1}_{B_R} \|_2^2,
\end{align*}
where $\varepsilon$, $\varepsilon_0$, $\varepsilon_1$ are fixed sufficiently small so that the expression in the brackets in the LHS is strictly positive. The latter is possible since $\delta<4$ and $p>\frac{2}{2-\sqrt{\delta}}$.
This ends the proof of Lemma \ref{cacc_prop}.
\end{proof}

\begin{lemma}
\label{sup_bound_prop}
Fix $\alpha>0$ by $\alpha(\alpha + 1) = 1 - \frac{\theta(d-2)}{d}$.
Then, for all $0<r<R \leq 1$, 
\begin{equation}
\label{sup_bound_ineq}
\sup_{B_\frac{R}{2}} u \leq C \left(\frac{1}{|B_R|}\langle u^{p\theta} \mathbf{1}_{B_R \cap \{u>0 \}} \rangle \right)^\frac{1}{p\theta} \left( \frac{|B_R \cap \{u>0\}|}{|B_R|}\right)^\frac{\alpha}{p\theta}+ R^{\frac{2}{p}}
\end{equation}
for  $C$ independent of $n$, $m$, $r$ and $R$. 
\end{lemma}
\begin{proof}
Step 1.~Fix a family of cut-off functions $\eta=\eta_{r,R} \in C_c^\infty$ such that
\begin{align}
\label{cut_off_prop}
\eta = 1 \quad \text{on } B_r \qquad \eta = 0 \quad \text{on } \mathbb{R}^d \setminus {{\overline B}_{\frac{R+r}{2}}}
,
\end{align}
and 
\begin{align}
\label{cut_off_nabla}
|\nabla \eta| \leq \frac{c}{R-r}\mathbf{1}_{B_{\frac{r+R}{2}}}, \quad  \frac{|\nabla \eta|^2}{\eta} \leq \frac{c}{(R-r)^{2}}\mathbf{1}_{B_{\frac{r+R}{2}}}
\end{align}
with constant $c$ independent of $r$, $R$.
Set $v = (u-k)_+$, where $k\in \mathbb{R}$ will be chosen later. Using Sobolev's embedding theorem, we have
\begin{align*}
\langle v^{\frac{pd}{d-2}} \mathbf{1}_{B_{r}} \rangle ^\frac{d-2}{d} &\leq \langle (v^\frac{p}{2} \eta^\frac{1}{2}) ^{\frac{2d}{d-2}} \mathbf{1}_{B_{\frac{R+r}{2}}} \rangle ^\frac{d-2}{d}\\
&\leq C_S^2 \langle |\nabla (v^\frac{p}{2} \eta^\frac{1}{2}) |^2 \mathbf{1}_{B_{\frac{R+r}{2}}} \rangle \\
&\leq 2C_S^2 \left(\langle |\nabla v^\frac{p}{2} |^2 \eta \mathbf{1}_{B_{\frac{R+r}{2}}} \rangle +  \langle v^p \frac{|\nabla \eta|^2}{\eta} \mathbf{1}_{B_{\frac{R+r}{2}}} \rangle\right). 
\end{align*}
We rewrite the latter, after applying H\"{o}lder's inequality in the second term, as follows:
\begin{align*}
\|v^\frac{p}{2}  \mathbf{1}_{B_{r}} \|^2_{\frac{2d}{d-2}} \leq C_1 \left( \|(\nabla v^\frac{p}{2}) \mathbf{1}_{B_{\frac{R+r}{2}}} \|_2^2 + \frac{|B_{R}|^\frac{1}{\theta'}}{(R - r)^2} \| v^\frac{p}{2} \mathbf{1}_{B_\frac{R+r}{2}} \|^2_{2\theta}\right).
\end{align*}
Next, we apply Lemma \ref{cacc_prop} to the first term in the RHS, obtaining
\begin{align*}
\|v^\frac{p}{2} \mathbf{1}_{B_{r}} \|^2_{\frac{2d}{d-2}} &\leq C_2 \biggl(\frac{|B_{R}|^\frac{1}{\theta'}}{(R - r)^2} \| v^\frac{p}{2} \mathbf{1}_{B_R} \|^2_{2\theta}+  \| |f - \mu u|^\frac{p}{2} \mathbf{1}_{v>0} \mathbf{1}_{B_{R}} \|_2^2\biggr).
\end{align*}
Next, applying H\"older's inequality in the second term in the RHS, we estimate
\begin{align*}
\| |f - \mu u|^\frac{p}{2} \mathbf{1}_{v>0} \mathbf{1}_{B_{R}} \|_2^2 \leq \|f - \mu u \|^p_{\infty}  |B_{R} \cap \{ v>0\}|^\frac{1}{\theta} |B_R|^\frac{1}{\theta'}{.}
\end{align*}
{Since the approximating semigroups are $L^\infty$ contractions, we have $\mu\|u\|_\infty\leq\|f\|_\infty$. Thus}
{
\begin{align*}
\|f-\mu u\|_\infty\leq\|f\|_\infty+\mu\|u\|_\infty\leq2\|f\|_\infty,
\end{align*}
}
{and the preceding estimate yields}
\begin{align*}
\|v^\frac{p}{2} \mathbf{1}_{B_{r}} \|^2_{\frac{2d}{d-2}} \leq C_2\biggl( \frac{|B_{R}|^\frac{1}{\theta'}}{(R - r)^2}  \| v^\frac{p}{2} \mathbf{1}_{B_R} \|^2_{2\theta} + |B_{R}|^\frac{1}{\theta'} 2^p\|f\|_\infty^p |B_{R} \cap \{ u>k\}|^\frac{1}{\theta}\biggr). 
\end{align*}
Representing $|B_R|^\frac{1}{\theta'} = |B_R|^{\frac{d-2}{d}+\frac{2}{d}-\frac{1}{\theta}}$, and dividing both sides of the previous inequality by $|B_R|^{\frac{d-2}{d}}$, we have
\begin{equation}
\begin{split}
\label{s}
\frac{1}{|B_R|^\frac{d-2}{d}} \|  v^{\frac{p}{2}} \mathbf{1}_{B_{r}} \|^2_{\frac{2d}{d-2}} &\leq C_2\biggl( |B_R|^\frac{2}{d}  \frac{1}{(R - r)^2}  \frac{1}{|B_R|^\frac{1}{\theta}} \|  v^{\frac{p}{2}} \mathbf{1}_{B_{R}} \|^2_{2\theta}  \\
&\quad + |B_R|^\frac{2}{d} 2^p\|f\|_\infty^p \left( \frac{|B_{R} \cap \{u > k\}|}{|B_R|} \right)^\frac{1}{\theta}\biggr).
\end{split}
\end{equation}
Next, note that if $h<k$, then  $(u-k)_+\leq (u-h)_+$. {Also, $(u-h)_+>k-h$ on $\{u>k\}$.} Therefore, by Chebyshev's inequality,
$$
| B_{R}\cap \{u>k\}  |^\frac{1}{\theta} \leq \frac{1}{(k-h)^p} \langle (u-h)_+^{p\theta} \mathbf{1}_{B_{R}}\rangle^{\frac{1}{\theta}}.
$$
Recalling that $v=(u-k)_+$ and applying the last inequality in \eqref{s}, we obtain
\begin{align*}
\bigg(\frac{\langle (u-k)_+^{\frac{pd}{d-2}} \mathbf{1}_{B_{r}} \rangle}{|B_R|}\bigg)^\frac{d-2}{d} \leq C_3 |B_R|^\frac{2}{d} \biggl( \frac{1}{(R - r)^2 } + \frac{1}{(k-h)^p}\bigg)  \bigg(\frac{\langle  (u-h)_+^{p\theta} \mathbf{1}_{B_{R}} \rangle}{|B_R|}\biggr)^\frac{1}{\theta}.
\end{align*}
Now, using H\"{o}lder's inequality and then applying the previous estimate, we get
\begin{align*}
&\frac{\langle (u-k)^{p\theta}_+ \mathbf{1}_{B_{r}} \rangle}{|B_R|} \leq   \bigg( \frac{\langle (u-k)^{\frac{pd}{d-2}}_+ \mathbf{1}_{B_{r}} \rangle}{|B_R|}\bigg)^\frac{\theta(d-2)}{d}  \bigg( \frac{|B_{r} \cap \{u>k\}|}{|B_R|}  \bigg)^{1-\frac{\theta(d-2)}{d} }\notag \\
&\leq C_3^\theta |B_R|^\frac{2\theta}{d} \left( \frac{1}{(R - r)^2 } + \frac{1}{(k-h)^p}\right)^\theta \bigg( \frac{\langle  (u-h)_+^{p\theta} \mathbf{1}_{B_{R}} \rangle}{|B_R|}\bigg)\bigg( \frac{|B_{R} \cap \{u>k\}|}{|B_R|} \bigg)^{1-\frac{\theta(d-2)}{d}}
\end{align*}
Multiplying this inequality by 
$\left( \frac{|B_{r} \cap \{u>k\}|}{|B_R|} \right)^\alpha  \left[ \leq \frac{1}{(k-h)^{p\theta\alpha}} \left( \frac{\langle (u-h)_+^{p\theta} \mathbf{1}_{B_R} \rangle}{|B_R|} \right)^\alpha \right]$,
and selecting $\alpha>0$ as in the statement of the lemma, i.e.\,such that $\alpha(\alpha + 1) = 1 - \frac{\theta(d-2)}{d}$, we get
\begin{align*}
&\frac{\left\langle (u-k)^{p\theta}_+ \mathbf{1}_{B_{r}} \right\rangle}{|B_R|}  \bigg( \frac{|B_{r} \cap \{u>k\}|}{|B_R|} \bigg)^\alpha \\
&\leq C_3^\theta |B_R|^\frac{2\theta}{d} \left( \frac{1}{(R - r)^2 } + \frac{1}{(k-h)^p}\right)^\theta \frac{1}{(k-h)^{p\theta\alpha}} \left( \frac{\langle (u-h)_+^{p\theta} \mathbf{1}_{B_{R}} \rangle}{|B_R|} \left( \frac{|B_{R} \cap \{u>k\}|}{|B_R|} \right)^\alpha \right)^{1+\alpha}
\end{align*}
At the next step, we are going to iterate this inequality.

\medskip

Step 2.~Now, define
$$r_m := \frac{R}{2} \left( 1 + \frac{1}{2^m} \right), \quad B_m := B_{r_m}$$
$$  k_m := \xi ( 1 - 2^{-m} ),$$
for a positive constant $\xi$ to be determined later.
Setting $$z_m \equiv z(k_m,r_m) := \frac{\left\langle (u-k_m)^{p\theta}_+ \mathbf{1}_{B_{m}} \right\rangle}{|B_R|}  \bigg( \frac{|B_{m} \cap \{u>k_m\}|}{|B_R|} \bigg)^\alpha,$$
we rewrite the previous inequality, upon selecting $k:=k_{m+1}$ and $h:=k_m$ there, as
\begin{align}
\label{z_m}
z_{m+1} \leq {C} \frac{|B_R|^\frac{2\theta}{d}}{R^{2\theta}}\left( 2^{2m} + 2^{mp} \frac{R^2}{\xi^p}\right)^\theta \frac{2^{mp\theta\alpha}}{\xi^{p\theta \alpha}}z^{1+\alpha}_m. 
\end{align}
{Here $C$ also absorbs the fixed factors coming from $r_m-r_{m+1}=R2^{-m-2}$ and $k_{m+1}-k_m=\xi2^{-m-1}$.}
In what follows, we restrict our choice of constant $\xi$ to those satisfying
\begin{align}
\label{xi_1}
\xi^p \geq R^2.
\end{align} 
Then, since $p \geq 2$,
$$z_{m+1} \leq \big(\frac{C_4}{\xi^{p\alpha}}\big)^\theta   2^{mp\theta(1+\alpha)} z_m^{1+\alpha}.$$
Therefore, setting $C_0 = 2^{p\theta(1+\alpha)}$ and $N =( \frac{C_4}{\xi^{p\alpha}})^\theta $, we have the first inequality in Lemma \ref{dg_lemma}, i.e.\,$z_{m+1} \leq N C_0^m z^{1+\alpha}_m$. To apply this lemma, we need to verify the second inequality there, i.e.\,$$z_0 \leq N^{-\frac{1}{\alpha}}C_0^{-\frac{1}{\alpha^2}},$$
where, recall,
$z_0 =\frac{\langle u^{p\theta} \mathbf{1}_{B_R \cap \{u>0\}} \rangle}{|B_R|}\big( \frac{|B_R \cap \{u>0\}|}{|B_R|}\big)^\alpha. $
The previous inequality holds, by the definition of $N$, if we select $\xi$ satisfying
\begin{align}
\label{xi_2}
\xi \geq 2^{\frac{1+\alpha}{\alpha^2}}C_4^\frac{1}{p\alpha}z_0^\frac{1}{p\theta} 
\end{align}
We combine \eqref{xi_1} and \eqref{xi_2} by taking $\xi = 2^{\frac{1+\alpha}{\alpha^2}}C_4^\frac{1}{p\alpha}z_0^\frac{1}{p\theta} + R^{\frac{2}{p}}$. Now Lemma \ref{dg_lemma} yields $z(\xi,\frac{R}{2})=0$, i.e.\,$\sup_{B_\frac{R}{2}} u \leq \xi$.
Hence,
$
\sup_{B_\frac{R}{2}} u \leq C z_0^\frac{1}{p\theta} + R^{\frac{2}{p}},
$ as claimed. The proof of Lemma \ref{sup_bound_prop} is completed.
\end{proof}

\medskip

Set
$$
\text{osc}(u,R):= \sup_{y, y' \in B_R} |u(y)-u(y')|.
$$

\begin{lemma}
\label{H}
{Let $p$ be fixed as above and let $n\geq1$ be an integer.}
Fix $ k_0 $ by
$
2k_0 = M(2R) + m(2R) := \sup_{B_{2R}} u + \inf_{B_{2R}} u.
$
Assume that $ |B_R \cap \{u > k_0\}| \leq \gamma |B_R| $ for some $\gamma < 1$. If
\begin{equation}
\label{o}
{\rm osc}(u, 2R) \geq 2^{n+1} C R^\frac{2}{p},
\end{equation}
then, for $ k_n := M(2R) - 2^{-n-1} {\rm osc}(u, 2R) $,
$$
\frac{|B_R \cap \{u > k_n\}|}{|B_R|} \leq c n^{-\frac{d}{2(d-1)}}.
$$
{The constant $c$ does not depend on $n$.}
\end{lemma}
\begin{proof}[Proof of Lemma \ref{H}]
Let $ h \in (k_0, k) $. Define
$$
w :=
\begin{cases} 
    (u - h)^{\frac{p}{2}} & \text{if } h < u < k, \\
    (k - h)^{\frac{p}{2}} & \text{if } u \geq k, \\
    0 & \text{if } u \leq h.
\end{cases}
$$
Note that $w=0$ in  $B_R \setminus (B_R \cap \{u>k_0\})$. {By assumption, the measure of this set is at least $(1-\gamma)|B_R|$, so Chebyshev's inequality and the Sobolev embedding theorem yield}
\begin{align*}
(k - h)^{\frac{p}{2}} \left| B_R \cap \{u > k\} \right|^{\frac{d-1}{d}} &\leq c_1 \langle w^{\frac{d}{d-1}} \mathbf{1}_{B_R} \rangle^{\frac{d-1}{d}} \\
&\leq c_2 \langle |\nabla w| \mathbf{1}_\Delta \rangle \\
&\leq c_2 |\Delta|^{\frac{1}{2}} \langle |\nabla (u - h)^{\frac{p}{2}}|^2 \mathbf{1}_{B_R \cap \{u > h\}} \rangle^{\frac{1}{2}},
\end{align*}
where $\Delta := B_R \cap \{u >h\} \setminus (B_R \cap \{u >k\}).$
{On the other hand, repeating the proof of Lemma \ref{cacc_prop}, using the splitting \eqref{pre-caccioppoli} with $v=(u-h)_+$ and estimating the terms $I_1$, $I_2$, $I_4$ in the same manner, but estimating the term $I_3$ via \eqref{int_est} rather than \eqref{F_estimate2}, we obtain}
\begin{align*}
\langle |\nabla (u-h)^\frac{p}{2}|^2 \mathbf{1}_{B_R \cap \{u>h\}}  \rangle &\leq \frac{C_1}{R^2} \langle (u-h)^p \mathbf{1}_{B_{2R} \cap \{u>h\}}  \rangle \\
&+ \frac{C_2}{R^2} \langle |Q-(Q)_R|^2 (u-h)^p \mathbf{1}_{B_{2R} \cap \{u>h\}}  \rangle \\
&+ C_3 \langle  |f-\mu u|^p \mathbf{1}_{B_{2R} \cap \{u>h\}}  \rangle{.}
\end{align*}
Applying the John-Nirenberg inequality \eqref{JN_ineq} in the second term, we get
\begin{align*}
\langle |\nabla (u-h)^\frac{p}{2}|^2 \mathbf{1}_{B_R \cap \{u>h\}}  \rangle &\leq C_4 R^{d-2} (M(2R)-h)^p + \frac{C_5}{R^2}(M(2R)-h)^p \cdot R^d \|Q\|^2_{{\rm BMO}} \\
&+ C_6\|f-\mu u\|^p_\infty R^d\\
&\leq CR^{d-2}(M(2R)-h)^p +  C R^d,
\end{align*}
(we have used $\|f-\mu u\|_\infty \leq 2\|f\|_\infty$).
For $ h \leq k_n $, we have $ M(2R) - h \geq M(2R) - k_n = 2^{-n-1}{\rm osc}(u,2R) \geq C R^\frac{2}{p}$, where we have used \eqref{o}, in which case
\begin{align*}
(k-h)^\frac{p}{2}|B_R \cap \{u > k\}|^{\frac{d-1}{d}} \leq C |\Delta|^\frac{1}{2}  R^{\frac{d-2}{2}} (M(2R) - h)^{\frac{p}{2}}.
\end{align*}
Now, choose an increasing finite sequence $k= k_i := M(2R) -2^{-i-1}{\rm osc}(u,2R)$ for $i\in \{1,2,\cdots, n\}$ and $h = k_{i-1}$. Then
$$
M(2R)-h = 2^{-i} {\rm osc}(u,2R), \quad  |k-h| = 2^{-i-1}{\rm osc}(u,2R)
$$
so 
$$
\left| B_R \cap \{u > k_n\} \right|^{\frac{2(d-1)}{d}} \leq \left| B_R \cap \{u > k_i\} \right|^{\frac{2(d-1)}{d}} \leq C|\Delta_i| R^{d-2},
$$
where \({\Delta_i = B_R \cap \{u >k_{i-1}\} \setminus (B_R \cap \{u >k_i\})}\). Summing up over $i$, we obtain
$$
n\left| B_R \cap \{u > k_n\} \right|^{\frac{2(d-1)}{d}} \leq CR^{d-2} \left| B_R \cap \{u > k_0\} \right| \leq C'R^{2(d-1)},
$$
and the claimed inequality follows. This ends the proof of Lemma \ref{H}.
\end{proof}

\medskip

We are in a position to end the proof of Proposition \ref{thm_holder}. Fix $k_0 = \frac{1}{2}(M(2R) + m(2R))$. Without loss of generality, $|B_R \cap \{u>k_0\}| \leq \frac{1}{2}|B_R|$ (otherwise replace $u$ by $-u$). {If ${\rm osc}(u,2R)=0$, there is nothing to prove. Otherwise, set $k_n = M(2R) - 2^{-n-1}{\rm osc}(u,2R)$ for $n\geq1$. Then $k_n-k_0=(\frac12-2^{-n-1}){\rm osc}(u,2R)>0$.} By Lemma \ref{sup_bound_prop} applied to $u-k_n$ ($k_n$ adds constant term $\mu k_n$ in the equation, but since $\{k_n\}$ is bounded, the constant in Lemma \ref{sup_bound_prop} can be chosen to be independent of $n$), we have
\begin{align}
\label{sup_half}
\sup_{B_{\frac{R}{2}}} (u-k_n) &\leq C_1 \bigg( \frac{1}{|B_R|}\langle (u-k_n)^{p\theta} \mathbf{1}_{B_R\cap \{u>k_n\}}\rangle\bigg)^\frac{1}{p\theta} \bigg(\frac{|B_R \cap \{u>k_n\}|}{|B_R|}\bigg)^\frac{\alpha}{p\theta} + R^\frac{2}{p} \notag\\
&\leq C_1 \sup_{B_R}(u-k_n)\bigg(\frac{|B_R \cap \{u>k_n\}|}{|B_R|}\bigg)^\frac{1+\alpha}{p\theta} +R^\frac{2}{p}.
\end{align}
Fix $n{\geq1}$ by $cn^{-\frac{d}{2(d-1)}}\leq (\frac{1}{2C_1})^\frac{p\theta}{1+\alpha}$. We consider two cases. First, let ${\rm osc}(u,2R) \geq 2^{n+1}R^\frac{2}{p}$. Then by Lemma \ref{H} (with, say, $C=1$) applied in the RHS of \eqref{sup_half},
\begin{align*}
M(R/2) - k_n &\leq  \frac{1}{2}(M(2R)-k_n) +R^\frac{2}{p}
\end{align*}
so
\begin{align*}
M(R/2) \leq M(2R) - \frac{1}{2^{n+1}} {\rm osc}(u,2R) + \frac{1}{2}\frac{1}{2^{n+1}}{\rm osc}(u,2R) + R^\frac{2}{p}, 
\end{align*}
which yields
\begin{align*}
{{\rm osc}(u,R/2)} &{\leq M(R/2)-m(R/2)}\\
&\leq M(2R) - m(2R) - \frac{1}{2}\frac{1}{2^{n+1}}{\rm osc}(u,2R) + R^\frac{2}{p}\\
&= \left( 1-\frac{1}{2^{n+2}}\right){\rm osc}(u,2R) + R^\frac{2}{p}.
\end{align*}
Next, if ${\rm osc}(u,2R) \leq 2^{n+1} R^\frac{2}{p}$, then
{
\begin{align*}
{\rm osc}(u,R/2)-{\rm osc}(u,2R)\leq0
\leq\frac12 R^{\frac{2}{p}}-\frac{1}{2^{n+2}}{\rm osc}(u,2R),
\end{align*}
}
{and so}
\begin{align*}
{\rm osc}(u, R/2)\leq (1- \frac{1}{2^{n+2}}) {\rm osc}(u,2R) + \frac{1}{2} R^\frac{2}{p}.
\end{align*}
This result provides the desired H\"older continuity of $u$ by applying the next Lemma \ref{scaling_lemma} with $\tau = \frac{1}{4}$, $\delta = \log_{\tau}(1 - {2^{-n-2}})$, and $0 < \beta < {\frac{2}{p}} \wedge \delta$. Note that the second inequality in Lemma {\ref{scaling_lemma}} is satisfied when $q = 1$ and ${\varphi(t):={\rm osc}(u,t)}$ is non-decreasing, which is our situation. \hfill \qed

\begin{lemma}[{\cite[Lemma 7.3]{G}}]
\label{scaling_lemma}
Let $\varphi(t)$ be a positive function, and assume that there {exist} {constants $q>0$, $B\geq0$, $0<\beta<\varkappa$,} and a number $0 < \tau < 1$ such that for every $0 < R < R_0$,
\[
\varphi(\tau R) \leq \tau^\varkappa \varphi(R) + B R^\beta
\]
and
\[
\varphi(t) \leq q \varphi(\tau^k R)
\]
for every $t$ in the interval $(\tau^{k+1} R, \tau^k R)$. Then, for every $0 < \rho < R < R_0$, we have
\[
\varphi(\rho) \leq C \left( \frac{\rho}{R} \right)^\beta \varphi(R) + {C} B \rho^\beta
\]
with a constant $C$ that depends only on $q$, $\tau$, $\varkappa$, and $\beta$.
\end{lemma}

\bigskip

\section{Proof of Proposition \ref{thm_conv} (Convergence)}

\label{thm_conv_sect}

The proof given below is close to \cite[Proof of Theorem 4.3]{KiS_theory}. This argument (in $L^2$) can, in principle, be replaced by an argument based on the Lions' variational approach that handles well both $b \in \mathbf{F}_\delta$ and $q \in \mathbf{BMO}^{-1}$ (regarding the latter, see \cite{QX}).

\medskip

In order to prove the existence of the limit, we need to construct first intermediate semigroups $e^{-t\Lambda(b_n,q)}$ in $L^p$. Here $b_n$ are bounded and smooth, but $q$ can be singular. 

\medskip

Step 1.~At the first step, we construct $e^{-t\Lambda(0,q)}$ in $L^2$. Here we work over complex numbers (in this regard, see Remark \ref{hol_rem}). Define sesquilinear form
$$
\tau[v,w]:=\langle \nabla v,\nabla w\rangle + \langle q \cdot \nabla v,w\rangle, \quad D(\tau)=W^{1,2},
$$
where $\langle v,w\rangle=\langle v\bar{w}\rangle$.

a) $\tau$ is bounded:
$
|\tau[v,w]| \leq C\|\nabla v\|_2\|\nabla w\|_2.
$
Indeed, by the compensated compactness estimate (Proposition \ref{cc_lem}), $|\langle q\cdot \nabla v,w\rangle| \leq C'\|\nabla v\|_2\|\nabla w\|_2$.

b) $\tau$ is a sectorial form:
$$
\Imag \tau[v,v] \leq K\Real \tau[v,v], \quad v \in D(\tau)
$$
for some constant $K>0$.
Indeed, writing $v=r+ie$, where $r,e$ are real-valued elements of $W^{1,2}$, we have
\begin{align*}
\tau[v,v] & =\langle |\nabla r|^2+|\nabla e|^2\rangle + \langle q \cdot \big[(\nabla r)r+(\nabla e)e \big]\rangle + i\langle q \cdot [(\nabla e)r - (\nabla r)e]\rangle,
\end{align*}
so, taking into account that the second term vanishes due to the anti-symmetry of $Q$, we have
$$
\Real \tau[v,v]=\langle |\nabla r|^2+|\nabla e|^2\rangle, \quad \Imag\tau[v,v]=\langle q \cdot [(\nabla e)r - (\nabla r)e]\rangle.
$$
Now, invoking again the compensated compactness, we obtain $\Imag\tau[v,v] \leq C\|Q\|_{\rm BMO}\Real \tau[v,v]$. 

c) $\Real \tau[v,v]$ is a closed form, i.e.\,if $v_k \rightarrow v$ in $L^2$ and $\Real \tau[v_k-v_l] \rightarrow 0$ as $k,l \rightarrow \infty$, then $\tau[v_k-v] \rightarrow 0$  (we only need to look at the real part of $\tau$ due to its sectoriality). But in our case the latter is just a re-statement that $W^{1,2}$ is a complete space.

\smallskip

Therefore, there exists a unique ($m$-sectorial) operator $\Lambda(0,q)$ such that
$$
\langle \Lambda(0,q)v,w\rangle=\tau[v,w], \quad v \in D(\Lambda(0,q)) \subset W^{1,2}, \quad w \in D(\tau)=W^{1,2},
$$
see \cite[Ch.\,VI, \S 2]{Ka}.
This operator, being $m$-sectorial, generates a holomorphic semigroup $e^{-t\Lambda(0,q)}$ in $L^2$. 

\begin{remark}
\label{hol_rem}
A property of holomorphic semigroups that we will need at Step 5 is as follows: for every $f \in L^2$, $t>0$, $e^{-t\Lambda(0,q)}f$ belongs to the domain $D(\Lambda(0,q))$.
\end{remark}

\medskip

Step 2.~We need to construct $e^{-t\Lambda(b_n,q)}$, i.e.\,to add the drift term $b_n \cdot \nabla$. So, we re-do what we did above for the sesquilinear form 
\begin{equation}
\label{tau}
\tau[v,w]:=\langle \nabla v,\nabla w\rangle + \langle q \cdot \nabla v,w\rangle + \langle b_n \cdot \nabla v,w\rangle.
\end{equation}
In particular, applying Cauchy-Schwarz' inequality to $\langle b_n \cdot \nabla v,w\rangle$, we obtain
\begin{equation}
\label{imag_real}
\Imag \tau[v,v] \leq K\big(\Real \tau[v,v]+C\langle v,v\rangle\big),
\end{equation}
where $C=C(\|b_n\|_\infty) \geq 0$. 
The cited results apply to the case \eqref{imag_real}, and we get ($m$-sectorial) generator $\Lambda(b_n,q)$ of a holomorphic semigroup in $L^2$ such that
\begin{equation}
\label{lambda_id}
\langle \Lambda(b_n,q)v,w\rangle=\tau[v,w] \quad \text{ for $\tau$ defined by \eqref{tau}},
\end{equation}
where $v \in D(\Lambda(b_n,q)) \subset W^{1,2}$, $w \in D(\tau)=W^{1,2}$. (Or we could appeal to the Hille perturbation theorem and work with the algebraic sum $\Lambda(b_n,q):=\Lambda(0,q) + b_n \cdot \nabla$, $D(\Lambda(b_n,q))=D(\Lambda(0,q))$, that still generates a holomorphic semigroup in $L^2$.)

\medskip

Step 3.~Let us now note that we have the following quasi-contraction estimate for $e^{-t\Lambda(b_n,q_m)}$, $n,m=1,2,\dots$, in $L^p$, $p>\frac{2}{2-\sqrt{\delta}}$, $p \geq 2$:
\begin{equation}
\label{qce}
\|e^{-t\Lambda(b_n,q_m)}f\|_{p} \leq e^{\omega t}\|f\|_p, \quad t \geq 0, \quad f \in C_c^\infty,
\end{equation}
for some $\omega$ independent of $n,m$. Indeed, setting $u_{n,m}=e^{-t(\omega + \Lambda(b_n,q_m))}f$,  we multiply the corresponding parabolic equation $(\omega + \partial_t-\Delta + (b_n+q_m)\cdot \nabla)u_{n,m}=0$ by $u_{n,m}|u_{n,m}|^{p-2}$ and repeat the proof of Proposition \ref{elem_lem} (without the weight there) with the obvious modifications for the time derivative term (see Step 5 below for details). The constant $\omega>0$ needs to be chosen to account e.g.\,for the constant $c_\delta$ resulting from the use of condition $b \in \mathbf{F}_\delta$. 

If we do include the weight $\rho_x$, then, again arguing as in the proof of Proposition \ref{elem_lem}, we obtain
a weighted quasi contraction estimate 
\begin{equation}
\label{qce_rho}
\|e^{-t\Lambda(b_n,q_m)}f\|_{L_{\rho_x}^p} \leq e^{\omega' t}\|f\|_{L_{\rho_x}^p}, \quad t \geq 0,
\end{equation}
for any $x \in \mathbb R^d$, with $\omega'$ independent of $n$, $m$ or $x$.
(Since $\rho_x \leq 1$, the result of Step 4 below implies that
\begin{equation}
\label{res_conv_rho}
u_{n,m}=(\mu+\Lambda(b_n,q_m))^{-1}f \rightarrow u_n=(\mu+\Lambda(b_n,q))^{-1}f \quad \text{ in $L_{\rho_x}^p$}.
\end{equation}
Moreover, \eqref{qce_rho} ensures that the resulting semigroup $e^{-t\Lambda(b_n,q)}$ is strongly continuous in $L^p_{\rho_x}$.) 

\medskip

Step 4.~From now on, we work over reals. Let us show convergence 
$$e^{-t\Lambda(b_n,q_m)} \rightarrow e^{-t\Lambda(b_n,q)} \text{ in $L^2$ loc.\,uniformly in $t \geq 0$}$$ 
as $m \rightarrow \infty$. Here $n$ is fixed. 

Since $n$ is fixed, it is easily seen that the operator norms of the resolvents $\|(\mu+\Lambda(b_n,q_m))^{-1}\|_{2 \rightarrow 2}$ are uniformly in $m$ bounded, provided $\mu=\mu(\|b_n\|_\infty)$ is fixed sufficiently large.
It suffices for us (see \cite[Ch.\,XI, \S 5]{Ka}) to show the convergence of the resolvents
$$
(\mu+\Lambda(b_n,q_m))^{-1}f \rightarrow (\mu+\Lambda(b_n,q))^{-1}f \quad \text{ in $L^2$ as $m \rightarrow \infty$}
$$
for all $f \in L_c^2$ (subscript $c$ means compact support).

The standard argument yields that $u_n=(\mu+\Lambda(b_n,q))^{-1}f$ is the unique weak solution to the elliptic equation
$
(\mu-\Delta+(b_n+q) \cdot \nabla)u_n=f,
$
where the former means that
$$
\mu\langle u_n,\varphi \rangle + \langle \nabla u_n,\nabla \varphi\rangle + \langle (b_n+q) \cdot \nabla u_n,\varphi\rangle = \langle f,\varphi\rangle, \quad \varphi \in C_c^\infty.
$$
(The compensated compactness estimate $|\langle q\cdot \nabla u_n,\varphi\rangle|=|\langle Q \cdot\nabla u_n,\nabla \varphi\rangle| \leq C'\|\nabla u_n\|_2\|\nabla \varphi\|_2$ of Proposition \ref{cc_lem} allows us to pass to test functions $\varphi \in W^{1,2}$.)
In turn, this uniqueness and the usual weak compactness argument show that $u_{n,m}=(\mu+\Lambda(b_n,q_m))^{-1}f$, i.e.\,solutions to the approximating elliptic equations $
(\mu-\Delta+(b_n+q_m) \cdot \nabla)u_{n,m}=f,
$
converge weakly in $W^{1,2}$ to the same limit $u_n$. Now we can appeal to the Rellich-Kondrashov theorem to obtain $u_{n,m} \rightarrow u_n$ in $L^2_{\loc}$, and further to upgrade this convergence to $u_{n,m} \rightarrow u_n$ in $L^2$ by ``cutting tails'' of $u_{n,m}$ at infinity uniformly in $m$ using the upper Gaussian bound on the heat kernel of $-\Delta + (b_n + q_m) \cdot \nabla$, see \cite{QX}, and taking into account that $f$ has compact support. (The constants in the upper Gaussian heat kernel bound, which yields the bound on the integral kernel of the resolvents, will depend on $n$. Since each $b_n$ is bounded, adding the drift term in $b_n \cdot \nabla$ does not  affect the proof of the upper bound in \cite{QX}  which employs Moser's iterations and the Davies device. In fact, one can account for $b_n \cdot \nabla$ by once again introducing a constant term in the operator that will absorb the contribution from $b_n \cdot \nabla$ in Moser's method.)

\medskip

Now, since $e^{-t\Lambda(b_n,q_m)}$ are $L^\infty$ contractions, we obtain by interpolation that
$$e^{-t\Lambda(b_n,q_m)} \rightarrow e^{-t\Lambda(b_n,q)} \text{ in $L^p$ loc.\,uniformly in $t \geq 0$}$$ 
for all $p \geq 2$. This implies the convergence of the resolvents: as $m \rightarrow \infty$,
\begin{equation}
\label{res_conv}
u_{n,m}=(\mu+\Lambda(b_n,q_m))^{-1}f \rightarrow u_n=(\mu+\Lambda(b_n,q))^{-1}f \quad \text{ in $L^p$}
\end{equation}
with $\mu$ independent of $n$, $m$.

\medskip

Step 5.~Having constructed the intermediate semigroups $e^{-t\Lambda(b_n,q)}$, our goal now is to show that they converge as $n \rightarrow \infty$. For the reader's convenience, at this step we give a proof that assumes additionally that $b_n \rightarrow b$ in $L^2$.
{Let us prove that, in this case, for every $f \in C_c^\infty$,}
\begin{equation}
\label{cauchy}
\{v_n(t):=e^{-t(\omega+\Lambda(b_n,q))}f\}_{n=1}^\infty \quad \text{ is a Cauchy sequence in $L^\infty([0,1],L^p)$}
\end{equation}
for $$p>\frac{2}{2-\sqrt{\delta}}, \quad p \geq 2, \text{ for some fixed $\omega$}.$$
Taking into account Remark \ref{hol_rem} about $v_n(t) \in D(\Lambda(b_n,q))$ for all $t>0$ and the identity \eqref{lambda_id}, we can write
\begin{equation}
\label{id_9}
\langle \partial_t v_n,\psi\rangle + \omega \langle v_n,\psi\rangle  + \langle \nabla v_n,\nabla \psi\rangle + \langle b_n \cdot \nabla v_n,\psi\rangle - \langle Q \cdot \nabla v_n,\nabla \psi\rangle=0,
\end{equation}
for all $\psi(t,\cdot) \in W^{1,2}$, where, recall, $\nabla Q=q$. 
Set $$h:=v_{n_1}-v_{n_2}.$$ Subtracting identities \eqref{id_9} for $v_{n_1}$ and $v_{n_2}$  from each other, we obtain
$$
\langle \partial_t h,\psi\rangle + \omega \langle h,\psi\rangle  + \langle \nabla h,\nabla \psi\rangle + \langle b_{n_1} \cdot \nabla h,\psi\rangle - \langle Q \cdot \nabla h,\nabla \psi\rangle= {-\langle(b_{n_1}-b_{n_2}) \cdot \nabla v_{n_2},\psi\rangle}.
$$
{We use the energy argument of Proposition \ref{elem_lem_par}(\textit{i}), whose spatial terms are estimated as in the proof of Proposition \ref{elem_lem}. Here $Q$ is no longer smooth and we are dealing with weak solutions of the parabolic equations rather than classical solutions.} {In this regard, the following comment is technical, but important. First, the $L^\infty$ contractivity that is inherited from the approximating semigroups yields
$$
\|h(t)\|_\infty\leq2\|f\|_\infty.
$$
For $t>0$, we also have $h(t)\in W^{1,2}$ by Remark \ref{hol_rem}. Since $p\geq2$, the functions $r\mapsto r|r|^{p-2}$, $r\mapsto r|r|^{\frac p2-1}$ and $r\mapsto|r|^{\frac p2}$ are Lipschitz continuous on the (bounded) range of $h(t)$ and vanish at zero. Therefore, extending these functions to Lipschitz functions on $\mathbb R$ and applying the Sobolev chain rule, taking into account the previous points, gives $h(t)|h(t)|^{p-2}$, $h(t)|h(t)|^{\frac p2-1}$, $|h(t)|^{\frac p2}\in W^{1,2}$.} So, we can take $\psi=h|h|^{p-2}$, obtaining
\begin{align}
\frac{1}{p}\partial_t \langle |h|^p\rangle + \omega \langle |h|^p\rangle +   \frac{4(p-1)}{p^2}\langle |\nabla |h|^{\frac{p}{2}}|^2 \rangle & +  \frac{2}{p}{\langle} b_{n_1} \cdot \nabla |h|^{\frac{p}{2}},|h|^{\frac{p}{2}}\rangle \notag \\
& \leq |\langle (b_{n_1}-b_{n_2}) \cdot \nabla v_{n_2},h|h|^{p-2}\rangle|, \label{h_ineq}
\end{align}
{where we have used anti-symmetry of $Q$.} 
We handle the term containing $b_{n_1}$ as in the proof of Proposition \ref{elem_lem}, i.e.\,first applying quadratic inequality and then $b_{n_1} \in \mathbf{F}_{\delta}$.
Now, handling the time derivative term as in the proof of Proposition \ref{elem_lem_par}, {and using $h(0)=0$,} we obtain
\begin{align}
\frac{1}{p}\sup_{s \in [0,t]}\langle |h(s)|^p\rangle + \biggl[\omega-2\frac{c_\delta}{\sqrt{\delta}} \biggr] \int_0^t \|h\|_p^p ds +  & \biggl[\frac{4(p-1)}{p^2} - \frac{2}{p}\sqrt{\delta} \biggr] \int_0^t\langle |\nabla |h|^{\frac{p}{2}}|^2 \rangle ds \notag \\
& \leq 2\int_0^t \|b_{n_1}-b_{n_2}\|_2 \|\nabla v_{n_2}\|_2 \|h\|_\infty^{p-1}ds. \label{ineq9}
\end{align}
Take $\frac{1}{2}\omega:=2\frac{c_\delta}{\sqrt{\delta}}$. Since $p>\frac{2}{2-\sqrt{\delta}}$, the expressions in square brackets are strictly positive. 
In the right-hand side of \eqref{ineq9}, $\|b_{n_1}-b_{n_2}\|_2 \rightarrow 0$ as $n_1, n_2 \rightarrow \infty$ and $\|h(s)\|_\infty^{p-1} \leq 2^{p-1}\|f\|^{p-1}_\infty$, $s \in [0,t]$, for all $n_1$, $n_2$. It remains to note that $\int_0^t \|\nabla v_{n_2}\|_2 ds$ is uniformly in $n_2$ bounded. Indeed, by \eqref{id_9} with $n=n_2$ and $\psi=v_{n_2}$ upon noting that
\begin{align*}
& |\langle b_{n_2} \cdot \nabla v_{n_2},v_{n_2}\rangle| + |\langle Q \cdot \nabla v_{n_2},\nabla v_{n_2}\rangle | \\
& = |\langle b_{n_2} \cdot \nabla v_{n_2},v_{n_2}\rangle|  \leq \frac{1}{2}\|b_{n_2}\|^2_2\|v_{n_2}\|^2_\infty + \frac{1}{2} \|\nabla v_{n_2}\|_2^2,
\end{align*}
where we have used again the anti-symmetry of $Q$.
Now we use $\sup_{n_2}\|b_{n_2}\|_2^2<\infty$ and the obvious a priori estimate $\|v_{n_2}(t)\|_\infty \leq \|f\|_\infty$, $t \geq 0$.
Thus, the right-hand side of \eqref{ineq9} tends to $0$ as $n_1$, $n_2 \rightarrow \infty$. Hence $h \rightarrow 0$ in $L^\infty([0,1],L^p)$, and \eqref{cauchy} follows. 
The latter and the contractivity estimate \eqref{qce} yield that the limit $v=L^p\mbox{-}\lim_n v_n$ (loc.\,uniformly in $t \geq 0$) determines a strongly continuous semigroup in $L^p$, say, $v(t)=:e^{-t\Lambda(b,q)}f$. In turn, the convergence of the semigroups yields the convergence of the resolvents 
$$
u_n=(\mu+\Lambda(b_n,q))^{-1}f \rightarrow u=(\mu+\Lambda(b,q))^{-1}f \quad \text{ in $L^p$}
$$
with $\mu$ independent of $n$ (proportional to $\omega$).
In view of \eqref{res_conv}, the existence of the limit $u=L^p\mbox{-}\lim_n u_n=L^p\mbox{-}\lim_n\lim_m u_{n,m}$ follows.

\medskip

Step 6.~Finally, we  prove convergence in the general case, i.e.\,we do not assume global convergence of $b_n$ to $b$ in $L^2$. We show that, for any $x \in \mathbb R^d$,
for every $f \in C_c^\infty$,
\begin{equation}
\label{cauchy2}
\{v_n(t):=e^{-t(\mu+\Lambda(b_n,q))}f\}_{n=1}^\infty \quad \text{ is a Cauchy sequence in $L^\infty([0,1],L_{\rho_x}^p)$}
\end{equation}
for
$p>\frac{2}{2-\sqrt{\delta}}$, $p \geq 2$, for some fixed $\mu$. Then, repeating the argument in the end of Step 5 but using \eqref{qce_rho} and \eqref{res_conv_rho}, we will obtain the claimed in Proposition \ref{thm_conv} existence of the limit $$u \;\biggl(=L^p_{\rho_x}\mbox{-}\lim_n u_n\biggr) \;=L^p_{\rho_x}\mbox{-}\lim_n\lim_m u_{n,m}.$$

Let us prove \eqref{cauchy2}. This time, taking $\psi=h|h|^{p-2}\rho_x$, we obtain
\begin{align}
\sup_{s \in [0,t]}\langle |h(s)|^p \rho_x\rangle + C_1 \int_0^t \langle |h|^p\rho_x \rangle ds +  & C_2 \int_0^t\langle |\nabla |h|^{\frac{p}{2}}|^2 \rho_x \rangle ds \notag \\
& \leq C_3 \int_0^t |\langle (b_{n_1}-b_{n_2}) \cdot \nabla v_{n_2},h|h|^{p-2} \rho_x\rangle|ds \label{h_rho}
\end{align}
for constants $C_1$-$C_3$ independent of $n$, $m$, constant $C_1$ being strictly positive provided that $\sigma$ in the definition of $\rho$ is fixed sufficiently small. Arguing as at the previous step, {using $\sup_n\langle |b_n|^2\rho_x\rangle<\infty$ (by the form-boundedness of $b_n$ applied to $\sqrt{\rho_x}$),} we obtain
$$
{\sup_{n_2}\int_0^1\langle|\nabla v_{n_2}(s)|^2\rho_x\rangle ds <\infty.}
$$
Therefore, {for $0<t\leq1$,}
{
\begin{align*}
&\int_0^t |\langle (b_{n_1}-b_{n_2}) \cdot \nabla v_{n_2},h|h|^{p-2} \rho_x\rangle|ds \\
&\quad\leq \langle |b_{n_1}-b_{n_2}|^2\rho_x\rangle^{\frac12}
(2\|f\|_\infty)^{p-1}t^{\frac12}
\biggl(\int_0^t \langle|\nabla v_{n_2}|^2\rho_x\rangle ds\biggr)^{\frac12}\\
&\quad\rightarrow 0 \quad \text{ as $n_1,n_2 \rightarrow \infty$}
\end{align*}
(use $\sup_{s\in[0,t]}\|h(s)\|_\infty \leq 2\|f\|_\infty$ and apply Lemma \ref{lem_b}).
}
Combining this with \eqref{h_rho}, we obtain the claimed convergence \eqref{cauchy2}.
\hfill \qed

\begin{remark}
\label{stronger_conv_rem}
If we try to obtain a stronger result about the existence of the limit $s\mbox{-}L^p\mbox{-}\lim_{n,m}u_{n,m}$ by extending the proof of Proposition \ref{thm_conv}  to the sequence
$h:=u_{n_1,m_1}-u_{n_2,m_2}$, then we get an extra term in the right-hand side of \eqref{h_ineq}: $|\langle (Q_{m_1}-Q_{m_2}) \cdot \nabla u_{n_2,m_2},\nabla (h|h|^{p-2})\rangle|$. It can be dealt with in two ways: 

\smallskip

(a) We can estimate
$$
|\langle (Q_{m_1}-Q_{m_2}) \cdot \nabla u_{n_2,m_2},\nabla (h|h|^{p-2})\rangle| \leq \|Q_{m_1}-Q_{m_2}\|_s\|\nabla u_{n_2,m_2}\|_{s'}(p-1)(2\|f\|_\infty)^{p-2}\|\nabla h\|_2, 
$$
$$
\frac{1}{s}+\frac{1}{s'}=\frac{1}{2}.
$$
So, assuming for the illustration purposes that we have global convergence $\|Q_{m_1}-Q_{m_2}\|_s \rightarrow 0$ as $m_1,m_2 \rightarrow \infty$, we need a bound on $\|\nabla u_{n_2,m_2}\|_{s'}$ for an $s'>2$. In principle, $s'$ can be chosen to be close to $2$. To obtain such an estimate, we can use Gehring-Giaquinta-Modica's lemma as in the proof of Lemma \ref{grad_lem}, but this, at least in the present form of the argument, requires us to consider the equation in $L^2$, hence we need to require $\delta<1$ rather than $\delta<4$ as in \eqref{cond}.

\smallskip

(b) Another option is to use the estimate 
$$
|\langle (Q_{m_1}-Q_{m_2}) \cdot \nabla u_{n_2,m_2},\nabla (h|h|^{p-2})\rangle| \leq \|Q_{m_1}-Q_{m_2}\|_{{\rm BMO}}\|\nabla u_{n_2,m_2}\|_{2}(2\|f\|_\infty)^{p-2}\|\nabla h\|_2,
$$
where $\|\nabla u_{n_2,m_2}\|_{2}$ can be estimated as in the proof of Proposition \ref{thm_conv}, but now to have convergence $\|Q_{m_1}-Q_{m_2}\|_{{\rm BMO}} \rightarrow 0$ as $m_1,m_2 \rightarrow \infty$ we need a stronger hypothesis on the matrix field $Q$ and thus on $q$, namely, that $Q$ has entries in ${\rm VMO}(\mathbb R^d)$.
\end{remark}

\hfill \qed

\bigskip

\section{Proof of Proposition \ref{sep_thm} (Separation property)}

\label{sep_thm_sect}

\medskip

Step 1.~Without loss of generality, $x=0$.
 We will need the following local estimate on solution $u=u_{n,m}$ of \eqref{eq10}: for all $\mu \geq \mu_0>0$,
\begin{equation}
\label{dg_ineq_u}
\sup_{B_{\frac{1}{2}}} |u| \leq K \biggl( \langle |u|^{p\theta}\mathbf{1}_{B_1}\rangle^{\frac{1}{p\theta}} + \big\langle |f|^{p\theta'}\mathbf{1}_{B_1}\big\rangle^{\frac{1}{p\theta'}} \biggr).
\end{equation}
In fact, it suffices to prove the previous estimate for $\sup_{B_{\frac{1}{2}}} u_+$ in the LHS.

To that end, we first establish the following Caccioppoli's inequality for $v:=(u-k)_+$, $k \geq 0$:
\begin{equation}
\label{ineq_h_n}
\|v\|^{p}_{L^{\frac{pd}{d-2}}(B_{r})} \leq C\biggl[(R-r)^{-2}|B_{R}|^{\frac{1}{\theta'}}\|v\|^{p}_{L^{p\theta}(B_{R})}+\|f\mathbf{1}_{u>k}\|_{L^p(B_R)}^p\biggr].
\end{equation}
To prove \eqref{ineq_h_n}, we argue as in the proof of Lemma \ref{cacc_prop}, but treat the term $\mu u$ differently: since $\mu$ and $k$ are non-negative, we have
$$
\mu (u-k) - \Delta (u-k) + (b+q)\cdot \nabla (u-k) \leq f.
$$
Therefore, multiplying the previous inequality by $v^{p-1}\eta$, with the cutoff function $\eta$ defined by \eqref{eta_ineq}, \eqref{eta_ineq2}, and repeating the proof of Lemma \ref{cacc_prop}, we obtain,
for all $0<r<R \leq 1$,
\begin{align*}
\|v^{\frac{p}{2}}\|^2_{W^{1,2}(B_{r})}  \leq  
C_1\bigg[\frac{1}{(R-r)^2} |B_{R}|^{\frac{1}{\theta'}} \|v^{\frac{p}{2}}\|_{L^{2\theta}(B_R)}^{{2}} + \|f\mathbf{1}_{u>k}\|_{L^p(B_R)}^p\bigg].
\end{align*}
The Sobolev embedding theorem now yields \eqref{ineq_h_n}.

Set $$R_m:=\frac{1}{2}+\frac{1}{2^{m+1}}, \quad m \geq 0,$$
so $B^m :=B_{R_m}$ is a decreasing sequence of balls converging to the ball of radius $\frac{1}{2}$. 
For the purposes of this proof, we can estimate $|B_{R}|^{\frac{1}{\theta'}} \leq 1$, which will make the iterations below converge slower, but will not change the sought estimate \eqref{dg_ineq_u}. Estimate \eqref{ineq_h_n} gives us 
\begin{align}
\|v\|^{p}_{L^{\frac{pd}{d-2}}(B_{r})} & \leq C_1 2^{2m}
\|v\|_{L^{p\theta}(B^{m})}^{p}+ C_2  \|f\mathbf{1}_{u>k}\|_{L^p(B^m)}^p \notag \\
& \leq C_1 2^{2m}
\|v\|_{L^{p\theta}(B^{m})}^{p} + C_2 H|B^{m} \cap \{v>0\}|^{\frac{1}{\theta}},
\label{theta_d_d-2_n}
\end{align}
where
$
H:=\langle |f|^{p\theta'}\mathbf{1}_{B^0}\rangle^{\frac{1}{\theta'}}
$
($B^0=B_1$ is the ball of radius 1).
On the other hand, by H\"{o}lder's inequality,
$$
 \|v\|_{L^{p\theta}(B^{m+1})}^{p\theta} \leq  \|v\|_{L^{\frac{pd}{d-2}}(B^{m+1})}^{p\theta} \biggl(|B^{m} \cap \{v>0\}| \biggr)^{1-\frac{(d-2)\theta}{d}}.
$$
Applying \eqref{theta_d_d-2_n} in the first multiple in the RHS, we obtain
$$
 \|v\|_{L^{p\theta}(B^{m+1})}^{p\theta} \leq C\biggl( 2^{2\theta m}\|v\|_{L^{p\theta}(B^{m})}^{p\theta}+ H^{\theta}|B^{m} \cap \{v>0\}| \biggr) \biggl(|B^{m} \cap \{v>0\}| \biggr)^{1-\frac{(d-2)\theta}{d}}.
$$

Put $$v_m:=(u-k_m)_+, \quad k_m:=\xi(1-2^{-m}) \uparrow \xi,$$ where constant $\xi>0$ will be chosen later.
Using $2^{2\theta m} \leq 2^{p\theta m}$ and dividing by $\xi^{p\theta}$, we obtain
\begin{align*}
&\frac{1}{\xi^{p\theta}}  \|v_{m+1}\|_{L^{p\theta}(B^{m+1})}^{p\theta} \\
&\leq C\biggl( \frac{2^{p\theta m}}{\xi^{p\theta}}\|v_{m+1}\|_{L^{p\theta}(B^{m})}^{p\theta}  + \frac{1}{\xi^{p\theta}} H^{\theta}|B^{m} \cap \{u>k_{m+1}\}|\biggr) \bigl(|B^{m} \cap \{u>k_{m+1}\}| \bigr)^{1-\frac{(d-2)\theta}{d}}.
\end{align*}
From now on, we require that constant $\xi$ satisfies $\xi^p \geq H$, so
\begin{align}
\label{dg_ineq7}
&\frac{1}{\xi^{p\theta}}  \|v_{m+1}\|_{L^{p\theta}(B^{m+1})}^{p\theta} \\
&\leq C\biggl( \frac{2^{p\theta m}}{\xi^{p\theta}}\|v_{m+1}\|_{L^{p\theta}(B^{m})}^{p\theta}  + |B^{m} \cap \{u>k_{m+1}\}|\biggr) \bigl(|B^{m} \cap \{u>k_{m+1}\}| \bigr)^{1-\frac{(d-2)\theta}{d}}. \notag
\end{align}
Now, 
\begin{align*}
|B^{m} \cap \{u>k_{m+1}\}| & = \big|B^{m} \cap \bigg\{\bigg(\frac{(u-k_m)_+}{k_{m+1}-k_m}\bigg)^{2\theta}>1\bigg\}\big| \\
&\leq (k_{m+1}-k_m)^{-p\theta} \langle v^{p\theta}_m \mathbf{1}_{B^{m}} \rangle  = \xi^{-p\theta}2^{p\theta(m+1)} \|v_m\|_{L^{p\theta}(B^{m})}^{p\theta},
\end{align*}
so using $\|v_{m+1}\|_{L^{p\theta}(B^{m})} \leq \|v_{m}\|_{L^{p\theta}(B^{m})}$ in \eqref{dg_ineq7} and applying the previous inequality, we obtain
$$
\frac{1}{\xi^{p\theta}} \|v_{m+1}\|_{L^{p\theta}(B^{m+1})}^{p\theta} \leq C 2^{p\theta m(2-\frac{\theta}{\theta_0})} \biggl(\frac{1}{\xi^{p\theta}}\|v_m\|_{L^{p\theta}(B^{m})}^{p\theta}\biggr)^{2-\frac{(d-2)\theta}{d}}.
$$

Denote
$
z_m:=\frac{1}{\xi^{p\theta}}\|v_m\|_{L^{p\theta}(B^{m})}^{p\theta}.
$
Then
$$
z_{m+1} \leq C\gamma^m z_m^{1+\alpha}, \quad m=0,1,2,\dots, \quad \alpha:=1-\frac{(d-2)\theta}{d},\;\; \gamma:=2^{p\theta (2-\frac{(d-2)\theta}{d})}
$$
and $z_0 = \frac{1}{\xi^{p\theta}}\langle u_+^{p\theta}\mathbf{1}_{B^0} \rangle \leq C^{-\frac{1}{\alpha}}\gamma^{-\frac{1}{\alpha^2}}$ (recall: $B^0:=B_{R_0} \equiv B_1$)
provided that we fix ${\xi}$ by $$\xi^{p\theta}:=C^{\frac{1}{\alpha}}\gamma^{\frac{1}{\alpha^2}}\langle u_+^{p\theta}\mathbf{1}_{B^0}\rangle + H^\theta.$$
Hence, by Lemma \ref{dg_lemma}, $z_m \rightarrow 0$ as $m \rightarrow \infty$. It follows that
$
\sup_{B_{1/2}}u_+ \leq \xi,
$
and the claimed inequality follows.

\medskip

Step 2.~Next, we bound $\langle |u|^{p\theta}\mathbf{1}_{B_1}\rangle^{\frac{1}{p\theta}}$ using Proposition \ref{elem_lem} with  $s=p\theta$, which allows us to conclude that if $\sigma$ (in the definition of weight $\rho$) is fixed sufficiently small, then for all $x \in \mathbb R^d$
$$
\sup_{B_{\frac{1}{2}}(x)} u_+ \leq K \biggl( \langle |f|^{p\theta}\rho_x\rangle^{\frac{1}{p\theta}} + \big\langle |f|^{p\theta'}\mathbf{1}_{B_1(x)}\big\rangle^{\frac{1}{p\theta'}} \biggr).
$$
This ends the proof of Proposition \ref{sep_thm}. \hfill \qed

\bigskip

\bigskip

\section{Proof of Lemma \ref{grad_lem}}
\label{grad_lem_proof}

It suffices to carry out the proof for $b_n$ and $q_m$, and then use the convergence result of Proposition \ref{thm_conv}. Thus, our goal is to show that
\begin{equation}
\label{unif_grad}
\sup_{n,m}\|\nabla u_{n,m}\|_{2+\varepsilon}<\infty
\end{equation}
for some $\varepsilon>0$ independent of $n,m$.

Put for brevity $b=b_n$, $q=q_m$ and $u=u_{n,m}$, so
$$
\big(\mu -\Delta + (b+q) \cdot \nabla\big)u=f.
$$

Let us fix some $1 < \theta < \frac{d}{d-2}$.

By Lemma \ref{cacc_prop} (with $p=2$ there, which is admissible since $\delta<1$),
the function $v:=(u-k)_+$ ($k \in \mathbb R$) satisfies Caccioppoli's inequality: for all $x \in \mathbb R^d$, $0<r<R{\leq1}$,
\begin{equation}
\label{cacc_ineq_g}
\|\nabla v\|_{L^2(B_r(x))}^2 \leq 
 \frac{K_1}{(R-r)^2} |B_{R}|^{\frac{1}{\theta'}} (1+\|Q\|_{{\rm BMO}}^2) \|v\|_{L^{2\theta}(B_R)}^2 + K_2 \| (f - \mu u) \mathbf{1}_{v>0} \|_{L^2(B_R)}^2,
\end{equation}
for constants $K_1$, $K_2$ independent of $k$, $r$, $R$ and $n$, $m$. 
(There is some abuse of notation: our $R$ here is not the radius of a fixed large ball in the proof of Theorem \ref{thm2}(\textit{ii}), but this should not cause a confusion.) 
We will obtain the sought bound \eqref{unif_grad} by applying a corollary of this Caccioppoli's inequality in the Gehring-Giaquinta-Modica lemma:

\begin{lemma}
\label{gehring_prop}
Assume that there exist constants $K \geq 1$, $1<\nu<\infty$ such that, for given $0\leq g \in L_{\loc}^{{\nu}}$, $0 \leq h \in L_{\loc}^\nu \cap L^\infty$ we have, for all $x \in \mathbb R^d$,
$$
\biggl(\frac{1}{|B_R|}\langle g^\nu \mathbf{1}_{B_R(x)}\rangle \biggr)^{\frac{1}{\nu}} \leq \frac{K}{|B_{2R}|}\langle g \mathbf{1}_{B_{2R}(x)}\rangle+ \biggl(\frac{1}{|B_{2R}|}\langle h^\nu \mathbf{1}_{B_{2R}(x)}\rangle \biggr)^{\frac{1}{\nu}}
$$
for all $0<R<\frac{1}{2}$. Then $g \in L_{\loc}^s$ for some $s>\nu$ and, for all $x \in \mathbb R^d$,
$$
\biggl(\frac{1}{|B_R|}\langle g^s \mathbf{1}_{B_R(x)}\rangle \biggr)^{\frac{1}{s}} \leq C_1\biggl(\frac{1}{|B_{2R}|}\langle g^\nu \mathbf{1}_{B_{2R}(x)}\rangle \biggr)^{\frac{1}{\nu}} + C_2\biggl(\frac{1}{|B_{2R}|}\langle h^s \mathbf{1}_{B_{2R}(x)}\rangle \biggr)^{\frac{1}{s}}.
$$
\end{lemma}

\begin{remark}
\label{gehring_expl_est}
The authors of \cite{KrSt} proved Lemma \ref{gehring_prop} with explicit constants independent of the dimension $d$.
\end{remark}

We are in a position to prove \eqref{unif_grad}. Put, for brevity, $x=0$.

\smallskip

Step 1. Set $(u_n)_{B_{2R}}:=\frac{1}{|B_{2R}|}\langle u_n \mathbf{1}_{B_{2R}}\rangle$. Applying \eqref{cacc_ineq_g} to the positive and the negative parts of $u_n-(u_n)_{B_{2R}}$, we obtain
\begin{equation}
\label{u_gehr}
\langle |\nabla u_n|^2 \mathbf{1}_{B_{R}}\rangle \leq \frac{K_1}{|B_{2R}|^{\frac{2}{d}}} |B_{2R}|^\frac{1}{\theta'} (1+\|Q\|_{{\rm BMO}}^2)\langle|u_n-(u_n)_{B_{2R}}|^{2\theta} \mathbf{1}_{B_{2R}}\rangle^\frac{1}{\theta} + K_2\langle |f-\mu u_n|^2 \mathbf{1}_{B_{2R}}\rangle, \quad 0<R<\frac{1}{2}.
\end{equation}
By the Sobolev-Poincar\'{e} inequality,
\begin{equation}
\label{sob_p}
\biggl( \frac{1}{|B_{2R}|} \langle |u_n-(u_n)_{B_{2R}}|^{2\theta} \mathbf{1}_{B_{2R}}\rangle\biggr)^{\frac{1}{2\theta}} \leq  C |B_R|^{\frac{1}{d}} \biggl( \frac{1}{|B_{2R}|}\langle |\nabla u_n|^{\frac{2\theta d}{d+2\theta}} \mathbf{1}_{B_{2R}}\rangle\biggr)^{\frac{d+2\theta}{2\theta d}},
\end{equation}
i.e.
$$
\langle |u_n-(u_n)_{B_{2R}}|^{2\theta} \mathbf{1}_{B_{2R}}\rangle^\frac{1}{\theta} \leq C^2|B_R|^{\frac{2}{d}+\frac{1}{\theta}} \biggl( \frac{1}{|B_{2R}|}\langle |\nabla u_n|^{\frac{2\theta d}{d+2\theta}} \mathbf{1}_{B_{2R}}\rangle\biggr)^{\frac{d+2\theta}{\theta d}}.
$$
Plug in the above estimate in {\eqref{u_gehr}}, and divide both sides by $|B_{R}|$; then, for appropriate constants $C_1$ and $c$, where $C_1$ depends on $\|Q\|_{{\rm BMO}}$, we obtain 
\begin{align*}
\frac{1}{|B_{R}|}\langle |\nabla u_n|^2 \mathbf{1}_{B_{R}}\rangle \leq C_1 \biggl( \frac{1}{|B_{2R}|}\langle |\nabla u_n|^{\frac{2\theta d}{d+2\theta}} \mathbf{1}_{B_{2R}}\rangle\biggr)^{\frac{d+2\theta}{\theta d}} + \frac{c}{|B_{2R}|}\langle |f-\mu u_n|^2 \mathbf{1}_{B_{2R}}\rangle. 
\end{align*}
Then the condition of the Gehring-Giaquinta-Modica lemma is verified with $g=|\nabla u_n|^{\frac{2\theta d}{d+2\theta}}$, $g^\nu=|\nabla u_n|^2$ (so $\nu=\frac{d+2\theta}{\theta d}$) and $h=(c^\frac{1}{2} |f-\mu u_n|)^{\frac{2\theta d}{d+2\theta}}$, $h^\nu=c |f-\mu u_n|^2$. Hence there exists $s>\frac{d+2\theta}{\theta d}$ such that
$$
\biggl(\frac{1}{|B_R|}\langle |\nabla u_n|^{s\frac{2\theta d}{d+2\theta}} \mathbf{1}_{B_R}\rangle \biggr)^{\frac{1}{s}} \leq C_1\biggl(\frac{1}{|B_{2R}|}  \langle |\nabla u_n|^2 \mathbf{1}_{B_{2R}}\rangle \biggr)^{\frac{\theta d}{d+2\theta}} + C_2\biggl(\frac{1}{|B_{2R}|}\langle |f-\mu u_n|^{s\frac{2\theta d}{d+2\theta}} \mathbf{1}_{B_{2R}}\rangle \biggr)^{\frac{1}{s}},
$$
where all constants are independent of $n$, or
$$
\frac{1}{|B_R|}\langle |\nabla u_n|^{s\frac{2\theta d}{d+2\theta}} \mathbf{1}_{B_R}\rangle \leq C_1'\biggl(\frac{1}{|B_{2R}|}  \langle |\nabla u_n|^2 \mathbf{1}_{B_{2R}}\rangle \biggr)^{s\frac{\theta d}{d+2\theta}} + C_2'\frac{1}{|B_{2R}|}\langle |f-\mu u_n|^{s\frac{2\theta d}{d+2\theta}} \mathbf{1}_{B_{2R}}\rangle.
$$

{
Fix some $R$, say, $R=\frac14$. We consider an equally spaced grid $c\mathbb Z^d$ in $\mathbb R^d$, with $c>0$ chosen sufficiently small depending only on $d$, so that the balls $B_{1/4}(x)$ centered at the nodes cover $\mathbb R^d$ and the balls $B_{1/2}(x)$ have bounded overlap. We apply the previous estimate on each ball and then sum up.
}
We obtain a global estimate
\begin{equation}
\label{glob_est_int}
\|\nabla u_n\|^{s\frac{2\theta d}{d+2\theta}}_{s\frac{2\theta d}{d+2\theta}} \leq C_3\sum_{x \in c\mathbb Z^d}\biggl(\frac{1}{|B_{{1/2}}|}  \langle |\nabla u_n|^2 \mathbf{1}_{B_{{1/2}}(x)}\rangle \biggr)^{s\frac{\theta d}{d+2\theta}} + C_4\|f-\mu u_n\|^{s\frac{2\theta d}{d+2\theta}}_{s\frac{2\theta d}{d+2\theta}}.
\end{equation}
To deal with the first term in the right-hand side, we split the grid into two parts: $I:=\{x \in c\mathbb Z^d \mid \frac{1}{|B_{{1/2}}|}  \langle |\nabla u_n|^2 \mathbf{1}_{B_{{1/2}}(x)}\rangle>1\}$ and its complement $I^c$. For the nodes in the complement we have, taking into account that 
$s\frac{\theta d}{d+2\theta}>1$,
\begin{align*}
\sum_{x \in I^c}\biggl(\frac{1}{|B_{{1/2}}|}  \langle |\nabla u_n|^2 \mathbf{1}_{B_{{1/2}}(x)}\rangle \biggr)^{s\frac{\theta d}{d+2\theta}} & \leq \sum_{x \in c\mathbb Z^d}\frac{1}{|B_{{1/2}}|}  \langle |\nabla u_n|^2 \mathbf{1}_{B_{{1/2}}(x)}\rangle \\
& \leq C_5 \langle |\nabla u_n|^2 \rangle.
\end{align*}
In turn, there are only finitely many nodes in $I$. In fact, the cardinality of $I$ can be estimated in terms of $\langle |\nabla u_n|^2\rangle$:
$$
|I| < \sum_{x \in I} \frac{1}{|B_{{1/2}}|}  \langle |\nabla u_n|^2 \mathbf{1}_{B_{{1/2}}(x)}\rangle \leq C_6 \langle |\nabla u_n|^2\rangle. 
$$
So, 
\begin{align*}
\sum_{x \in I}\biggl(\frac{1}{|B_{{1/2}}|}  \langle |\nabla u_n|^2 \mathbf{1}_{B_{{1/2}}(x)}\rangle \biggr)^{s\frac{\theta d}{d+2\theta}} & \leq \sum_{x \in I}\biggl(\frac{1}{|B_{{1/2}}|}  \langle |\nabla u_n|^2\rangle \biggr)^{s\frac{\theta d}{d+2\theta}} \\
& \leq C_7 \langle |\nabla u_n|^2\rangle^{1+s\frac{\theta d}{d+2\theta}}.
\end{align*}
We arrive at a global estimate
\begin{equation*}
\|\nabla u_n\|^{s\frac{2\theta d}{d+2\theta}}_{s\frac{2\theta d}{d+2\theta}} \leq C_8 \biggl(\|\nabla u_n\|_2^2 + \|\nabla u_n\|_2^{2+s\frac{2\theta d}{d+2\theta}} \biggr) + C_4\|f-\mu u_n\|^{s\frac{2\theta d}{d+2\theta}}_{s\frac{2\theta d}{d+2\theta}}.
\end{equation*}

\smallskip

{
Step 2. Let us show that in the right-hand side of the estimate of Step 1 we have $\sup_n\|\nabla u_n\|_2^2<\infty$. To this end, we multiply $(\mu-\Delta+(b_n+q_m)\cdot\nabla)u_n=f$ by $u_n$ and integrate. Since ${\rm div\,}q_m=0$, we obtain
$$
\mu\|u_n\|_2^2+\|\nabla u_n\|_2^2+\langle b_n\cdot\nabla u_n,u_n\rangle=\langle f,u_n\rangle.
$$
By $b_n\in\mathbf F_\delta$, uniformly in $n$, and the quadratic inequality, for every $\varepsilon>0$,
\begin{align*}
|\langle b_n\cdot\nabla u_n,u_n\rangle|
&\leq\|b_nu_n\|_2\|\nabla u_n\|_2\\
&\leq\sqrt\delta\|\nabla u_n\|_2^2+\sqrt{c_\delta}\|u_n\|_2\|\nabla u_n\|_2\\
&\leq(\sqrt\delta+\varepsilon)\|\nabla u_n\|_2^2+\frac{c_\delta}{4\varepsilon}\|u_n\|_2^2.
\end{align*}
We fix $0<\varepsilon<1-\sqrt\delta$, which is possible since $\delta<1$. It follows that
\begin{equation}
\label{e}
\left(\mu-\frac{c_\delta}{4\varepsilon}\right)\|u_n\|_2^2
+(1-\sqrt\delta-\varepsilon)\|\nabla u_n\|_2^2
\leq\langle f,u_n\rangle.
\end{equation}
Applying the quadratic inequality in the right-hand side, we obtain
$$
\left(\mu-\frac{c_\delta}{4\varepsilon}-\frac12\right)\|u_n\|_2^2
+(1-\sqrt\delta-\varepsilon)\|\nabla u_n\|_2^2
\leq\frac12\|f\|_2^2.
$$
Hence $\sup_n\big(\|u_n\|_2^2+\|\nabla u_n\|_2^2\big)<\infty$ for $\mu\geq\mu_0:=\frac{c_\delta}{4\varepsilon}+1$, uniformly also in $m$.
}

\medskip

Step 3.~Next, $\|u_n\|_2 \leq C\|f\|_2$ and a priori bound $\|u_n\|_\infty \leq \|f\|_\infty$ yield $\sup_n\|u_n\|_{s\frac{2\theta d}{d+2\theta}}<\infty$. Hence $\sup_n\|f-\mu u_n\|^2_{s\frac{2\theta d}{d+2\theta}}<\infty$.

\medskip

Steps 1-3 give us the sought gradient bound $\sup_{n}\|\nabla u_n\|_{s\frac{2\theta d}{d+2\theta}}<\infty$, which thus ends the proof. \hfill \qed

\bigskip

\section{Proof of Theorem \ref{thm2_a}}

\label{proof_thm2a_sect}

\subsection*{Proof of ({\textit{i}})} We modify the proof of Theorem \ref{thm2}, i.e.\,we verify conditions of Trotter's approximation theorem, but now for the Feller realizations in $C_\infty$ of
$$
\Lambda(a_n,b_n,q_m)=-a_n\cdot \nabla^2 + (b_n + q_m) \cdot \nabla.
$$
Condition $1^\circ$) of Trotter's theorem is obvious.

Let us verify conditions $2^\circ$) and $3^\circ$).
To this end, we note that Propositions \ref{thm_holder}, \ref{sep_thm} and \ref{emb_thm}, i.e.\,a priori H\"{o}lder continuity of solutions, separation and embedding properties, are still valid for operators $\Lambda(a_n,b_n,q_m)$ (in fact, under more general condition $\delta<4\xi^2$) since we can put these operators in divergence form
$$
\Lambda(a_n,b_n,q_m)=-\nabla \cdot a_n \cdot \nabla + (\tilde{b}_{n} + q_m) \cdot \nabla, \quad \tilde{b}_{n}=\nabla a_n + b_n \in \mathbf{F}_\delta,
$$
so De Giorgi's method applies. In the analogue of Proposition \ref{emb_thm}, the form-bounded part of the right-hand side is initially $\tilde b_n^i f$. We will explain below how to handle $b_n^i f$ as well.

\begin{remark} Condition $\delta<4\xi^2$ is seen from the following calculation, which we have to repeat several times (also, with the cutoff function $\eta$) when extending Propositions \ref{emb_thm}, \ref{thm_holder}, \ref{sep_thm} to include matrix fields $a_n$. We multiply elliptic equation $(\mu-\nabla \cdot a_n \cdot \nabla + (\tilde{b}_n+q_m)\cdot \nabla)u=0$ by $u^{p-1} \geq 0$ and integrate by parts, obtaining, after taking into account ${\rm div\,}q_m=0$,
$$
\mu\langle u^p\rangle + \frac{4(p-1)}{p^2}\langle a_n \cdot \nabla u^{\frac{p}{2}}, \nabla u^{\frac{p}{2}}\rangle + \frac{2}{p}\langle \tilde{b}_{n} \cdot \nabla u^{\frac{p}{2}},u^{\frac{p}{2}}\rangle=0.
$$
Since $a_n \in H_{\xi}$,
$$
\mu\langle u^p\rangle + \frac{4(p-1)}{p^2}\xi \langle |\nabla u^{\frac{p}{2}}|^2 \rangle + \frac{2}{p}\langle \tilde{b}_{n} \cdot \nabla u^{\frac{p}{2}},u^{\frac{p}{2}}\rangle\leq 0.
$$
Applying the quadratic inequality in the last term, we arrive at
$$
\mu\langle u^p\rangle + \frac{4(p-1)}{p^2}\xi \langle |\nabla u^{\frac{p}{2}}|^2 \rangle \leq \frac{2}{p}\biggl(\alpha \langle |\tilde{b}_{n}|^2,u^p \rangle + \frac{1}{4\alpha} \langle |\nabla u^{\frac{p}{2}}|^2 \rangle\biggr).
$$
Now, using $\tilde{b}_{n} \in \mathbf{F}_\delta$ and selecting $\alpha=\frac{1}{2\sqrt{\delta}}$ (we can take $\delta>0$ without loss of generality), we obtain
\begin{equation*}
\biggl(\mu-\frac{c_\delta}{p\sqrt\delta}\biggr)\langle u^p\rangle + \biggl[\frac{4(p-1)}{p^2}\xi-\frac{2}{p}\sqrt{\delta}\biggr]\langle |\nabla u^{\frac{p}{2}}|^2 \rangle \leq 0.
\end{equation*}
So, $\delta<4\xi^2$ is exactly the condition that ensures that $\frac{4(p-1)}{p^2}\xi-\frac{2}{p}\sqrt{\delta}>0$ for some finite $p \geq 2$ and hence, for sufficiently large $\mu$, gives us an energy inequality; that is, we need $p>\frac{2}{2-\xi^{-1}\sqrt{\delta}}$.
\end{remark}

\medskip

Proposition \ref{thm_conv} is replaced by a simpler convergence result
\begin{equation}
\label{conv_l2}
u:=L^2_{\loc}\mbox{-}\lim_n\lim_{m}u_{n,m},
\end{equation}
where $u_{n,m}=(\mu+\Lambda(a_n,b_n,q_m))^{-1}f$, $f \in C_c^\infty$. (As is explained in the proof of Theorem \ref{thm2}(\textit{i}), we need convergence in \textit{some} topology to establish the approximation uniqueness.) Let us prove \eqref{conv_l2}. Due to our more restrictive assumption $\delta<\xi^2$ we can work in $L^2$ rather than $L^p$. The standard energy inequality argument (cf.\,Section \ref{aux_sect}) gives
$$
(\mu-\mu_0)\|u_{n,m}\|_2^2+C_1\|\nabla u_{n,m}\|_2^2\leq \|f\|_2\|u_{n,m}\|_2,
$$
where $C_1>0$ and $\mu_0$ are independent of $n$, $m$. Here we use $\xi-\sqrt\delta>0$ and the cancellation of the term containing $q_m$. Fix $\mu>\mu_0$.

For each fixed $n$, the weak compactness argument in $W^{1,2}$ and the Rellich-Kondrashov theorem give a subsequential limit $u_n$ in $L^2_{\loc}$. Since $Q_m\rightarrow Q$ in $[L^2_{\loc}]^{d\times d}$, this limit satisfies
\begin{equation}
\label{id_temp}
\mu \langle u_n,\varphi\rangle+\langle a_n\cdot \nabla u_n,\nabla \varphi\rangle + \langle (\tilde b_n+q)\cdot \nabla u_n,\varphi\rangle=\langle f,\varphi\rangle, \quad \varphi \in C_c^\infty.
\end{equation}
The form-boundedness of $\tilde b_n$ and the compensated compactness estimate of Proposition \ref{cc_lem} allow us to extend \eqref{id_temp} to test functions $\varphi \in W^{1,2}$, i.e.\,$u_n$ is the standard weak solution to the elliptic equation $(\mu-\nabla \cdot a_n\cdot \nabla + (\tilde b_n+q)\cdot \nabla)u_n=f$ in $L^2$. The same energy inequality gives uniqueness, so the whole sequence converges: $u_n=L^2_{\loc}\mbox{-}\lim_m u_{n,m}$.

Now, the convergence $\tilde b_n\rightarrow\tilde b:=\nabla a+b$ in $[L^2_{\loc}]^d$ and the convergence $a_n \rightarrow a$ a.e.\,on $\mathbb R^d$, {both included in the definition of $[a,b]$ in Section \ref{diff_coeff_sect},} applied in the same weak compactness argument, give us, via the uniqueness of the weak solution in $L^2$, the sought convergence $u_n \rightarrow u$ in $L^2_{\loc}$. Here $u$ satisfies \eqref{id_temp} with $a$, $\tilde b$ instead of $a_n$, $\tilde b_n$, and this identity extends in the same way to $\varphi \in W^{1,2}$. Thus, $u$ is the standard weak solution to $(\mu-\nabla \cdot a\cdot \nabla + (\tilde b+q)\cdot \nabla)u=f$ in $L^2$.
In the case $q=0$, this is essentially how the divergence form operator $-\nabla \cdot a \cdot \nabla + b \cdot \nabla$ with form-bounded $b$ was treated in \cite[Theorem 4.3]{KiS_theory}.

Now, armed with the above analogues of Propositions \ref{thm_holder}, \ref{thm_conv}, \ref{sep_thm} and \ref{emb_thm} for $\Lambda(a_n,b_n,q_m)$, we verify condition $2^\circ$) of Trotter's theorem in the same way as in the proof of Theorem \ref{thm2}: a priori H\"{o}lder continuity upgrades the convergence to locally uniform convergence, and the separation property allows us to cut the tails uniformly in $n$, $m$. We then extend the convergence from $C_c^\infty$ to $C_\infty$ by condition $1^\circ$).

\medskip

Condition $3^\circ$) requires a comment. Fix $g \in C_c^\infty$. The coefficients of each approximating operator are bounded, so approximation with smooth cutoff functions gives $\mathcal S\subset D(\Lambda(a_n,b_n,q_m))$. Since $(\mu-\Delta)^{-1}g\in\mathcal S$, the resolvent identity yields
\begin{align*}
\mu (\mu+\Lambda(a_n,b_n,q_m))^{-1} g - \mu(\mu-\Delta)^{-1} g & = \mu (\mu+\Lambda(a_n,b_n,q_m))^{-1} (a_n-I) \cdot \nabla^2 (\mu-\Delta)^{-1} g \\
& \quad - \mu (\mu+\Lambda(a_n,b_n,q_m))^{-1} (b_n+q_m) \cdot \nabla (\mu-\Delta)^{-1} g.
\end{align*}
Since $\mu(\mu-\Delta)^{-1}g\rightarrow g$ uniformly on $\mathbb R^d$ as $\mu\rightarrow\infty$, it suffices to show convergence
\begin{align}
\|(\mu+\Lambda(a_n,b_n,q_m))^{-1} (a_n-I) & \cdot \mu (\mu-\Delta)^{-1}\nabla^2 g\|_\infty\rightarrow 0 \label{cc1}\\[2mm]
& \text{and} \notag\\[2mm]
\|(\mu+\Lambda(a_n,b_n,q_m))^{-1} (b_n+q_m) & \cdot \mu (\mu-\Delta)^{-1}\nabla g\|_\infty\rightarrow 0 \label{cc2}\\
& \text{as }\mu\rightarrow\infty\quad\text{uniformly in $n$, $m$}.\notag
\end{align}
\begin{samepage}
To prove \eqref{cc2}, we apply the analogue of Proposition \ref{emb_thm} to $w_{n,m}:=(\mu+\Lambda(a_n,b_n,q_m))^{-1}(b_n^i+q_m^i)f$, where $f:=\mu(\mu-\Delta)^{-1}\nabla_i g$. Here the matrix divergence is understood as in Section \ref{notations_sect}. We use the decomposition
$$
b_n^i+q_m^i=\tilde b_n^i+\sum_{j=1}^d\nabla_j(Q_m^{ji}-a_n^{ji}).
$$
\end{samepage}
In the proof of Proposition \ref{emb_thm}, the anti-symmetry of $Q_m$ is used for the drift term, while the term containing $Q_m^i$ in the right-hand side is estimated only by integration by parts and the quadratic inequality. Thus, this proof gives \eqref{w_est_nm} with $A_m$ replaced by
$$
A_{n,m}:=|(Q_m-a_n)e_i||\nabla f|+(1+|(Q_m-a_n)e_i|)|f|,
$$
where $e_i$ is the $i$-th coordinate vector, and with constants independent of $n$, $m$. The estimates in the proof of \eqref{bd_task} control the terms containing $Q_m^i$. Since $a_n$ are uniformly bounded, the remaining terms are controlled by
$$
\|f\|_s\leq\|\nabla_i g\|_s,\qquad \|\nabla f\|_s\leq\|\nabla\nabla_i g\|_s,\qquad s=p\theta,\;p\theta'.
$$
Consequently, $\sup_{\mu\geq1,n,m}\|A_{n,m}\|_s<\infty$ for both values of $s$, and \eqref{cc2} follows from the negative powers of $\mu-\mu_0$ in \eqref{w_est_nm}.

The proof of \eqref{cc1} is even easier. Indeed, by condition $1^\circ$) and the uniform boundedness of $a_n$, its left-hand side is at most $C\mu^{-1}\|\nabla^2g\|_\infty$, where $C$ is independent of $n$, $m$. Thus, condition $3^\circ$) is verified, and Trotter's theorem gives assertion (\textit{i}).

We also note that the intermediate and limiting semigroups preserve probability. Indeed, choose $\chi_R\in C_c^\infty$, $0\leq\chi_R\leq1$, such that $\chi_R=1$ on $B_R$ and $\chi_R\uparrow1$. The separation property, applied by monotone approximation to $1-\chi_R$, gives, for every compact set $K$ and all sufficiently large $\mu$,
$$
\lim_{R\rightarrow\infty}\sup_{n,m}\sup_{x\in K}(\mu+\Lambda(a_n,b_n,q_m))^{-1}(1-\chi_R)(x)=0.
$$
Here the resolvents are extended to bounded Borel functions by their transition kernels. Since the approximating semigroups preserve probability, $\mu(\mu+\Lambda(a_n,b_n,q_m))^{-1}1=1$. Passing to the limits $m\rightarrow\infty$, $n\rightarrow\infty$ and then $R\rightarrow\infty$ yields $\mu(\mu+\Lambda(a,b,q))^{-1}1=1$, and the same argument applies to each intermediate semigroup. By positivity and contractivity, $e^{-t\Lambda(a,b,q)}1$ is non-increasing in $t$, so this resolvent identity implies $e^{-t\Lambda(a,b,q)}1=1$ for all $t\geq0$.

\subsection*{Proof of ({\textit{ii}})} The approximation uniqueness follows from the characterization of the limit in (\textit{i}) as the unique weak solution in $L^2$. Let us prove the relaxed approximation uniqueness. We write $\tilde b_n=\nabla a_n+b_n$ and $\tilde b=\nabla a+b$, and assume now only that $\tilde b_n\overset{w}{\rightarrow}\tilde b$ in $[L^2_{\loc}]^d$, with all the other conditions on $a_n$, $b_n$ unchanged.

Fix $f\in C_c^\infty$ and $\mu>\mu_0$, where $\mu_0$ is independent of $n$ and is chosen sufficiently large. Put $u_n=(\mu+\Lambda(a_n,b_n,q))^{-1}f$, $u=(\mu+\Lambda(a,b,q))^{-1}f$ and $h_n=u_n-u$. The energy inequality and the $L^\infty$ contraction estimate give
$$
\sup_n\|h_n\|_{W^{1,2}}<\infty,\qquad \sup_n\|h_n\|_\infty\leq2\mu^{-1}\|f\|_\infty.
$$
The difference $h_n$ satisfies
\begin{align*}
\mu\langle h_n,\varphi\rangle+\langle a_n\cdot\nabla h_n,\nabla\varphi\rangle+\langle(\tilde b_n+q)\cdot\nabla h_n,\varphi\rangle
&=\langle(a-a_n)\cdot\nabla u,\nabla\varphi\rangle\\
&\quad+\langle(\tilde b-\tilde b_n)\cdot\nabla u,\varphi\rangle
\end{align*}
for all $\varphi\in W^{1,2}$. Taking $\varphi=h_n$, using the cancellation of the term containing $q$ and applying the form-bound to $\tilde b_n$, we obtain
\begin{align*}
(\mu-\mu_0)\|h_n\|_2^2+C_1\|\nabla h_n\|_2^2
&\leq|\langle(a-a_n)\cdot\nabla u,\nabla h_n\rangle|\\
&\quad+|\langle(\tilde b-\tilde b_n)\cdot\nabla u,h_n\rangle|,
\end{align*}
where $C_1>0$ is independent of $n$. The first term tends to zero since
$$
|\langle(a-a_n)\cdot\nabla u,\nabla h_n\rangle|\leq\|\,|a-a_n||\nabla u|\,\|_2\|\nabla h_n\|_2\rightarrow0
$$
by the Dominated convergence theorem and the uniform energy bound.

Let us consider the second term. The form-bounds of $\tilde b_n$, $\tilde b$ and the uniform bound on $h_n$ in $W^{1,2}$ give
$$
\sup_n\|\,|\tilde b-\tilde b_n|h_n\|_2<\infty.
$$
Hence, for a constant $C$ independent of $n$, $R$,
$$
|\langle\mathbf1_{\mathbb R^d\setminus B_R}(\tilde b-\tilde b_n)\cdot\nabla u,h_n\rangle|\leq C\|\mathbf1_{\mathbb R^d\setminus B_R}\nabla u\|_2\rightarrow0
$$
as $R\rightarrow\infty$, uniformly in $n$.

The Rellich-Kondrashov theorem and a diagonal argument allow us to pass from any subsequence to a further subsequence, denoted by $\{h_n\}$, such that $h_n\rightarrow g$ a.e.\,on $\mathbb R^d$, where $g\in L^\infty$. Inside each $B_R$, by the Dominated convergence theorem and the uniform $L^\infty$ bound,
$$
\|\mathbf1_{B_R}(h_n-g)\nabla u\|_2\rightarrow0.
$$
Since $\tilde b_n$ are bounded in $[L^2(B_R)]^d$ and converge weakly there to $\tilde b$, it follows that
\begin{align*}
\langle\mathbf1_{B_R}(\tilde b-\tilde b_n)\cdot\nabla u,h_n\rangle
&=\langle\mathbf1_{B_R}(\tilde b-\tilde b_n)\cdot\nabla u,g\rangle\\
&\quad+\langle\mathbf1_{B_R}(\tilde b-\tilde b_n)\cdot\nabla u,h_n-g\rangle\rightarrow0.
\end{align*}
Combining this convergence with the uniform estimate outside $B_R$, and then letting $R\rightarrow\infty$, we obtain $h_n\rightarrow0$ in $W^{1,2}$. Since the argument applies to every subsequence, the whole sequence converges. The a priori H\"{o}lder continuity and separation properties now give convergence in $C_\infty$, and Trotter's theorem yields the asserted convergence of the Feller semigroups.

% Finally, suppose that $b_n\overset{w}{\rightarrow}b$ in $[L^2]^d$ and all the other conditions on $a_n$, $b_n$ are unchanged. The uniform boundedness of $a_n$ and their a.e.\,convergence imply $\nabla a_n\rightarrow\nabla a$ in the sense of distributions. Thus, $\tilde b_n\rightarrow\tilde b$ in the sense of distributions. The uniform form-bound gives a uniform bound on $\tilde b_n$ in $[L^2(B_R)]^d$ for every $R<\infty$, so this distributional convergence implies $\tilde b_n\overset{w}{\rightarrow}\tilde b$ in $[L^2_{\loc}]^d$, as required.

\subsection*{Proof of ({\textit{iii}})} By (\textit{i}), the semigroups $e^{-t\Lambda(a_n,b_n,q)}$ and $e^{-t\Lambda(a,b,q)}$ are strongly continuous conservative Feller semigroups on $C_\infty$. Hence they determine strong Markov families $\{\mathbb P_x^n\}$ and $\{\mathbb P_x\}$ on $\mathbf D$, see \cite[Ch.4, Theorem 2.7]{EKu}. We equip $\mathbf D$ with the Skorohod topology and $\mathbf C$ with the topology of uniform convergence on compact time intervals. The strong convergence of the semigroups in (\textit{i}) and \cite[Ch.4, Theorem 2.5]{EKu}, applied first as $m\rightarrow\infty$ for each fixed $n$ and then as $n\rightarrow\infty$, yield
$$
\mathbb P_x^n=w\mbox{-}\mathcal P(\mathbf D)\mbox{-}\lim_m\mathbb P_x^{n,m},\qquad
\mathbb P_x=w\mbox{-}\mathcal P(\mathbf D)\mbox{-}\lim_n\mathbb P_x^n.
$$
Since $\mathbf C$ is closed in $\mathbf D$ and all $\mathbb P_x^{n,m}$ are concentrated on $\mathbf C$, the same is true for $\mathbb P_x^n$ and $\mathbb P_x$. On $\mathbf C$ the Skorohod topology agrees with the topology of uniform convergence on compact time intervals, so these limits also hold in $\mathcal P(\mathbf C)$. Finally, the generator identity gives the asserted martingale property for every $v\in D(\Lambda(a,b,q))$, see \cite[Ch.\,VII, \S 1]{RevuzYor}. This martingale is continuous since $v,\Lambda(a,b,q)v\in C_\infty$.

\subsection*{Proof of ({\textit{iv}})} Fix $T>0$, and let $r$ be as in the statement of (\textit{iv}). Choose
$$
1<r_0<\min\{r,\xi/\sqrt{\delta},2\}, \qquad r_0'=\frac{r_0}{r_0-1}.
$$
This is possible since $\delta<\xi^2$. Since $\nu_0$ is a probability density,
interpolation gives $\nu_0\in L^2\cap L^{2r_0}$.
Put $\tilde b_n=\nabla a_n+b_n$, and let $\nu^{n,m}$ be the probability density
of the approximating process with initial density $\nu_0$. The Fokker--Planck
equation takes the form
$$
\partial_s\nu^{n,m}-\nabla\cdot a_n\cdot\nabla\nu^{n,m}
-{\rm div\,}\big[(\tilde b_n+q_m)\nu^{n,m}\big]=0,
\qquad \nu^{n,m}(0)=\nu_0.
$$
The argument in Corollary \ref{elem_cor_fp}, with the ellipticity constant
$\xi$ included, applies for $2\leq s<2\xi/\sqrt{\delta}$.
Indeed, the coefficient of $\|\nabla(\nu^{n,m})^{s/2}\|_2^2$ before accounting
for $c_\delta$ is
$$
\frac{4(s-1)\xi}{s^2}-\frac{2(s-1)\sqrt{\delta}}s>0.
$$
Therefore, applying this estimate with $s=2$ and $s=2r_0$, and using
Gronwall's inequality to account for $c_\delta$, we obtain
\begin{equation}
\label{nu_bound_a}
\sup_{n,m}\biggl[
\sup_{0\leq s\leq T}\bigl(\|\nu^{n,m}(s)\|_2^2+
\|\nu^{n,m}(s)\|_{2r_0}^{2r_0}\bigr)
+\int_0^T\|\nabla\nu^{n,m}(s)\|_2^2ds\biggr]\leq C_T.
\end{equation}
 The estimate can be extended from smooth compactly supported
initial densities by approximation.

We next provide the analogue of Lemma \ref{A_conv_lem}, i.e.\,establish convergence \eqref{A_cauchy_a} below. 

Fix $1\leq i\leq d$,
and write
$$
c=\tilde b_{n_1}^i-\tilde b_{n_2}^i,\qquad
H=\bigl(Q_{m_1}-Q_{m_2}-a_{n_1}+a_{n_2}\bigr)e_i,\qquad
F=b_{n_1}^i+q_{m_1}^i-b_{n_2}^i-q_{m_2}^i,
$$
where $e_i$ is the vector having $1$ in the $i$-th coordinate and $0$ in all other coordinates. Let us choose $f\in C_c^\infty$ which is
identically $1$ on the fixed ball containing the supports of all the drifts,
and put $K=\sprt f$. Then
\begin{equation}
\label{source_decomp_a}
F=fc+f\,{\rm div\,}H=fc+{\rm div\,}(fH)-H\cdot\nabla f.
\end{equation}
For $0<t\leq T$, let $u$ be the solution to the terminal-value problem
$$
\bigl(\partial_s+\nabla\cdot a_n\cdot\nabla
-(\tilde b_n+q_m)\cdot\nabla\bigr)u=-F,
\qquad 0<s<t,\qquad u(t)=0.
$$
For every fixed $2\leq p<\infty$, the proof of Proposition
\ref{elem_lem_par}, applied after reversing time, gives
\begin{equation}
\label{u_bound_a}
\sup_{0\leq s\leq t}\|u(s)\|_p^p+
\int_0^t\|\nabla |u(s)|^{p/2}\|_2^2ds\leq C_{T,p}
\end{equation}
uniformly in $n,m,n_1,n_2,m_1,m_2$ and in $t\leq T$. 
We will use this inequality with $p=2,4,2r_0'$, as in \eqref{prelim_bd}. 
% (Regarding the proof of \eqref{u_bound_a}, in detail, the vector fields $\tilde b_{n_1}-\tilde b_{n_2}$ are form-bounded with form-bound $4\delta$
% and constant $4c_\delta$, while $\|H\|_{L^p(K)}$ is uniformly bounded
% for every finite $p$ by the definition of $[q]$ and the boundedness of $a_n$.
% Using Cauchy-Schwarz and Young's inequalities, we obtain for
% every $\varepsilon>0$ (write $v=|u|^{p/2}$):
% \begin{align*}
% |\langle fc,u|u|^{p-2}\rangle|
% &\leq \bigl\langle |c|^2|u|^p\bigr\rangle^{1/2}
       % \bigl\langle f^2|u|^{p-2}\bigr\rangle^{1/2}
 % \leq\varepsilon\|\nabla v\|_2^2+C_{\varepsilon,p}\|u\|_p^p+C_{\varepsilon,p},\\
% |\langle {\rm div\,}(fH),u|u|^{p-2}\rangle|
% &\leq \frac{2(p-1)}p\langle |fH||u|^{p/2-1},|\nabla v|\rangle
 % \leq\varepsilon\|\nabla v\|_2^2+C_{\varepsilon,p}\|u\|_p^p+C_{\varepsilon,p},\\
% |\langle H\cdot\nabla f,u|u|^{p-2}\rangle|
% &\leq C_p\|u\|_p^p+C_p.
% \end{align*}
% The coefficient of $\|\nabla v\|_2^2$ given by the principal part and
% the drift is positive for every $p\geq2$, since $\sqrt{\delta}<\xi$.
% Selecting $\varepsilon$ sufficiently small and e.g.\,applying Gronwall's inequality
% yields \eqref{u_bound_a}.) 

Next, let us write $X^{n,m}$ for an approximating process with initial density $\nu_0$.
The Markov property, as in the proof of Lemma \ref{A_conv_lem}, gives
$$
\mathbf E\biggl|\int_0^t F(X_s^{n,m})ds\biggr|^2
=2\int_0^t\langle Fu,\nu^{n,m}\rangle ds.
$$
Using $F=fc+f\,{\rm div\,}H$ and integrating by parts, the integral on the
right, apart from its factor $2$, is
\begin{align*}
\int_0^t\langle fcu,\nu^{n,m}\rangle ds
-\int_0^t\langle H,f\nu^{n,m}\nabla u\rangle ds-\int_0^t\langle H,fu\nabla\nu^{n,m}\rangle ds
-\int_0^t\langle H,u\nu^{n,m}\nabla f\rangle ds.
\end{align*}
H\"older's inequality in space and time, together with
\eqref{nu_bound_a} and \eqref{u_bound_a}, bounds these four terms in absolute
value, respectively, by $C_T\|c\|_{L^2(K)}$, $C_T\|H\|_{L^{2r_0'}(K)}$, $C_T\|H\|_{L^4(K)}$ and $C_T\|H\|_{L^{2r_0'}(K)}$.
It follows that
\begin{align}
\label{A_cauchy_a}
\sup_{n,m}\sup_{0\leq t\leq T}
&\mathbf E\biggl|\int_0^t
\bigl(b_{n_1}+q_{m_1}-b_{n_2}-q_{m_2}\bigr)(X_s^{n,m})ds\biggr|^2\\
&\leq C_T\bigl[
\|\tilde b_{n_1}-\tilde b_{n_2}\|_{L^2(K)}+
\|a_{n_1}-a_{n_2}\|_{L^{2r_0'}(K)}+
\|Q_{m_1}-Q_{m_2}\|_{L^{2r_0'}(K)}\bigr]
\rightarrow 0,\notag
\end{align}
where the convergence follows from $[a,b]$, the boundedness of $a_n$ and its
a.e.\,convergence, and the definition of $[q]$.

Set
$$
\mathbb P_{\nu_0}:=\int_{\mathbb R^d}\mathbb P_x\nu_0(x)dx,
\qquad
\mathbb P_{\nu_0}^{n,m}:=\int_{\mathbb R^d}\mathbb P_x^{n,m}\nu_0(x)dx,
$$
and, for a continuous process $Y$, put
$$
A_t^{n,m}(Y):=\int_0^t(b_n+q_m)(Y_s)ds.
$$
By (\textit{iii}), $\mathbb P_{\nu_0}^{n,m}\rightarrow\mathbb P_{\nu_0}$ (recall: first in $m \rightarrow \infty$ with $n$ fixed, then in $n \rightarrow \infty$). Also, for fixed $n_1,m_1,n_2,m_2$, the function
$$
\omega\mapsto\left|A_t^{n_1,m_1}(\omega)-A_t^{n_2,m_2}(\omega)\right|^2
$$
is bounded and continuous on $\mathbf C$. Hence \eqref{A_cauchy_a} remains valid if we take instead expectation with respect to the limiting probability measure $\mathbb P_{\nu_0}$. Taking $n_1=n_2$ first provides convergence of $A_t^{n,m}$ in $m$ for each fixed $n$. Letting $n_1,n_2\rightarrow\infty$ then shows that, for any such process $X$, the limit
$$
A_t=L^2(\Omega')\mbox{-}\lim_n\lim_m A_t^{n,m}(X)
$$
exists for every $t \geq 0$.

Let us choose $n_k,m_k\rightarrow\infty$ so that
$\mathbb P_{\nu_0}^{n_k,m_k}\rightarrow\mathbb P_{\nu_0}$ weakly on $\mathbf C$, and write $X^k=X^{n_k,m_k}$. On the original probability space define
$$
M_t^k:=X_t^k-X_0^k+A_t^{n_k,m_k}(X^k)
=\sqrt{2}\int_0^t\sigma_{n_k}(X_s^k)dB_s.
$$
These are continuous martingales, and
$$
\langle M^{k,i},M^{k,j}\rangle_t
=2\int_0^t a_{n_k}^{ij}(X_s^k)ds.
$$
The uniform boundedness of $a_{n_k}$ and the Burkholder-Davis-Gundy inequality give
$$
\mathbf E|M_t^k-M_s^k|^4\leq C|t-s|^2,
\qquad
\sup_k\mathbf E\sup_{s\leq T}|M_s^k|^4\leq C_T.
$$
Thus the laws of $(X^k,M^k)$ are tight on $\mathbf C\times\mathbf C$. Passing to a subsequence and applying the Skorohod representation theorem, we obtain continuous processes, denoted in the same way, on a probability space $(\Omega',\mathcal F',\mathbf P')$, such that
\begin{equation}
\label{conv_ev}
(X^k,M^k)\rightarrow(X,M)\quad\mathbf P'\text{-a.s.\,locally uniformly in time}.
\end{equation}
Here $X$ has law $\mathbb P_{\nu_0}$ and $X_0$ has density $\nu_0$.
For fixed $n,m$, boundedness and continuity of $b_n+q_m$ imply
$A_t^{n,m}(X^k)\rightarrow A_t^{n,m}(X)$ in $L^2(\Omega')$.
We represent
\begin{align*}
A_t^{n_k,m_k}(X^k)-A_t
=\big[A_t^{n_k,m_k}(X^k)-A_t^{n,m}(X^k)\big]
+\big[A_t^{n,m}(X^k)-A_t^{n,m}(X)\big]+\big[A_t^{n,m}(X)-A_t\big].
\end{align*}
Using \eqref{A_cauchy_a} in the first term, and then letting $k\rightarrow\infty$ and $n,m\rightarrow\infty$, we obtain
$$
A_t^{n_k,m_k}(X^k)\rightarrow A_t\quad\text{in }L^2(\Omega').
$$
Therefore, \eqref{conv_ev} yields
$$
M_t=X_t-X_0+A_t\quad\mathbf P'\text{-a.s.}
$$
Thus $A$ has the continuous version $M-X+X_0$, which we use from now on.

\medskip

It remains to describe martingale $M$. 

\medskip

Step 1.~Let us first establish convergence 
\begin{equation}
\label{bracket_conv_a}
\mathbf E'\int_0^T
|a_{n_k}(X_s^k)-a(X_s)|ds\longrightarrow 0.
\end{equation}
Indeed, put
$p=\frac{2r_0}{2r_0-1}$. By \eqref{nu_bound_a}, for every
$h\in C_c^\infty$ and $0\leq s\leq T$,
$$
|\mathbf E'h(X_s^k)|
=\bigl|\langle h,\nu^{n_k,m_k}(s)\rangle\bigr|
\leq C_T\|h\|_p.
$$
By \eqref{conv_ev}, the left-hand side converges to
$|\mathbf E'h(X_s)|$. Thus, by duality, the distribution of $X_s$
has a density in $L^{2r_0}$ whose norm is bounded by $C_T$,
uniformly in $s\leq T$. Applying H\"older's inequality to these
densities and integrating in time, we obtain the following bound on expected occupation times:
$$
\sup_k\mathbf E'\int_0^T|h(X_s^k)|ds,\;\;\mathbf E'\int_0^T|h(X_s)|ds
\leq C_T\|h\|_p
$$
for every Borel function $h\in L^p$.

We choose a bounded symmetric Borel representative of $a$ satisfying
$\xi I\leq a(x)\leq\xi^{-1}I$ everywhere. The previous estimate,
applied to $h$ being the indicator of a Lebesgue measure zero set, shows that changing
$a$ on such a set does not affect its time integrals along $X^k$
or $X$.

We will also need a variant of the above occupation estimate for bounded functions (in fact, diffusion coefficients $a$) that are, naturally, not in $L^p$. So, we will localize it. 
Namely, the tightness of the laws of
$X^k$ on $\mathbf C$ yields
$$
\varepsilon_R:=
\sup_k\mathbf P'\biggl(\sup_{0 \leq s\leq T}|X_s^k|>R\biggr)
+\mathbf P'\biggl(\sup_{0 \leq s\leq T}|X_s|>R\biggr)
\longrightarrow0
\quad\text{as }R\rightarrow\infty.
$$
Since $a_{n_k}$ and $a$ are uniformly bounded, we can split the time
integral according to whether $X_s^k\in B_R$ to obtain a localized occupation estimate:
\begin{align*}
\mathbf E'\int_0^T
|a_{n_k}(X_s^k)-a(X_s^k)|ds
&\leq C_T\|a_{n_k}-a\|_{L^p(B_R)}
+C_T\varepsilon_R.
\end{align*}

Now, we note that, for each fixed $R$, the first term tends to zero by the a.e.\
convergence of $a_{n_k}$ and the Dominated convergence theorem.
Taking first $k\rightarrow\infty$ and then $R\rightarrow\infty$,
we obtain
\begin{equation}
\label{conv_aa8}
\mathbf E'\int_0^T
|a_{n_k}(X_s^k)-a(X_s^k)|ds\longrightarrow0.
\end{equation}
Next, let us compare $a(X_s^k)$ and $a(X_s)$. Let $a^{(l)}$ be smooth
mollifications of $a$. These matrices are uniformly bounded and
converge to $a$ in $L^p(B_R)$ for every fixed $R$. We write
\begin{align*}
a(X_s^k)-a(X_s)
= &\bigl(a-a^{(l)}\bigr)(X_s^k)
+\bigl(a^{(l)}(X_s^k)-a^{(l)}(X_s)\bigr)+\bigl(a^{(l)}-a\bigr)(X_s).
\end{align*}
Applying the occupation estimate on $B_R$ to the first and third
terms, and using boundedness on its complement, yields
\begin{align*}
\mathbf E'\int_0^T|a(X_s^k)-a(X_s)|ds
\leq & C_T\|a-a^{(l)}\|_{L^p(B_R)}
+C_T\varepsilon_R\\
&+\mathbf E'\int_0^T
|a^{(l)}(X_s^k)-a^{(l)}(X_s)|ds.
\end{align*}
For fixed $l$, the last term tends to zero by \eqref{conv_ev},
continuity of $a^{(l)}$ and the Dominated convergence theorem.
Taking successively $k\rightarrow\infty$, $l\rightarrow\infty$
and $R\rightarrow\infty$, we conclude that
$$
\mathbf E'\int_0^T|a(X_s^k)-a(X_s)|ds\longrightarrow0.
$$
Together with convergence \eqref{conv_aa8}, this yields convergence \eqref{bracket_conv_a}.

\medskip

Step 2.~We now pass to the limit in the martingale problem. On the
original probability space, $M^k$ and
$$
M_t^{k,i}M_t^{k,j}
-2\int_0^t a_{n_k}^{ij}(X_r^k)dr
$$
are martingales with respect to the filtration generated by
$(X^k,M^k)$. 
A monotone class argument shows that, with respect to the usual augmentation $\{\mathcal F'_t\}$ of the filtration generated by $(X,M)$, $M$ is a continuous martingale and
$$
\langle M^i,M^j\rangle_t=2\int_0^t a^{ij}(X_s)ds.
$$
Finally, uniform ellipticity allows us to define
$$
B_t:=\frac{1}{\sqrt{2}}\int_0^t\sigma(X_s)^{-1}dM_s.
$$
Then $\langle B^i,B^j\rangle_t=\delta_{ij}t$, so L\'evy's characterization of Brownian motion implies that $B$ is a $d$-dimensional $\mathcal F'_t$-Brownian motion. Then
$$
M_t=\sqrt{2}\int_0^t\sigma(X_s)dB_s,
$$
which ends the proof.

\hfill\qed

\bigskip

\appendix

\section{D.R.\,Adams' estimates}
\label{adams_sect}

The following are special cases of estimates proved by D.R.\,Adams. Let $V \geq 0$.

\begin{lemma}[{\cite[Theorem 7.3]{A}}] 
\label{adams_est}
Let $0<\alpha<d$, $1<q<\infty$. Let $s>1$. If 
$$
\sup_{x \in \mathbb R^d, r>0} r^{\alpha q}\biggl(\frac{1}{|B_r|}\big\langle V^s \mathbf{1}_{B_r(x)} \big\rangle \biggr)^{\frac{1}{s}}<\infty,
$$
then, for all $\varphi \in \mathcal S$,
$$
\|V^{\frac{1}{q}}\varphi\|_{q} \leq C \|(-\Delta)^{\frac{\alpha}{2}}\varphi\|_q.
$$

\end{lemma}

\begin{lemma}[{\cite{A1}}]
\label{adams_lem}
Let $1 < p < q < \infty$, $p < d$. Then 
$$
\|V^{\frac{1}{q}}\varphi\|_{q} \leq C \|(-\Delta)^{\frac{1}{2}}\varphi\|_p,
$$
if and only if 
\[
\sup_{x \in \mathbb R^d, r > 0} r^{-q\left( \frac{d}{p} - 1 \right)} \big\langle V \mathbf{1}_{B_r(x)} \big\rangle<\infty. 
\]
\end{lemma}

Thus, Morrey class is responsible for the $L^p(\mathbb R^d,dx) \rightarrow L^q(\mathbb R^d,V dx)$, $p<q$, estimates. It is only sufficient for the $L^q(\mathbb R^d,dx) \rightarrow L^q(\mathbb R^d,V dx)$ estimates (already the larger Chang-Wilson-Wolff class shows that Morrey class is not necessary, see Section \ref{classes_sect}). In turn, the $L^q(\mathbb R^d,dx) \rightarrow L^q(\mathbb R^d,V dx)$ estimates can be used e.g.\,to prove weak well-posedness of SDEs via \eqref{frac_repr}, see Section \ref{wfb_sect}.

\bigskip

\section{Multiplicative form-boundedness and Morrey class $M_1$}

\label{m_sect}

\begin{proposition}\label{multiplicative} 
Let $b \in [L^1_{\loc}]^d$. Then
\begin{equation}
\label{mult_ineq6}
\langle |b|\varphi,\varphi\rangle \leq \delta \|\nabla \varphi\|_2\|\varphi\|_2 \quad \forall\,\varphi \in C_c^\infty(\mathbb R^d)
\end{equation}
if and only if
\begin{align}
\label{Morrey}
\langle |b| \mathbf{1}_{B_r(x)} \rangle \leq K r^{d-1}
\end{align}
for some constant $K$ independent of $r>0$ and $x \in \mathbb R^d$, i.e.\,$|b| \in M_1$; then $K$ is proportional to $\delta$.
{In other words, $b \in \mathbf{MF}_\delta$ with $c_\delta=0$ (see Definition \ref{mf_def}) if and only if $|b| \in M_1$.}

\end{proposition}

This result was proved by Mazya \cite[Theorem 1.4.7]{M}. The proof in \cite{M} uses the Besicovitch covering theorem. Below we give a shorter and elementary proof due to Krylov \cite{Kr_m}.

\begin{proof} Let us prove \eqref{Morrey} $\Rightarrow$ \eqref{mult_ineq6}.

Step 1.~
By Adams' Lemma \ref{adams_lem},
$$
\langle |b|,|\varphi|^q\rangle^{\frac{1}{q}} \leq C {K^{\frac{1}{q}}} \|\nabla \varphi\|_2, \quad q=2\frac{d-1}{d-2}>2,
$$
for every $\varphi \in C_c^\infty$.
Hence, by H\"{o}lder's inequality, for $\varphi$ with support in a ball of radius $1$,
\begin{align*}
\langle |b|, \varphi^2 \rangle & \leq \langle |b|\mathbf{1}_{B_1}\rangle^{\frac{q-2}{q}} \langle |b|,|\varphi|^q\rangle^{\frac{2}{q}} \\
& \leq C_0 K \|\nabla \varphi\|^2_2.
\end{align*}
{Choose $\zeta\in C_c^\infty(B_1)$, $0\leq\zeta\leq1$, $\zeta=1$ on $B_{1/2}$, and put $\zeta_x(\cdot):=\zeta(\cdot-x)$. Using this cutoff function, we now obtain, for arbitrary $\varphi\in C_c^\infty$, }
$$
\langle |b|, \varphi^2 \zeta_x^2 \rangle \leq C_1 K \|(\nabla \varphi)\zeta_x\|^2_2 + C_1K \|(\nabla \zeta_x)\varphi\|_2^2.
$$
{We now sum up over all $x\in (2\sqrt d)^{-1}\mathbb Z^d$. The balls $B_{1/2}(x)$ cover $\mathbb R^d$, while the balls $B_1(x)$ have overlap bounded by a constant depending only on $d$. We obtain}
\begin{equation}
\label{w_ineq}
\langle |b|, \varphi^2 \rangle \leq C_2 K\big( \|\nabla \varphi\|^2_2 + \|\varphi\|_2^2\big).
\end{equation}

\smallskip

Step 2.~Put $b_\lambda(x):=b(x/\lambda)$, $\lambda>0$. The Morrey norm $\|\cdot\|_{M_1}$ scales as follows:
$$
\|b_\lambda\|_{M_1}=\lambda \|b\|_{M_1},
$$
so
$$
\langle |b_\lambda|\mathbf{1}_{B_r(x)}\rangle \leq K \lambda r^{d-1}, \quad x \in \mathbb R^d,\;r>0.
$$
Therefore, by \eqref{w_ineq},
\begin{equation}
\label{w_ineq2}
\langle |b_\lambda|, \varphi^2 \rangle \leq C_2 K \lambda \big( \|\nabla \varphi\|^2_2 + \|\varphi\|_2^2\big), \quad \lambda>0.
\end{equation}
We now deduce the sought inequality \eqref{mult_ineq6} from the family of inequalities \eqref{w_ineq2}. Put $\psi(\cdot):=\varphi(\lambda \cdot)$. Then
\begin{align*}
\langle |b|,\psi^2 \rangle & = \lambda^{-d} \langle |b_\lambda|,\varphi^2\rangle \\
& (\text{{use \eqref{w_ineq2}}}) \\
& \leq \lambda^{-d} C_2 K \lambda  \big(\|\nabla \varphi\|^2_2 + \|\varphi\|_2^2 \big) \\
& = C_2 K \big(\lambda^{-1} \|\nabla \psi\|^2_2 +  \lambda {\|\psi\|_2^2}\big).
\end{align*}
It remains to minimize the right-hand side in $\lambda$, i.e.\,take $
\lambda=\frac{\|\nabla \psi\|_2}{\|\psi\|_2},
$
 to obtain \eqref{mult_ineq6} with $\delta=2C_2 K$.

\medskip

The reciprocal \eqref{mult_ineq6} $\Rightarrow$ \eqref{Morrey} is easier. Choosing in \eqref{mult_ineq6} test functions {\( \varphi=\varphi_r \in C_c^\infty(B_{2r}(x)) \), $0\leq\varphi_r\leq1$, $\varphi_r=1$ on $B_r(x)$, such that} \( \| \nabla \varphi \|_2 \leq c r^{\frac{d-2}{2}} \), \( \| \varphi \|_2 \leq c r^{d/2} \), we obtain
$
\langle |b|, \mathbf{1}_{B_r(x)} \rangle \leq {C\delta} r^{d-1}$,
i.e.\,$|b| \in M_1$. This completes the proof.
\end{proof}

\bigskip

\makeatletter
\let\savedtocwrite\@tocwrite
\let\@tocwrite\@gobbletwo
\section*{AI usage }
\let\@tocwrite\savedtocwrite
\makeatother

This paper was written in 2024-2025 without the use of AI. To prepare the present revised version, we used OpenAI ChatGPT Astra rather lightly, to find possible remaining errors, both in the mathematics and the language, after we have taken into account referees' comments. It did find a number of remaining typos, minor errors and inconsistencies, which we have corrected.

\bigskip


\bigskip


\begin{thebibliography}{99}


\bibitem[A1]{A} D.\,R.\,Adams, Weighted nonlinear potential theory, {\em Trans. Amer. Math. Soc.},
297 (1986), 73-94.

\bibitem[A2]{A1}
D.\,R.\,Adams, A trace inequality for generalized potentials. {\em Stud. Math.} 48 (1973) 99-105.


\bibitem[AX]{AX} D.\,R.\,Adams and J.\,Xiao, Elliptic-to-parabolic Morrey spaces--potentials--capacities with applications to certain evolution PDE, {\em J.\,Lond. Math.\,Soc.} 111(4) (2025), e70131.


\bibitem[ABK]{ABK} S.\,Armstrong, A.\,Bou-Rabee and T.\,Kuusi, Superdiffusive central limit theorem for a Brownian particle in a critically-correlated incompressible random drift, {\em Invent. Math.} (2026), published online.

\bibitem[BGL]{BGL}
D.\,Bakry, I.\,Gentil and M.\,Ledoux,
Analysis and Geometry of Markov Diffusion Operators,
Grundlehren der mathematischen Wissenschaften 348,
Springer, 2014.

\bibitem[B]{B} R.\,Bass, Diffusions and Elliptic Operators, {Springer}, 1997.



\bibitem[BC]{BC} R.~Bass and Z.-Q.~Chen, \newblock Brownian motion with singular drift.
\newblock {\em Ann. Probab.} 31 (2003), 791-817.

\bibitem[Bil]{Bil} P.\,Billingsley, Convergence of Probability Measures. Second Edition. {Wiley}, 1999.

\bibitem[BFGM]{BFGM} L.~Beck, F.~Flandoli, M.~Gubinelli and M.~Maurelli, 
\newblock Stochastic ODEs and stochastic linear PDEs with critical
drift: regularity, duality and uniqueness. 
\newblock{\em Electron. J. Probab.}, 24 (2019),  Paper No. 136, 72 pp.


\bibitem[BS]{BS}
A.G. Belyi and Yu.A. Semenov. 
\newblock On the $L^p$-theory of Schr\"{o}dinger semigroups. II.  {\em Siberian Math.~J.}, 31 (1991), 540-549.



\bibitem[BRS]{BRS} V.I.\,Bogachev, M.\,R\"{o}ckner and S.V.\,Shaposhnikov, On the Ambrosio–Figalli–Trevisan superposition principle for probability solutions to Fokker–Planck–Kolmogorov Equations. {\em J. Dynam.
Differential Equations} 33 (2021), p.\,715-739.


\bibitem[BKRS]{BKRS} V.I.\,Bogachev, N.V.\,Krylov, M.\,R\"{o}ckner and S.V.\,Shaposhnikov, Fokker-Planck-Kolmogorov Equations, {\em Amer.\,Math.\,Soc.}, 2015.


\bibitem[BK]{BK} S.\,E.\,Boutiah and D.\,Kinzebulatov, Upper bound on heat kernels of finite particle systems of Keller-Segel type, {\em Preprint}, arXiv:2508.10892v4.

\bibitem[BJW]{BJW} D.\,Bresch, P.-E.\,Jabin and Z.\,Wang, Mean field limit and quantitative estimates with singular attractive kernels, {\em Duke Math. J.} 172 (2023), 2591-2641.


\bibitem[CC]{CC} G. Cannizzaro and K. Chouk. Multidimensional SDEs with singular drift and universal construction of the
polymer measure with white noise potential, {\em Ann. Probab.} 46(3) (2018), 1710-1763.

\bibitem[C]{C} P.\,Cattiaux, Entropy on the path space and application to singular diffusions and mean-field models, {\em Preprint}, arXiv:2404.09552.

\bibitem[CP]{CP} P.\,Cattiaux and L.\,P\'{e}d\`{e}ches, The $2$-$D$ stochastic Keller-Segel particle model: existence and uniqueness, {\em ALEA, Lat. Am. J. Probab. Math. Stat.} 13 (2016),  447-463.

\bibitem[CWW]{CWW} S.Y.A.\;Chang, J.M.\;Wilson and T.H.\;Wolff, Some weighted norm inequalities concerning
the Schr\"{o}dinger operator, \newblock { \em Comment.\;Math.\;Helvetici} 60 (1985), 217-246.


\bibitem[CM]{CM} P.-E.\,Chaudru de Raynal and S.\,Menozzi, On multidimensional stable-driven stochastic differential equations with
Besov drift. {\em Electron. J. Probab.} 27 (2022),  1-52.


\bibitem[CJM]{CJM} P.E.\,Chaudru de Raynal, J.-F.\,Jabir and S.\,Menozzi, Multidimensional stable-driven McKean-Vlasov SDEs with distributional interaction kernel -- a regularization by noise perspective, {\em Stoch. Partial Differ. Equ. Anal. Comput.} 13 (2025), 367-420.

\bibitem[CJM2]{CJM2} P.E.\,Chaudru de Raynal, J.-F.\,Jabir and S.\,Menozzi, Multidimensional stable-driven McKean-Vlasov SDEs with distributional interaction kernel: critical thresholds and related models, {\em Preprint}, arXiv:2302.09900.


\bibitem[CE]{CE} A.\,Cheskidov and T.\,Eguchi, Global-in-time smooth solution to the Navier–Stokes equations for small initial data in frequency localized Koch–Tataru's space, {\em J. Differential Equations} 477 (2026), article no. 114534.

\bibitem[Ch]{Ch} A.\,S.\,Cherny, On the uniqueness in law and the pathwise uniqueness for stochastic
differential equations, {\em Theory Probab.\,Appl.} 46(3) (2002), 406-419.


\bibitem[CF]{CF}  F. Chiarenza and M. Frasca, A remark on a paper by C. Fefferman, {\em Proc. Amer. Math. Soc.} 108 (1990),
407-409. 


\bibitem[CLMS]{CLMS} R. Coifman, P.-L. Lions, Y. Meyer and S. Semmes, Compensated compactness and Hardy spaces,
{\em J. Math. Pures Appl.} 72 (1992) 247-286.


\bibitem[CPZ]{CPZ} L. Corrias, B. Perthame and H. Zaag, 
Global solutions of some chemotaxis
and angiogenesis systems in high
space dimensions, {\em Milan J.\,Math.} 72 (2004), 1-28.

\bibitem[DD]{DD} F.\,Delarue and R.\,Diel, Rough paths and 1d SDE with a time dependent distributional drift: application to
polymers, {\em Probab. Theory Related Fields} 165 (2016), 1-63.


\bibitem[EK]{EKu} S.\,N.\,Ethier and T.G.\,Kurtz, Markov Processes: Characterization and Convergence. {\em Wiley}, 2005.

\bibitem[Fe]{F}  C. Fefferman, The uncertainty principle, {\em Bull. Amer. Math. Soc.} 9 (1983), 129-206.


\bibitem[FMT]{FMT} V.\,Felli, E.\,M.\,Marchini and S.\,Terracini, On Schrödinger operators with multipolar
inverse-square potentials, {\em J.\,Funct.\,Anal.} 250 (2007), 265-316.

\bibitem[Fl]{Fl} F.\,Flandoli, Regularization by additive noise, in Random Perturbation of PDEs and Fluid Dynamic Models, {\em Lecture Notes in Mathematics} 2015, Springer (2011), 17-69. 

\bibitem[FGP]{FGP} F.\,Flandoli, M.\,Gubinelli and E.\,Priola, Well‐posedness of the transport equation by stochastic perturbation, {\em Invent.\,Math.} 180 (2010), 1-53.

\bibitem[FGP2]{FGP2}
F.\,Flandoli, M.\,Gubinelli and E.\,Priola,
Full well-posedness of point vortex dynamics corresponding to
stochastic 2D Euler equations,
{\em Stochastic Process. Appl.} 121 (2011), 1445-1463.

\bibitem[FIR]{FIR} F. Flandoli, E. Issoglio and F. Russo, Multidimensional stochastic differential equations with distributional
drift. {\em Trans. Amer. Math. Soc.} 369 (2017), 1665–1688.

\bibitem[F]{Fo} N.\,Fournier, Stochastic particles for the Keller-Segel equation, {\em CIRM}, 2024.

\bibitem[FJ]{FJ} N.\,Fournier and B.\,Jourdain, Stochastic particle approximation of the Keller-Segel and two-dimensional generalization of Bessel process, {\em Ann.\,Appl. Probab.} 27 (2017), 2807-2861.

\bibitem[FT]{FT} N.\,Fournier and Y.\,Tardy, A simple proof of non-explosion for measure solutions of the Keller-Segel equation, {\em Kinetic and Related Models} 16(2) (2023), 178-186.


\bibitem[G]{G} E.\,Giusti, Direct Methods in the Calculus of Variations, {\em World Scientific}, 2003.

\bibitem[GS]{GS} M.\,Glazkov and T.\,Shilkin, On the $L_1$-stability for parabolic equations
with a supercritical drift term, {\em J. Elliptic Parabol. Equ.} 12 (2026), 1121-1145.

\bibitem[Gr]{Graf_M} L.\,Grafakos, Modern Fourier Analysis. Second Edition. {\em Springer}, 2009.


\bibitem[GP]{GP} L.\,Gr\"{a}fner and N.\,Perkowski, Weak well-posedness of energy solutions to singular SDEs with
supercritical distributional drift, {\em Preprint}, arXiv:2407.09046.

\bibitem[Gru]{Gru} S.\,Grube, Strong solutions to degenerate SDEs and uniqueness for degenerate
Fokker–Planck equations, {\em Stoch. Partial Differ. Equ. Anal. Comput.} 14 (2026), 1362-1392.


\bibitem[HZ]{HZ} Z.\,Hao and X.\,Zhang, SDEs with supercritical distributional drifts, {\em Comm. Math. Phys.} 406 (2025), article no. 250.

\bibitem[H]{H} T.\,Hara, A refined subsolution estimate of weak subsolutions to second order
linear elliptic equations with a singular vector field, {\em Tokyo J.\,Math.} 38(1) (2015), 75-98.




\bibitem[HL]{HL} T.\,Hoffmann-Ostenhof and A.\,Laptev, Hardy inequalities with homogeneous weights, {\em J.\,Funct.\,Anal.} 268 (2015), 3278-3289.

\bibitem[HHLT]{HHLT} M.\,Hoffmann-Ostenhof, T.\,Hoffmann-Ostenhof, A.\,Laptev and J.\,Tidblom, Many-particle Hardy inequalities, {\em J.\,Lond.\,Math.\,Soc.} 77 (2008), 99-114.


\bibitem[JL]{JL} W.\,Jager and S.\,Luckhaus, On explosions of solutions to a system of partial differential equations modelling chemotaxis, {\em Trans.\,Amer.\,Math.\,Soc.} 329 (1992), 819-824.

\bibitem[Ka]{Ka} T.\,Kato, Perturbation Theory for Linear Operators, \newblock {Springer-Verlag Berlin Heidelberg}, 1995.


\bibitem[KSa]{KSa} R.~Kerman and E.~Sawyer, The trace inequality and eigenvalue estimates for Schrödinger operators, {\em Ann. Inst. Fourier} 36(4) (1986), 207-228.


\bibitem[K1]{Ki_super} D.\,Kinzebulatov, A new approach to the $L^p$-theory of $-\Delta + b\cdot\nabla$, and its applications to Feller processes with general drifts,
\newblock {\em Ann.~Sc.~Norm.~Sup.~Pisa (5)} 17 (2017), 685-711.




\bibitem[K2]{Ki_Orlicz} D.\,Kinzebulatov, Laplacian with singular drift in a critical borderline case, {\em Math. Nachr.} 298(2) (2025), 511-526.

\bibitem[K3]{Ki_survey} D.\,Kinzebulatov, Form-boundedness and SDEs with singular drift, INdAM Meeting 2022: Kolmogorov Operators and Their Applications, {\em INdAM Series} 56 (2024), 147-261.

 \bibitem[K4]{Ki_Morrey} D.\,Kinzebulatov, Parabolic equations and SDEs with time-inhomogeneous Morrey drift, {\em NoDEA Nonlinear Differ. Equ. Appl.} 32 (2025), article no. 117.


\bibitem[K5]{Ki_multi} D.\,Kinzebulatov, On particle systems and critical strengths of general singular interactions, {\em Ann. Inst. Henri Poincaré Probab. Stat.} 62(2) (2026), 1191-1251.


\bibitem[K6]{Ki_revisited} D.\,Kinzebulatov, Regularity theory of Kolmogorov operator revisited, {\em Canadian Math. Bull.} 64 (2021), 725-736.

\bibitem[K7]{Ki_measure} D.\,Kinzebulatov, Feller generators with measure-valued drifts, {\em Potential Anal.} 48 (2018), 207-222.

\bibitem[K8]{Ki_intrinsic} D.\,Kinzebulatov, Intrinsic Hardy inequality for fractional Kolmogorov operator, {\em Preprint}, arXiv:2603.23621.


\bibitem[KM1]{KiM_strong} D.\,Kinzebulatov and K.R.\,Madou, Strong solutions of SDEs with singular (form-bounded) drift via Roeckner-Zhao approach, {\em Stoch.\,Dyn.} 24(7) (2024), article no. 2450050.

 \bibitem[KM2]{KiM_JDE} D.\,Kinzebulatov and K.R.\,Madou, Stochastic equations with time-dependent singular drift, {\em J.\,Differential Equations} 337 (2022), 255-293.



\bibitem[KS1]{KiS_brownian} D.\,Kinzebulatov and Yu.A.\,Sem\"{e}nov, Brownian motion with general drift, \newblock{\em Stoch. Proc. Appl.} 130 (2020), 2737-2750.




\bibitem[KS2]{KiS_sharp} D.\,Kinzebulatov and Yu.\,A.\,Sem\"{e}nov, Sharp solvability for singular SDEs, {\em Electron.\,J.\,Probab.} 28 (2023), article no. 69, 1-15. 


\bibitem[KS3]{KiS_theory} D.\,Kinzebulatov and Yu.\,A.\,Sem\"{e}nov, On the theory of the Kolmogorov operator in the spaces $L^p$ and $C_\infty$, {\em Ann. Sc. Norm. Sup. Pisa (5)} 21 (2020), 1573-1647.


\bibitem[KS4]{KiS_MAAN} D.\,Kinzebulatov and Yu.\,A.\,Sem\"{e}nov, Heat kernel bounds for parabolic equations with singular (form-bounded) vector fields, {\em Math.\,Ann.} 384 (2022), 1883-1929.


\bibitem[KS5]{KiS_JDE} D.\,Kinzebulatov and Yu.\,A.\,Sem\"{e}nov, Regularity for parabolic equations with singular non-zero divergence vector fields, {\em J.\,Differential Equations} 381 (2024), 293-339.


\bibitem[KS6]{KiS_feller} D.\,Kinzebulatov and Yu.\,A.\,Sem\"{e}nov, {Feller generators with singular drifts in the critical range}, {\em J. Differential Equations} 433 (2025), article no. 113262.

\bibitem[KS7]{KiS_Osaka} D.\,Kinzebulatov and Yu.\,A.\,Sem\"{e}nov, Feller generators and stochastic differential equations with singular (form-bounded) drift,  {\em Osaka J.\,Math.} 58 (2021), 855-883.


 \bibitem[KS8]{KiS_note} D.\,Kinzebulatov and Yu.A.\,Sem\"{e}nov, Remarks on parabolic Kolmogorov operator, {\em Theory Probab.\,Appl.} 70(2) (2025), 282-300.



\bibitem[KSS]{KiSS_transport} D.\,Kinzebulatov, Yu.\,A.\,Sem\"{e}nov and R.\,Song, Stochastic transport equation with singular drift, {\em Ann.\,Inst.\,Henri Poincar\'{e} (B) Probab. Stat.} 60(1) (2024), 731-752.



\bibitem[KV]{KiV} D.\,Kinzebulatov and R.\,Vafadar, On divergence-free (form-bounded type) drifts, {\em Discrete Contin. Dyn.\,Syst.\,Ser.\,S.} 17 (2024), 2083-2107.




\bibitem[KT]{KT}{H.\,Koch and D.\,Tataru}, Well-posedness for the Navier-Stokes equations, {\em Advances in Mathematics}  157(1) (2001), 22-35. 

\bibitem[KoS]{KS} V.\,F.\,Kovalenko and Yu.\,A.\,Sem\"{e}nov,
{\newblock $C_0$-semigroups in $L^p(\mathbb R^d)$ and $C_\infty(\mathbb R^d)$ spaces generated by differential expression $\Delta+b\cdot\nabla$}, {\em Theory Probab. Appl.} 35 (1990), 443-453. 

\bibitem[KPS]{KPS} V.\,F.\,Kovalenko, M.\,A.\,Perelmuter and Yu.\,A.\,Sem\"{e}nov, Schr\"{o}dinger operators with $L^{l/2}_{w}$($R^{l}$)-potentials, \newblock{\em J.\,Math.\,Phys.} 22  (1981), 1033-1044. 


\bibitem[KrS]{KrSt} J.\,Kristensen and B.\,Stroffolini, The Gehring lemma: dimension free estimates, {\em Nonlinear Analysis} 177(B) (2018), 601-610.



\bibitem[Kr1]{Kr1} N.V. Krylov, Once again on weak solutions of time inhomogeneous It\^{o}'s equations with VMO diffusion and Morrey drift, {\em Electron. J. Probab.} 29, paper 95, 1-19 (2024).



\bibitem[Kr2]{Kr2} N.V. Krylov, On strong solutions of It\^{o}'s equations with $D\sigma$ and $b$ in Morrey classes containing $L^d$, {\em Ann. Probab.} 51(5) (2023) 1729-1751.


\bibitem[Kr3]{Kr3} N.V. Krylov, On strong solutions of time inhomogeneous It\^{o}'s equations with
Morrey diffusion gradient and drift. A supercritical case, {\em Stoch. Proc. Appl.}, 185 (2025) 104619.

\bibitem[Kr4]{Kr3_5} N.V.\,Krylov, On parabolic equations in Morrey spaces with VMO $a$ and Morrey $b$, $c$, {\em NoDEA Nonlinear Differ. Equ. Appl.} 32 (2025), article no. 9.


\bibitem[Kr5]{Kr4}
N.\,V.\,Krylov,
Singular Time-Homogeneous It\^{o} Equations and PDEs,
{\em Mathematical Surveys and Monographs} 298,
American Mathematical Society, Providence, RI, 2026.


\bibitem[Kr6]{Kr_b2}
N.\,V.\,Krylov,
Essentials of Real Analysis and Morrey--Sobolev spaces for second
order elliptic and parabolic PDEs with singular first order coefficients,
{\em J.\,Math.\,Sci.} 294(6) (2025), 802-899.

\bibitem[Kr7]{Kr_m} N.V.\,Krylov, personal communication.


\bibitem[KrR]{KR} N.\,V.\,Krylov and M.\,R\"{o}ckner. 
\newblock Strong solutions of stochastic equations with singular time dependent drift. 
\newblock {\em Probab. Theory Related Fields} 131 (2005), 154-196.


\bibitem[Le]{Le}
C.\,L\'eonard, Girsanov theory under a finite entropy condition,
in {\em S\'eminaire de Probabilit\'es XLIV},
Lecture Notes in Mathematics 2046,
Springer, Heidelberg, 2012, 429-465.




\bibitem[LWZ]{LWZ} J.\,Liang, L.\,Wang and C.\,Zhou, $C^0$-regularity for solutions of elliptic equations with distributional coefficients, {\em J. Differential Equations} 470 (2026), article no. 114389.


\bibitem[LS]{LS}
V.~A.~Liskevich and Yu.~A.~Sem\"{e}nov, 
\newblock Some problems on Markov semigroups,
\newblock{\em 
``Schr\"{o}dinger Operators, Markov Semigroups, Wavelet Analysis, Operator Algebras''
(M. Demuth et al., Eds.), Mathematical Topics: Advances in Partial Differential
Equations, Vol. 11, Akademie Verlag, Berlin} (1996), 163-217.


\bibitem[MK]{MK}
A.\,J.\,Majda and P.\,R.\,Kramer, Simplified models for turbulent diffusion:
theory, numerical modelling, and physical phenomena, {\em Physics Reports} 314 (1999), 237-574. 


\bibitem[M]{M}
V.\,G.\,Mazya, Sobolev Spaces: with Applications to Elliptic Partial
Differential Equations,
2nd edition,
Grundlehren der mathematischen Wissenschaften 342,
Springer, 2011.


\bibitem[MV]{MV}
V.\,G.\,Mazya and I.\,E.\,Verbitsky, Form boundedness of the general
second-order differential operator, {\em Comm.\,Pure Appl.\,Math.} 59 (2006), 1286-1329. 




\bibitem[MV2]{MV2} V.\,G.\,Mazya and I.\,E.\,Verbitsky, Infinitesimal form boundedness and Trudinger's
subordination for the Schr\"{o}dinger operator, {\em Invent. Math.} 162 (2005), 81-136.


\bibitem[NY]{NY} E.\,Nakai and K.\,Yabuta, Pointwise multipliers for functions
of bounded mean oscillation, {\em J.\,Math.\,Soc.\,Japan} 37(2) (1985), 207-218.

\bibitem[N]{N} J.\,Nash, Continuity of solutions of parabolic and elliptic equations, \newblock {\em Amer.\,J.\,Math.} 80(4) (1958), 931-954.


\bibitem[ORT]{ORT} A.\,Ohashi, F.\,Russo, and A.\,Teixeira. SDEs for Bessel processes in low dimension and path-dependent extensions. {\em ALEA} 20 (2023), 1111-1138.



\bibitem[QX]{QX} Z.\,Qian and G.\,Xi, Parabolic equations with singular divergence-free drift vector fields, {\em J. Lond. Math. Soc.} 100(1) (2019), 17-40.

\bibitem[QX2]{QX_supercritical} Z.\,Qian and G.\,Xi, Parabolic equations with divergence-free drift in space $L_t^lL_x^q$, {\em Indiana Univ.\,Math.\,J.} 68(3) (2019), 761-797.

\bibitem[RY]{RevuzYor} D.\,Revuz and M.\,Yor, Continuous Martingales and Brownian Motion, {Springer}, 1991.

\bibitem[RZ]{RZ_weak} M.\,R\"{o}ckner and G.\,Zhao, {SDEs with critical time dependent drifts: weak solutions}, {\em Bernoulli} 29 (2023), 757-784.

\bibitem[RZ2]{RZ_strong} M.\,R\"{o}ckner and G.\,Zhao,  SDEs with critical time dependent drifts: strong solutions, {\em Probab. Theory Related Fields} 192 (2025), 1071–1111.

\bibitem[S]{S} Yu.\,A.\,Sem\"{e}nov, \newblock Regularity theorems for parabolic equations, \newblock {\em J.\,Funct.\,Anal.} 231 (2006), 375-417.


\bibitem[SS\v{S}Z]{SSSZ} G.\,Seregin, L.\,Silvestre, V.\,\v{S}verak and A.\,Zlato\v{s}, On divergence-free drifts, {\em J. Differential Equations} 252(1) (2012), 505-540.



\bibitem[T]{T} Y.\,Tardy, Weak convergence of the empirical measure for the Keller-Segel model in both sub-critical and critical cases, {\em Electron. J. Probab.} 29 (2024), Paper No. 142, 35 pp.


\bibitem[Tr]{Tr} D.\,Trevisan,  Well-posedness of multidimensional diffusion processes with weakly differentiable
coefficients, {\em Electron. J. Probab.} 21 (2016), Paper No. 22, 1-41.

\bibitem[V]{V} A.\,Yu. Veretennikov, Strong solutions and explicit formulas for solutions of stochastic integral equations, {\em Math. USSR-Sbornik} 39 (1981) 387-403.


\bibitem[Za]{Z_Kato} Q.\,S.\,Zhang, Gaussian bounds for the fundamental solutions
of $\nabla(A\nabla u) + B \nabla u - u_t = 0$, {\em Manuscripta Math.} 93 (1997), 381-390.

\bibitem[Za2]{Z_supercritical} Q.\,S.\,Zhang, A strong regularity result for parabolic equations, {\em Comm.\,Math.\,Phys.} 244 (2004), 245-260.

\bibitem[Za3]{Z_supercritical2} Q.\,S.\,Zhang, Local estimates of two linear parabolic equations with singular coefficients, {\em Pacific J.\,Math.} 223 (2006), 367-396.

\bibitem[Z]{Z_a} X.\, Zhang, \newblock Stochastic differential equations with Sobolev coefficients and applications, {\em Ann. Appl. Probab.} 26
 (2016), 2697–2732.


\bibitem[ZZ1]{ZZ} X.\,Zhang and G.\,Zhao, Stochastic Lagrangian paths for Leray solutions of $3D$ Navier-Stokes equations, {\em  Comm.\,Math.\,Phys.}, 381(2) (2021), 491-525.


\bibitem[ZZ2]{ZZ2} X.\,Zhang and G.\,Zhao, Heat kernels and ergodicity of SDEs with distributional drifts.  {\em Preprint}, arXiv:1710.10537, 2017.




\end{thebibliography}
\end{document}